\documentclass[12pt, leqno]{article}
\usepackage{amsmath,amsfonts,amsthm,amscd,amssymb,verbatim} 
\usepackage[all]{xy}

\numberwithin{equation}{section}

\theoremstyle{plain}
\newtheorem{theorem}{Theorem}[section]
\newtheorem{definition}[theorem]{Definition}
\newtheorem{proposition}[theorem]{Proposition}
\newtheorem{lemma}[theorem]{Lemma}
\newtheorem{corollary}[theorem]{Corollary}

\newtheorem{question}[theorem]{Question}

\theoremstyle{definition}
\newtheorem{remark}[theorem]{Remark}
\newtheorem{example}[theorem]{Example}

\renewcommand{\b}{\bullet}
\newcommand{\beast}{\begin{eqnarray*}}
\newcommand{\east}{\end{eqnarray*}}

\newcommand{\N}{{\Bbb N}}
\newcommand{\Z}{{\Bbb Z}}
\renewcommand{\P}{{\Bbb P}}

\newcommand{\Q}{{\Bbb Q}}
\newcommand{\R}{{\Bbb R}}
\newcommand{\C}{{\Bbb C}}

\newcommand{\G}{{\Bbb G}}
\newcommand{\fG}{\wh{{\Bbb G}}}

\newcommand{\F}{{\Bbb F}}

\newcommand{\Af}{{\Bbb A}}
\newcommand{\fAf}{\wh{{\Bbb A}}}

\newcommand{\ba}{{\bold a}}
\newcommand{\bb}{{\bold b}}

\newcommand{\Spec}{{\mathrm{Spec}}\,}

\newcommand{\Spf}{{\mathrm{Spf}}\,}
\newcommand{\Spm}{{\mathrm{Spm}}\,}

\newcommand{\lra}{\longrightarrow}
\newcommand{\ra}{\rightarrow}
\newcommand{\hra}{\hookrightarrow}
\newcommand{\hla}{\hookleftarrow}

\newcommand{\lla}{\longleftarrow}

\newcommand{\ti}[1]{\widetilde{#1}}
\newcommand{\wt}[1]{\widetilde{#1}}

\newcommand{\ol}[1]{\overline{#1}}
\newcommand{\os}{\overset}

\newcommand{\et}{{\mathrm{et}}}
\newcommand{\liset}{{\mathrm{lis}\text{-}\mathrm{et}}}

\newcommand{\conv}{{\mathrm{conv}}}

\newcommand{\Hom}{{\mathrm{Hom}}}

\newcommand{\Ker}{{\mathrm{Ker}}}
\newcommand{\im}{{\mathrm{Im}}}
\newcommand{\Coker}{{\mathrm{Coker}}}

\newcommand{\Aut}{{\mathrm{Aut}}}

\newcommand{\sat}{{\mathrm{sat}}}

\newcommand{\Sch}{{\mathrm{Sch}}}
\newcommand{\LSch}{{\mathrm{LSch}}}
\newcommand{\Sta}{{\mathrm{Sta}}}
\newcommand{\LSta}{{\mathrm{LSta}}}

\newcommand{\id}{{\mathrm{id}}}

\newcommand{\Rep}{{\mathrm{Rep}}}
\newcommand{\End}{{\mathrm{End}}}

\newcommand{\res}{{\mathrm{res}}}

\newcommand{\Coh}{{\mathrm{Coh}}}

\newcommand{\red}{{\mathrm{red}}}

\newcommand{\dlog}{{\mathrm{dlog}}\,}

\newcommand{\an}{{\mathrm{an}}}

\newcommand{\Frac}{{\mathrm{Frac}}\,}

\newcommand{\cC}{{\cal C}}

\newcommand{\cE}{{\cal E}}
\newcommand{\cF}{{\cal F}}
\newcommand{\cG}{{\cal G}}

\newcommand{\cO}{{\cal O}}

\newcommand{\cS}{{\cal S}}
\newcommand{\cT}{{\cal T}}
\newcommand{\cU}{{\cal U}}
\newcommand{\cV}{{\cal V}}
\newcommand{\cX}{{\cal X}}
\newcommand{\cY}{{\cal Y}}
\newcommand{\cZ}{{\cal Z}}
\newcommand{\fm}{\mathfrak{m}}

\newcommand{\fD}{{\frak D}}

\newcommand{\fU}{{\frak U}}

\newcommand{\fX}{{\mathfrak{X}}}
\newcommand{\fY}{{\mathfrak{Y}}}

\newcommand{\NM}{{\mathrm{NM}}}
\newcommand{\LNM}{{\mathrm{LNM}}}
\newcommand{\ULNM}{{\mathrm{ULNM}}}
\newcommand{\SLNM}{{\mathrm{SLNM}}}

\newcommand{\NID}{{\mathrm{(NID)}}}
\newcommand{\NLD}{{\mathrm{(NLD)}}}

\newcommand{\NRD}{{\mathrm{(NRD)}}}
\newcommand{\SNLD}{{\mathrm{(SNLD)}}}
\renewcommand{\c}{\circ}
\newcommand{\Sm}{{\mathrm{Sm}}}

\newcommand{\gss}{{\mathrm{gss}}}
\newcommand{\sing}{{\mathrm{sing}}}

\newcommand{\rk}{{\mathrm{rk}}\,}

\newcommand{\0}{{\bold 0}}

\newcommand{\vv}{{\bold v}}

\newcommand{\wh}{\widehat}
\renewcommand{\wt}{\widetilde}

\renewcommand{\res}{{\mathrm{res}}}
\renewcommand{\d}{\dagger}

\newcommand{\lam}{\lambda}

\newcommand{\pa}{\partial}

\renewcommand{\0}{{\bold 0}}
\renewcommand{\1}{{\bold 1}}

\renewcommand{\End}{{\mathrm{End}}}

\newcommand{\Isoc}{{\mathrm{Isoc}}}
\newcommand{\Isocd}{\Isoc^{\dagger}}
\newcommand{\FIsoc}{F\text{-}\Isoc}
\newcommand{\FIsocd}{F\text{-}\Isocd}

\newcommand{\fin}{{\mathrm{fin}}}
\newcommand{\codim}{{\mathrm{codim}}}
\newcommand{\Repfin}{{\mathrm{Rep}}^{\mathrm{fin}}}
\newcommand{\sm}{\mathrm{sm}}
\newcommand{\Isocl}{\Isoc^{\log}}
\newcommand{\FIsocl}{\FIsoc^{\log}}
\newcommand{\Par}{{\mathrm{Par}}}
\newcommand{\PIsoc}{\Par\text{-}\Isocl}
\newcommand{\PFIsoc}{\Par\text{-}F\text{-}\Isocl}
\newcommand{\Vect}{{\mathrm{Vect}}}
\newcommand{\PVect}{\Par\text{-}{\mathrm{Vect}}}
\newcommand{\FVect}{F\text{-}{\mathrm{Vect}}}
\newcommand{\PFVect}{\Par\text{-}F\text{-}{\mathrm{Vect}}}
\newcommand{\FLatt}{F\text{-}{\mathrm{Latt}}}
\newcommand{\PFLatt}{\Par\text{-}F\text{-}{\mathrm{Latt}}}
\newcommand{\bX}{{\bold X}}

\newcommand{\bZ}{{\bold Z}}
\newcommand{\Sigmass}{\Sigma\text{-}{\mathrm{ss}}}
\newcommand{\olSigmass}{\ol{\Sigma}\text{-}{\mathrm{ss}}}

\renewcommand{\ss}{\text{-}{\mathrm{ss}}}

\begin{document}
\title{Parabolic log convergent isocrystals}
\author{Atsushi Shiho
\footnote{
Graduate School of Mathematical Sciences, 
University of Tokyo, 3-8-1 Komaba, Meguro-ku, Tokyo 153-8914, JAPAN. 
E-mail address: shiho@ms.u-tokyo.ac.jp \, 
Mathematics Subject Classification (2000): 12H25, 14F35.}}
\date{}
\maketitle

\begin{abstract}
In this paper, we introduce the notion of parabolic 
log convergent isocrystals on smooth varieties endowed with 
a simple normal crossing divisor, 
which is a kind of $p$-adic analogue of the notion of parabolic 
bundles on smooth varieties defined by Seshadri, Maruyama-Yokogawa, 
Iyer-Simpson, Borne. We prove that the equivalence between 
the category of $p$-adic representations of the fundamental group 
and the category of unit-root convergent $F$-isocrystals (proven by 
Crew) induces the equivalence between the category of 
 $p$-adic representations 
of the tame fundamental group and the category of 
semisimply adjusted parabolic unit-root log convergent $F$-isocrystals. 
We also prove equivalences which relate 
categories of log convergent isocrystals on certain fine log algebraic 
stacks with some conditions and 
categories of adjusted parabolic log convergent isocrystals with 
some conditions. 
We also give an interpretation of unit-rootness in terms of 
the generic semistability with slope $0$. 
Our result can be regarded as a $p$-adic analogue of some results of 
Seshadri, Mehta-Seshadri, Iyer-Simpson and Borne. 
\end{abstract}

\tableofcontents

\section*{Introduction}
For a proper smooth curve $X$ over $\C$ of genus $\geq 2$, 
Narasimhan-Seshadri \cite{ns} proved the equivalence 
\begin{equation}\label{eq0}
\left( 
\begin{aligned}
& \text{irreducible unitary} \\ 
& \text{representation of $\pi_1(X^{\an})$} 
\end{aligned}
\right) 
\os{=}{\lra} 
\left( 
\begin{aligned}
& \text{vector bundle on $X$} \\ 
& \text{stable of degree $0$} 
\end{aligned}
\right). 
\end{equation}
(Here, for a scheme $S$ of finite type over $\C$, 
$S^{\an}$ denotes the analytic space associated to $S$.) 
In the case where a given smooth curve $X$ over $\C$ is open with 
compactification $\ol{X}$, there are two approaches to prove an 
analogue of the above result: 
In \cite{seshadri}, Seshadri took a `stacky' approach and proved 
the equivalence 
\begin{equation}\label{eq1}
\left( 
\begin{aligned}
& \text{irreducible unitary} \\ 
& \text{representation of $\pi_1(X^{\an})$} 
\end{aligned}
\right) 
\os{=}{\lra} 
\varinjlim_{Y \ra X \in \cG_X}
\left( 
\begin{aligned}
& \text{vector bundle on $[\ol{Y}/G_Y]$} \\ 
& \text{stable of degree $0$} 
\end{aligned}
\right), 
\end{equation}
where $\cG_X$ denotes the category of finite etale Galois covering 
of $X$, $\ol{Y}$ denotes the smooth compactification of $Y$, 
$G_Y := \Aut (Y/X)$ and  $[\ol{Y}/G_Y]$ denotes the quotient 
stack. (Note that a vector bundle on  $[\ol{Y}/G_Y]$ is nothing but 
a $G_Y$-equivariant vector bundle on $\ol{Y}$.) On the other hand, 
in \cite{ms}, Mehta-Seshadri took a `parabolic' approach: 
They defined the notion of parabolic vector bundles on $(\ol{X},Z)$ 
(where $Z := \ol{X} \setminus X$) and proved the following equivalence: 
\begin{equation}\label{eq2}
\left( 
\begin{aligned}
& \text{irreducible unitary} \\ 
& \text{representation of $\pi_1(X^{\an})$} 
\end{aligned}
\right) 
\os{=}{\lra} 
\left( 
\begin{aligned}
& \text{parabolic vector} \\ 
& \text{bundle on $(\ol{X},Z)$} \\ 
& \text{parabolic stable} \\
& \text{of parabolic degree $0$} 
\end{aligned}
\right). 
\end{equation}
Note that the notion of parabolic vector bundles and the moduli of them 
are studied also by Maruyama-Yokogawa \cite{my}, 
including the higher-dimensional
 case. \par 
Let us recall that there exists another `stacky' interpretation: 
For an open immersion $X \hra \ol{X}$ of smooth varieties over $\C$ 
such that $Z := \ol{X} \setminus X$ is a simple normal crossing divisor, 
Iyer-Simpson \cite{is} and Borne \cite{borne1} \cite{borne2} 
introduced the notion of 
`stack of roots' $(\ol{X},Z)^{1/n}$ ($n \in \N$) and 
established the equivalence 
\begin{equation}\label{eq3}
\left( 
\begin{aligned}
& \text{parabolic vector} \\ 
& \text{bundle on $(\ol{X},Z)$} 
\end{aligned}
\right) 
\os{=}{\lra} 
\varinjlim_{n} 
\left( 
\begin{aligned}
& \text{vector bundle} \\ 
& \text{on $(\ol{X},Z)^{1/n}$} 
\end{aligned}
\right)
\end{equation}
such that the parabolic degree on the left hand side coincides with 
the degree on the right hand side. (There is also a related work by 
Biswas \cite{biswas}.) 
So, when $X$ is a curve, the equivalences \eqref{eq2} and 
\eqref{eq3} imply the equivalence 
\begin{equation}\label{eq4}
\left( 
\begin{aligned}
& \text{irreducible unitary} \\ 
& \text{representation of $\pi_1(X^{\an})$} 
\end{aligned}
\right) 
\os{=}{\lra} 
\varinjlim_n 
\left( 
\begin{aligned}
& \text{vector bundle on $(\ol{X},Z)^{1/n}$} \\ 
& \text{stable of degree $0$} 
\end{aligned}
\right). 
\end{equation}

A generalization of \eqref{eq0} to higher-dimensional case and 
non-unitary case (where Higgs bundle appears) 
is established by many people including Donaldson \cite{donaldson}, 
Mehta-Ramanathan \cite{mr}, Uhlenbeck-Yau \cite{uy}, 
Corlette \cite{corlette} and Simpson \cite{simpson}. As for 
open case, a generalization of \eqref{eq2} to higher-dimensional 
non-unitary case (where parabolic Higgs bundle appears) 
is given also by many people including Simpson \cite{simpsonopen}, 
Jost-Zuo \cite{jz} and Mochizuki 
\cite{mochizuki1} \cite{mochizuki1.5} \cite{mochizuki2}. \par 
Now let us turn to the $p$-adic situation. Let $p$ be a prime, let 
$q$ be a fixed power of $p$ and let 
$K$ be a complete discrete 
valuation field of characteristic zero with ring of integers $O_K$ and 
perfect residue field $k$ of characteristic $p>0$ containing $\F_q$. 
Assume moreover that we have an endomorphism $\sigma: K \lra K$ 
which respects $O_K$ and which lifts the $q$-th power map on $k$. 
Let $K^{\sigma}$ be the fixed field of $\sigma$. Then, for a connected 
smooth $k$-variety $X$ (not necessarily proper), Crew \cite{crew} 
proved the equivalence 
\begin{equation}\label{eq5}
G: \Rep_{K^{\sigma}}(\pi_1(X)) \os{=}{\lra} \FIsoc(X)^{\circ} 
\end{equation}
between the category $\Rep_{K^{\sigma}}(\pi_1(X))$ of finite dimensional 
continuous representation of the algebraic fundamental group 
$\pi_1(X)$ of $X$ over $K^{\sigma}$ and 
the category $\FIsoc(X)^{\circ}$ of unit-root convergent $F$-isocrystals 
on $X$ over $K$. When $X$ is proper smooth, we regard 
\eqref{eq5} as an analogue of \eqref{eq0}. (When $X$ is not proper, 
the categories in \eqref{eq5} are considered 
to be too big.) \par 
Now let 
$X \hra \ol{X}$ be an open immersion of smooth $k$-varieties such that 
$Z := \ol{X} \setminus X$ is a simple normal crossing divisor. 
Weng \cite{weng} 
raised a question on the construction of the $p$-adic analogues of 
\eqref{eq1} and \eqref{eq2} for $(X,\ol{X})$ (at least for curves), starting 
from the equivalence \eqref{eq5}. In this paper, we will prove 
several equivalences which can be regarded as $p$-adic analogues 
of \eqref{eq1}, \eqref{eq2} and \eqref{eq4} (so we think it answers 
the question of Weng in some sense). \par 
Let us explain our main results 
more precisely. Let $X,\ol{X}$ be as above and let 
$\ol{X} \setminus X =: Z = \bigcup_{i=1}^r Z_i$ be the decomposition 
of $Z$ into irreducible components. For $1 \leq i \leq r$, 
let $v_i$ be the discrete valuation of $k(X)$ corresponding to the 
generic point of $Z_i$, let $k(X)_{v_i}$ be the completion of $k(X)$ with 
respect to $v_i$ and let $I_{v_i}$ be the inertia group of 
$k(X)_{v_i}$. (Then we have homomorphisms $I_{v_i} \lra \pi_1(X)$ which are 
well-defined up to conjugate.) Let us define 
$\Repfin_{K^{\sigma}}(\pi_1(X))$ by 
$$ \Repfin_{K^{\sigma}}(\pi_1(X)) := 
\{\rho \in \Rep_{K^{\sigma}}(\pi_1(X)) \,|\, 
\forall i, \rho|_{I_{v_i}} \text{ has finite image}\}. $$
Then first we prove the equivalence 
\begin{equation}\label{seq1}
\Repfin_{K^{\sigma}}(\pi_1(X)) \os{=}{\lra} 
\varinjlim_{Y\ra X \in \cG_X} \FIsoc ([\ol{Y}^{\sm}/G_Y])^{\circ}, 
\end{equation}
where $\cG_X$ is the category of finite etale Galois covering 
of $X$, $\ol{Y}$ is the normalization of $\ol{X}$ in $k(Y)$, 
$\ol{Y}^{\sm}$ is the smooth locus of $\ol{Y}$, 
$G_Y := \Aut (Y/X)$, $[\ol{Y}^{\sm}/G_Y]$ is the quotient 
stack and the right hand side is the limit of the category of unit-root 
convergent $F$-isocrystals on stacks $[\ol{Y}^{\sm}/G_Y]$
 (which we will define in Section 2). This is a $p$-adic analogue of 
\eqref{eq1}. The above equivalence induces the 
equivalence 
\begin{equation}\label{seq2}
\Rep_{K^{\sigma}}(\pi_1^t(X)) \os{=}{\lra} 
\varinjlim_{Y\ra X \in \cG_X^t} \FIsoc ([\ol{Y}^{\sm}/G_Y])^{\circ}, 
\end{equation}
where $\pi_1^t(X)$ is the tame fundamental group of $X$ (tamely ramified at 
the valuations $v_i \,(1 \leq i \leq r)$), 
$\cG_X^t$ is the category of finite etale Galois covering 
of $X$ tamely ramified at $v_i \,(1 \leq i \leq r)$ and the other notations 
are the same as before. Next, we prove the equivalence 
\begin{equation}\label{seq3}
\Rep_{K^{\sigma}}(\pi_1^t(X)) \os{=}{\lra} 
\varinjlim_{(n,p)=1} \FIsoc ((\ol{X},Z)^{1/n})^{\circ}, 
\end{equation}
where the right hand side is the limit of the category of unit-root 
convergent $F$-isocrystals on stacks of roots $(\ol{X},Z)^{1/n}$. 
This is a $p$-adic analogue of \eqref{eq4}. Also, we introduce 
the category $\PFIsoc(\ol{X},Z)^{\circ}_{\0\ss}$ 
of semisimply adjusted parabolic unit-root log convergent $F$-isocrystals 
on $(\ol{X},Z)$ and prove the equivalence 
\begin{equation}\label{seq4}
\Rep_{K^{\sigma}}(\pi_1^t(X)) \os{=}{\lra} 
\PFIsoc (\ol{X},Z)^{\circ}_{\0\ss}, 
\end{equation}
which is a $p$-adic analogue of \eqref{eq2}. The key ingredients of 
the proof are results of Tsuzuki in \cite{tsuzuki} and results of 
the author in \cite{sigma} and \cite{purity}. We also discuss 
the relations among the variants (without Frobenius structure, 
with log structure and with exponent condition) of the categories 
on the right hand side of \eqref{seq2}, \eqref{seq3} and \eqref{seq4}. \par 
In the case where $\ol{X}$ is liftable to a smooth formal scheme 
$\ol{\cX}_{\circ}$ over $\Spf W(k)$ 
together with a suitable lift of Frobenius endomorphism 
$F_{\circ}:\ol{\cX}_{\circ} \lra \ol{\cX}_{\circ}$, 
the equivalence \eqref{eq5} of Crew 
factors through an equivalence of Katz (\cite{katzcrew}, see also \cite{crew})
$$ G: \Rep_{O_K^{\sigma}}(\pi_1(\ol{X})) \os{=}{\lra} 
\FLatt(\ol{\cX})^{\c} $$ 
between the category $\Rep_{O_K^{\sigma}}(\pi_1(\ol{X}))$ 
of 
continuous representations of $\pi_1(\ol{X})$ to free 
$O_K^{\sigma}$-modules of finite 
rank (here $O_K^{\sigma} := K^{\sigma} \cap O_K$) 
and 
the category $\FLatt(\ol{\cX})^{\c}$ of unit-root $F$-lattices 
on $\ol{\cX} := \ol{\cX}_{\circ} \otimes_{W(k)} O_K$. 
We also prove in the paper 
that, when $(\ol{X},Z)$ lifts to 
a smooth formal scheme 
$(\ol{\cX}_{\circ},\cZ_{\circ})$ over $\Spf W(k)$ 
endowed with a relative simple normal crossing divisor 
together with a lift of Frobenius endomorphism 
$F_{\circ}: (\ol{\cX}_{\circ},\cZ_{\circ}) \lra 
(\ol{\cX}_{\circ},\cZ_{\circ})$ (endomorphism as log formal schemes), 
there exist equivalences of the form 
\begin{equation}\label{seq2-}
\Rep_{O_K^{\sigma}}(\pi_1^t(X)) \os{=}{\lra} 
\varinjlim_{Y\ra X \in \cG_X^t} \FLatt ([\ol{\cY}^{\sm}/G_Y])^{\c}, 
\end{equation}
\begin{equation}\label{seq3-}
\Rep_{O_K^{\sigma}}(\pi_1^t(X)) \os{=}{\lra} 
\varinjlim_{(n,p)=1} \FLatt ((\ol{\cX},\cZ)^{1/n})^{\c} 
\end{equation}
(where $(\ol{\cX},\cZ) := (\ol{\cX}_{\circ}, \cZ_{\circ}) \otimes_{W(k)} O_K$, 
$\ol{\cY}^{\sm}$ is a certain lift of the smooth locus $\ol{Y}^{\sm}$ 
of the normalization $\ol{Y}$ of $\ol{X}$ in $k(Y)$, 
the left hand sides are the category 
of continuous representations of $\pi_1^t(X)$ to free 
$O_K^{\sigma}$-modules of finite rank 
and 
the right hand sides are the limits of the categories of 
unit-root $F$-lattices on the ind-stacks 
$[\ol{\cY}^{\sm}/G_Y]$, $(\ol{\cX},\cZ)^{1/n}$, respectively)
which can be proven in the same way as \eqref{seq2} and \eqref{seq3}. 
Moreover, we will introduce the category 
$\PFLatt(\ol{\cX},\cZ)^{\c}$ of 
of locally abelian 
parabolic unit-root $F$-lattices on $(\ol{\cX},\cZ)$ and 
prove the equivalence 
\begin{equation}\label{seq4-}
\Rep_{O_K^{\sigma}}(\pi_1^t(X)) \os{=}{\lra} 
\PFLatt (\ol{\cX},\cZ)^{\c}. 
\end{equation}

Note that, 
in the $p$-adic equivalences we have explained above, the notion of 
`the stability of degree $0$' does not appear, which appears in 
 \eqref{eq1}, \eqref{eq2}, \eqref{eq4}. To see the $p$-adic analogue 
 of this notion more clearly, we introduce the notion of 
generic semistability (gss) and the invariant $\mu$ 
for objects in the category $\FIsoc([\ol{Y}^{\sm}/G_Y])$ 
(resp. $\FIsoc((\ol{X},Z)^{1/n}), \PFIsoc(\ol{X},Z)_{\0}$) of 
convergent $F$-isocrystals on $[\ol{Y}^{\sm}/G_Y]$ (resp. 
convergent $F$-isocrystals on $(\ol{X},Z)^{1/n}$, 
adjusted parabolic log convergent $F$-isocrystals on 
$(\ol{X},Z)$) (where the notations are as in \eqref{seq1}, \eqref{seq3} and 
\eqref{seq4}) and rewrite the equivalences 
\eqref{seq1}, 
\eqref{seq2}, \eqref{seq3} and 
\eqref{seq4} as 
\begin{align}
& \Rep_{K^{\sigma}}^{\fin}(\pi_1(X)) \os{=}{\lra} 
\varinjlim_{Y \ra X \in \cG_X} 
\FIsoc([\ol{Y}^{\sm}/G_Y])^{\gss,\mu=0}, \label{sseq1} \\ 
& \Rep_{K^{\sigma}}(\pi_1^t(X)) \os{=}{\lra} 
\varinjlim_{Y \ra X \in \cG^t_X} 
\FIsoc([\ol{Y}^{\sm}/G_Y])^{\gss,\mu=0}, \label{sseq2} \\ 
& \Rep_{K^{\sigma}}(\pi_1^t(X)) \os{=}{\lra} 
\varinjlim_{(n,p)=1} 
\FIsoc((\ol{X},Z)^{1/n})^{\gss,\mu=0}, \label{sseq3} \\ 
& \Rep_{K^{\sigma}}(\pi_1^t(X)) \os{=}{\lra} 
\PFIsoc(\ol{X},Z)^{\gss,\mu=0}_{\0}, \label{sseq4} 
\end{align}
where ${}^{\gss,\mu=0}$ means the subcategory consisting 
of generically semistable objects with $\mu=0$. The proof is 
an easy application 
of some results of Katz \cite{katzslope} and Crew \cite{crewsp}, 
\cite{crew}. Also, 
 we introduce the notion of 
generic semistability (gss) and the invariant $\mu$ 
for objects in the $\Q$-linearization $\FLatt([\ol{\cY}^{\sm}/G_Y])_{\Q}$ 
(resp. $\FLatt((\ol{\cX},\cZ)^{1/n})_{\Q}, \PFLatt(\ol{\cX},\cZ)$) 
of the category of 
$F$-lattices on $[\ol{\cY}^{\sm}/G_Y]$ (resp. 
$F$-lattices on $(\ol{\cX},\cZ)^{1/n}$, 
locally abelian parabolic $F$-lattices on 
$(\ol{\cX},\cZ)$) (where the notations are as in 
\eqref{seq2-}, \eqref{seq3-} and 
\eqref{seq4-}) and rewrite the $\Q$-linearization of the equivalences 
\eqref{seq2-}, \eqref{seq3-} and 
\eqref{seq4-} as 
\begin{align}
& \Rep_{K^{\sigma}}(\pi_1^t(X)) \os{=}{\lra} 
\varinjlim_{Y \ra X \in \cG^t_X} 
\FLatt([\ol{\cY}^{\sm}/G_Y])_{\Q}^{\gss,\mu=0}, \label{sseq2-} \\ 
& \Rep_{K^{\sigma}}(\pi_1^t(X)) \os{=}{\lra} 
\varinjlim_{(n,p)=1} 
\FLatt((\ol{\cX},\cZ)^{1/n})_{\Q}^{\gss,\mu=0}, \label{sseq3-} \\ 
& \Rep_{K^{\sigma}}(\pi_1^t(X)) \os{=}{\lra} 
\PFLatt(\ol{\cX},\cZ)_{\Q}^{\gss,\mu=0}. \label{sseq4-} 
\end{align}

The equivalences \eqref{sseq1}--\eqref{sseq4-} might be better 
$p$-adic analogues of \eqref{eq1}, \eqref{eq2}, \eqref{eq4}, 
but they are not good in the following point: 
In the equivalences 
\eqref{sseq1}, \eqref{sseq2}, \eqref{sseq3} and \eqref{sseq4}, 
an object in the category on the right hand side contains 
an isocrystal structure ($=$ $p$-adic version of 
connection structure) unlike the equivalences \eqref{eq1}, 
\eqref{eq2}, \eqref{eq4}. In the equivalences 
\eqref{sseq2-}, \eqref{sseq3-} and \eqref{sseq4-}, 
an object in the category on the right hand side contains 
a lattice structure ($=$ $p$-adic version of 
metric) unlike the equivalences \eqref{eq1}, 
\eqref{eq2}, \eqref{eq4}. To overcome this, we introduce the category 
$\FVect([\ol{\cY}^{\sm}/G_Y]_K)$ (resp. 
$\FVect((\ol{\cX},\cZ)^{1/n}_K)$, 
$\PFVect((\ol{\cX},\cZ)_K)$) of `$F$-vector bundles on 
rigid analytic stack $[\ol{\cY}^{\sm}/G_Y]_K$') 
(resp. `$F$-vector bundles on 
rigid analytic stack $(\ol{\cX},\cZ)^{1/n}_K$', 
`locally abelian parabolic $F$-vector bundles on 
log rigid analytic space $(\ol{\cX},\cZ)_K$') and the notion of 
 generic semistablity and the invariant $\mu$ for objects in it. 
(Attention: We do not develop the general theory of rigid analytic stacks 
nor log rigid spaces. We only define the above categories.) 
An object in these categories does not contain an information on 
isocrystals nor lattices. Then, in the case of curves, we can rewrite 
the equivalences \eqref{sseq2-}, \eqref{sseq3-} and \eqref{sseq4-} 
further to obtain the equivalences 
\begin{align}
& \Rep_{K^{\sigma}}(\pi_1^t(X)) \os{=}{\lra} 
\varinjlim_{Y \ra X \in \cG^t_X} 
\FVect([\ol{\cY}/G_Y]_K)^{\gss,\mu=0}, \label{sseq2+} \\ 
& \Rep_{K^{\sigma}}(\pi_1^t(X)) \os{=}{\lra} 
\varinjlim_{(n,p)=1} 
\FVect((\ol{\cX},\cZ)^{1/n}_K)^{\gss,\mu=0}, \label{sseq3+} \\ 
& \Rep_{K^{\sigma}}(\pi_1^t(X)) \os{=}{\lra} 
\PFVect((\ol{\cX},\cZ)_K)^{\gss,\mu=0}, \label{sseq4+} 
\end{align}
which will be further better $p$-adic analogues of 
\eqref{eq1}, \eqref{eq2}, \eqref{eq4}. These 
equivalences are essentially conjectured by Weng \cite{weng} as a 
micro reciprocity law in log rigid analytic geometry. \par 
In the case of $p$-torsion coefficient, Ogus-Vologodsky \cite{ov}
and Gros-Le Stum-Quir\'os \cite{glsq} prove 
the Simpson correspondence between the category of 
integrable connections and Higgs bundles, and the logarithmic version of it 
is proved by Schepler \cite{schepler}. 
Our results are different from theirs because 
we treat $p$-adic coefficient. On the other hand, we have to say that 
our results are not fully developed in 
the sense that the Higgs bundle does not appear in our equivalence. 
We expect that certain generalization of the results of 
Schepler (to the $p$-adic 
coefficient case) is related to certain generalization (to the 
Higgs case) of our result. \par 
The content of each section is as follows: 
In the first section, we review the definition and some results concerning 
certain properties on log-$\nabla$-modules and isocrystals which we developed 
in \cite{sigma}. Note that we also add some results which were not proved 
there but useful in this paper. 
In the second section, we give a definition of the category of 
(log) convergent isocrystals on (fine log) algebraic stacks and prove 
the equivalences \eqref{seq1}, \eqref{seq2} and \eqref{seq3}. 
We also relate the variants of right hand sides of 
\eqref{seq2} and \eqref{seq3} without Frobenius structures, with log 
strucutes and with exponent conditions in the case of curves. In the third 
section, we introduce the category 
of semisimply adjusted parabolic unit-root log convergent $F$-isocrystals 
and prove the equivalence \eqref{seq4}. In the course of the proof, 
we prove the equivalence of the variants of right hand sides of 
\eqref{seq3} and \eqref{seq4} without Frobenius structures, with log 
structures and with exponent conditions. In the fourth 
section, we work in the lifted situation and prove the equiavelences 
\eqref{seq2-}. \eqref{seq3-} and \eqref{seq4-}. We also prove 
a comparison result between the category of vector bundles on 
certain ind-stacks 
and the category of parabolic vector bundles on formal schemes 
which is a formal version of the results of 
Iyer-Simpson \cite{is} and Borne \cite{borne1} \cite{borne2}. 
In the fifth section, we introduce the notion of generic 
semistablity and the invariant $\mu$ for objects in several 
categories and prove the equivalences 
\eqref{sseq2}--\eqref{sseq4+}, using results of 
Katz \cite{katzslope} and Crew \cite{crewsp}, 
\cite{crew}. \par 
The author would like to thank to Professor Lin Weng 
for useful discussion and for sending 
the author the preprint version of the paper
 \cite{weng}, which made the author to consider the 
topics in this paper. 
The author is partly supported by Grant-in-Aid for Young Scientists (B) 
21740003 (representative: Atsushi Shiho) from the Ministry of Education, 
Culture, Sports, Science and Technology, Japan 
and 
Grant-in-Aid for Scientific Research (B) 22340001 
(representative: Nobuo Tsuzuki) from 
Japan Society for the Promotion of Science. 

\section*{Convention}
Throughout this paper, $p$ is a fixed prime number and 
$q$ is a fixed power of $p$. $K$ is a complete discrete valuation 
field of caracteristic zero with ring of integers $O_K$ and 
perfect residue field $k$ containing $\F_q$. 
The maximal ideal of $O_K$ is denoted by $\fm_K$. 
We fix a valuation $|\cdot|: K \lra \R_{\geq 0}$ induced by the 
discrete valuation on $K$ and let us put $\Gamma^* := \sqrt{|K^{\times}|} \cup 
\{0\} \subseteq \R_{\geq 0}$. 
Assume moreover that there exists an endomorphism $\sigma:K \lra K$ 
inducing the endomorphism $O_K \lra O_K$ (denoted also by 
$\sigma$) which lifts the $q$-th power map on $k$. 
Let $K^{\sigma}$ be the fixed field of $\sigma$ and let 
$O_K^{\sigma} := K^{\sigma} \cap O_K$. \par 
The category of schemes separated of finite type over $k$ is 
denoted by $\Sch$. Following \cite{kedlayaI}, a variety over $k$ 
(or a $k$-variety) means an object in $\Sch$ which is reduced. 
For $X \in \Sch$ with $X$ connected, 
let $\Rep_{K^{\sigma}}(\pi_1(X))$ 
be the category of 
finite dimensional continuous representations of the fundamental 
group $\pi_1(X)$ of $X$ over $K^{\sigma}$ and 
let $\Rep_{O_K^{\sigma}}(\pi_1(X))$ 
be the category of 
continuous representations of the fundamental 
group $\pi_1(X)$ of $X$ to free $O_K^{\sigma}$-modules of 
finite rank. For $X \in \Sch$, 
let $\Sm_{K^{\sigma}}(X)$ be the category of smooth $K^{\sigma}$-sheaves 
on $X_{\et}$ and let 
$\Sm_{O_K^{\sigma}}(X)$ be the category of smooth $O_K^{\sigma}$-sheaves 
on $X_{\et}$. 
We have the well-known equivalences 
$\Rep_{K^{\sigma}}(\pi_1(X)) \cong \Sm_{K^{\sigma}}(X), 
\Rep_{O_K^{\sigma}}(\pi_1(X)) \cong \Sm_{O_K^{\sigma}}(X)$ for 
$X \in \Sch$ with $X$ connected. For a $p$-adic formal 
scheme $\cX$ separated of finite type over $\Spf O_K$, we define 
the categories $\Sm_{K^{\sigma}}(\cX)$, $\Sm_{O_K^{\sigma}}(\cX)$ 
in the same way. \par 
For $X \in \Sch$, we denote the category of 
convergent isocrystals (resp. convergent $F$-isocrystals, 
unit-root convergent $F$-isocrystals) on $X$ over $K$ by 
$\Isoc(X)$ (resp. $\FIsoc(X), \FIsoc(X)^{\circ}$). 
(For precise definition and basic properties, see 
\cite{ogus}, \cite{berthelot}, \cite{lestum} and \cite{crew}. 
For the definition of $\FIsoc(X)$ and $\FIsoc(X)^{\c}$, 
we follow the definition in 
\cite{crew}, not that in \cite{ogus}.) 
For an open immersion $X \hra \ol{X}$ in $\Sch$, we denote 
the category of 
overconvergent isocrystals (resp. overconvergent $F$-isocrystals, 
unit-root overconvergent $F$-isocrystals) on $(X,\ol{X})$ 
over $K$ by 
$\Isocd(X,\ol{X})$ (resp. $\FIsocd(X,\ol{X}), \FIsocd(X,\ol{X})^{\circ}$). 
(For precise definition and basic properties, see 
\cite{berthelot}, \cite{lestum}, \cite{crew} and \cite{kedlayaI}.) 
Let $\LSch$ be the category of fine log schemes separated of 
finite type over $k$. 
For $(X,M_X) \in \LSch$,  
the category of locally free log convergent isocrystals on 
$(X,M_X)$ over $K$ (called locally free isocrystals on 
the log convergent site $((X,M_X)/\Spf O_K)_{\conv}$ in 
\cite{sigma}, \cite{relativeI}) by $\Isocl(X,M_X)$. 
The category of locally free log convergent $F$-isocrystals on 
$(X,M_X)$ over $K$ (that is, the category of pairs 
$(\cE,\Psi)$ consisting of $\cE \in \Isocl(X,M_X)$ and 
an isomorphism $\Psi: F^*\cE \os{=}{\lra} \cE$, where 
$F$ is the $\sigma$-linear endofunctor on $\Isocl(X,M_X)$ 
induced by $q$-th power map on $(X,M_X)$ and $\sigma$) by 
$\FIsocl(X,M_X)$. (For precise definition and basic properties, 
see \cite{kedlayaI}, \cite{relativeI}. See also \cite{crysI}, 
\cite{crysII}.) \par 
For a functor $\Phi: \cC \lra \Sch$, we define the category 
$\Sm_{K^{\sigma}}(\Phi)$ of smooth $K^{\sigma}$-sheaves on $\Phi$ 
(resp.  the category 
$\Sm_{O_K^{\sigma}}(\Phi)$ of smooth $O_K^{\sigma}$-sheaves on $\Phi$) 
as the 
category of pairs $$(\{\cE_Y\}_{Y \in {\rm Ob}(\cC)}, 
\{\varphi_{\cE}\}_{\varphi:Y \ra Y' \in 
{\rm Mor}(\cC)}),$$ where $\cE_Y \in \Sm_{K^{\sigma}}(Y)$ 
(resp. $\cE_Y \in \Sm_{O_K^{\sigma}}(Y)$) 
and $\varphi_{\cE}$ is an isomorphism 
$\Phi(\varphi)^*\cE_{Y'} \os{=}{\lra} \cE_Y$ in $\Sm_{K^{\sigma}}(Y)$ 
(resp. $\Sm_{O_K^{\sigma}}(Y)$) satisfying the cocycle condition 
$\varphi_{\cE} \circ \Phi(\varphi)^* \varphi'_{\cE} = 
(\varphi' \circ \varphi)_{\cE}$ for $Y \os{\varphi}{\lra} Y' 
\os{\varphi'}{\lra} Y''$ in $\cC$. 
For a functor $\Phi: \cC \lra \Sch$, we define the category 
$\Isoc(\Phi)$ of convergent isocrystals on $\Phi$ over $K$ 
as the 
category of pairs $$(\{\cE_Y\}_{Y \in {\rm Ob}(\cC)}, 
\{\varphi_{\cE}\}_{\varphi:Y \ra Y' \in 
{\rm Mor}(\cC)}),$$ where $\cE_Y \in \Isoc(\Phi(Y))$ 
and $\varphi_{\cE}$ is an isomorphism 
$\Phi(\varphi)^*\cE_{Y'} \os{=}{\lra} \cE_Y$ in $\Isoc(\Phi(Y))$ 
satisfying the cocycle condition 
$\varphi_{\cE} \circ \Phi(\varphi)^* \varphi'_{\cE} = 
(\varphi' \circ \varphi)_{\cE}$ for $Y \os{\varphi}{\lra} Y' 
\os{\varphi'}{\lra} Y''$ in $\cC$. 
We can also define the category $\FIsoc(\Phi)$ of 
convergent $F$-isocrystals on $\Phi$ over $K$ and the category 
$\FIsoc(\Phi)^{\circ}$ of unit-root convergent $F$-isocrystals 
on $\Phi$ over $K$ in the same way. Similarly, for a functor 
$\Phi: \cC \lra \text{(open immersions in $\Sch$)}$, 
we can define the category 
$\Isocd(\Phi)$ of overconvergent isocrystals on $\Phi$ over $K$, 
 the category 
$\FIsocd(\Phi)$ of overconvergent $F$-isocrystals on $\Phi$ over $K$ and  
the category 
$\FIsocd(\Phi)^{\circ}$ of unit-root 
overconvergent $F$-isocrystals on $\Phi$ over $K$. 
Also, for a functor 
$\Phi: \cC \lra \LSch$, 
we can define the category 
$\Isocl(\Phi)$ of locally free 
log convergent isocrystals on $\Phi$ over $K$ and 
the category 
$\FIsocl(\Phi)$ of locally free 
log convergent $F$-isocrystals on $\Phi$ over $K$. \par 
Note that a diagram of schemes $X_{\b}$ 
in $\Sch$ can be regarded as a functor $\cC \lra \Sch$ as above. 
So we can define the categories $\Sm_{K^{\sigma}}(X_{\b}), 
\Sm_{O_K^{\sigma}}(X_{\b}), 
\Isoc(X_{\b}), \allowbreak \FIsoc(X_{\b}), 
\FIsoc(X_{\b})^{\circ}$ in the above way. 
Also, for a diagram $X_{\b} \hra \ol{X}_{\b}$ of open immersions in 
$\Sch$, we can define the categories $\Isocd(X_{\b},\ol{X}_{\b}), 
\FIsocd(X_{\b},\ol{X}_{\b}), \allowbreak 
\FIsocd(X_{\b},\ol{X}_{\b})^{\circ}$ 
and for a diagram $(X_{\b},M_{X_{\b}})$ in 
$\LSch$, we can define the categories $\Isocl(X_{\b},M_{X_{\b}}), 
\FIsocl(X_{\b},M_{X_{\b}})$. \par 
For a $p$-adic formal scheme $\cX$ separated 
of finite type over $\Spf O_K$, we denote the associated rigid 
space over $K$ by $\cX_K$. For a fine log structure $M$ on 
a (formal) scheme $X$, we put $\ol{M} := M/\cO^{\times}_X$. 
A morphism of log (formal) schemes $f:(X,M_X) \lra (Y,M_Y)$ is called 
strict when $f^*M_Y = M_X$. When $X$ 
is a smooth scheme over $O_K/\fm_K^a$ for 
some $a$ or a $p$-adic formal scheme smooth over $\Spf O_K$ and $Z$ is a 
relative simple normal crossing divisor on $X$, we denote by $(X,Z)$ 
the fine log (formal) scheme with underlying (formal) scheme $X$ whose 
log structure is associated to $Z$. \par 
For subsets $\Sigma, \Sigma'$ of the form 
$\Sigma = \prod_{i=1}^r \Sigma_i, \Sigma' = \prod_{i=1}^r \Sigma'_i$ in 
$\Z_p^r$ and $n \in \N, a = (a_i)_{1 \leq i \leq r} \in \N^r$, 
we define $\Sigma + \Sigma', n\Sigma, a\Sigma$ by 
\begin{align*}
& \Sigma + \Sigma' := \prod_{i=1}^r \{\xi_i + \xi'_i \,|\, \xi_i \in \Sigma_i, 
\xi'_i \in \Sigma'_i\}, \\ 
& n\Sigma := \prod_{i=1}^r \{n\xi_i \,|\, \xi_i \in \Sigma_i\}, \qquad 
a\Sigma := \prod_{i=1}^r \{a_i\xi_i \,|\, \xi_i \in \Sigma_i\}. 
\end{align*}
Also, the set $\{0\}^r$ in $\Z_p^r$ is denoted simply by $\0$. \par 
Finally, a discrete valuation always means a discrete valuation of 
rank one.

\section{Log-$\nabla$-modules and isocrystals} 

In this section, we review the definition and some results concerning 
certain properties on log-$\nabla$-modules and isocrystals which we developed 
in \cite{sigma} (which generalizes some results in \cite{kedlayaI}).
We also add some more terminologies and 
results which were not treated there but useful in this paper. \par 

\subsection{Log-$\nabla$-modules}

Let $L$ be a field containing $K$ complete with respect to a 
multiplicative norm (also denoted by $| \cdot |$) which extends 
the given absolute value on $K$. 
For a morphism $f: \fX \lra \fY$ of 
rigid spaces over $L$, a $\nabla$-module of on $\fX$ relative to $\fY$ 
is defined to be a pair $(E,\nabla)$ consisting of a coherent module $E$ on 
$\fX$ endowed with an integrable $f^{-1}\cO_{\fY}$-linear connection 
$\nabla: E \lra E \otimes_{\cO_{\fX}} \Omega^1_{\fX/\fY}$. 
In the case $\fY = \Spm L$, we omit the term `relative to $\fY$'. 
We denote the category of $\nabla$-modules on $\fX$ by $\NM_{\fX}$. \par 
For a morphism $f: \fX \lra \fY$ of rigid spaces over $L$ and elements 
$x_1,...,x_r$ in $\Gamma(\fX,\cO_{\fX})$, a log-$\nabla$-module 
 on $\fX$ with respect to $x_1,...,x_r$ relative to $\fY$ 
is defined to be a pair $(E,\nabla)$ consisting of a locally free 
module of finite rank $E$ on 
$\fX$ endowed with an integrable $f^{-1}\cO_{\fY}$-linear log connection 
$\nabla: E \lra E \otimes_{\cO_{\fX}} \omega^1_{\fX/\fY}$. (Here 
$\omega^1_{\fX/\fY}$ is defined by 
\begin{equation}\label{deflogdiff}
\omega^1_{\fX/\fY} := (\Omega^1_{\fX/\fY} \oplus \bigoplus_{i=1}^r 
\cO_{\fX} \cdot 
\dlog x_i)/N,
\end{equation}
where $N$ is the sheaf locally generated by 
$(dx_i,0)-(0,x_i\dlog x_i)$ $(1 \leq i \leq r)$.) 
In the case $\fY = \Spm L$, we omit the term `relative to $\fY$'. 
We denote the category of 
log-$\nabla$-modules on $\fX$ with respect to $x_1,...,x_r$ 
by $\LNM_{\fX}$. 

\begin{remark}\label{rem1.1}
When $\cX$ is a $p$-adic formal scheme smooth separated of finite type 
over $\Spf O_K$ and $\cZ = \bigcup_{i=1}^r \cZ_i$ 
is a relative simple normal crossing divisor on $\cX$ 
(with each $\cZ_i$ irreducible), we have sections 
$x_1, ..., x_r \in \Gamma(\cX,\cO_{\cX})$ which cut out $\cZ_1,...,\cZ_r$ 
Zariski locally on $\cX$. So, locally on $\cX$, 
we can define the notion of log-$\nabla$-modules on 
$\cX_K$ with respect to $x_1,...,x_r$ (relative to $\Spm K$). 
In this case, we see that this definition is independent of 
the choice of $x_i$'s above, because $\omega^1_{\cX_K/\Spm K}$ 
defined in \eqref{deflogdiff} is nothing but the coherent sheaf on 
$\cX_K$ induced by the log differential module $\Omega^1_{\cX}(\log \cZ)$ 
on the formal scheme $\cX$. So, in this case, we can define 
the notion of log-$\nabla$-module on $(\cX_K, \cZ_K)$ globally, by 
patching the above local definition of log-$\nabla$-modules on 
$\cX$ with respect to $x_1,...,x_r$. We denote the category 
of log-$\nabla$-modules on $(\cX_K, \cZ_K)$ by $\LNM_{(\cX_K,\cZ_K)}$ 
and also by $\LNM_{\cX_K}$ when there will be no confusion on $\cZ_K$. 
\end{remark}

Next, let $\fX$ be a smooth rigid space over $L$ endowed with 
sections $x_1,...,x_r \in \Gamma(\fX,\cO_{\fX})$ whose zero loci are 
affinoid, smooth and meet transversally. Let us put 
$\fD_i := \{x_i = 0\}$ and $M_i:={\rm Im}(\Omega^1_{\fX/K} \oplus 
\bigoplus_{j\not=i}\cO_{\fX}\,\dlog x_j \lra 
\omega^1_{\fX})$. Then the composite map 
$$ 
E \os{\nabla}{\lra} E \otimes_{\cO_{\fX}} \omega^1_{\fX} \lra 
E \otimes_{\cO_{\fX}} (\omega^1_{\fX}/M_i) \cong 
E |_{\fD_i} \dlog x_i \cong E |_{\fD_i}
$$ 
naturally induces an element $\res_i$ in 
$\End_{\cO_{\fD_i}}(E |_{\fD_i})$, which we call the 
residue of $(E,\nabla)$ along $\fD_i$. By \cite[1.24]{curvecut}, 
we can take the minimal monic polynomial $P_i(x) \in K[x]$ satisfying 
$P_i(\res_i)=0$. We call the roots of $P_i(x)$ the exponents of 
$(E,\nabla)$ along $\fD_i$. For $\fX, x_1,...,x_r, \fD_1,...,\fD_r$ 
as above and 
$\Sigma := \prod_{i=1}^r \Sigma_i \subseteq \Z_p^r$, 
we denote the category of log-$\nabla$-modules 
on $\fX$ with respect to $x_1,...,x_r$ whose exponents along $\fD_i$ 
are contained in $\Sigma_i\,(1 \leq i \leq r)$ by $\LNM_{\fX,\Sigma}$. 
In the situation of Remark \ref{rem1.1}, the category 
$\LNM_{(\cX_K,\cZ_K),\Sigma}$ is defined by patching this definition. 
\par 
We call an interval $I$ in $[0,\infty)$ aligned if any endpoint 
of $I$ at which it is closed is contained in 
$\Gamma^*$. We call an interval 
$I$ in $[0,\infty)$ quasi-open if it is open at non-zero endpoints. 
For an aligned interval $I$, we define the rigid space $A_L^n(I)$ by 
$A_L^n(I) := \{(t_1,...,t_n) \in 
\Af^{n,\an}_L \,|\, \forall i, |t_i| \in I\}$. \par 
Following \cite[3.2.4]{kedlayaI}, we use the following convention: 
For a smooth 
affinoid rigid space $\fX$ over $L$ 
we put 
$\omega^1_{\fX \times A^n_L[0,0]} := 
\Omega^1_{\fX} \oplus \bigoplus_{i=1}^n\cO_{\fX} \dlog t_i$, where 
$\dlog t_i$ is the free generator `corresponding to 
the $i$-th coordinate of $A^n_L[0,0]$'. Using this, we can define the 
notion of a log-$\nabla$-module $(E,\nabla)$ 
on $\fX \times A_L^n[0,0]$ with respect to 
$t_1, ..., t_n$ and the notion of the residue, the exponents 
of $(E,\nabla)$ along $\{t_i=0\}$ in natural way: 
To give a log-$\nabla$-module 
on $\fX \times A_L^n[0,0]$ with respect to 
$t_1, ..., t_n$ is equivalent to give a 
$\nabla$-module $(E,\nabla)$
on $\fX$ and commuting endomorphisms $\pa_i := t_i 
\dfrac{\pa}{\pa t_i}$ of $(E,\nabla)$ 
$(1 \leq i \leq n)$. Also, we can define the 
category $\LNM_{\fX \times A^n_L[0,0],\Sigma}$ for $\Sigma = 
\prod_{i=1}^{n}\Sigma_i \subseteq \Z_p^{n}$ as above: 
A log-$\nabla$-module 
on $\fX \times A_L^n[0,0]$ with respect to 
$t_1, ..., t_n$, regarded as a $\nabla$-module $(E,\nabla)$ 
on $\fX$ endowed with commuting endomorphisms $\pa_i := t_i 
\dfrac{\pa}{\pa t_i}$ $(1 \leq i \leq n)$, is in the category 
$\LNM_{\fX \times A^n_L[0,0],\Sigma}$ if and only if 
all the eigenvalues of $\pa_i$ are 
in $\Sigma_{i}$ $(1 \leq i \leq n)$. \par 
For an aligned interval $I \subseteq [0,\infty)$ and 
$\xi := (\xi_1, ..., \xi_n) \in \Z_p^n$, we define the log-$\nabla$-module 
$(M_{\xi},\nabla_{M_{\xi}})$ on $A^n_L(I)$ with respect to 
$t_1, ..., t_n$ (which are the coordinates) 
as the log-$\nabla$-module 
$(\cO_{A_L^n(I)},d + \sum_{i=1}^n \xi_i \dlog t_i)$. 
Following \cite[1.3]{sigma} (cf. \cite[3.2.5]{kedlayaI}), we define the notion of $\Sigma$-constance and 
$\Sigma$-unipotence of log-$\nabla$-modules as follows (we also introduce 
the notion of $\Sigma$-semisimplicity):

\begin{definition}\label{unipdef}
Let $\fX$ be a smooth rigid space over $L$. 
Let $I \subseteq [0,\infty)$ 
be an aligned interval and 
fix $\Sigma := \prod_{i=1}^{n}\Sigma_i \subseteq 
\Z_p^{n}$. \\
$(1)$ \,\,\, 
An object $(E,\nabla)$ in 
$\LNM_{\fX \times A_L^n(I),\Sigma}$ 
$(\fX \times A_L^n(I)$ is endowed with $t_1, ..., t_n$, 
where $t_i$'s are the 
coordinates in $A_L^n(I))$ 
is called $\Sigma$-constant if 
$(E,\nabla)$ has the form 
$\pi_1^*(F,\nabla_F) \otimes \pi_2^*(M_{\xi},\nabla_{M_{\xi}})$ for some 
$\nabla$-module $(F,\nabla_F)$ on $\fX$ and $\xi \in \Sigma$, where 
$\pi_1: \fX \times A_L^n(I) \lra \fX$, $\pi_2:\fX \times A_L^n(I) 
\lra A_L^n(I)$ 
denote the projections. An object in 
$\LNM_{X \times A_L^n(I),\Sigma}$ is called $\Sigma$-semisimple if 
it is a direct sum of $\Sigma$-constant ones. \\
$(2)$ \,\,\, 
An object $(E,\nabla)$ in 
$\LNM_{\fX \times A_L^n(I),\Sigma}$ is called $\Sigma$-unipotent if 
$(E,\nabla)$ admits a filtration 
$$ 0 = E_0 \subset E_1 \subset \cdots \subset E_m=E $$ 
by sub log-$\nabla$-modules whose successive quotients are 
$\Sigma$-constant log-$\nabla$-modules. 

We denote the category of $\Sigma$-semisimple $(\Sigma$-unipotent$)$ 
log-$\nabla$-modules on 
$\fX \times A_L^n(I)$ with respect to $t_1, ..., t_n$ by 
$\SLNM_{\fX \times A^n_L(I),\Sigma}\, (\ULNM_{\fX \times A^n_L(I),\Sigma})$. 
Note that we have $\SLNM_{\fX \times A^n_L(I),\Sigma} \subseteq 
\ULNM_{\fX \times A^n_L(I),\Sigma}$. 
\end{definition}

Here we give a remark which is the same as \cite[1.4]{sigma}: 
When $I$ does not 
contain $0$, the log-$\nabla$-modules 
$M_{\xi}$ and $M_{\xi'}$ $(\xi,\xi' \in \Sigma)$ are isomorphic if 
$\xi-\xi'$ is contained in $\Z^n$. So we see that 
the notion of $\Sigma$-semisimplicity and $\Sigma$-unipotence
only depends on the image 
$\ol{\Sigma}$ of $\Sigma$ in $\Z_p^n/\Z^n$ in the following sense: 
An object $(E,\nabla)$ in 
$\LNM_{\fX \times A_L^n(I),\Sigma}$ is 
$\Sigma$-semisimple ($\Sigma$-unipotent) if and only if 
it is $\tau(\ol{\Sigma})$-semisimple 
($\tau(\ol{\Sigma})$-unipotent) for some (or any) section 
$\tau: \Z_p^n/\Z^n \lra \Z_p^n$ of the form $\tau = \prod_{i=1}^n \tau_i$ 
of the canonical projection 
$\Z_p^n \lra \Z_p^n/\Z^n$. So, in this case, we will say also that 
$(E,\nabla)$ is $\ol{\Sigma}$-semisimple 
($\ol{\Sigma}$-unipotent), by abuse of terminology. \par 
Before giving properties on $\Sigma$-semisimple and $\Sigma$-unipotent 
log-$\nabla$-modules, we recall (and introduce) several terminologies on 
subsets in $\Z_p^r$. 
(Some of them are used not in this subsection but in later sections.) 
Recall that an element $\alpha$ in $\Z_p$ is 
called $p$-adically non-Liouville if we have the equalities 
$$ \varliminf_{n\to\infty} |\alpha + n |^{1/n} = 
\varliminf_{n\to\infty} |\alpha - n |^{1/n} = 1. $$

\begin{definition}\label{nl}
\begin{enumerate}
\item 
A subset 
$\Sigma$ in $\Z_p$ is $\NID$ $($resp. $\NRD)$ if, 
for any $\alpha,\beta \in \Sigma$, $\alpha - \beta$ is not a 
non-zero integer $($resp. $\alpha - \beta$ is not a 
non-zero rational number$)$. 
\item 
A subset 
$\Sigma$ in $\Z_p$ is $\NLD$ $($resp. $\SNLD)$ if, 
for any $\alpha,\beta \in \Sigma$, $\alpha - \beta$ is a 
$p$-adically non-Liouville number 
$($resp. for any $\alpha,\beta \in \Sigma$ and $a \in \Z_{(p)}$, 
$\alpha - \beta + a$ is a $p$-adically non-Liouville number$)$. 
\item 
A subset $\Sigma$ in $\Z_p^r$ of the form $\Sigma = \prod_{i=1}^r \Sigma_i$ 
is $\NID$ $($resp. $\NRD, \NLD, \allowbreak \SNLD)$ 
if so is each $\Sigma_i \,(1 \leq i \leq r)$. 
\item 
A subset $\ol{\Sigma}$ in $\Z_p^r/\Z^r$ $($resp. $\Z_p^r/\Z_{(p)}^r)$ 
of the form 
$\ol{\Sigma} = \prod_{i=1}^r \ol{\Sigma}_i$ is 
$\NLD$ if, for any section $\tau: \Z_p^r/\Z^r \lra \Z_p^r$ 
$($resp. $\tau: \Z_p^r/\Z_{(p)}^r \lra \Z_p^r)$ 
of the natural projection $\Z_p^r \lra \Z_p^r/\Z^r$ 
$($resp. $\Z_p^r \lra \Z_p^r/\Z_{(p)}^r)$ 
of the form $\tau = \prod_{i=1}^r \tau_i$, $\tau(\ol{\Sigma})$ is 
$\NLD$. $($Note that, when this condition is satisfied, 
$\tau(\ol{\Sigma})$ is $\NID$ and $\NLD$ $($resp. 
$\NRD$ and $\SNLD)$. 
Note also that, if $\Sigma = \prod_{i=1}^r \Sigma_i \subseteq \Z_p^r$ is $\NID$ and $\NLD$ $($resp. $\NRD$ and $\SNLD)$, the image $\ol{\Sigma}$ of 
$\Sigma$ in $(\Z_p/Z)^r$ $($resp. $\Z_p^r/\Z_{(p)}^r)$ is $\NLD$ and we have 
$\Sigma = \tau(\ol{\Sigma})$ for some section $\tau$ as above.$)$ 
\end{enumerate}
\end{definition}

We have the following property, which we use in the later sections.

\begin{lemma}\label{lem1.4}
Let $\Sigma := \prod_{i=1}^r \Sigma_i$ be a subset in $\Z_p^r$ which is 
$\NRD$ $($resp. $\SNLD)$. Then, for any $m \in \N$ prime to $p$, 
$m\Sigma$ $($see Convention for definition$)$ is $\NID$ $($resp. $\NLD)$. 
\end{lemma}

\begin{proof}
We only prove that $m\Sigma$ is $\NLD$ when $\Sigma$ is $\SNLD$, because 
the other assertion is easy. We may assume that $r=1$. Take any 
$\alpha, \beta \in \Sigma$ and put $\xi := m(\alpha - \beta)$. 
Then, for any $n \in \N$ with $n = mq + r \,(q,r \in \N, 0 \leq r < m)$, 
we have 
\begin{align*}
1 & \geq |\xi \pm n |^{1/n} = |m|^{1/n} 
\left|
\left(\alpha - \beta \pm \frac{r}{m} \right) \pm q 
\right|^{\frac{1}{q}\cdot \frac{q}{n}} \\ 
& \geq 
\left|
\left(\alpha - \beta \pm \frac{r}{m}\right) \pm q 
\right|^{\frac{1}{q} \cdot 
\frac{1}{m+1}} \\ 
& \geq 
\min_{a \in (-1,1) \cap \frac{1}{m}\Z} 
(|(\alpha - \beta + a) \pm q |^{\frac{1}{q} \cdot \frac{1}{m+1}})  
\end{align*}
and the limit inferior of the right hand side as $n\to\infty$ is $1$ 
since $(\alpha - \beta + a)$'s are $p$-adically non-Liouville by assumption.  
So we have $\varliminf_{n\to\infty} |\xi \pm n |^{1/n} = 1$ and so 
$m\Sigma$ is $\NLD$. 
\end{proof}

Now let us recall the following proposition, which is partly 
proven in \cite[1.13, 1.15]{sigma}: 

\begin{proposition}\label{homexteq}
Let $\Sigma = \prod_{i=1}^n\Sigma_i$ be a subset of $\Z_p^n$ which is 
$\NID$ and $\NLD$. Then, for a smooth rigid space $\fX$ over $L$ and 
a quasi-open interval $I \subseteq [0,\infty)$, 
 we have the equivalences of categories 
\begin{align*}
& \cU_I: \LNM_{\fX \times A_L^n[0,0],\Sigma} = 
\ULNM_{\fX \times A_L^n[0,0],\Sigma} \os{=}{\lra} 
\ULNM_{\fX \times A_L^n(I),\Sigma}, \\ 
&  \cU_I: \SLNM_{\fX \times A_L^n[0,0],\Sigma} \os{=}{\lra} 
\SLNM_{\fX \times A_L^n(I),\Sigma} 
\end{align*}
defined as follows$:$ An object $((E,\nabla),\{\partial_i\})$ in 
$\ULNM_{\fX \times A_L^n[0,0],\Sigma}$ or 
$\SLNM_{\fX \times A_L^n[0,0],\Sigma}$ 
$($so $(E,\nabla)$ is a $\nabla$-module on $\fX)$ 
is sent to 
$\pi^*E$ $($where $\pi$ is the projection $\fX \times A_L^n(I) \lra \fX)$ 
endowed with the log connection 
$$ \vv 
\mapsto \pi^*\nabla(\vv) + \sum_{i=1}^n \pi^*(\pa_i)(\vv) \dlog t_i. $$
\end{proposition}

\begin{proof} 
The first equivalence is a special case of \cite[1.15]{sigma}. 
The essential surjectivity of the second functor follows from 
the definition of $\Sigma$-semisimplicity and the full faithfulness 
of the second functor follows from that of the first functor. 
\end{proof}

We also prove certain full faithfulness 
property of log-$\nabla$-modules on 
relative annulus, which is a variant of \cite[1.14]{sigma}. 

\begin{proposition}\label{uniqueinj0}
Let $\fX$ be a smooth rigid space over $L$, let $\lambda$ be an 
element in $(0,1) \cap \Gamma^*$ and let 
$\Sigma_1 := \prod_{j=1}^n \Sigma_{1j}, \Sigma_2 := \prod_{j=1}^n\Sigma_{2j}$ 
be subsets of $\Z_p^n$ such that for any $j$ and for any 
$\xi_1 \in \Sigma_{1j}, \xi_2 \in \Sigma_{2j}$, we have 
$\xi_1 - \xi_2 \in \Z_p \setminus \Z_{<0}$. For $i=1,2$, let 
$E_i := (E_i,\nabla_i)$ be a $\Sigma_i$-unipotent 
log-$\nabla$-module on $\fX \times A_L^n[0,1)$ with respect to 
$t_1,...,t_n$ $($where $t_1,...,t_n$ are the canonical coordinate of 
$A_L^n[0,1))$ and let us put $E'_i := E_i |_{\fX \times A_K^n[\lambda,1)}$, 
which is a $\Sigma_i$-unipotent 
log-$\nabla$-module on $\fX \times A_L^n[\lambda,1)$. Then the 
restriction functor induces the isomorphism 
$$ \Hom (E_1,E_2) \os{=}{\lra} \Hom (E'_1,E'_2), $$
where $\Hom$ means the set of homomorphisms as log-$\nabla$-modules. 
\end{proposition}

\begin{proof}
The proof is analogous to that of \cite[1.14]{sigma}. 
First, by Proposition \ref{homexteq}, there exists a 
$\Sigma_i$-unipotent log-$\nabla$-module $F_i \,(i=1,2)$ on 
$\fX \times A_L^n[0,0]$ with respect to $t_1,...,t_n$ (which are 
`the coodinate of $A_L^n[0,0]$') such that $E_i = \cU_{[0,1)}(F_i)$, and 
we have $\Hom(F_1,F_2) \os{=}{\lra} \Hom(E_1,E_2)$, 
$E'_i = \cU_{[\lambda,1)}(F_i)$. So it suffices to prove that 
the functor $\cU_{[\lambda,1)}$ induces the isomorphism 
$$ \Hom (F_1,F_2) \os{=}{\lra} \Hom (\cU_{[\lambda,1)}(F_1), 
\cU_{[\lambda,1)}(F_2)). $$ 
In the following, for a log-$\nabla$-module $N$ on $\fX \times A_L^n(J)$ 
($J = [0,0]$ or $[\lambda,1)$) with respect to $t_1,...,t_n$, 
we denote the associated log-$\nabla$-module 
on $\fX \times A_L^n(J)$ with respect to $t_1,...,t_n$ relative to $\fX$ 
by $\ol{N}$. \par 
Let us put 
$F = F_1^{\vee}\otimes F_2$
(as log-$\nabla$-module). Then it suffices to prove that the map 
$$ 
H^a(\fX, 
F \otimes \omega^{\b}_{\fX \times A^n_L[0,0]}) \lra 
H^a(\fX \times A_L^n[\lambda,1), 
\cU_{[\lambda,1)}(F) \otimes \omega^{\b}_{\fX \times A^n_L[\lambda,1)}) $$
is an isomorphism for $a=0$ and injective for $a=1$. 
We may assume that $\fX$ is affinoid by considering the 
spectral sequence induced by admissible hypercovering by affinoids. 
By the same technique and the five lemma, 
we may assume that $\fX$ is affinoid and that $F$ has the form 
$F_0 \otimes \pi^*M_{\xi}$ 
(where $\pi$ is `the projection $\fX \times A^n_L[0,0] \lra A^n_L[0,0]$') 
for some 
$\xi := \xi_2 - \xi_1 \, (\xi_i \in \Sigma_i)$ and 
for some $\nabla$-module 
$F_0$ on $\fX$ with $F_0$ free as $\cO_{\fX}$-module. \par 
By considering the Katz-Oda type spectral sequence for the diagram 
$\fX \times A^n_L(J) \allowbreak \lra \fX \lra \Spm K$ for $J = 
[0,0]$ or $[\lambda,1)$, we obtain the spectral sequence 
\begin{align*}
E_2^{a,b}& = H^a(\Gamma(\fX,\omega^{\b}_{\fX/L}) \otimes 
H^b(\fX\times A^n_L(J), \ol{F_J}
\otimes\omega^{\b}_{\fX \times A^n_L(J)/\fX})) \\ 
& \,\Longrightarrow\, 
H^{a+b}(\fX\times A^n_L(J), F_J \otimes\omega^{\b}_{\fX\times A^n_L(J)/L}), 
\end{align*}
where $F_J = F$ (resp. $\cU_{[\lambda,1)}(F)$) when $J=[0,0]$ (resp. 
$J=[\lambda, 1)$). 
From this, we see that it suffices to prove the map 
\begin{equation}\label{rest}
H^a(\fX, \ol{F} \otimes 
\omega^{\b}_{\fX \times A^n_L[0,0]/\fX}) \lra 
H^a(\fX \times A_L^n[\lambda,1), \ol{\cU_{[\lambda,1)}(F)} \otimes 
\omega^{\b}_{\fX \times A^n_L[\lambda,1)/\fX})
\end{equation} 
induced by $\cU_{[\lambda,1)}$ 
is an isomorphism for $a=0$ and injective for $a=1$. \par 
Since 
$\ol{F}$ is a finite direct sum of $\pi^*M_{\xi}$, we may assume that 
$\ol{F} = \pi^*M_{\xi}$. If we put $\xi = (\eta_j)_{j=1}^n$, 
we have $\eta_j \in \Z_p \setminus \Z_{> 0}$ for all $j$ 
by assumption on $\Sigma_i$'s. Then, when $a=0$, the both hand sides on 
\eqref{rest} are equal to $0$ if there exists some $j$ with 
$\eta_j \in \Z_p \setminus \Z$ and 
equal to $\Gamma(\fX, \cO_{\fX})\prod_{j=1}^n t_j^{-\eta_j}$ 
if $\eta_j \in \Z_{\leq 0}$ for all $j$. Hence 
they are isomorphic when $a=0$. \par 
For general $a$, the left hand side (resp. the right hand side) of 
\eqref{rest} (in the case $\ol{F} = \pi^*M_{\xi}$) 
is the $a$-th cohomology of the left hand side 
(resp. the right hand side) of the following map of complexes 
which is induced by $\cU_{[\lambda,1)}$: 
\begin{equation}\label{resti}
\Gamma(\fX, \ol{\pi^*M_{\xi}} \otimes 
\omega^{\b}_{\fX \times A^n_L[0,0]/\fX}) \lra 
\Gamma(\fX \times A_L^n[\lambda,1), \ol{\cU_{[\lambda,1)}(\pi^*M_{\xi})} 
\otimes 
\omega^{\b}_{\fX \times A^n_L[\lambda,1)/\fX}). 
\end{equation}
(Here the complex is defined as the log de Rham complex 
associated to the log-$\nabla$-module structure on 
$\ol{\pi^*M_{\xi}}$ and $\ol{\cU_{[\lambda,1)}(\pi^*M_{\xi})}$.) 
Since the map 
$$ 
\Gamma(\fX \times A_L^n[\lambda,1), \ol{\cU_{[\lambda,1)}(\pi^*M_{\xi})} 
\otimes 
\omega^{\b}_{\fX \times A^n_L[\lambda,1)/\fX}) \lra 
\Gamma(\fX, \ol{\pi^*M_{\xi}} \otimes 
\omega^{\b}_{\fX \times A^n_L[0,0]/\fX}) 
$$ 
of `taking the constant coefficient' gives a left inverse of \eqref{rest}, 
we see that the map \eqref{resti} induces injection on 
cohomologies. So the map \eqref{rest} is injective for any $a$ and 
the proof is finished. 
\end{proof}

\subsection{Isocrystals}

In this subsection, we review the definition of certain properties 
on isocrystals 
which we introduced in \cite{sigma} and recall some results proven there. 
We also add some more new terminologies and results which are 
useful in later sections. \par 
To do this, first we recall terminologies on frames. 

\begin{definition}[{\cite[2.2.4, 4.2.1]{kedlayaI}}] 
\begin{enumerate}
\item 
A frame $($or an affine frame$)$ is a triple 
$(X,\ol{X},\allowbreak \ol{\cX})$ consisting of 
$k$-varieties $X, \ol{X}$ and a $p$-adic affine formal scheme 
$\ol{\cX}$ of finite type over $\Spf O_K$ 
endowed with a closed immersion 
$i: \ol{X} \hra \ol{\cX}$ over $\Spf O_K$ and an open immersion 
$j: X \hra \ol{X}$ over $k$ such that 
$\ol{\cX}$ is formally smooth over $O_K$ on a neighborhood 
of $X$. 
We say that the frame encloses the pair $(X,\ol{X})$. 
\item 
A small frame is a frame $(X,\ol{X},\ol{\cX})$ such that 
$\ol{X}$ is isomorphic $($via $i$ as in $(1))$ 
to $\ol{\cX} \times_{\Spf O_K} \Spec k$ 
and that there exists 
an element $f \in \Gamma(\ol{X},\cO_{\ol{X}})$ with $X=\{f\not=0\}$. 
\end{enumerate}
\end{definition} 

\begin{remark}
In \cite{kedlayaI}, a frame is written as a tuple 
$(X,\ol{X},\ol{\cX},i,j)$, but we denote it here simply as a 
triple $(X,\ol{X},\ol{\cX})$. 
\end{remark}

\begin{definition}[{\cite[3.3]{sigma}}]\label{sigma3.3} 
Let $X \hra \ol{X}$ be an open immersion of smooth $k$-varieties such 
that $Z := \ol{X} \setminus X$ is a simple normal crossing 
divisor and let $Z = \bigcup_{i=1}^r Z_i$ be a decomposition of 
$Z$ such that $Z=\bigcup_{\scriptstyle 1 \leq i \leq r \atop 
\scriptstyle Z_i \not= \emptyset} Z_i$ gives the decomposition of 
$Z$ into irreducible components. 
A standard small frame enclosing $(X,\ol{X})$ 
is a small frame $\bX := (X,\ol{X},\ol{\cX})$ enclosing $(X,\ol{X})$ 
which satisfies the following condition$:$ 
There exist $t_1, ..., t_r \in \Gamma(\ol{\cX},\cO_{\ol{\cX}})$ such that, 
if we denote the zero locus of $t_i$ in $\ol{\cX}$ by $\cZ_i$, 
each $\cZ_i$ is irreducible $($possibly empty$)$ and that 
$\cZ = \bigcup_{i=1}^r\cZ_i$ $($which we call the lift of $Z)$ 
is a relative simple normal crossing divisor of $\ol{\cX}$ 
satisfying $Z_i = \cZ_i \times_{\ol{\cZ}} \ol{X}$. 
We call a pair $(\bX, (t_1,...,t_r))$ a charted standard small frame. 
When $r=1$, we call $\bX$ a smooth standard small frame and 
the pair $(\bX,t_1)$ a charted smooth standard small frame. 
\end{definition} 

We also introduce the notion of charted smooth standard small frame 
with generic point as follows (It is essentially already appeared in 
the paper \cite{sigma}, but it is convenient to give it a name): 

\begin{definition}
Let $X \hra \ol{X}$ be an open immersion of smooth $k$-varieties such 
that $Z := \ol{X} \setminus X$ is an irreducible smooth 
divisor. A charted smooth standard small frame 
with generic point enclosing $(X,\ol{X})$ is a triple 
$(\bX,t_1,L)$ consisting of a charted smooth standard small frame 
$(\bX,t_1) = ((X,\ol{X},\ol{\cX}),t_1)$ and an injection 
$\Gamma(\cZ,\cO_{\cZ}) \hra L$ $($where $\cZ := \{t_1=0\}$ is the 
lift of $Z)$ into a field $L$ endowed with a complete multiplicative 
norm which restricts to the supremum norm on $\Gamma(\cZ,\cO_{\cZ})$. 
\end{definition}

We introduce the notion of a morphism of charted smooth standard small frames 
(with generic points) as follows: 

\begin{definition}
A morphism $f:((X',\ol{X}',\ol{\cX}'),t') \lra ((X,\ol{X},\ol{\cX}),t)$ 
of charted smooth standard small frames is morphisms 
$X' \lra X, \ol{X}' \lra \ol{X}, \ol{\cX}' \os{\ti{f}}{\lra} \ol{\cX}$ 
compatible with the structure of smooth standard small frames 
such that $\ti{f}^*t = {t'}^n$ for some positive integer $n$. 
A morphism $((X',\ol{X}',\ol{\cX}'),t',L') \lra ((X,\ol{X},\ol{\cX}),t,L)$ 
of charted smooth standard small frames with generic points is a 
morphism $f: ((X',\ol{X}',\ol{\cX}'),t') \lra ((X,\ol{X},\ol{\cX}),t)$ 
of charted smooth standard small frames endowed with a continuous 
morphism $g: L \lra L'$ of fields such that, if we put 
$\cZ \allowbreak := \allowbreak \{t=0\}, \cZ' := \{t'=0\}$, the diagram 
\begin{equation*}
\begin{CD}
\Gamma(\cZ,\cO_{\cZ}) @>{f^*}>> \Gamma(\cZ',\cO_{\cZ'}) \\ 
@V{\cap}VV @V{\cap}VV \\ 
L @>g>> L'
\end{CD}
\end{equation*}
is commutative, where $f^*$ is the homomorphism induced by $f$. 
\end{definition}

As for the existence of a morphism of charted smooth standard small frames 
(with generic points), the following lemma will be useful later. 

\begin{lemma}\label{lifting}
Let $X \hra \ol{X}$ $(X' \hra \ol{X}')$ be an open immersion of 
affine smooth $k$-varieties such that $Z = \ol{X} \setminus X$ 
$(Z' = \ol{X}' \setminus X')$ is an irreducible smooth divisor defined as 
the zero section of $\ol{t} \in \Gamma(\ol{X},\cO_{\ol{X}})$ 
$(\ol{t}' \in \Gamma(\ol{X}',\cO_{\ol{X}'}))$. Let 
$f: (\ol{X}',Z') \lra (\ol{X},Z)$ be a log smooth morphism such that 
$f^*\ol{t} = {\ol{t}'}^n$ for some $n \in \N$ prime to $p$. 
Denote the morphism of pairs of schemes 
$(X',\ol{X}') \lra (X,\ol{X})$ induced by $f$ also by $f$. Assume 
moreover that we are given a charted smooth standard small frame 
$((X,\ol{X},\ol{\cX}),t)$ enclosing $(X,\ol{X})$ such that $t \in 
\Gamma(\ol{\cX},\cO_{\ol{\cX}})$ is a lift of $\ol{t}$. Then there exists 
a charted smooth standard small frame 
$((X',\ol{X}',\ol{\cX}'),t')$ enclosing $(X',\ol{X}')$ 
such that $t' \in 
\Gamma(\ol{\cX}',\cO_{\ol{\cX}'})$ is a lift of $\ol{t}'$ and a 
morphism $\ti{f}: ((X',\ol{X}',\ol{\cX}'),t') \lra ((X,\ol{X},\ol{\cX}),t)$ 
with $\ti{f}^*t = {t'}^n$ which is compatible with 
$f:(X',\ol{X}') \lra (X,\ol{X})$. 
\end{lemma}

\begin{proof}
In the situation of the lemma, we have a diagram 
\begin{equation*}
\begin{CD}
(\ol{X}',Z') @>f>> (\ol{X},Z) \\ 
@VVV @VVV \\ 
(\Spec k[\ol{t}'], \{\ol{t}'=0\}) @>{f_0}>> (\Spec k[\ol{t}], \{\ol{t}=0\}), 
\end{CD}
\end{equation*}
where the vertical arrows are induced from $\ol{t}, \ol{t}'$ and 
$f_0$ is induced by the ring homomorphism 
$k[\ol{t}] \lra k[\ol{t}']: \, \ol{t} \mapsto {\ol{t}'}^n$. 
Then $f$ is factorized as 
\begin{equation}\label{qwerty}
(\ol{X}',Z') \lra (\ol{X},Z) \times_{(\Spec k[\ol{t}], \{\ol{t}=0\})} 
(\Spec k[\ol{t}'], \{\ol{t}'=0\}) \lra  (\ol{X},Z). 
\end{equation}
Then $f$ is log smooth by assumption and 
the second morphism in \eqref{qwerty} is log etale because so is $f_0$. 
Hence the first morphism in \eqref{qwerty} is also log smooth by 
\cite[3.5]{kato}. Since it is strict by assumption, it induces 
the smooth morphism $g: \ol{X}' \lra \ol{X} \times_{\Spec k[\ol{t}]} 
\Spec k[\ol{t}']$ of affine schemes. Hence there exists a 
formal scheme $\ol{\cX}'$ with $\ol{\cX}' \otimes_{O_K} k = \ol{X}'$ and 
a smooth morphism $\ti{g}: \ol{\cX}' \lra \ol{\cX} 
\times_{\Spf O_K\{t\},\ti{f}_0} \Spf O_K\{t'\}$ lifting $g$, where 
$\ti{f}_0:\Spf O_K\{t'\} \lra \Spf O_K\{t\}$ is the morphism induced by 
the ring homomorphism $O_K\{t\} \lra O_K\{t'\}; \, t \mapsto {t'}^n$. 
Then $((X',\ol{X}',\ol{\cX}'),t':=\ti{g}^*t')$ defines a 
charted smooth standard small frame, and the composite 
$$ 
 \ol{\cX}' \os{\ti{g}}{\lra} \ol{\cX} 
\times_{\Spf O_K\{t\},\ti{f}_0} \Spf O_K\{t'\} \os{\text{proj.}}{\lra} 
 \ol{\cX} $$ 
induces the morphism $\ti{f}$ of charted smooth standard small frames 
as in the statement of the lemma. So we are done. 
\end{proof}

\begin{remark}\label{liftingrem}
In this remark, let the notation be as in Lemma \ref{lifting} and 
we denote the zero locus of $t$ ($t'$) in $\ol{\cX}$ ($\ol{\cX}'$) 
by $\cZ$ ($\cZ'$). Then, 
by looking the proof of Lemma \ref{lifting} carefully, we see the following: 
The morphism $(\ol{\cX}',\cZ') \lra (\ol{\cX},\cZ)$ defining 
$\ti{f}$ is log smooth, and it is strict smooth when $n=1$. Also, if the morphism $Z' \lra Z$ induced by $f$ is an isomorphism, the morphism 
$\cZ' \lra \cZ$ induced by $\ti{f}$ 
is also an isomorphism because it is 
a lift of the morphism $Z' \lra Z$ to a morphism of $p$-adic smooth formal 
schemes over $\Spf O_K$. 
So, if the morphism $Z' \lra Z$ induced by $f$ is an isomorphism and 
if we are given 
a charted smooth standard small frame with generic point 
$((X,\ol{X},\ol{\cX}),t,L)$, the morphism $\ti{f}$ can be enriched to 
a morphism of charted smooth standard small frames with generic points of 
the form 
$(\ti{f},\id): 
((X',\ol{X}',\ol{\cX}'),t',L) \lra ((X,\ol{X},\ol{\cX}),t,L).$
\end{remark}

Let $j: X \hra \ol{X}$ be an open immersion of smooth $k$-varieties 
such that $\ol{X} \setminus X =: Z$ is a 
simple normal crossing divisor. Then we have the log scheme 
$(\ol{X},Z)$ (see Convention for this notation) and 
the category of locally free log convergent 
isocrystals $\Isocl(\ol{X},Z)$ on $(\ol{X},Z)$ over $K$. 
We recall the notion of `having exponents in $\Sigma$' for an object in 
$\Isocl(\ol{X},Z)$, following \cite[3.7]{sigma}. (We also introduce 
the notion of `having exponents in $\Sigma$ with semisimple residues'.) 

\begin{definition}\label{defsigexpo}
Let $X \hra \ol{X}$ be an open immersion of smooth $k$-varieties 
such that $Z:=\ol{X} \setminus X$ is a simple normal crossing divisor and let 
$Z = \bigcup_{i=1}^r Z_i$ be the decomposition of $Z$ by irreducible 
components. Let $\Sigma = \prod_{i=1}^r\Sigma_i$ be a 
subset of $\Z_p^r$. 
Then we say that an object $\cE$ in $\Isocl(\ol{X},Z)$ 
has exponents in $\Sigma$ $($resp. has exponents in $\Sigma$ with 
semisimple residues$)$ 
if there exist an affine open covering 
$\ol{X} = \bigcup_{\alpha\in \Delta} \ol{U}_{\alpha}$ and 
charted standard small frames 
$(
(U_{\alpha},\ol{U}_{\alpha},\ol{\cX}_{\alpha}), 
(t_{\alpha,1}, ..., t_{\alpha,r}))$ enclosing 
$(U_{\alpha},\ol{U}_{\alpha})$ $(\alpha \in \Delta$, where we put 
$U_{\alpha} := X \cap \ol{U}_{\alpha})$ such that, for 
any $\alpha \in \Delta$ and any $i$ $(1 \leq i \leq r)$, 
all the exponents of the log-$\nabla$-module $E_{\cE, \alpha}$ 
on $\ol{\cX}_{\alpha,K}$ induced by $\cE$ along the locus 
$\{t_{\alpha,i}=0\}$ are contained in $\Sigma_i$ $($resp. 
all the exponents of the log-$\nabla$-module $E_{\cE, \alpha}$ 
on $\ol{\cX}_{\alpha,K}$ induced by $\cE$ along the locus 
$\{t_{\alpha,i}=0\}$ are contained in $\Sigma_i$ and there exists 
some polynomial $P_{\alpha,i}(x) \in \Z_p[x]$ 
without any multiple roots such that 
$P_{\alpha,i}(\res_i)=0$ holds, 
where $\res_i$ is the residue of $E_{\cE, \alpha}$ 
along $\{t_{\alpha,i}=0\})$. 
We denote the category of objects in $\Isocl(\ol{X},Z)$ 
having exponents in $\Sigma$ $($resp. having exponents in $\Sigma$ with 
semisimple residues$)$ by $\Isocl(\ol{X},Z)'_{\Sigma}$ $($resp. 
$\Isocl(\ol{X},Z)'_{\Sigmass})$. 
\end{definition}

As a variant of \cite[3.8]{sigma}, we can prove the following: 

\begin{lemma}\label{wdlemma}
Let $(X,\ol{X}), Z := \bigcup_{i=1}^r Z_i, \Sigma$ be as above and let 
$\cE$ be an object in the category $\Isocl(\ol{X},Z)$. Then$:$ 
\begin{enumerate}
\item 
$\cE$ 
has exponents in $\Sigma$ $($resp. has exponents in $\Sigma$ with 
semisimple residues$)$ if and only if the following condition 
is satisfied$:$ For any 
affine open subscheme $\ol{U} \hra \ol{X}$, 
any charted standard small frame 
$((U,\ol{U},\ol{\cX}), 
(t_{1}, ..., t_{r}))$ enclosing 
$(U,\ol{U})$ $($where we put 
$U:= X \cap \ol{U})$ and for any $i$ $(1 \leq i \leq r)$, 
all the exponents of the log-$\nabla$-module $E_{\cE}$ 
on $\ol{\cX}_{K}$ induced by $\cE$ along the locus $\{t_{i}=0\}$ 
are contained in $\Sigma_i$ $($resp. 
all the exponents of the log-$\nabla$-module $E_{\cE}$ 
on $\ol{\cX}_{K}$ induced by $\cE$ along the locus 
$\{t_{i}=0\}$ are contained in $\Sigma_i$ and there exists 
some polynomial $P_i(x) \in \Z_p[x]$ 
without multiple roots such that 
$P_i(\res_i)=0$ holds, where $\res_i$ is the residue of $E_{\cE}$ 
along $\{t_{i}=0\})$. 
\item 
$\cE$ 
has exponents in $\Sigma$ $($resp. has exponents in $\Sigma$ with 
semisimple residues$)$ if and only if the following condition 
is satisfied$:$ 
there exist affine open subschemes 
$\ol{U}^{(\alpha)} \subseteq \ol{X} \setminus Z_{\sing}$ 
$($where $Z_{\sing}$ is the set of singular points of $Z)$ 
containing 
the generic point of $Z_{\alpha}$  $(1 \leq \alpha \leq r)$, 
charted smooth standard small frames 
$(
(U^{(\alpha)},\ol{U}^{(\alpha)},
\ol{\cX}^{(\alpha)}), 
t^{(\alpha)})$ enclosing 
$(U^{(\alpha)},\ol{U}^{(\alpha)})$ $($where we put 
$U^{(\alpha)} := X \cap \ol{U}^{(\alpha)})$, 
such that, for any $\alpha$ $(1 \leq \alpha \leq r)$, 
all the exponents of the log-$\nabla$-module $E_{\cE}^{(\alpha)}$ 
on $\ol{\cX}^{(\alpha)}_{K}$ induced by $\cE$ along the locus 
$\{t^{(\alpha)}=0\}$ 
are contained in $\Sigma_{\alpha}$ $($resp. 
all the exponents of the log-$\nabla$-module $E_{\cE}^{(\alpha)}$ 
on $\ol{\cX}^{(\alpha)}_{K}$ induced by $\cE$ along the locus 
$\{t^{(\alpha)}=0\}$ are contained in $\Sigma_{\alpha}$ and there exists 
some polynomial $P^{(\alpha)}(x) \in \Z_p[x]$ 
without multiple roots such that 
$P^{(\alpha)}(\res^{(\alpha)})=0$ holds, where 
$\res^{(\alpha)}$ is the residue of $E_{\cE,\alpha}$ 
along $\{t^{(\alpha)}=0\})$. 
\end{enumerate}
\end{lemma}

\begin{proof}
First we prove (1). 
In the case of `having exponents in $\Sigma$', this is proven in 
\cite[3.8]{sigma}. In the case of `having exponents in $\Sigma$ with 
semisimple residues', we can prove the lemma in the same way, as follows: 
Let $\cE$ be an object in $\Isocl(\ol{X},Z)$ 
with exponents in $\Sigma$ with semisimple residues and take 
$((U,\ol{U},\ol{\cX}), (t_{1}, ..., t_{r}))$, $E_{\cE}$ as in 
(1). It suffices to prove that there exists 
some polynomial $P_i(x) \in \Z_p[x]$ without multiple roots such that 
$P_i(\res_i)=0$ holds, where $\res_i$ is the residue of $E_{\cE}$ 
along $\{t_{i}=0\}$. \par 
Let $\ol{X} = \bigcup_{\alpha \in \Delta} \ol{U}_{\alpha}$ and take 
$(
(U_{\alpha},\ol{U}_{\alpha},\ol{\cX}_{\alpha}), 
(t_{\alpha,1}, ..., t_{\alpha,r}))$, $E_{\cE,\alpha}$ and $P_{\alpha,i}(x) 
\in \Z_p[x]$ 
as in Definition \ref{defsigexpo}. By shrinking $\ol{\cX}$ and 
$\ol{\cX}_{\alpha}$ appropriately, we may assume that 
$\ol{U} = \ol{U}_{\alpha}$ for some $\alpha$. Then, in the proof of 
\cite[3.8]{sigma}, we constructed the diagram 
\begin{equation}\label{wdeq1}
\ol{\cX}_{\alpha,K}=]\ol{U}[_{\ol{\cX}_{\alpha}} \os{\pi_1}{\lla}\,
]\ol{U}[^{\log}_{\ol{\cX}_{\alpha} \times \ol{\cX}} \,\cong\, 
]\ol{U}[_{\hat{{\Bbb A}}^d \times \ol{\cX}} \os{s}{\lla}\, ]\ol{U}[_{\ol{\cX}} 
= \ol{\cX}_K 
\end{equation}
and proved that the residue of $E_{\cE,\alpha}$ along $\{t_{\alpha,i}=0\}$ 
is pulled back by \eqref{wdeq1} 
to the residue $\res_i$ of $E_{\cE}$ along $\{t_{i}=0\}$. 
 So we have $P_{\alpha,i}(\res_i)=0$ and so we have proved (1). \par
Next we prove (2). Let us take 
$((U,\ol{U},\ol{\cX}), (t_{1}, ..., t_{r}))$, $E_{\cE}$ as in 
(1). Then it suffices to prove that 
all the exponents of $E_{\cE}$ along 
$\{t_{i}=0\}$ are contained in $\Sigma_i$ (and there exists 
some polynomial $P_i(x) \in \Z_p[x]$ without multiple roots such that 
$P_i(\res_i)=0$ holds, where $\res_i$ is the residue of $E_{\cE}$ 
along $\{t_{i}=0\}$). Let us consider the intersection 
$\ol{U} \cap \ol{U}^{(i)} \cap Z_i$. When it is 
empty, we have $\ol{U} \cap Z_{i} = \emptyset$ and so the 
locus $\{t_{i}=0\}$ is empty. Hence the assertion 
we have to prove is vacuous in this case. So let us consider the case 
where $\ol{U} \cap \ol{U}^{(i)} \cap Z_i$ is non-empty. 
Then we can take a non-empty 
affine open formal subscheme $\ol{\cX}_{i}$ of $\ol{\cX}$ with 
$\ol{\cX}_i \otimes_{O_K} k \subseteq \ol{U} \cap \ol{U}^{(i)}$ and 
$\ol{\cX}_i \cap \{t_i = 0\}$ non-empty. 
Let 
$F_{\cE} := {\cE}\mathrm{nd} (E_{\cE})$. Then, by the argument in 
\cite[between 1.2 and 1.3]{sigma}, the restriction map 
$$\Gamma(\{t_i = 0\}, F_{\cE}) \lra 
\Gamma(\ol{\cX}_{i,K} \cap \{t_i = 0\}, F_{\cE})$$
is injective. So, it suffices to show that 
$E_{\cE}|_{\ol{\cX}_{i,K}}$ has exponents in $\Sigma_i$ (and 
$P^{(i)}(\res_i) = 0$), that is, we may assume that 
$\ol{U} \subseteq \ol{U}^{(i)}$ to prove the desired assertion. 
Then, by shrinking $\ol{U}^{(i)}$, we may assume that 
$\ol{U} = \ol{U}^{(i)}$. In this case, we have the diagram 
\eqref{wdeq1} with $\ol{\cX}_{\alpha,K}$ replaced by $\ol{\cX}^{(i)}_K$. 
So the residue of $E_{\cE}^{(i)}$ along $\{t^{(i)}=0\}$ 
is pulled back 
to the residue $\res_i$ of $E_{\cE}$ along $\{t_{i}=0\}$. 
Hence $E_{\cE}$ has exponents in $\Sigma_i$ (and $P^{(i)}(\res_i)=0$) and 
thus we have also proved (2). 
\end{proof} 

Next we recall the notion of `having $\Sigma$-unipotent generic 
monodromy' for an object in $\Isocd(X,\ol{X})$. We introduce also the 
notion of `having $\Sigma$-semisimple generic 
monodromy'. 

\begin{definition}\label{defsiggenmono}
Let $X \hra \ol{X}$ be an open immersion of smooth $k$-varieties 
such that $Z:=\ol{X} \setminus X$ is a simple normal crossing divisor.
Let $Z = \bigcup_{i=1}^rZ_i$ be the decomposition of $Z$ by 
irreducible components and let 
$Z_{\sing}$ be the set of singular points of $Z$. 
Let $\Sigma = \prod_{i=1}^r\Sigma_i$ be a 
subset of $\Z_p^r$. Then we say that an overconvergent isocrystal 
$\cE$ on $(X,\ol{X})$ over $K$ has $\Sigma$-unipotent generic monodromy 
$($resp. $\Sigma$-semisimple generic monodromy$)$ 
if there exist affine open subschemes 
$\ol{U}_{\alpha} \subseteq \ol{X} \setminus Z_{\sing}$ containing 
the generic point of $Z_{\alpha}$  $(1 \leq \alpha \leq r)$, 
charted smooth standard small frames with generic point 
$(
(U_{\alpha},\ol{U}_{\alpha},\ol{\cX}_{\alpha}), 
t_{\alpha},L_{\alpha})$ enclosing 
$(U_{\alpha},\ol{U}_{\alpha})$ $($where we put 
$U_{\alpha} := X \cap \ol{U}_{\alpha})$ 
satisfying 
the following condition$:$ For 
any $1 \leq \alpha \leq r$, there exists some $\lam \in (0,1)\cap\Gamma^*$ 
such that the $\nabla$-module $E_{\cE, \alpha}$ associated to 
$\cE$ is defined on $\{x \in \ol{\cX}_{\alpha,K} \,\vert\, 
|t_{\alpha}(x)| \geq \lam\}$ and that the restriction of $E_{\cE,\alpha}$ 
to $A^1_{L_{\alpha}}[\lam,1)$ is $\Sigma_{\alpha}$-unipotent $($resp. 
$\Sigma_{\alpha}$-semisimple$)$. \par 
We denote the category of objects in $\Isocd(X,\ol{X})$ 
having $\Sigma$-unipotent generic monodromy $($resp. 
having $\Sigma$-semisimple generic 
monodromy$)$ by $\Isocd(X, \ol{X})'_{\Sigma}$ $($resp. 
$\Isocd(X, \ol{X})'_{\Sigmass})$. 
\end{definition} 

Note that the notion of having $\Sigma$-unipotent generic monodromy 
($\Sigma$-semisimple generic monodromy) 
depends 
only on the image $\ol{\Sigma}$ of $\Sigma$ in $\Z_p^r/\Z^r$ in the 
sense that $\cE$ has $\Sigma$-unipotent generic monodromy 
($\Sigma$-semisimple generic monodromy) 
if and only if 
$\cE$ has $\tau(\ol{\Sigma})$-unipotent generic monodromy 
($\tau(\ol{\Sigma})$-semisimple generic monodromy) 
for some (or any) 
section $\tau:\Z_p^r/\Z^r \lra \Z_p$ of the form $\tau = \prod_{i=1}^r 
\tau_i$ of the canonical projection 
$\Z_p^r \lra \Z_p^r/\Z^r$. (See \cite[1.4]{sigma} or the paragraph 
after Definition \ref{unipdef}.) Hence 
it is allowed to say 
that $\cE$ has $\ol{\Sigma}$-unipotent generic 
monodromy ($\ol{\Sigma}$-semisimple generic 
monodromy) and denote by 
$\Isocd(X, \ol{X})'_{\ol{\Sigma}}$ $($resp. 
$\Isocd(X, \ol{X})'_{\olSigmass})$
by abuse of terminology. \par 
Then the main theorem of \cite{sigma} is described as follows. 
(We add also the `semisimple' variant of the theorem.)

\begin{theorem}\label{sigmamain}
Let $X \hra \ol{X}$ be an open immersion of smooth $k$-varieties 
such that $Z :=\ol{X} \setminus X$ is a simple normal crossing divisor and let 
$Z=\bigcup_{i=1}^rZ_i$ be the decomposition of $Z$ into irreducible 
components. Let $\ol{\Sigma} := \prod_{i=1}^r\ol{\Sigma}_i$ be a subset of 
$\Z_p^r/\Z^r$ which is $\NLD$ and let $\tau:\Z_p^r/\Z^r \lra \Z_p$ be
 a section of the form $\tau = \prod_{i=1}^r 
\tau_i$ of the canonical projection 
$\Z_p^r \lra \Z_p^r/\Z^r$. Then we have the canonical 
equivalences of categories 
\begin{align}
& j^{\dagger}: \Isocl(\ol{X},Z)'_{\tau(\ol{\Sigma})} \os{=}{\lra} 
  \Isocd(X,\ol{X})'_{\ol{\Sigma}}, \label{sigmaeq1} \\ 
& j^{\dagger}: \Isocl(\ol{X},Z)'_{\tau(\ol{\Sigma})\text{-}{\rm ss}} 
\os{=}{\lra} \Isocd(X,\ol{X})'_{\olSigmass}, \label{sigmaeq2}
\end{align}
which are defined by the restriction. 
\end{theorem}

\begin{proof}
The equivalence \eqref{sigmaeq1} follows from \cite[3.12, 3.16, 3.17]{sigma}. 
Let us prove the equivalence \eqref{sigmaeq2}. First let us take 
$\cE \in \Isocl(\ol{X},Z)'_{\tau(\ol{\Sigma})\text{-}{\rm ss}}$ 
and let us take open subschemes 
$\ol{U}_{\alpha} \subseteq \ol{X} \setminus Z_{\sing}$ containing 
the generic point of $Z_{\alpha}$  $(1 \leq \alpha \leq r)$ and 
charted smooth standard small frames with generic point 
$(
(U_{\alpha},\ol{U}_{\alpha},\ol{\cX}_{\alpha}), 
t_{\alpha}, L_{\alpha})$ enclosing 
$(U_{\alpha},\ol{U}_{\alpha})$ $($where we put 
$U_{\alpha} := X \cap \ol{U}_{\alpha})$. 
Let $E_{\cE, \alpha}$ be the log-$\nabla$-module on 
$\ol{\cX}_{\alpha,K}$ induced by $\cE$, let 
$E_{\cE,L,\alpha}$ be the restriction of $E_{\cE, \alpha}$ to 
$A^1_{L_{\alpha}}[0,1)$ and let $\cZ_{\alpha}$ be the zero locus of 
$t_{\alpha}$ in $\ol{\cX}_{\alpha}$. 
Then, by \cite[3.6, 2.12]{sigma}, $E_{\cE,\alpha}|_{\cZ_{\alpha,K} 
\times A_K^1[0,1)}$ 
is $\Sigma_{\alpha}$-unipotent. So $E_{\cE,L,\alpha}$ is 
also 
$\Sigma_{\alpha}$-unipotent and hence 
$E_{\cE,L,\alpha} = \cU_{[0,1)}(E,\pa)$ 
for some $(E,\pa) \in \ULNM_{A^1_{L_{\alpha}}[0,0]}$. 
On the other hand, by Lemma \ref{wdlemma}, there exists some 
$P_{\alpha}(x) \in \Z_p[x]$ without 
multiple roots such that $P_{\alpha}(\res_{\alpha})=0$, where 
$\res_{\alpha}$ denotes the residue of $E_{\cE,\alpha}$ along 
$\cZ_{\alpha,K}$. Then the residue $\res_{L,\alpha}$ 
of $E_{\cE,L,\alpha}$ satisfies the same equation. 
Note that, when we restrict $E_{\cE,L,\alpha}$ 
to the locus $\{t=0\}$ (where $t$ denotes the coordinate of 
$A_{L_{\alpha}}^1[0,1)$), we obtain $(E,\pa)$. So we have the 
equation $P_{\alpha}(\pa)=0$, that is, $\pa$ acts semisimply on 
$E$. Hence $E_{\cE,L,\alpha} = \cU_{[0,1)}(E,\pa)$ is $\Sigma$-semisimple 
and so is the restriction of $E_{\cE,L,\alpha}$ to 
$A_{L_{\alpha}}^1[\lambda,1)$ for any $\lambda \in [0,1) \cap \Gamma^*$. 
Therefore, we have $j^{\dagger}\cE \in  \Isocd(X,\ol{X})'_{\olSigmass}$, 
that is, the functor \eqref{sigmaeq2} is well-defined. \par 
The full faithfulness of \eqref{sigmaeq2} follows from 
that of \eqref{sigmaeq1}. Let us prove the essential surjectivity of 
\eqref{sigmaeq2}. Let us take $\cE' \in \Isocd(X,\ol{X})'_{\olSigmass}$ 
and take $\cE \in \Isocl(\ol{X},Z)'_{\tau(\ol{\Sigma})}$ with 
$j^{\dagger}\cE = \cE'$. Let us take 
$(
(U_{\alpha},\ol{U}_{\alpha},\ol{\cX}_{\alpha}), 
t_{\alpha}, L_{\alpha})$, $\cZ_{\alpha}$ and $\lambda$ such that 
the $\nabla$-module $E_{\cE', \alpha}$ associated to 
$\cE'$ is defined on $\{x \in \ol{\cX}_{\alpha,K} \,\vert\, 
|t_{\alpha}(x)| \geq \lam\}$ and that the restriction 
$E_{\cE',L,\alpha}$ of $E_{\cE',\alpha}$ 
to $A^1_{L_{\alpha}}[\lam,1)$ is $\Sigma_{\alpha}$-semisimple. 
By definition of $\cE$, $E_{\cE',L,\alpha}$ extends to 
the log-$\nabla$-module $E_{\cE,L,\alpha}$ on $A_{L_{\alpha}}^1[0,1)$ 
induced by $\cE$ and it is $\Sigma_{\alpha}$-unipotent 
by \cite[3.6, 2.12]{sigma}. Now let us note the equivalences 
\begin{align*}
& \ULNM_{A_{L_{\alpha}}[0,0],\Sigma_{\alpha}} \os{\cU_{[0,1)}}{\lra} 
\ULNM_{A_{L_{\alpha}}[0,1),\Sigma_{\alpha}} \os{=}{\lra} 
\ULNM_{A_{L_{\alpha}}(\lambda,1),\Sigma_{\alpha}}, \\ 
& \SLNM_{A_{L_{\alpha}}[0,0],\Sigma_{\alpha}} \os{\cU_{[0,1)}}{\lra} 
\SLNM_{A_{L_{\alpha}}[0,1),\Sigma_{\alpha}} \os{=}{\lra} 
\SLNM_{A_{L_{\alpha}}(\lambda,1),\Sigma_{\alpha}} 
\end{align*}
induced by \eqref{homexteq}: By the first equivalences, we see that 
there exists an object $(E,\pa) \in 
\ULNM_{A_{L_{\alpha}}[0,0],\Sigma_{\alpha}}$ which is sent to 
$E_{\cE,L,\alpha} \in \ULNM_{A_{L_{\alpha}}[0,1),\Sigma_{\alpha}}$ 
and it is sent to 
$E_{\cE',L,\alpha}|_{A_{L_{\alpha}}(\lambda,1)}$. 
Then, since 
 $E_{\cE',L,\alpha}|_{A_{L_{\alpha}}(\lambda,1)}$ belongs to 
$\SLNM_{A_{L_{\alpha}}(\lambda,1),\Sigma_{\alpha}}$, we see by 
the second equivalences that  $(E,\pa)$ is actually in 
$\SLNM_{A_{L_{\alpha}[0,0]},\Sigma_{\alpha}}$. Hence we have 
$E_{\cE,L,\alpha} \in \SLNM_{A_{L_{\alpha}[0,1)},\Sigma_{\alpha}}$ and 
this implies that there exists a polynomial $P_{\alpha}(x) \in \Z_p[x]$ 
without multiple roots such that 
the residue $\res$ of $E_{\cE,L,\alpha}$ along $\{t=0\}$ (where $t$ denotes the coordinate of $A_{L_{\alpha}}^1[0,1)$) satisfies $P_{\alpha}(\res)=0$. 
Then, since the restriction map 
$\Gamma(\cZ_{\alpha,K}, \cE{\rm nd}(E_{\cE,\alpha} |_{\cZ_{\alpha,K}})) \lra 
\Gamma(\Spm L_{\alpha}, \cE{\rm nd}(E_{\cE,L,\alpha} |_{\{t=0\}}))$
is injective, the residue $\res_{\alpha}$ of $E_{\cE,\alpha}$ along 
$\cZ_{\alpha,K}$ also satisfies $P_{\alpha}(\res_{\alpha})=0$. Hence 
$\cE$ belongs to $\Isocl(\ol{X},Z)'_{\tau(\ol{\Sigma})\text{-}{\rm ss}}$, 
by Lemma \ref{wdlemma}. So we are done. 
\end{proof}

We have also the following
variant of the full faithfulness of the functor \eqref{sigmaeq2}, which 
is useful in this paper: 

\begin{proposition}\label{uniqueinj}
Let $X \hra \ol{X}, Z :=\ol{X} \setminus X$ be as above and 
let $\ol{\Sigma} := \prod_{i=1}^r\ol{\Sigma}_i$ be a subset of 
$\Z_p^r/\Z^r$ which is $\NLD$. Let 
$\tau_i:\Z_p^r/\Z^r \lra \Z_p^r \,(i=1,2)$ be sections of the form 
$\tau_i = \prod_{j=1}^r \tau_{ij}$ of the canonical projection 
$\Z_p^r \lra \Z_p^r/\Z^r$ such that 
for any $j$ and any $\xi \in \ol{\Sigma}_j$, $\tau_{1j}(\xi) \geq 
\tau_{2j}(\xi)$. 
Then, for any $\cE_i \in 
\Isocl(\ol{X},Z)'_{\tau_i(\ol{\Sigma})} \,(i=1,2)$, 
the homomophism 
$$ \Hom(\cE_1,\cE_2) \lra \Hom (j^{\dagger}\cE_1, j^{\dagger}\cE_2) $$ 
is an isomorphism. 
\end{proposition}

\begin{proof}
Since we may work Zariski locally on $\ol{X}$, 
we may assume that there exists 
a charted standard small frame 
$((X,\ol{X},\ol{\cX}), (t_{1}, ..., t_{r}))$ enclosing 
$(X,\ol{X})$. Let $\wt{E}_i$ (resp. $E_i$) be the log-$\nabla$-module 
(resp. $\nabla$-module) on $\ol{\cX}_K$ (resp. on a strict 
neighborhood of $]X[_{\ol{\cX}}$ in $\ol{\cX}_K$) induced by 
$\cE_i$ (resp. $j^{\dagger}\cE_i$) \,$(i=1,2)$. 
We may assume that $E_i$'s are both defined on $\fY := \{x \in 
\ol{\cX}_K \,\vert\, \forall i, |t_i(x)| \geq \lam\}.$ It suffices to prove the isomorphism 
\begin{equation}\label{e*}
\Hom_{\ol{\cX}_K}(\wt{E}_1,\wt{E}_2) \os{=}{\lra} \Hom_{\fY}(E_1,E_2).
\end{equation}
(Here $\Hom$ denotes the set of homomorphism as (log-)$\nabla$-modules.) 
Let us consider the admissible covering 
$\ol{\cX}_K = \bigcup_{I \subseteq \{1,...,r\}} \fX_I$, where 
$\fX_I$ is defined by 
$$ 
\fX_I := \{x \in \ol{\cX}_K\,\vert\, 
|t_i(x)|<1 \,(i \in I), \, |t_i(x)|\geq \lam 
\,(i \notin I)\}. 
$$
This covering induces the admissible covering 
$\fY = \bigcup_{I \subseteq \{1,...,r\}} \fY_I$, where 
$$ 
\fY_I := 
\{x \in \ol{\cX}_K\,\vert\, 
\lam \leq |t_i(x)|<1 \,(i \in I), \, |t_i(x)|\geq \lam 
\,(i \notin I)\}. 
$$ 
For $i=1,2$, $\wt{E}_i$ is $\prod_{j\in I}\tau_i(\Sigma_{ij})$-unipotent on 
$$\fX_I = 
\{x \in \ol{\cX}_K\,\vert\,t_i(x)=0 \,(i \in I), \, 
|t_i(x)|\geq \lam \,(i \notin I)\} \times A^{|I|}[0,1)$$ 
by \cite[3.6, 2.12]{sigma}. Therefore, we have the 
isomorphism 
$$  \Hom_{\fX_I}(\wt{E}_1,\wt{E}_2) \os{=}{\lra} \Hom_{\fY_I}(E_1,E_2) $$
by Proposition \ref{uniqueinj0}. By 
noting the equalities 
\begin{align*}
\fX_I \cap \fX_J = 
\{x \in P_K\,\vert\, \lam 
\leq |t_i(x)| < 1 \,& (i \in (I\cup J) \setminus (I\cap J)), \\ 
& \lam \leq |t_i(x)| \,(i \notin I \cup J)\} \times 
A^{|I\cap J|}_K[0,1),
\end{align*} 
\begin{align*}
\fY_I \cap \fY_J = 
\{x \in P_K\,\vert\, \lam \leq |t_i(x)| < 1 \,& 
(i \in (I\cup J) \setminus (I\cap J)), \\ 
& \lam \leq |t_i(x)| \,(i \notin I \cup J)\} \times 
A^{|I\cap J|}_K[\lam,1), 
\end{align*}  
we also see the 
isomorphism 
$$  \Hom_{\fX_I\cap\fX_J}(\wt{E}_1,\wt{E}_2) 
\os{=}{\lra} \Hom_{\fY_I\cap\fY_J}(E_1,E_2) $$
by the same argument. So we have the isomorphism \eqref{e*}. 
\end{proof}

Finally in this section, we give a slightly different formulation 
concerning the definition of log convergent isocrystals 
having exponents in $\Sigma$ (with semisimple residues) and 
overconvergent isocrystals having $\Sigma$-unipotent ($\Sigma$-semisimple) 
generic monodromy. The formulation given below is useful when we 
discuss the functoriality. 

\begin{definition}\label{defsigexpo2}
Let $X \hra \ol{X}$ be an open immersion of smooth $k$-varieties such that 
$Z = \ol{X} \setminus X$ is a simple normal crossing divisor and 
assume that we are given a family of simple normal crossing 
subdivisors $\{Z_i\}_{i=1}^r$ of $Z$ with $Z = \sum_{i=1}^r Z_i$ 
$($we call such a family $\{Z_i\}_{i=1}^r$ a decomposition of $Z)$
 and a subset $\Sigma = \prod_{i=1}^r \Sigma_i$ of $\Z_p^r$. 
Let $Z_i = \bigcup_{j=1}^{r_i} Z_{ij}$ be the decomposition of $Z_i$ 
into irreducible components and let us put $\Sigma' := \prod_{i=1}^r 
\Sigma_i^{r_i} \subseteq \Z_p^{\sum_{i=1}^r r_i}$. Then we say that 
an object $\cE$ in $\Isocl(\ol{X},Z)$ has exponents in $\Sigma$ 
$($resp. has exponents in $\Sigma$ with semisimple residues$)$ with 
respect to the decomposition $\{Z_i\}_{i=1}^r$ when it has 
exponents in $\Sigma'$ $($resp. it has 
exponents in $\Sigma'$ with semisimple residues$)$ in the sense of 
Definition \ref{defsigexpo}. 
We denote the category of objects in $\Isocl(\ol{X},Z)$ 
having exponents in $\Sigma$ $($resp. having exponents in $\Sigma$ with 
semisimple residues$)$ with 
respect to the decomposition $\{Z_i\}_{i=1}^r$ 
by $\Isocl(\ol{X},Z)_{\Sigma}$ $($resp. 
$\Isocl(\ol{X},Z)_{\Sigmass})$. 
\end{definition}

\begin{definition}\label{defsiggenmono2}
Let $X \hra \ol{X}$ be an open immersion of smooth $k$-varieties such that 
$Z = \ol{X} \setminus X$ is a simple normal crossing divisor and 
assume that we are given a decomposition $\{Z_i\}_{i=1}^r$ of $Z$ 
in the sense of Definition \ref{defsigexpo2} and 
a subset $\Sigma = \prod_{i=1}^r \Sigma_i$ of $\Z_p^r$ or 
$\Z_p^r/\Z^r$. 
Let $Z_i = \bigcup_{j=1}^{r_i} Z_{ij}$ be the decomposition of $Z_i$ 
into irreducible components and let us put $\Sigma' := \prod_{i=1}^r 
\Sigma_i^{r_i}$, which is a subset of $\Z_p^{\sum_{i=1}^r r_i}$ or 
$(\Z_p/\Z)^{\sum_{i=1}^r r_i}$. Then we say that 
an object $\cE$ in $\Isocd(X,\ol{X})$ has $\Sigma$-unipotent generic 
monodromy 
$($resp. has $\Sigma$-semisimple generic monodromy$)$ with 
respect to the decomposition $\{Z_i\}_{i=1}^r$ when it has 
$\Sigma'$-unipotent generic monodromy $($resp. it has 
$\Sigma'$-semisimple generic monodromy$)$ in the sense of 
Definition \ref{defsiggenmono}. 
We denote the category of objects in $\Isocd(X,\ol{X})$ 
having $\Sigma$-unipotent generic monodromy 
$($resp. $\Sigma$-semisimple generic monodromy$)$ with 
respect to the decomposition $\{Z_i\}_{i=1}^r$ 
by $\Isocd(X,\ol{X})_{\Sigma}$ $($resp. 
$\Isocd(X,\ol{X})_{\Sigmass})$. 
\end{definition}

Note that, when each $Z_i$ is irreducible and non-empty, we have 
$\Isocl(\ol{X},Z)_{\Sigma} = \Isocl(\ol{X},Z)'_{\Sigma}$,  
$\Isocl(\ol{X},Z)_{\Sigmass} = \Isocl(\ol{X},Z)'_{\Sigmass}$, 
$\Isocd(X,\ol{X})_{\Sigma} = \Isocd(X,\ol{X})'_{\Sigma}$ and 
$\Isocd(X,\ol{X})_{\Sigmass} = \Isocd(X,\ol{X})'_{\Sigmass}$. 
We prove here certain functoriality results for the categories of the form 
$\Isocl(\ol{X},Z)_{\Sigma(\ss)}$. 

\begin{proposition}\label{sig_pre1}
Let $X,\ol{X},Z$ be as above and assume we are given a decomposition 
$\{Z_i\}_{i=1}^r$ of $Z$ and a subset $\Sigma = \prod_{i=1}^r \Sigma_i$ 
of $\Z_p^r$. Let $f: \ol{X}' \lra \ol{X}$ one of the following$:$ 
\begin{enumerate}
\item 
$f$ is an open immersion such that the image of $\ol{X}'$ contains 
all the generic points of $Z$. 
\item 
$f$ is the morphism of the form $\ol{X}' = \coprod_{\beta} \ol{X}_{\beta} \lra 
\ol{X}$ for an open covering $\ol{X} = \bigcup_{\beta}\ol{X}_{\beta}$ by 
finite number of open subschemes. 
\end{enumerate}
Let us put $Z' := Z \times_{\ol{X}} \ol{X}'$ and consider that we are given 
the decomposition $\{Z'_i\}_{i=1}^r = \{Z_i \times_{\ol{X}} \ol{X}'\}$ of 
$Z'$. Then, for an object $\cE$ in $\Isocl(\ol{X},Z)$, it is actually 
contained in 
$\Isocl(\ol{X},Z)_{\Sigma(\ss)}$ 
if and only if $f^*\cE \in \Isocl(\ol{X}',Z')$ 
is contained in $\Isocl(\ol{X}',Z')_{\Sigma(\ss)}$. 
\end{proposition}

\begin{proof}
We may assume that each $Z_i$ is irreducible (and so 
$Z = \bigcup_{i=1}^r Z_i$ is the decomposition of $Z$ into irreducible 
components). First consider the case (1). In this case, 
$Z' = \bigcup_{i=1}^r Z'_i$ is also the decomposition of $Z$ into irreducible 
components. If we assume that $\cE$ is in the category 
$\Isocl(\ol{X},Z)_{\Sigma(\ss)}$, we can take charted smooth standard 
small frames $((U^{(\alpha)},\ol{U}^{(\alpha)},\ol{\cX}^{(\alpha)}),
t^{(\alpha)})$\,$(1 \leq \alpha \leq r)$ satisfying the conclusion of 
Lemma \ref{wdlemma}(2). Then, by shrinking $\ol{\cX}^{(\alpha)}$, we can 
assume that each $\ol{U}^{(\alpha)}$ is contained in $\ol{X}'$. Then 
we have $f^*\cE \in \Isocl(\ol{X}',Z')_{\Sigma(\ss)}$ by 
Lemma \ref{wdlemma}(2). Conversely, if we assume that $f^*\cE$ is 
in the category 
$\Isocl(\ol{X}',Z')_{\Sigma(\ss)}$, we can take charted smooth standard 
small frames $((U^{(\alpha)},\ol{U}^{(\alpha)},\ol{\cX}^{(\alpha)}),
t^{(\alpha)})$\,($1 \leq \alpha \leq r$) for $(\ol{X}', Z')$ and $f^*\cE$
satisfying the conclusion of 
Lemma \ref{wdlemma}(2). Then these charted smooth standard 
small frames satisfy the conclusion of 
Lemma \ref{wdlemma}(2) for $(\ol{X},Z)$ and $\cE$. So 
$\cE$ is in the category $\Isocl(\ol{X},Z)_{\Sigma(\ss)}$. So we are done 
in this case. \par 
Next consider the case (2). 
In this case, $Z'_i = 
\bigcup_{\scriptstyle \beta \atop \scriptstyle Z'_i \cap \ol{X}_{\beta} \not= 
\emptyset} Z'_i \cap \ol{X}_{\beta}$ gives the decomposition of 
$Z'_i$ into irreducible components. If we assume that $\cE$ is in the category 
$\Isocl(\ol{X},Z)_{\Sigma(\ss)}$, we can take charted smooth standard 
small frames $((U^{(\alpha)},\ol{U}^{(\alpha)},\ol{\cX}^{(\alpha)}),
t^{(\alpha)})$\,$(1 \leq \alpha \leq r)$ satisfying the conclusion of 
Lemma \ref{wdlemma}(2). Then 
$$ ((U^{(\alpha)} \cap \ol{X}_{\beta},\ol{U}^{(\alpha)} \cap \ol{X}_{\beta}, 
\ol{\cX}^{(\alpha)}_{\beta}), t^{(\alpha)})$$
(where $\ol{\cX}^{(\alpha)}_{\beta}$ is the open formal subscheme of 
$\ol{\cX}^{(\alpha)}$ whose special fiber is equal to 
$\ol{U}^{(\alpha)} \cap \ol{X}_{\beta}$) 
for $(\alpha, \beta)$ with $Z'_{\alpha} \cap \ol{X}_{\beta} \not= \emptyset$ 
satisfies the conclusion of Lemma \ref{wdlemma}(2) for 
$(\ol{X}',Z')$, $f^*\cE$ and 
$\Sigma' := 
\prod_{\scriptstyle (\alpha, \beta) \atop \scriptstyle Z'_{\alpha} \cap \ol{X}_{\beta} \not= \emptyset} \Sigma_{\alpha}$. So we have $f^*\cE \in 
\Isocl(\ol{X}',Z')_{\Sigma(\ss)}$. Conversely, if 
$f^*\cE$ is in $\Isocl(\ol{X}',Z')_{\Sigma(\ss)}$, we can take 
charted smooth standard 
small frames $((U^{(\alpha,\beta)},\ol{U}^{(\alpha,\beta)},
\ol{\cX}^{(\alpha,\beta)}),
t^{(\alpha,\beta)})$\,($(\alpha, \beta)$ runs through the indices with 
$Z'_{\alpha} \cap \ol{X}_{\beta} \not= \emptyset$) with 
$U^{(\alpha,\beta)}$ containing the generic point of 
$Z'_{\alpha} \cap \ol{X}_{\beta}$ such that they 
satisfy the conclusion of Lemma \ref{wdlemma}(2) for $(\ol{X}',Z'), f^*\cE$ 
and $\Sigma'$. Noting the fact that $Z'_{\alpha} \cap \ol{X}_{\beta}$ is 
nonempty for some $\beta$ for any $1 \leq \alpha \leq r$, we see that 
these charted smooth standard 
small frames satisfy the conclusion of 
Lemma \ref{wdlemma}(2) for $(\ol{X},Z)$ and $\cE$. So 
$\cE$ is in the category $\Isocl(\ol{X},Z)_{\Sigma(\ss)}$. So we are done. 
\end{proof}

\begin{proposition}\label{sig_pre2}
Let $X \hra \ol{X}$ $(X' \hra \ol{X}')$ be an open immersion of 
smooth $k$-varieties such that $Z = \ol{X} \setminus X$ 
$(Z' = \ol{X}' \setminus X')$ is a simple normal crossing divisor and 
let us assume given a decomposition $\{Z_i\}_{i=1}^r$ $(\{Z'_i\}_{i=1}^r)$ 
of $Z$ $(Z')$. Let $\Sigma = \prod_{i=1}^r \Sigma_i$ be a subset of 
$\Z_p^r$ and let $f: (\ol{X}',Z') \lra (\ol{X},Z)$ be a morphism. Then$:$ 
\begin{enumerate}
\item 
If $f$ is a surjective strict smooth morphism with $f^*Z_i = Z'_i 
\,(1 \leq i \leq r)$, an object $\cE \in \Isocl(\ol{X},Z)$ is in 
$\Isocl(\ol{X},Z)_{\Sigma(\ss)}$ if and only if $f^*\cE$ is in 
$\Isocl(\ol{X}',Z')_{\Sigma(\ss)}$. 
\item 
If $f$ is log smooth and if there exists some 
$n := (n_i)_{i=1}^r$ with $n_i$'s prime to $p$ 
such that $f^*Z_i = n_iZ'_i \,(1 \leq i \leq r)$ 
etale locally on $\ol{X}$ and $\ol{X}'$, the pull-back $f^*\cE$ 
of an object $\cE$ in $\Isocl(\ol{X},Z)_{\Sigma(\ss)}$ is contained in 
$\Isocl(\ol{X}',Z')_{n\Sigma(\ss)}$. 
\end{enumerate}
\end{proposition}

\begin{proof}
First let us prove (1). By Proposition \ref{sig_pre1}(1), 
we may replace $\ol{X}, \ol{X}'$ by 
$\ol{X} \setminus Z_{\sing}, f^{-1}(\ol{X} \setminus Z_{\sing})$, 
respectively. (Then $Z, Z'$ will be smooth.) Then, 
for an open covering $\ol{X} = \bigcup_{\beta} \ol{X}_{\beta}$ by 
finite number of open subschemes, 
we may replace $\ol{X}, \ol{X}'$ by 
$\coprod_{\beta} \ol{X}_{\beta}, 
\coprod_{\beta}f^{-1}(\ol{X}_{\beta})$ 
by Proposition \ref{sig_pre1}(2). Then, 
to prove the assertion in this situation, 
we may replace $\ol{X}, \ol{X}'$ by 
$\ol{X}_{\beta}, f^{-1}(\ol{X}_{\beta})$. 
So we may assume that $\ol{X}$ is affine connected, $Z$ is a connected 
smooth divisor in $\ol{X}$ and $r=1$ (and so $Z=Z_1$). 
Moreover, we may assume that there exists a charted smooth standard small 
frame $((X, \ol{X},\ol{\cX}),t)$ enclosing 
$(X, \ol{X})$. Moreover, by taking an 
open covering $\ol{X}' = \bigcup_{\beta} 
\ol{X}'_{\beta}$ by finite number of affine opens 
and replacing $\ol{X}'$ by $\coprod_{\beta} \ol{X}'_{\beta}$, 
we may assume that $\ol{X}'$ is a disjoint union of connected affine schemes 
$\ol{X}'_{\beta}$. Then, by Lemma \ref{lifting} and Remark \ref{liftingrem}, 
there exists a morphism of charted smooth standard small frames 
$\ti{f}_{\beta}: 
((\ol{X}'_{\beta} \setminus Z', \ol{X}'_{\beta},\ol{\cX}_{\beta}),t'_{\beta}) 
\lra ((X,\ol{X},\ol{\cX}),t)$ compatible with the composite 
$(\ol{X}'_{\beta} \setminus Z', \ol{X}'_{\beta}) \hra (X',\ol{X}') \os{f}{\lra}  (X,\ol{X})$ such that $\ti{f}: \ol{\cX}'_{\beta} \lra \ol{\cX}$ is 
smooth and that $\ti{f}^*t = t'_{\beta}$. Then 
$\cE$ induces the log-$\nabla$-module $(E,\nabla)$ on $\ol{\cX}_K$ with respect  to $t$, $f^*\cE$ induces the log-$\nabla$-module $(E'_{\beta},
\nabla'_{\beta})$ on $\ol{\cX}'_{\beta,K}$ with respect to $t'_{\beta}$ for 
each $\beta$ and we have $\ti{f}_{\beta,K}^*(E,\nabla) = 
(E'_{\beta},\nabla'_{\beta})$, where $\ti{f}_{\beta,K}: \ol{\cX}'_{\beta,K} 
\lra \ol{\cX}_{K}$ is the morphism induced by $\ti{f}_{\beta}$. 
By definition, $\cE$ has exponents in 
$\Sigma$ (with semisimple residues) if and only if 
the exponents of $(E,\nabla)$ along $\{t=0\}$ is contained in $\Sigma$ 
(and there exists some $P(x) \in \Z_p[x]$ without multiple roots 
with $P(\res)=0$, where $\res$ is the residue of 
$(E,\nabla)$ along $\{t=0\}$) and $f^*\cE$ has exponents in 
$\Sigma$ (with semisimple residues) if and only if, for any $\beta$, 
the exponents of $(E'_{\beta},\nabla'_{\beta})$ 
along $\{t'_{\beta}=0\}$ is contained in $\Sigma$ 
(and there exists some $P_{\beta}(x) \in \Z_p[x]$ without multiple roots 
with $P_{\beta}(\res'_{\beta})=0$, where 
$\res'_{\beta}$ is the residue of $(E'_{\beta},\nabla'_{\beta})$ 
along $\{t'_{\beta}=0\}$). 
Note now that, since $\coprod_{\beta} \ti{f}_{\beta}$ is 
surjective and smooth, the induced map 
$$ \End(E |_{\{t=0\}}) \lra \prod_{\beta} 
\End(E'_{\beta}|_{\{t'_{\beta}=0\}}), $$ 
which sends $\res$ to $(\res'_{\beta})_{\beta}$ 
is injective. From this injectivity, we see that, for any $P(x) \in \Z_p[x]$, 
we have $P(\res) =0$ if and only if $P(\res_{\beta})=0$ for 
any $\beta$. So 
the exponents of $(E,\nabla)$ along $\{t=0\}$ is contained in $\Sigma$ 
(and there exists some $P(x) \in \Z_p[x]$ without multiple roots 
with $P(\res)=0$) if and only if 
the exponents of $(E'_{\beta},\nabla'_{\beta})$ 
along $\cZ_{\beta}$ is contained in $\Sigma$ 
(and there exists some $P_{\beta}(x) \in \Z_p[x]$ without multiple roots 
with $P_{\beta}(\res'_{\beta})=0$) for any $\beta$. 
So $\cE$ has exponents in 
$\Sigma$ (with semisimple residues) 
if and only if $f^*\cE$ has exponents in 
$\Sigma$ (with semisimple residues). Hence we have proved the assertion 
(1). \par 
Next we prove (2). By the argument as in the proof of (1), we may assume that 
$\ol{X}$ is affine connected, $Z$ is a connected 
smooth divisor in $\ol{X}$ and $r=1$ (and so $Z=Z_1$). 
By using Proposition \ref{sig_pre1}(1), we may assume that $Z'$ is smooth, and 
by replacing $\ol{X}'$ by its affine open subschemes, we may assume 
 $\ol{X}'$ is affine connected and $Z'$ is a connected smooth divisor. 
Since we may work etale locally on $\ol{X}$ and $\ol{X}'$ by (1), 
we may assume that $Z, Z'$ are defined as the zero locus of $\ol{t}, \ol{t}'$ 
with $f^*\ol{t} = {\ol{t}'}^n$. 
Also, we may assume that there exists a charted 
smooth standard small frame 
$((X, \ol{X},\ol{\cX}),t)$ enclosing 
$(X, \ol{X})$ such that $t$ lifts $\ol{t}$. Then, 
by Lemma \ref{lifting}, there exists 
a charted smooth standard small frame 
$((X',\ol{X}',\ol{\cX}'),t')$ enclosing $(X',\ol{X}')$ 
such that $t' \in 
\Gamma(\ol{\cX}',\cO_{\ol{\cX}'})$ is a lift of $\ol{t}'$ and a 
morphism $\ti{f}: ((X',\ol{X}',\ol{\cX}'),t') \lra ((X,\ol{X},\ol{\cX}),t)$ 
with $\ti{f}^*t = {t'}^n$ which is compatible with 
$f:(X',\ol{X}') \lra (X,\ol{X})$. Then 
$\cE$ induces the log-$\nabla$-module $(E,\nabla)$ on $\ol{\cX}_K$ with respect  to $t$, $f^*\cE$ induces the log-$\nabla$-module $(E',
\nabla')$ on $\ol{\cX}'_{K}$ with respect to $t'$ and 
we have $\ti{f}_{K}^*(E,\nabla) = 
(E',\nabla')$, where $\ti{f}_{K}: \ol{\cX}'_{K} 
\lra \ol{\cX}_{K}$ is the morphism induced by $\ti{f}$. 
When $\cE$ has exponents in 
$\Sigma$ (with semisimple residues), 
the exponents of $(E,\nabla)$ along $\{t=0\}$ is contained in $\Sigma$ 
(and there exists some $P(x) \in \Z_p[x]$ without multiple roots 
with $P(\res)=0$, where $\res$ is the residue of 
$(E,\nabla)$ along $\{t=0\}$). Then, by the equality $\ti{f}^*t = {t'}^n$, 
we see that the residue $\res'$ of $(E',\nabla')$ 
along $\{t'=0\}$ satisfies $\res' = n\ti{f}^*_K(\res)$. Hence 
the exponents of $(E',\nabla')$ 
along $\{t'=0\}$ is contained in $n \Sigma$ 
(and we have $P(\res'/n)=0$). Hence $f^*\cE$ has exponents in 
$n\Sigma$ (with semisimple residues) and so we have proven 
the assertion (2). 
\end{proof}

\section{Convergent isocrystals on stacks}
In this section, we give a definition of the category of 
(log) convergent isocrystals on (fine log) algebraic stacks and prove 
the equivalences \eqref{seq1}, \eqref{seq2} and \eqref{seq3}. 

\subsection{Basic definitions}

In this subsection, we give a definition of the category of 
(log) convergent isocrystals on (fine log) algebraic stacks. 
For an algebraic stack $X$ of finite type over $k$, we define the 
category $\Sch/X$ as follows: An object in $\Sch/X$ is an object $Y$ in $\Sch$ 
 endowed with a $1$-morphism $a_Y: Y \lra X$. A morphism 
$(Y,a_Y) \lra (Y',a_{Y'})$ in $\Sch/X$ is a morphism of schemes 
$\varphi: Y \lra Y'$ endowed with a $2$-isomorphism 
$\tau: a_{Y'} \circ \varphi \os{=}{\lra} a_Y$. 
(Note that, when $X$ is in $\Sch$, the category $\Sch/X$ is nothing 
but the category of objects in $\Sch$ over $X$.) 
Using this category, we define the notion of convergent isocrystals on 
algebraic stacks as follows: 

\begin{definition}\label{stackconv}
Let $X$ be an algebraic stack of finite type over $k$. 
We define the category $\Isoc(X)$ of convergent isocrystals on $X$ 
over $K$ $($resp. the category $\FIsoc(X)$ of 
convergent $F$-isocrystals on $X$ over $K$, the category 
$\FIsoc(X)^{\circ}$ of unit-root convergent $F$-isocrystals on $X$ over $K)$ 
as the 
category of pairs $$(\{\cE_Y\}_{Y \in {\rm Ob}(\Sch/X)}, 
\{\varphi_{\cE}\}_{\varphi:Y \ra Y' \in 
{\rm Mor}(\Sch/X)}),$$ where $\cE_Y \in \Isoc(Y)$ 
$($resp. $\cE_Y \in \FIsoc(Y)$, $\cE_Y \in \FIsoc(Y)^{\circ})$
and $\varphi_{\cE}$ is an isomorphism 
$\varphi^*\cE_{Y'} \os{=}{\lra} \cE_Y$ in $\Isoc(Y)$ 
$($resp. $\FIsoc(Y)$, 
$\FIsoc(Y)^{\circ})$ satisfying the cocycle condition 
$\varphi_{\cE} \circ \varphi^* \varphi'_{\cE} = 
(\varphi' \circ \varphi)_{\cE}$ for $Y \os{\varphi}{\lra} Y' 
\os{\varphi'}{\lra} Y''$ in $\Sch/X$. 
\end{definition}

For a scheme $X$ in $\Sch$, 
the above definition of the categories 
$\Isoc(X)$, $\FIsoc(X)$, $\FIsoc(X)^{\circ}$ 
coincides with the usual definition: The 
equivalence is given by 
\begin{align*}
& 
\cE \mapsto 
(\{a_Y^*\cE \}_{(Y,a_Y) \in {\rm Ob}(\Sch/X)}, 
\{\varphi_{\cE}: \varphi^*a_{Y'}^*\cE \os{=}{\ra} a_Y^*\cE 
\}_{\varphi:Y \ra Y' \in {\rm Mor}(\Sch/X)}), \\ 
& (\{\cE_Y\}_{(Y,a_Y) \in {\rm Ob}(\Sch/X)}, 
\{\varphi_{\cE}\}_{\varphi:Y \ra Y' \in 
{\rm Mor}(\Sch/X)}) \mapsto \cE_X. 
\end{align*}

Let $X$ be a Deligne-Mumford stack of finite type over $k$, 
let $\epsilon: X_0 \lra X$ be an etale surjective morphism from 
a scheme $X_0$ in $\Sch$ and let 
$X_n \,(n=0,1,2)$ 
be the $(n+1)$-fold fiber product of $X_0$ over $X$. Then we have 
a $2$-truncated simplicial scheme $X_{\b}$ endowed with a morphism 
$X_{\b} \lra X$ and we have canonical functors 
\begin{align}
& \Isoc(X) \lra \Isoc(X_{\b}), \,\,\,\, \FIsoc(X) \lra \FIsoc(X_{\b}), 
\label{func} 
\\ 
& \FIsoc(X)^{\circ} \lra \FIsoc(X_{\b})^{\circ}. \nonumber 
\end{align} 
Then we have the following: 

\begin{proposition}\label{func_eq}
Let the notations be as above. Then the functors \eqref{func} are 
equivalences. 
\end{proposition}

\begin{proof}
We treat only the case for $\Isoc(X)$. (The other cases can be proven 
in the same way.) 
We define the inverse functor as follows: Given an object 
$\cE_{\b} \in \Isoc(X_{\b})$ and $Y \lra X$ in $\Sch/X$, 
let us put $Y_{\b} := X_{\b} \times_X Y$. Then $\cE_{\b}$ naturally 
induces an object in $\Isoc(Y_{\b})$, which we denote by 
$\cE'_{\b}$. Then, since $Y_{\b} \lra Y$ is an etale \v{C}ech 
hypercovering of schemes, we have the equivalence 
$\Isoc(Y) \os{=}{\lra} \Isoc(Y_{\b})$ by \cite[4.4]{ogus}. So 
$\cE'_{\b}$ induces an object $\cE_Y$ in $\Isoc(Y)$. Then 
$\{\cE_Y\}_{Y \in \Sch/X}$ forms in a narutal way an object in 
$\Isoc(X)$. One can prove that this functor gives 
the desired inverse functor by using the etale descent property 
for the category of convergent isocrystals \cite[4.4]{ogus} again. 
\end{proof}

\begin{example}\label{exam}
Let $X$ be an object in $\Sch$ and let $G$ be a finite etale group scheme 
over $k$ acting on $X$. Then we have the quotient stack $[X/G]$, which is a 
Deligne-Mumford stack. $X_0: = X \lra [X/G]$ is a surjective 
finite etale morphism and if we denote the $(n+1)$-fold fiber 
product of $X$ over $[X/G]$ by $X_n \,(n=0,1,2)$, we have 
$X_n \cong X \times_k G^{n} $. In this case, an object in 
$\Isoc([X/G]) \cong \Isoc(X_{\b})$ 
is nothing but a convergent isocrystal $\cE$ on $X$ endowed with 
an equivariant action of $G$. 
\end{example}

In the following, we generalize the above definition of the category 
of convergent ($F$-)isocrystals on algebraic stacks to the case with 
log structures. First recall that the notion of a fine log structure 
on an algebraic stack $X$ is defined in \cite[5.1]{olsson} as 
a sheaf of monoids $M_X$ on the lisse-etale site $X_{\liset}$ endowed 
with a monoid homomorphism 
$M_X \lra \cO_{X_{\liset}}$ such that, for any object $U$ in 
$X_{\liset}$, the log structure $M_X |_{U_{\et}}$ on $U_{\et}$ is fine and 
that, for any morphism $\varphi:U' \lra U$ in $X_{\liset}$, the 
map $\varphi^*(M_X|_{U_{\et}}) \lra M_X|_{U'_{\et}}$ is an 
isomorphism. 
We call a pair $(X,M_X)$ of an algebraic stack $X$ and a fine log 
structure $M_X$ on it a fine log algebraic stack. 
For a fine log algebraic stack $(X,M_X)$ and a morphism of algebraic 
stacks $f:Y \lra X$, one can define the pull-back log structure 
$f^*M_X$ on $Y_{\liset}$ (see \cite[p.773]{olsson}). When $Y$ is 
a scheme, the restriction $f^*M_X |_{Y_{\et}}$ of $f^*M_X$ to the 
etale site $Y_{\et}$ gives a fine log structure on $Y$ 
(see \cite[5.3]{olsson}). Now 
we define the notion of locally free 
log convergent isocrystals on fine log algebraic stacks as follows: 

\begin{definition}\label{logstackconv}
Let $(X,M_X)$ be a fine log algebraic stack of finite type over $k$. 
We define the category $\Isocl(X,M_X)$ of locally free log 
convergent isocrystals on $(X,M_X)$ 
over $K$ $($resp. the category $\FIsocl(X,M_X)$ of locally free log 
convergent $F$-isocrystals on $(X,M_X)$ over $K)$ 
as the category of pairs $$(\{\cE_Y\}_{(Y,a_Y) \in {\rm Ob}(\Sch/X)}, 
\{\varphi_{\cE}\}_{\varphi:Y \ra Y' \in 
{\rm Mor}(\Sch/X)}),$$ where $\cE_Y \in \Isocl(Y,a_Y^*M_X|_{Y_{\et}})$ 
$($resp. $\cE_Y \in \FIsocl(Y,a_Y^*M_X|_{Y_{\et}}))$
and $\varphi_{\cE}$ is an isomorphism 
$\varphi^*\cE_{Y'} \os{=}{\lra} \cE_Y$ in $\Isocl(Y,a_Y^*M_X|_{Y_{\et}})$  $($resp. $\FIsocl(Y,a_Y^*M_X|_{Y_{\et}}))$ satisfying the cocycle condition 
$\varphi_{\cE} \circ \varphi^* \varphi'_{\cE} = 
(\varphi' \circ \varphi)_{\cE}$ for $Y \os{\varphi}{\lra} Y' 
\os{\varphi'}{\lra} Y''$ in $\Sch/X$. 
\end{definition}

Note that, 
for $(Y,M_Y) \in \LSch$ and a strict etale \v{C}ech hypercovering 
$(Y_{\b}, M_{Y_{\b}}) \lra (Y,M_Y)$, we can prove the equivalence of 
categories 
$\Isocl(Y,M_Y) \os{=}{\lra} \Isocl(Y_{\b}, \allowbreak M_{Y_{\b}})$ 
in the same way as \cite[4.4]{ogus}. (See also 
\cite[5.1.7]{crysI}.) So we can prove the following proposition 
in the same way as Proposition \ref{func_eq} (so we omit the proof): 

\begin{proposition}
Let $(X,M_X)$ be a fine log Deligne-Mumford stack of finite type over $k$, 
let $\epsilon: X_0 \lra X$ be an etale surjective morphism from 
a scheme $X_0$ in $\Sch$. Let 
$X_n \,(n=0,1,2)$ 
be the $(n+1)$-fold fiber product of $X_0$ over $X$ and let 
$M_{X_n}$ be the restriction of $M_X$ to $X_{n,\et}$. 
Then the canonical functors 
\begin{equation*}
\Isocl(X,M_X) \lra \Isocl(X_{\b},M_{X_{\b}}), \,\, 
\FIsocl(X,M_X) \lra \FIsocl(X_{\b},M_{X_{\b}}) 
\end{equation*}
are equivalences. 
\end{proposition}

Once we define the categories 
$\Isoc(X), \FIsoc(X), \FIsoc(X)^{\circ}$ (resp. 
$\Isocl(X, \allowbreak M_X), \FIsocl(X,M_X)$) for algebraic stacks $X$ 
(resp. fine log algebraic stacks $(X, \allowbreak M_X)$), we can generalize 
the definition to the case of diagram of algebraic stacks 
(resp. diagram of fine log algebraic stacks) as follows. 

\begin{definition} 
Let $\Sta$ $($resp. $\LSta)$ be the $2$-category of 
algebraic stacks $($resp. fine log algebraic stacks$)$ 
of finite type over $k$. Then, for a category $\cC$ 
$($regarded also as a discrete $2$-category$)$
and a $2$-functor $\Phi: \cC \lra \Sta$ $($resp. $\Phi: \cC \lra \LSta)$, 
we define the category 
$\Isoc(\Phi)$, $\FIsoc(\Phi)$, $\FIsoc(\Phi)^{\circ}$ 
$($resp. $\Isocl(\Phi)$, $\FIsocl(\Phi))$ 
as the 
category of pairs $$(\{\cE_Y\}_{Y \in {\rm Ob}(\cC)}, 
\{\varphi_{\cE}\}_{\varphi:Y \ra Y' \in 
{\rm Mor}(\cC)}),$$ where $\cE_Y$ is an object in 
$\Isoc(\Phi(Y))$, $\FIsoc(\Phi(Y))$, $\FIsoc(\Phi(Y))^{\circ}$ 
$($resp. $\Isocl \allowbreak (\Phi(Y))$, $\FIsocl(\Phi(Y)))$
and $\varphi_{\cE}$ is an isomorphism 
$\Phi(\varphi)^*\cE_{Y'} \os{=}{\lra} \cE_Y$ 
satisfying the cocycle condition 
$\varphi_{\cE} \circ \Phi(\varphi)^* \varphi'_{\cE} = 
(\varphi' \circ \varphi)_{\cE}$ for $Y \os{\varphi}{\lra} Y' 
\os{\varphi'}{\lra} Y''$ in $\cC$. 
\end{definition}

In particular, we can define the categories 
$\Isoc(X_{\b})$, $\FIsoc(X_{\b})$, $\FIsoc(X_{\b})^{\circ}$
(resp. $\Isocl(X_{\b},M_{X_{\b}})$, $\FIsocl(X_{\b},M_{X_{\b}})$) 
for $2$-truncated simplicial algebraic stacks $X_{\b}$
(resp. $2$-truncated simplicial fine log algebraic stacks 
$(X_{\b},M_{X_{\b}})$). It is easy to see that, for 
$X \in \Sta$ and a $2$-truncated etale \v{C}ech hypercovering 
$X_{\b} \lra X$, we have the equivalences 
$$ \Isoc(X) \os{=}{\ra} 
\Isoc(X_{\b}), \,\,\, \FIsoc(X) \os{=}{\ra} 
\FIsoc(X_{\b}), \,\,\, \FIsoc(X)^{\circ} \os{=}{\ra} 
\FIsoc(X_{\b})^{\circ} $$ 
and for $X \in \LSta$ and a $2$-truncated strict etale \v{C}ech hypercovering 
$(X_{\b},M_{X_{\b}}) \lra (X,M_X)$, we have the equivalences 
$$ \Isocl(X,M_X) \os{=}{\ra} 
\Isocl(X_{\b},M_{X_{\b}}), \,\,\, \FIsocl(X,M_X) \os{=}{\ra} 
\FIsocl(X_{\b},M_{X_{\b}}). $$ 

Next we give a definition of the category of locally free log convergent 
isocrystals on certain algebraic stacks `with exponent condition'. 
Let $(X,M_X)$ be a fine log algebraic stack of finite type over $k$ 
satisfying the following condition $(*)$: \\
\quad \\
$(*)$ \,\,\, There exists a  smooth surjective morphism 
$a_Y: Y \lra X$ with $Y \in \Sch$ such that $Y$ is 
smooth over $k$ and that 
$M_Y :=  a_Y^*M_X |_{Y_{\et}}$ is associated to a simple normal 
crossing divisor $Z$ on $Y$. \\
\quad \\
Under this condition, we define the notion of a decomposition of $M_X$ as 
follows: 

\begin{definition}\label{defdecomp}
Let $(X,M_X)$ be as above. Then a family of sub fine log structures 
$\{M_{X,i}\}_{i=1}^r$ of $M_X$ 
is called a decomposition of $M_X$ 
$($with respect to $a_Y:Y \lra X)$ 
if the log structure 
$M_{Y,i} := a_Y^*M_{X,i}|_{Y_{\et}}$ is associated to a simple normal 
crossing subdivisor $Z_i$ of $Z$ $(1 \leq i \leq r)$ with 
$Z = \sum_{i=1}^rZ_i$. 
\end{definition}

Concerning the above definition, 
we have the following independence result. 

\begin{lemma}
Let $(X,M_X)$ be as above and let $\{M_{X,i}\}_{i=1}^r$ be 
a family of sub fine log structures of $M_X$. Then, if  
$\{M_{X,i}\}_{i=1}^r$ is a decomposition of $M_X$ with respect to 
one morphism $a_Y:Y \lra X$ as in $(*)$, it is 
a decomposition of $M_X$ with respect to 
any morphism $a_Y:Y \lra X$ as in $(*)$. 
\end{lemma}

\begin{proof}
Assume that $\{M_{X,i}\}_{i=1}^r$ is a decomposition of $M_X$ with respect to 
a morphism $a_Y:Y \lra X$ as in $(*)$ and let us take another morphism 
$a_{Y'}: Y' \lra X$ as in $(*)$. Then, by taking a suitable surjective etale 
morphism $f: Y'' \lra Y \times_X Y'$, we see that the composite 
$a_{Y''}: Y''\os{f}{\lra} Y \times_X Y' \lra X$ also satisfies the condition 
$(*)$. Also, the morphisms $Y'' \lra Y, Y'' \lra Y'$ induced by 
projection are smooth. 
For $m=0,1,2$, let $Y''_m$ be the $(m+1)$-fold fiber product of $Y''$ over 
$Y'$. Let $\pi_j: Y''_1 \lra Y''_0 = Y'' \,(j=0,1)$ be the projections and 
let $a_{Y''_m}$ be the natural morphism $Y''_m \lra X$ induced by 
(any) projection $Y''_m \lra Y''$ and $a_{Y''}$. \par 
By assumption, $a_Y^*M_{X}|_{Y_{\et}}$ 
is associated to a simple normal crossing 
divisor $Z$ on $Y$ and $a_Y^*M_{X,i}|_{Y_{\et}}$ 
is associated to a simple normal 
crossing subdivisor $Z_i$ of $Z$ with $Z= \sum_{i=1}^r Z_i$. 
By \cite[5.3]{olsson}, we see that the log structure 
$a_Y^*M_X$ (resp. $a_Y^*M_{X,i}$) on $Y_{\liset}$ is also 
associated to $Z$ (resp. $Z_i$). So, if we define 
$Z''$ (resp. $Z''_i$) to be the pull-back of $Z$ (resp. $Z_i$) to 
$Y''$, we see that the log structure $a_{Y''}^*M_X$ (resp. 
$a_{Y''}^*M_{X,i}$) is associated to $Z''$ (resp. $Z''_i$), and 
by pulling it back by $\pi_j$, we see that 
the log structure $a_{Y''_1}^*M_X$ (resp. 
$a_{Y''_1}^*M_{X,i}$) is associated to $\pi_0^*Z''$ (resp. $\pi_0^*Z''_i$) 
and also to $\pi_1^*Z''$ (resp. $\pi_1^*Z''_i$). 
So we have the equality $\pi_0^*Z'' = \pi_1^*Z''$ (resp. 
$\pi_0^*Z''_i = \pi_1^*Z''_i$) and it satisfies the cocycle condition on 
$Y''_2$. So there exists a normal crossing divisor $Z'$ (resp. $Z'_i$) 
on $Y'$ which is pulled back to $Z''$ (resp. $Z''_i$) by the morphism 
$Y'' \lra Y'$. (Then we see also the equality $Z' = \sum_{i=1}^r Z'_i$.) 
Let $N$ (resp. $N_i$) be the log structure on 
$Y'_{\et}$ associated to $Z'$ (resp. $Z'_i$). Then 
$N$ and $M_{Y'} := a_{Y'}^*M_X |_{Y'_{\et}}$ 
(resp. $N_i$ and $M_{Y',i} := a_{Y'}^*M_{X,i}|_{Y'_{\et}}$) 
coincide on $Y''_{\b,\et}$. So they are equal by 
\cite[Theorem A.1]{olsson}. Hence 
$M_{Y'}$ 
(resp. $M_{Y',i}$) 
is associated to $Z'$ (resp. $Z'_i$). So, by the condition $(*)$, 
we see that $Z'$ (which is a priori a normal crossing divisor) is in 
fact a simple normal crossing divisor and so are $Z'_i$'s. 
Hence $\{M_{X,i}\}_{i=1}^r$ is also a decomposition of $M_X$ with respect to 
the morphism $a_{Y'}:Y' \lra X$. So we are done. 
\end{proof}

Now we define the category of locally free log convergent isocrystals 
with exponents in $\Sigma$ for certain fine log algebraic stacks, 
as follows: 

\begin{definition}\label{defstacksigma}
Let $(X,M_X)$ be a fine log algebraic stack of finite type 
over $k$ satisfying the condition 
$(*)$ above and let $\{M_{X,i}\}_{i=1}^r$ be a decomposition of 
$M_X$. Let us take $\Sigma = \prod_{i=1}^r \Sigma_i \subseteq \Z_p^r$ 
and let us take $a_Y: Y \lra X, M_Y, Z = \sum_{i=1}^r Z_i$ as in 
$(*)$ and Definition \ref{defdecomp}. 
Then we say that a locally free log convergent isocrystal $\cE$ on 
$(X,M_X)$ over $K$ 
has exponents in $\Sigma$ $($resp. has exponents in $\Sigma$ 
with semisimple residues$)$ with respect to the decomposition 
$\{M_{X,i}\}_{i=1}^r$ 
if the object $\cE_{Y} \in \Isocl(Y,M_Y)$ 
induced by $\cE$ has exponents in $\Sigma$ $($resp. 
has exponents in $\Sigma$ with semisimple residues$)$ with 
respect to the decomposition $\{Z_i\}_{i=1}^r$ of $Z$. 
We denote the category of locally free log convergent isocrystals on 
$(X,M_X)$ having exponents in $\Sigma$ 
$($resp. having exponents in $\Sigma$ with semisimple residues$)$ 
by 
$\Isocl(X,M_X)_{\Sigma}$ $($resp. $\Isocl(X,M_X)_{\Sigma\ss})$. 
\end{definition} 

We have the following independence result. 

\begin{lemma}
The above defininition is independent of the choice of 
$a_Y: Y \lra X$ as in $(*)$. 
\end{lemma}

\begin{proof}
Let $\cE \in \Isocl(X,M_X)$ and 
let us take two morphisms 
$a_Y: Y \lra X$ 
$a_{Y'}: Y' \lra X$ as in $(*)$. Then, by taking a suitable surjective etale 
morphism $f: Y'' \lra Y \times_X Y'$, we see that the composite 
$a_{Y''}: Y''\os{f}{\lra} Y \times_X Y' \lra X$ also satisfies the condition 
$(*)$, and the morphisms $Y'' \lra Y, Y'' \lra Y'$ induced by 
projection are surjective and smooth. So, it suffices to prove the 
following: For two morphism 
$a_Y: Y \lra X$, $a_{Y'}: Y' \lra X$ as in $(*)$ 
and a surjective smooth morphism 
$f: Y' \lra Y$ with $a_Y \circ f$ $2$-isomorphic to $a_{Y'}$, 
$\cE$ has exponents in $\Sigma$ (with semisimple residues) for the morphism 
$a_Y: Y \lra X$ if and only if so does it for the morphism 
$a_{Y'}: Y' \lra X$. If we take $Z = \sum_{i=1}^r Z_i$ as in 
$(*)$ and Definition 
\ref{defdecomp} and if we put $Z' := f^{*}Z, Z'_i := f^{*}Z_i$, 
the claim we should prove is rewritten as follows: 
$a_Y^*\cE$ has exponents in $\Sigma$ (with semisimple residues) with respect 
to the decomposition $\{Z_i\}_i$ of $Z$ if and only if 
$a_{Y'}^*\cE \cong f^*(a_Y^*\cE)$ has exponents in 
$\Sigma$ (with semisimple residues) with respect 
to the decomposition $\{Z'_i\}$ of $Z'$. This follows from 
Proposition \ref{sig_pre2}(1). So we are done. 
\end{proof}

We can also define the category of locally free log convergent isocrystals 
with exponent condition for certain diagram of fine log algebraic 
stacks. 

\begin{definition}\label{flas_ex}
Let $(X_{\b},M_{X_{\b}})$ be a diagram of 
fine log algebraic stacks of finite type over $k$ 
indexed by a small category $\cC$ such that 
each $(X_{c},M_{X_{c}}) \,(c \in \cC)$ satisfies the condition $(*)$ and 
let $\{M_{X_{\b},i}\}_{i=1}^r$ be a family of sub fine log structures 
of $M_{X_{\b}}$ such that, for any $c \in \cC$, the induced family 
$\{M_{X_c,i}\}_{i=1}^r$ gives a decomposition of 
$M_{X_c}$. $($We call such a family $\{M_{X_{\b},i}\}_{i=1}^r$ a 
decomposition of $M_{X_{\b}}.)$ 
Then we say that an object $\cE_{\b}$ in $\Isocl(X_{\b},M_{X_{\b}})$  
has exponents in $\Sigma$ $($resp. has exponents in $\Sigma$ 
with semisimple residues$)$ with respect to the decomposition 
$\{M_{X_{\b},i}\}_{i=1}^r$ 
if, for any $c \in \cC$, the object $\cE_{c} \in \Isocl(X_{c},M_{X_{c}})$
induced by $\cE_{\b}$ is contained in $\Isocl(X_{c},M_{X_{c}})_{\Sigma}$ 
$($resp. $\Isocl(X_{c},M_{X_{c}})_{\Sigma\ss}).$ 
We denote the category of objects in in $\Isocl(X_{\b},M_{X_{\b}})$ 
having exponents in $\Sigma$ $($resp. having exponents in $\Sigma$ 
with semisimple residues$)$ with respect to the decomposition 
$\{M_{X_{\b},i}\}_{i=1}^r$ by $\Isocl(X_{\b},M_{X_{\b}})_{\Sigma}$ 
$($resp. $\Isocl(X_{\b},M_{X_{\b}})_{\Sigma\ss}).$ 
\end{definition}

We have a functoriality property of the category of the form 
$\Isocl(X,M_X)_{\Sigma(\ss)}$ for certain morphism of 
fine log algebraic stacks. 

\begin{proposition}\label{ex_ex}
Let $(X,M_X)$ be a fine log algebraic stack of finite type over $k$ 
satisfying the 
condition $(*)$ and let $\{M_{X,i}\}_{i=1}^r$ be a decomposition 
of $M_X$. Let us assume given another fine log algebraic stack 
$(X',M_{X'})$ of finite type over $k$ and a strict smooth morphism 
$f:(X',M_{X'}) \lra (X,M_X)$ over $k$. 
Then, $(X',M_{X'})$ also satisfies the 
condition $(*)$ and $\{f^*M_{X,i}\}_{i=1}^r$ gives a decomposition 
of $M_{X'}$. Moreover, for a subset $\Sigma = \prod_{i=1}^r \Sigma_i$ 
in $\Z_p^r$, $f$ induces the functor 
$f^*:\Isocl(X,M_X)_{\Sigma(\ss)} \lra \Isocl(X',M_{X'})_{\Sigma(\ss)}$. 
\end{proposition}

\begin{proof}
Let us take $a_Y: Y \lra X, Z = \sum_i Z_i$ as in $(*)$. Then 
we can take a diagram 
$a_{Y'}: Y' \lra Y \times_X X' \lra X'$ for some $Y' \in \Sch$ such that the 
first map is surjective etale. Denote the composite 
$Y' \lra Y \times_X X' \os{\text{proj.}}{\lra} Y$ by $b$. 
Then $a_{Y'}$ is surjective smooth and the log structure 
$a_{Y'}^*M_{X'}$ is equal to $a_{Y'}^*f^*M_X \cong b^*a_Y^*M_X$, 
which is equal to 
the log structure associated to the simple normal crossing divisor 
$b^{-1}(Z)$ in $Y'$. So $(X',M_{X'})$ satisfies the 
condition $(*)$. Moreover, since the log structure 
$a_{Y'}^*f^*M_{X,i}|_{Y'_{\et}} \cong 
b^*a_Y^*M_{X,i} |_{Y'_{\et}}$ is associated to 
the simple normal crossing subdivisor 
$b^{-1}(Z_i)$ of $b^{-1}(Z)$ in $Y'$, we see that $\{f^*M_{X,i}\}_{i=1}^r$ gives a decomposition 
of $M_{X'}$. \par 
Let us prove the last asssertion. Let us take an object $\cE$ in 
$\Isocl(X,M_X)_{\Sigma(\ss)}$. Then the object $\cE_Y$ in $\Isocl(Y,Z)$ induced by $\cE$ is contained in $\Isocl(Y,Z)_{\Sigma(\ss)}$. Since 
the morphism $b: (Y',b^{-1}(Z)) \lra (Y,Z)$ is strict smooth, 
$b^*(\cE_Y) \cong 
(f^*\cE)_{Y'}(:=$ the object in $\Isocl(Y',b^{-1}(Z))$ induced 
by $f^*\cE$) is contained in the category 
$\Isocl \allowbreak (Y',b^{-1}(Z))_{\Sigma(\ss)}$ by 
Proposition \ref{sig_pre2}(1). Hence we have $f^*\cE \in 
\Isocl(X',M_{X'})_{\Sigma(\ss)}$, as desired. 
\end{proof}

\begin{corollary}\label{starlogcor}
Let $(X,M_X)$ be a fine log algebraic stack of finite type over $k$ 
satisfying the condition $(*)$ and let $(X_{\b},M_{X_{\b}}) \lra 
(X,M_X)$ be a $2$-truncated strict etale \v{C}ech hypercovering. Let 
$\{M_{X,i}\}_{i=1}^r$ be a decomposition of $M_X$, let 
$\{M_{X_{\b},i}\}$ be the induced decomposition of $M_{X_{\b}}$ and 
let $\Sigma = \prod_{i=1}^r \Sigma_i$ be a subset of $\Z_p^r$. Then we have 
an equivalence of categories 
\begin{equation}\label{galn}
\Isocl(X,M_X)_{\Sigma(\ss)} \os{=}{\lra} 
\Isocl(X_{\b},M_{X_{\b}})_{\Sigma(\ss)}. 
\end{equation}
\end{corollary}

\begin{proof}
By Proposition \ref{ex_ex} (see also Definition \ref{flas_ex}), 
we have the functor \eqref{galn}, and it is fully faithful because 
the functor $\Isocl(X,M_X) {\lra} 
\Isocl(X_{\b},M_{X_{\b}})$ is an equivalence. Let us prove the 
essential surjectivity. So let us take an object $\cE_{\b} \in 
\Isocl(X_{\b},M_{X_{\b}})_{\Sigma(\ss)}$ and take an object 
$\cE \in \Isocl(X,M_X)$ which is sent to $\cE_{\b}$. 
Then 
there exists a morphism $a_{Y}: Y \lra X_0$ and $Z \subseteq Y$ 
as in $(*)$ for $X_0$ and the object $\cE_{0,Y}$ in $\Isocl(Y,Z)$ 
induced by $\cE_0$ is actually contained in $\Isocl(Y,Z)_{\Sigma(\ss)}$. 
Then the composite $Y \os{a_Y}{\lra} X_0 \lra X$ and $Z \subseteq Y$ 
satisfy the 
condition $(*)$ for $X$ and the object $\cE_{0,Y}$ is equal to 
the object $\cE_Y$ in $\Isocl(Y,Z)$ 
induced by $\cE$. Hence we have $\cE \in \Isocl(X,M_X)_{\Sigma(\ss)}$, 
as desired. 
\end{proof}

Finally in this subsection, we give how to define fine log structures 
on Deligne-Mumford stacks and give examples which are useful 
in this paper. \par 
Let $X$ be a Deligne-Mumford stack of finite type over $k$, 
let $\epsilon: X_0 \lra X$ be an etale surjective morphism from 
a scheme $X_0$ separated of finite type over $k$ and let 
$X_n \,(n=0,1,2)$ 
be the $(n+1)$-fold fiber product of $X_0$ over $X$. 
Assume that there exists a fine log structure $M_{X_{\b}}$ on 
the $2$-trucated simplicial scheme $X_{\b}$ such that the transition 
maps $(X_a,M_{X_{a}}) \lra (X_{b},M_{X_b}) \,(a,b \in \{0,1,2\})$ are 
all strict. Then we can define the fine log structure $M_X$ on $X$ 
in the following way: For an object $Y \lra X$ in $X_{\liset}$, 
we have the $2$-truncated etale \v{C}ech hypercovering 
$X_{\b} \times_X Y \lra Y$. Then the pull-back $M_{X_{\b}\times_X Y}$ 
of $M_{X_{\b}}$ to $X_{\b} \times_X Y$ descends to the fine log 
structure $M_Y$ on $Y_{\et}$ and it is functorial with respect to $Y$. 
So $\{M_Y\}_{Y}$ defines a fine log structure on $X$. 

\begin{example}\label{exam1}
Let $(X,M_X)$ be a connected Noetherian 
fs log scheme and 
let $(Y,M_Y) \lra (X,M_X)$ be a finite Kummer log etale 
(finite log etale of Kummer type in the terminology in \cite{nakayama}) 
Galois covering with Galois group $G$. For $m=0,1,2$, let 
$(Y_m,M_{Y_m})$ be the $(m+1)$-fold fiber product of $(Y,M_Y)$ over 
$(X,M_X)$ in the category of fs log schemes. 
Then we have $(Y_m,M_{Y_m}) \cong (Y_0,M_{Y_0}) \times G^m$ 
naturally (see \cite{illusie}) and so
$Y_{\b} \lra [Y/G]$ gives an etale \v{C}ech hypercovering of 
$[Y/G]$. So the log structure $M_{Y_{\b}}$ on $Y_{\b}$ induces the 
log structure $M_{[Y/G]}$ on $[Y/G]$. 
\end{example}

\begin{example}\label{exam2}
Let $X$ be a scheme smooth separated of finite type over $k$ and 
let $M_X$ be a log structure on $X$ associated to a simple normal 
crossing divisor $Z = \bigcup_{i=1}^r Z_i$ (with each $Z_i$ irreducible). 
Let $f: (Y,M_Y) \lra (X,M_X)$ be a finite Kummer log etale Galois 
covering with Galois group $G$, let $Y^{\sm}$ be the smooth locus 
of $Y$ and put $M_{Y^{\sm}} := M_Y|_{Y^{\sm}}$. Also, let 
$X'$ be the image of $Y^{\sm}$ in $X$ and put $M_{X'} := M_X |_{X'}$. 
Then $Y^{\sm}$ is $G$-stable and $f$ induces the morphism 
$f': (Y^{\sm}, M_{Y^{\sm}}) \lra (X',M_{X'})$, which is again 
a finite Kummer log etale Galois covering with Galois group $G$. 
So, by Example \ref{exam1}, 
the log structure $M_{[Y^{\sm}/G]}$ is 
defined. (By construction, we have $M_{[Y^{\sm}/G]} = 
M_{[Y/G]} |_{[Y^{\sm}/G]}$, where $M_{[Y/G]}$ is as in Example \ref{exam1}.) 
Note that, by \cite[5.2]{niziol} and the perfectness of $k$, 
the log structure $M_{Y^{\sm}}$ is associated to some 
normal crossing divisor $Z'$ and the Kummer type assumption of $f$ implies 
that $Z'$ is in fact a simple normal crossing divisor. 
For any $i$, $({f'}^*(X' \cap Z_i))_{\red}$ defines a simple normal crossing 
subdivisor $Z'_i$ of $Z'$ which is $G$-stable and we have 
$Z' = \sum_{i=1}^r Z'_i$. Then, the log structure on 
$Y^{\sm}_{\b} := X' \times_X Y_{\b} \cong Y^{\sm} \times G^{\b}$ 
(where $Y_{\b}$ is as in Example \ref{exam1}) 
associated to $Z'_i \times G^{\b}$ induces a 
sub log structure $M_{[Y^{\sm}/G],i}$ 
of $M_{[Y^{\sm}/G]}$ and the family 
$\{M_{[Y^{\sm}/G],i}\}_{i=1}^r$ defines a decomposition of 
$M_{[Y^{\sm}/G]}$. 
\end{example}

\subsection{First stacky equivalence}

In this subsection, we give a proof of the equivalences 
\eqref{seq1} and \eqref{seq2}. To do so, first let us recall 
the equivalence of Crew. \par 
Let $X$ be a connected smooth varieties over $k$. 
Crew proved in \cite{crew} the equivalence 
\begin{equation}\label{eq5'}
G: \Rep_{K^{\sigma}}(\pi_1(X)) \os{=}{\lra} \FIsoc(X)^{\circ} 
\end{equation}
(which is equal to \eqref{eq5}). 
Let us recall the definition of the functor \eqref{eq5'}, using 
the notion of convergent site which is introduced in \cite{ogus}. 
Let $\rho$ be an object in $\Rep_{K^{\sigma}}(\pi_1(X))$ and let 
$\cF := \cF_{0,\Q}$ be the corresponding object in 
$\Sm_{K^{\sigma}}(X) = \Sm_{O_K^{\sigma}}(X)_{\Q}$. 
(Here, for an additive category $\cC$, $\cC_{\Q}$ denotes the 
$\Q$-linearlization of it and for an object $A$ in 
$\cC$, $A_{\Q}$ denotes the object $A$ regarded as an object in 
$\cC_{\Q}$.) For an enlargement $(T,z_T)$ in the convergent site 
$(X/\Spf O_K)_{\conv}$ of $X$ over $\Spf O_K$ (that is, 
a $p$-adic formal scheme $T$ of finite type over $\Spf O_K$ and 
a morphism $z_T: T_0 := (T \otimes_{O_K} O_K/\fm_K)_{\red} \lra X$ 
over $k$), let $\cF_{0,T}$ be the object in 
$\Sm_{O^{\sigma}_K}(T)$ corresponding to $z_T^{-1}(\cF_0)$ via the 
equivalence $\Sm_{O^{\sigma}_K}(T) \cong \Sm_{O^{\sigma}_K}(T_0)$ 
and let us define $\cE_T$ by $\cE_T := 
(\cF_{0,T} \otimes_{O^{\sigma}_K} \cO_T)_{\Q}$ as an object 
in the $\Q$-linearized category of the category of coherent sheaves on 
$T_{\et}$. Then $\cE := \{\cE_T\}_{(T,z_T)}$ defines a locally free 
convergent isocrystal on $X$ over $K$. Moreover, the $q$-th 
power Frobenius endomorphism $F:X \lra X$ induces the equivalence 
$F^{-1}: \Sm_{O^{\sigma}_K}(X) \os{=}{\lra} \Sm_{O^{\sigma}_K}(X)$ 
with $F^{-1}(\cF_0) \cong \cF_0$, and so we have the isomorphism 
$$\Psi_T: (F^*\cE)_T = (F^{-1}(\cF_0)_T 
\otimes_{O^{\sigma}_K} \cO_T)_{\Q} \os{=}{\lra} 
(\cF_{0,T} \otimes_{O^{\sigma}_K} \cO_T)_{\Q} = \cE_T$$
 for any enlargement $(T,z_T)$, where 
 $(F^*\cE)_T$ denotes the sheaf on $T_{\et}$ induced by the locally free 
convergent isocrystal $F^*\cE$ and $F^{-1}(\cF_0)_T$ is the object in 
$\Sm_{O^{\sigma}_K}(T)$ corresponding to $z_T^{-1}(F^{-1}(\cF_0))$ via the 
equivalence $\Sm_{O^{\sigma}_K}(T) \cong \Sm_{O^{\sigma}_K}(T_0)$. 
Then, if we put $\Psi := \{\Psi_T\}_{(T,z_T)}$, we see that the pair 
$(\cE,\Psi)$ defines an object in $\FIsoc(X)^{\circ}$, which we define 
to be the image of $\rho$ by the functor $G$. In view of this 
description, we see that Crew's equivalence \eqref{eq5'} is 
written as the equivalence 
\begin{equation}\label{neweq}
G: \Sm_{K^{\sigma}}(X) \os{=}{\lra} \FIsoc(X)^{\circ}. 
\end{equation}
(The choice of a base point in the definition of $\pi_1(X)$ 
is not essential.) Also, we see easily the functoriality of the 
equivalence \eqref{neweq}. \par 
Now we proceed to prove the equivalences 
\eqref{seq1} and \eqref{seq2}. In the following in this subsection, 
let $X \hra \ol{X}$ be an open immersion of connected smooth varieties over 
$k$ such that $\ol{X} \setminus X =: Z = \bigcup_{i=1}^r Z_i$ is a 
simple normal crossing divisor (with each $Z_i$ irreducible). 
For $1 \leq i \leq r$, 
let $v_i$ be the discrete valuation of $k(X)$ corresponding to the 
generic point of $Z_i$, let $k(X)_{v_i}$ be the completion of $k(X)$ with 
respect to $v_i$ and let $I_{v_i}$ be the inertia group of 
$k(X)_{v_i}$. (Then we have homomorphisms $I_{v_i} \lra \pi_1(X)$ which are 
well-defined up to conjugate.) Let us define 
$\Repfin_{K^{\sigma}}(\pi_1(X))$ by 
$$ \Repfin_{K^{\sigma}}(\pi_1(X)) := 
\{\rho \in \Rep_{K^{\sigma}}(\pi_1(X)) \,|\, 
\forall i, \rho|_{I_{v_i}} \text{ has finite image}\}.$$
Let us also define categories consisting of isocrystals. 
Let 
$\cG_X$ be the category of finite etale Galois (connected) coverings 
of $X$ and for an object $Y \ra X$ in $\cG_X$, let 
$\ol{Y}$ be the normalization of $\ol{X}$ in $k(Y)$, let 
$\ol{Y}^{\sm}$ be the smooth locus of $\ol{Y}$ and put 
$G_Y := \Aut (Y/X)$. Then $G_Y$ acts on $\ol{Y}^{\sm}$ and so 
we can define the categories $\Isoc([\ol{Y}^{\sm}/G_Y])$ 
(resp. $\FIsoc([\ol{Y}^{\sm}/G_Y])$, $\FIsoc([\ol{Y}^{\sm}/G_Y])^{\circ}$) 
of convergent isocrystals (resp. convergent $F$-isocrystals, 
unit-root convergent $F$-isocrystals) on the quotient 
stack $[\ol{Y}^{\sm}/G_Y]$ over $K$. 
Then we have the following: 

\begin{proposition}\label{wd}
Let $Y \ra X, Y' \ra X$ be objects in $\cG_X$ and let 
$f:Y' \lra Y$ be a morphism in $\cG_X$. Then we have the canonical 
functors 
\begin{align*}
& f^*: \FIsoc([\ol{Y}^{\sm}/G_{Y}]) \lra \FIsoc([\ol{Y'}^{\sm}/G_{Y'}]), \\ 
& f^*: \FIsoc([\ol{Y}^{\sm}/G_{Y}])^{\circ} 
\lra \FIsoc([\ol{Y'}^{\sm}/G_{Y'}])^{\circ}
\end{align*}
induced by $f$. When $X$ is a curve, we also have the canonical functor 
$$f^*: \Isoc([\ol{Y}^{\sm}/G_{Y}]) \lra \Isoc([\ol{Y'}^{\sm}/G_{Y'}]). $$
\end{proposition}

\begin{proof}
The latter assertion is obvious because we have the equalities 
$\ol{Y}^{\sm} = \ol{Y}, \ol{Y'}^{\sm} \allowbreak 
= \ol{Y'}$ in one dimensional case 
and we have a morphism $[\ol{Y'}/G_{Y'}] \lra [\ol{Y}/G_{Y}]$. 
Let us prove the former asssertion. Let us put 
$\ol{Y'}^{\sm,\circ} := f^{-1}(\ol{Y}^{\sm}) \cap \ol{Y'}^{\sm}$. 
Then it is a $G_{Y'}$-stable open subscheme of $\ol{Y'}^{\sm}$ 
such that the complement $\ol{Y'}^{\sm} \setminus \ol{Y'}^{\sm,\circ}$ 
has codimension at least $2$ in $\ol{Y'}^{\sm}$. So, for $n =0,1,2$, 
we have the equivalence of categories 
$\FIsoc(\ol{Y'}^{\sm} \times G_{Y'}^n) \os{=}{\lra} 
\FIsoc(\ol{Y'}^{\sm,\circ} \times G_{Y'}^n)$
 by \cite[3.1]{purity} and this implies the equivalence 
$\FIsoc([\ol{Y'}^{\sm}/G_{Y'}]) \os{=}{\lra} 
\FIsoc([\ol{Y'}^{\sm,\circ}/G_{Y'}])$ (see Example \ref{exam}). 
Then we can define the desired functor as the composite 
$$ 
\FIsoc([\ol{Y}^{\sm}/G_{Y}]) \lra 
\FIsoc([\ol{Y'}^{\sm,\circ}/G_{Y'}]) \os{=}{\lla} 
\FIsoc([\ol{Y'}^{\sm}/G_{Y'}]), $$
where the first functor is induced by the morphism 
$[\ol{Y'}^{\sm,\circ}/G_{Y'}] \lra [\ol{Y}^{\sm}/G_{Y'}]$ and the 
second functor is the equivalence above. We can define 
the functor also in the unit-root case in the same way. 
\end{proof}

By Proposition \ref{wd}, we can define the limit 
$$\varinjlim_{Y\ra X \in \cG_X} \FIsoc ([\ol{Y}^{\sm}/G_Y]), \qquad
\varinjlim_{Y\ra X \in \cG_X} \FIsoc ([\ol{Y}^{\sm}/G_Y])^{\circ} $$ 
and we can define the limit 
$\varinjlim_{Y\ra X \in \cG_X} \Isoc ([\ol{Y}^{\sm}/G_Y])$ when 
$X$ is a curve. Now we prove the following theorem: 

\begin{theorem}\label{thm1}
Let the notation be as above. 
\begin{enumerate}
\item There exists an equivalence of categories 
\begin{equation}\label{eqeq1} 
\Repfin_{K^{\sigma}}(\pi_1(X)) \os{=}{\lra} \FIsocd(X,\ol{X})^{\circ}
\end{equation}
which is compatible with the equivalence \eqref{eq5'} of Crew. 
\item 
There exists an equivalence of categories 
\begin{equation}\label{eqeq2}
\Repfin_{K^{\sigma}}(\pi_1(X)) \os{=}{\lra} 
\varinjlim_{Y\ra X \in \cG_X} \FIsoc ([\ol{Y}^{\sm}/G_Y])^{\circ} 
\end{equation}
and a natural restriction functor 
\begin{equation}\label{eqeq3}
\varinjlim_{Y\ra X \in \cG_X} 
\FIsoc ([\ol{Y}^{\sm}/G_Y]) {\lra} \FIsocd(X,\ol{X})
\end{equation}
which makes the following diagram commutative$:$
\begin{equation}\label{eqeq4}
\begin{CD}
\Repfin_{K^{\sigma}}(\pi_1(X)) @>{\text{\eqref{eqeq2}}}>> 
\varinjlim_{Y\ra X \in \cG_X} \FIsoc ([\ol{Y}^{\sm}/G_Y])^{\circ} \\ 
@| @V{\text{\eqref{eqeq3}}^{\circ}}VV \\ 
\Repfin_{K^{\sigma}}(\pi_1(X)) @>{\text{\eqref{eqeq1}}}>> 
\FIsocd(X,\ol{X})^{\circ}. 
\end{CD}
\end{equation}
When $X$ is a curve, there exists also a natural restriction 
functor 
\begin{equation}\label{eqeq3'}
\varinjlim_{Y\ra X \in \cG_X} 
\Isoc ([\ol{Y}^{\sm}/G_Y]) {\lra} \Isocd(X,\ol{X})
\end{equation}
with $F\text{-\eqref{eqeq3'}} = \text{\eqref{eqeq3}}$. 
\end{enumerate}
\end{theorem}
The equivalence \eqref{eqeq2} is the same as \eqref{seq1}, which is 
a $p$-adic version of \eqref{eq1}. 

\begin{proof}
The equivalence \eqref{eqeq1} is already proven in 
\cite[4.2]{purity}. (In one dimensional case, this is a result of Tsuzuki 
\cite{tsuzukicurve}. 
In higher dimensional case, there is a related result by Kedlaya 
\cite[2.3.7, 2.3.9]{kedlayaswanII}.) \par 
Let us prove the equivalence \eqref{eqeq2}. 
Let $\rho: \pi_1(X) \lra GL_d(K^{\sigma})$ 
be an object in $\Repfin_{K^{\sigma}}(\pi_1(X))$. 
By \cite[4.1, 4.2]{purity} or \cite[2.3.7, 2.3.9]{kedlayaswanII}, 
there exists an object $\varphi: Y \lra X$ in 
$\cG_X$ such that, for any discrete valuation $v$ of $k(Y)$ centered on 
$\ol{Y} \setminus Y$, the restriction of $\rho$ to the inertia group 
$I_v$ at $v$ is trivial. In particular, $\rho |_{\pi_1(Y)}$ is 
unramified at any generic points of $\ol{Y}^{\sm} \setminus Y$ 
(note that $\ol{Y}^{\sm} \setminus Y$ is generically smooth). 
So, by Zariski-Nagata purity, we see that $\rho|_{\pi_1(Y)}$ factors 
through $\pi_1(\ol{Y}^{\sm})$. 
In this way, $\rho \in \Repfin_{K^{\sigma}}(\pi_1(X)) \subseteq 
\Rep_{K^{\sigma}}(\pi_1(X)) \cong \Sm_{K^{\sigma}}(X)$ 
induces a $G_Y$-equivariant object in $\Sm_{K^{\sigma}}(\ol{Y}^{\sm})$ 
corresponding to $\rho|_{\pi_1(Y)}$. \par 
On the other hand, let $\rho$ be an object in $\Rep_{K^{\sigma}}(\pi_1(X))$ 
such that $\rho |_{\pi_1(Y)}$ factors through $\pi_1(\ol{Y}^{\sm})$ for some 
$Y \lra X$ in $\cG_X$. 
Let $\ol{X}'$ be the image of $\ol{Y}^{\sm}$ in $\ol{X}$. 
Then $v_i$'s are centered on $\ol{X}'$ and 
so we can take an extension $v'_i$ of $v_i$ to $k(Y)$ whose center 
$x$ is contained in $\ol{Y}^{\sm}$. 
Let 
$k(Y)_{v'_i}$ be the completion of $k(Y)$ with respect to $v'_i$ and 
denote the valuation ring by $O_{v'_i}$. Then the composite 
$$ \Spec k(Y)_{v'_i} \lra \Spec k(Y) \lra Y \lra \ol{Y}^{\sm} $$ 
factors as 
$$ \Spec k(Y)_{v'_i} \lra \Spec O_{v'_i} \lra \Spec 
\cO_{\ol{Y}^{\sm},x} \lra \ol{Y}^{\sm}. $$ 
So we see that the restriction of 
$\rho |_{\pi_1(Y)}: \pi_1(Y) \lra \pi_1(\ol{Y}^{\sm}) \lra GL_d(K^{\sigma})$ 
to $I_{v'_i}$ is trivial. 
Therefore, 
$\rho |_{I_{v_i}}$ factors through the finite group $I_{v_i}/I_{v'_i}$ and 
we see that $\rho$ is contained in 
$\Repfin_{K^{\sigma}}(\pi_1(X))$. \par 
Let $G_Y\text{-}\Sm_{K^{\sigma}}(\ol{Y}^{\sm})$
be the category of $G_Y$-equivariant objects in 
the category $\Sm_{K^{\sigma}}(\ol{Y}^{\sm})$. 
Then we see from the argument 
in the previous two paragraphs that the 
above construction gives an equivalence 
\begin{equation}\label{eqeqeq11}
\Repfin_{K^{\sigma}}(\pi_1(X)) \os{=}{\lra} 
\varinjlim_{Y \ra X \in \cG_X} 
G_Y\text{-}\Sm_{K^{\sigma}}(\ol{Y}^{\sm}). 
\end{equation}
For $m=0,1,2$, let $\ol{Y}^{\sm}_m$ be the $(m+1)$-fold 
fiber product of $\ol{Y}^{\sm}$ over $[\ol{Y}^{\sm}/G_Y]$. Then 
they form a $2$-truncated simplicial scheme $\ol{Y}^{\sm}_{\b}$ and 
we have the canonical equivalence 
$G_Y\text{-}\Sm_{K^{\sigma}}(\ol{Y}^{\sm}) = 
\Sm_{K^{\sigma}}(\ol{Y}^{\sm}_{\b})$. Hence, 
by Crew's equivalence and Example \ref{exam}, we have the 
equivalence of categories 
\begin{align}
\varinjlim_{Y \ra X \in \cG_X} 
G_Y\text{-}\Sm_{K^{\sigma}}(\ol{Y}^{\sm}) & \os{=}{\lra} 
\varinjlim_{Y \ra X \in \cG_X} 
\Sm_{K^{\sigma}}(\ol{Y}^{\sm}_{\b}) \label{grepfisoc} \\ 
& \os{=}{\lra} 
\varinjlim_{Y\ra X \in \cG_X} \FIsoc (\ol{Y}^{\sm}_{\b})^{\circ} \nonumber \\ 
& \os{=}{\lra} 
\varinjlim_{Y\ra X \in \cG_X} \FIsoc ([\ol{Y}^{\sm}/G_Y])^{\circ}. \nonumber 
\end{align}
Combining this equivalence with \eqref{eqeqeq11}, we obtain the 
equivalence \eqref{eqeq2}. \par 
Next let us define the functor \eqref{eqeq3}. 
Let $\ol{Y}^{\sm}_{\b}$ be as above, let 
$Y_{\b}$ be the $2$-truncated simplicial scheme such that 
$Y_m \,(m=0,1,2)$ is naturally the $(m+1)$-fold fiber product of 
$Y$ over $X$ and let $\ol{X}'$ be the image of $\ol{Y}^{\sm}$ in $\ol{X}$. 
We would like to define the composite functor 
\begin{align}
\FIsoc([\ol{Y}^{\sm}/G_Y]) & 
\os{=}{\lra} 
\FIsoc(\ol{Y}^{\sm}_{\b}) \label{precompos} \\ 
& \lra \FIsocd(Y_{\b}, \ol{Y}^{\sm}_{\b}) \nonumber \\ 
& \os{=}{\lla} \FIsocd(X, \ol{X}') \nonumber \\ 
& \os{=}{\lla} \FIsocd(X, \ol{X}). \nonumber 
\end{align} 
In order that the functor is well-defined, we should prove that 
the third arrow in the above composite is an equivalence. 
(The fourth arrow is an equivalence by \cite[3.1]{purity}.) 
Then it suffices to prove the equivalence 
\begin{equation}\label{2barbox}
\Isocd(X, \ol{X}') \os{=}{\lra} \Isocd(Y_{\b}, \ol{Y}^{\sm}_{\b}). 
\end{equation}
Note that the right hand side in \eqref{2barbox} 
is the category of objects in 
$\Isocd(Y_{0}, \ol{Y}^{\sm}_0)$ endowed with equivariant 
$G_Y$-action. 
Then, if we denote the projection $\ol{Y}^{\sm}_0 \lra \ol{X}'$ by $\pi$, 
we have the functor 
\begin{equation}\label{aa2barbox}
\Isocd(Y_{\b}, \ol{Y}^{\sm}_{\b}) \lra 
\Isocd(X, \ol{X}'); 
\quad \cE \mapsto (\pi_*\cE)^{G_Y}, 
\end{equation}
where $\pi_*$ is the push-out functor defined by Tsuzuki \cite{tsuzuki}. 
By \cite{tsuzuki} and \cite[2.6.8]{kedlayaI}, we have the 
trace morphisms $(\pi_*\pi^*\cE)^{G_Y} \lra \cE$, 
$\pi^*((\pi_*\cE)^{G_Y}) \lra \cE$, and they are 
isomorphic in $\Isoc(X)$ and $\Isoc(Y_{\b})$ by etale descent. Since 
$\Isocd(X,\ol{X}') \lra \Isoc(X)$, 
$\Isocd(Y_{\b}, \ol{Y}^{\sm}_{\b}) \lra \Isoc(Y_{\b})$ are exact and 
faithful, they are actually isomorphic. 
Hence \eqref{aa2barbox} is a quasi-inverse of 
\eqref{2barbox} and so \eqref{2barbox} is an equivalence. 
So the functor \eqref{precompos} is well-defined and it induces the 
functor \eqref{eqeq3}. When $X$ is a curve, we can define the 
functor \eqref{eqeq3'} in the same way using \eqref{2barbox}. 
(Note that we do not have to use \cite[3.1]{purity} because 
$\ol{X} = \ol{X}'$ in the case of curves.) \par 
To prove the commutativity of the diagram \eqref{eqeq4}, it suffice to 
check it after we compose with the restriction functor 
\begin{equation}\label{hoho}
\FIsocd(X,\ol{X})^{\circ} \lra \FIsoc(X)^{\circ},  
\end{equation}
which is known to be fully faithful (\cite{tsuzuki}). 
Then $\text{\eqref{hoho}$\circ$\eqref{precompos}}^{\circ}$ 
is equal to the composite 
\begin{align*}
\FIsoc([\ol{Y}^{\sm}/G_Y])^{\circ} & \os{=}{\lra} 
\FIsoc(\ol{Y}^{\sm}_{\b})^{\circ} \\ 
& \lra \FIsoc(Y_{\b})^{\circ} \os{=}{\lla} 
\FIsoc(X)^{\circ}. 
\end{align*}
Hence $\text{\eqref{hoho}}
\circ\text{\eqref{precompos}}^{\circ}\circ\text{\eqref{eqeq2}}^{\circ}$ 
is the composite 
\begin{align*}
\Repfin_{K^{\sigma}}(\pi_1(X)) & \os{\text{\eqref{eqeqeq11}}}{\lra}
\varinjlim_{Y\ra X \in \cG_X} 
G_Y\text{-}\Sm_{K^{\sigma}}(\ol{Y}^{\sm}) \\ 
& \os{=}{\lra} 
\varinjlim_{Y\ra X \in \cG_X} 
\Sm_{K^{\sigma}}(\ol{Y}^{\sm}_{\b}) \\ 
& \os{G}{\lra} 
\varinjlim_{Y\ra X \in \cG_X} 
\FIsoc(\ol{Y}^{\sm}_{\b})^{\circ} \\ 
& \lra 
\varinjlim_{Y\ra X \in \cG_X} 
\FIsoc(Y_{\b})^{\circ} \os{=}{\lla} \FIsoc(X)^{\circ}. 
\end{align*}
By the functoriality of \eqref{neweq}, it is equal to 
the composite 
\begin{align*}
\Repfin_{K^{\sigma}}(\pi_1(X)) & \os{\text{\eqref{eqeqeq11}}}{\lra}
\varinjlim_{Y\ra X \in \cG_X} 
G_Y\text{-}\Sm_{K^{\sigma}}(\ol{Y}^{\sm}) \\ 
& \os{=}{\lra} 
\varinjlim_{Y\ra X \in \cG_X} 
\Sm_{K^{\sigma}}(\ol{Y}^{\sm}_{\b}) \\ 
& \lra 
\varinjlim_{Y\ra X \in \cG_X} 
\Sm_{K^{\sigma}}(Y_{\b}) \\ 
& \os{G}{\lra} 
\varinjlim_{Y\ra X \in \cG_X} 
\FIsoc(Y_{\b})^{\circ} \os{=}{\lla} \FIsoc(X)^{\circ}, 
\end{align*}
and by definition of \eqref{eqeqeq11}, it is rewritten as the composite 
\begin{align*}
\Repfin_{K^{\sigma}}(\pi_1(X)) & \os{\subset}{\lra}
\Rep_{K^{\sigma}}(\pi_1(X)) \\ 
& \os{=}{\lra} \Sm_{K^{\sigma}}(X) \\ 
& \os{=}{\lra} 
\varinjlim_{Y\ra X \in \cG_X} 
\Sm_{K^{\sigma}}(Y_{\b}) \\ 
& \os{G}{\lra} 
\varinjlim_{Y\ra X \in \cG_X} 
\FIsoc(Y_{\b})^{\circ} \os{=}{\lla} \FIsoc(X)^{\circ}. 
\end{align*}
Again by the functoriality of \eqref{neweq} and the assertion (1), 
it is further rewritten as the composite 
$$ 
\Repfin_{K^{\sigma}}(\pi_1(X)) \os{\text{\eqref{eqeq1}}}{\lra} 
\FIsocd(X,\ol{X})^{\circ} \os{\text{\eqref{hoho}}}{\lra} 
\FIsoc(X)^{\circ}. $$
So we have proved the commutativity of the diagram \eqref{eqeq4} and 
we are done. 
\end{proof}

Next let us consider `the tame variant' of the equivalence 
\eqref{eqeq2}. Let $\pi_1^t(X)$ be 
the tame fundamental group of $X$ (tamely ramified at 
the valuations $v_i \,(1 \leq i \leq r)$) and let 
$\Rep_{K^{\sigma}}(\pi_1^t(X))$ be the category of finite dimensional 
continuous representations of $\pi_1^t(X)$ over $K^{\sigma}$. 
On the other hand, 
let $\cG_X^t$ be the category of finite etale Galois tame covering 
of $X$ (tamely ramified at $v_i \,(1 \leq i \leq r)$). 
For an object $Y \ra X$ in $\cG_X^t$, let 
$\ol{Y}, \ol{Y}^{\sm}, G_Y$ be as before. 
Then we have the following: 

\begin{theorem}\label{cor1}
Let the notations be as above. Then 
there exists an equivalence of categories 
\begin{equation}\label{eqeq2t}
\Rep_{K^{\sigma}}(\pi_1^t(X)) \os{=}{\lra} 
\varinjlim_{Y\ra X \in \cG_X^t} \FIsoc ([\ol{Y}^{\sm}/G_Y])^{\circ} 
\end{equation}
and a natural restriction functor 
\begin{equation}\label{eqeq3t}
\varinjlim_{Y\ra X \in \cG_X^t} 
\FIsoc ([\ol{Y}^{\sm}/G_Y]) \lra \FIsocd(X,\ol{X})
\end{equation}
which makes the following diagram commutative$:$
\begin{equation}\label{eqeq4t}
\begin{CD}
\Rep_{K^{\sigma}}(\pi_1^t(X)) @>{\text{\eqref{eqeq2t}}}>> 
\varinjlim_{Y\ra X \in \cG_X^t} \FIsoc ([\ol{Y}^{\sm}/G_Y])^{\circ} \\ 
@V{\cap}VV @V{\text{\eqref{eqeq3t}}^{\circ}}VV \\ 
\Repfin_{K^{\sigma}}(\pi_1(X)) @>{\text{\eqref{eqeq1}}}>> 
\FIsocd(X,\ol{X})^{\circ}. 
\end{CD}
\end{equation}
When $X$ is a curve, we have also a natural restriction functor 
\begin{equation}\label{eqeq3t'}
\varinjlim_{Y\ra X \in \cG_X^t} 
\Isoc ([\ol{Y}^{\sm}/G_Y]) \lra \Isocd(X,\ol{X})
\end{equation}
with $F\text{-\eqref{eqeq3t'}}=\text{\eqref{eqeq3t}}$. 
\end{theorem}
The equivalence \eqref{eqeq2t} is the same as \eqref{seq2}, which is 
a tame $p$-adic version of \eqref{eq1}. Before the proof of 
Theorem \ref{cor1}, we prove a lemma. 

\begin{lemma}\label{cor1lem}
Let $X \hra \ol{X}$, $v_i \,(1 \leq i \leq r)$ be as above and 
let us take $\rho \in \Rep_{K^{\sigma}}(\pi_1^t(X))$. 
Let $v$ be a discrete valuation on $k(X)$ centered on $\ol{X}$, 
let $k(X)_v$ be the completion of $k(X)$ with respect to $v$ and 
let $I_v$ be the inertia group of $k(X)_v$. Then 
$|\im(\rho|_{I_v})|$ is finite and prime to $p$. 
$($In particular, we have the inclusion 
$\Rep_{K^{\sigma}}(\pi_1^t(X)) \subseteq 
\Rep_{K^{\sigma}}^{\fin}(\pi_1(X)).)$ 
\end{lemma}

\begin{proof}
Let us take a suitable $O_{K}^{\sigma}$-lattice $\pi_1^t(X) \lra 
GL_d(O_{K}^{\sigma})$ of $\rho$. First we prove the lemma in the case 
where $v=v_i$ for some $i$. In this case, $\rho|_{I_v}$ factors through 
the tame quotient $I^t_v$ by definition. Note that 
$I^t_v$ is a pro-prime-to-$p$ group and that 
$N := \Ker(GL_d(O_K^{\sigma}) \lra GL_d(O_K^{\sigma}/2pO_K^{\sigma}))$ 
is a pro-$p$ group. So the image of 
$\rho|_{I_v}: I_v \lra I_v^t \lra GL_d(O_{K}^{\sigma})$ is isomorphic to 
the image of the composite 
$I_v \lra I_v^t \lra GL_d(O_{K}^{\sigma}) \lra GL_d(O_{K}^{\sigma})/N$.
So $|\im(\rho|_{I_v})|$ is finite. Then, since $I^t_v$ is a pro-prime-to-$p$ 
group, we see also that  $|\im(\rho|_{I_v})|$ is prime to $p$. So we are 
done in this case. Note also that, since we have the isomorphism 
$I_v^t \cong \prod_{l\not=p} \Z_l$ and $\Ker(I_v \lra I^t_v)$ is 
a pro-$p$-group, $\rho|_{I_v}$ factors through 
$I_v/J$ for any $J \lhd I_v$ such that $|I_v/J|$ is prime to $p$ and divisible 
by $|\im(\rho|_{I_v})|$. \par 
Next we prove the lemma for general $v$. Let $x$ be the center of $v$ and 
take an open neighborhood $\ol{U} \subseteq \ol{X}$ of $x$ such that 
$\ol{U}$ admits a smooth morphism $f:\ol{U} \lra \Af_k^r = \Spec 
k[t_1,...,t_r]$ with $f^{-1}(\{t_i=0\}) = \ol{U} \cap Z_i \, (1 \leq i 
\leq r)$. Let us put $U := \ol{U} \cap X = f^{-1}(\G_{m,k}^r)$. 
Let us take a positive integer $n$ prime to $p$ which is divisible 
by all the $|\im (\rho|_{I_{v_i}})|$'s for $1 \leq i \leq r$ 
(there exists such $n$ by the lemma for $v_i$'s). Let 
$n: \Af_k^r \lra \Af_k^r$ be the $n$-th power map and put 
$\ol{U}' := \ol{U} \times_{\Af_k^r,n}\Af_k^r$, 
$U' := U \times_{\Af_k^r,n}\Af_k^r$. Let us denote the 
inverse image of the $i$-th coodinate hyperplane by the projection 
$\ol{U}' = \ol{U} \times_{\Af_k^r,n}\Af_k^r \lra \Af_k^r$ by $Z'_i$ 
\,$(1 \leq i \leq r)$. Then we have $\ol{U}' \setminus U' = :Z' 
= \bigcup_{i=1}^r Z'_i$ and it is a simple normal crossing divisor. 
For $i$ with $Z'_i \not= \emptyset$, let $v'_i$ be the discrete 
valuation on $k(U')$ corresponding to the generic point of $Z'_i$, 
let $k(U')_{v'_i}$ be the completion of $k(U')$ with respect to $v'_i$ 
and let $I_{v'_i}$ be the inertia group of $k(U')_{v'_i}$. 
Then, by definition, $v'_i$ is an extension of $v_i$ and we have 
$|I_{v_i}/I_{v'_i}| = n$. So, by the argument in the previous paragraph, 
$\rho|_{I_{v_i}}$ factors through $I_{v_i}/I_{v'_i}$. So 
$\rho|_{I_{v'_i}}$ is trivial and hence $\rho|_{\pi_1(U')}$ factors 
through $\pi_1(\ol{U}')$. Now let $v'$ be an extension of the 
valuation $v$ to $k(U')$ centered on $\ol{U}'$ 
and let $x' \in \ol{U}'$ be the center of $v'$. Let 
$k(U')_{v'}$ be the completion of $k(U')$ with respect to $v'$ and 
denote the valuation ring by $O_{v'}$. Then the composite 
$$ \Spec k(U')_{v'} \lra \Spec k(U') \lra U' \lra \ol{U}' $$ 
factors as 
$$ \Spec k(U')_{v'} \lra \Spec O_{v'} \lra \Spec 
\cO_{\ol{U}',x'} \lra \ol{U}'. $$ 
So we see that the restriction of 
$\rho |_{\pi_1(U')}: \pi_1(U') \lra \pi_1(\ol{U}') \lra GL_d(\cO_K^{\sigma})$ 
to $I_{v'}$ is trivial. Hence $\rho|_{I_v}$ factors through 
$I_v/I_{v'}$. Since $|I_v/I_{v'}|$ divides 
$[k(U'):k(U)] = n^r$, we can conclude that $|\im (\rho |_{I_v})|$ is 
finite and prime to $p$. So we are done. 
\end{proof}

Now we prove Theorem \ref{cor1}, using the above lemma. 

\begin{proof}[Proof of Theorem \ref{cor1}]
First let us prove the equivalence \eqref{eqeq2t}. 
Let $\rho$ be an object in $\Rep_{K^{\sigma}}(\pi_1^t(X))$ and 
take a suitable $O_K^{\sigma}$-lattice $\pi_1^t(X) \lra 
GL_d(O_K^{\sigma})$ of $\rho$. Let 
$Y \lra X$ be a finite etale Galois tame covering of $X$ 
(tamely ramified along $v_i$'s) which corresponds to 
the subgroup $\Ker(\pi_1^t(X) \os{\rho}{\ra} GL_d(O_K^{\sigma}) 
\twoheadrightarrow GL_d(O_K^{\sigma}/2pO_K^{\sigma}))$. 
Let $v$ be any discrete valuation of $k(Y)$ centered on $\ol{Y}$. 
Then $v|_{k(X)}$ is a discrete valuation of $k(X)$ centered on $\ol{X}$. 
So, by Lemma \ref{cor1lem}, $|\im (\rho|_{I_{v|_{k(X)}}})|$ is finite and 
prime to $p$. Hence so is $|\im (\rho|_{I_{v}})|$. 
On the other hand, $\im (\rho|_{I_v})$ is contained in 
$\Ker(GL_d(O_K^{\sigma}) 
\twoheadrightarrow GL_d(O_K^{\sigma}/2pO_K^{\sigma}))$, which is a 
pro-$p$ group. Hence $\rho|_{I_v}$ is trivial. By using this fact, 
we see in the same way as the proof of Theorem \ref{thm1} that 
$\rho |_{\pi_1(Y)}$ factors through $\pi_1(\ol{Y}^{\sm})$ and so 
induces a $G_Y$-equivariant object of 
$\Sm_{K^{\sigma}}(\ol{Y}^{\sm})$. \par 
On the other hand, let $\rho$ be an object in $\Rep_{K^{\sigma}}(\pi_1(X))$ 
such that $\rho |_{\pi_1(Y)}$ factors through $\pi_1(\ol{Y}^{\sm})$ for some 
$Y \ra X$ in $\cG_X^t$. Then we see in the same way as the proof of 
Theorem \ref{thm1} that, for 
an extension $v'_i$ of $v_i$ to $k(Y)$ $\,(1 \leq i \leq r)$ centered on 
$\ol{Y}^{\sm}$, $\rho|_{I_{v'_i}}$ is trivial. Then 
$\rho |_{I_{v_i}}$ factors through $I_{v_i}/I_{v'_i}$ and 
this is a quotient of the tame inertia quotient $I^t_{v_i}$ of $I_{v_i}$ 
because $Y \lra X$ is in $\cG_X^t$. 
Hence $\rho$ factors through $\pi_1^t(X)$. So we have an equivalence 
\begin{equation}\label{eqeqeq1}
\Repfin_{K^{\sigma}}(\pi_1^t(X)) \os{=}{\lra} 
\varinjlim_{Y \ra X \in \cG_X^t} 
G_Y\text{-}\Sm_{K^{\sigma}}(\ol{Y}^{\sm}). 
\end{equation}
(Here $G_Y\text{-}\Sm_{K^{\sigma}}(\ol{Y}^{\sm})$ is as in 
the proof of Theorem \ref{thm1}.)  
Combining this with the equivalence 
$$ \varinjlim_{Y \ra X \in \cG_X^t} 
G_Y\text{-}\Sm_{K^{\sigma}}(\ol{Y}^{\sm}) \os{=}{\lra} 
\varinjlim_{Y\ra X \in \cG_X^t} \FIsoc ([\ol{Y}^{\sm}/G_Y])^{\circ} $$
which is defined in the same way as \eqref{grepfisoc}, 
we obtain the equivalence \eqref{eqeq2t}. \par 
The functor \eqref{eqeq3t} is defined as the composite of the 
canonical `inclusion functor' 
$\varinjlim_{Y\ra X \in \cG_X^t} \FIsoc ([\ol{Y}^{\sm}/G_Y]) 
\lra \varinjlim_{Y\ra X \in \cG_X} \FIsoc ([\ol{Y}^{\sm}/G_Y])$ 
and the functor \eqref{eqeq3}. 
By construction, we have the commutative diagram 
\begin{equation*}
\begin{CD}
\Rep(\pi_1^t(X)) @>{\text{\eqref{eqeq2t}}}>> 
\varinjlim_{Y\ra X \in \cG_X^t} \FIsoc ([\ol{Y}^{\sm}/G_Y])^{\circ} \\
@V{\cap}VV @V{\text{incl.}}VV \\ 
\Repfin(\pi_1(X)) @>{\text{\eqref{eqeq2}}}>> 
\varinjlim_{Y\ra X \in \cG_X} \FIsoc ([\ol{Y}^{\sm}/G_Y])^{\circ} 
\end{CD}
\end{equation*}
(where $\text{incl.}$ denotes the `inclusion functor'). 
By combining this with \eqref{eqeq3}, we obtain 
the commutative diagram \eqref{eqeq4t}. 
When $X$ is a curve, we define the functor \eqref{eqeq3t'} as 
 the composite of the 
canonical `inclusion functor' 
$\varinjlim_{Y\ra X \in \cG_X^t} \Isoc ([\ol{Y}^{\sm}/G_Y]) 
\lra \varinjlim_{Y\ra X \in \cG_X} \Isoc ([\ol{Y}^{\sm}/G_Y])$ 
and the functor \eqref{eqeq3'}. Then it is easy to see the 
equality $F\text{-\eqref{eqeq3t'}}=\text{\eqref{eqeq3t}}$. 
So we are done. 
\end{proof}

\subsection{Stack of roots}

In this subsection, we recall the notion of stack of roots, which is 
treated in  
\cite{cadman}, \cite{borne1}, \cite{borne2}, \cite{is}. 
We also define the canonical log structure on it and we also introduce a 
`bisimplicial resolution' of it which we use later. \par 
For $r \in \N$, let 
$[\Af_k^r/\G_{m,k}^r]$ be the stack over $k$ which is the 
quotient of $\Af_k^r$ by the canonical action of 
$\G_{m,k}^r$. It is known \cite[5.13]{olsson} that it 
classifies 
pairs $(M,\gamma)$, where $M$ is a fine log structure and 
$\gamma: \N^r \lra \ol{M} := M/\cO^{\times}$ is a homomorphism of 
monoid sheaves which lifts to a chart etale locally. 
By \cite[5.14, 5.15]{olsson}, there exists the canonical fine log structure 
$M_{[\Af_k^r/\G_{m,k}^r]}$ on $[\Af_k^r/\G_{m,k}^r]$ 
endowed with $\N^r \lra \ol{M_{[\Af_k^r/\G_{m,k}^r]}} (:= 
M_{[\Af_k^r/\G_{m,k}^r]}/\cO^{\times}_{[\Af_k^r/\G_{m,k}^r]})$ 
such that the 
above $(M,\gamma)$ 
is realized as the pull-back of $M_{[\Af_k^r/\G_{m,k}^r]}$. \par 
Now let 
$X \hra \ol{X}$ be an open immersion of smooth varieties over 
$k$ such that $\ol{X} \setminus X =: Z = \bigcup_{i=1}^r Z_i$ is a 
simple normal crossing divisor (each $Z_i$ being irreducible). 
Denote the fine 
log structure on $\ol{X}$ associated to $Z$ by $M_{\ol{X}}$. 
Then we have the 
morphism $\ol{X} \lra [\Af_k^r/\G_{m,k}^r]$ defined by $(M_{\ol{X}},\gamma)$, 
where $\gamma$ is the homomorphism of 
monoid sheaves 
$$\N^r \lra \ol{M_{\ol{X}}} = \bigoplus_{i=1}^r \N_{Z_i} $$
(where $\N_{Z_i}$ is the direct image to $\ol{X}$ 
of the constant sheaf on $Z_i$ with fiber $\N$) 
induced by the maps $\N \lra \N_{Z_i} \, (1 \leq i \leq r)$ 
which are adjoint to the identity. 
For $n \in \N$ prime to $p$, let 
$n: [\Af_k^r/\G_{m,k}^r] \lra [\Af_k^r/\G_{m,k}^r]$ be the morphism 
induced by the $n$-th 
power map. Then we define 
$$ (\ol{X},Z)^{1/n} := \ol{X} \times_{[\Af_k^r/\G_{m,k}^r],n} 
[\Af_k^r/\G_{m,k}^r] $$ 
and call it the stack of ($n$-th) roots of $(\ol{X},Z)$. (Note that 
we always assume that $n$ is prime to $p$ in this paper.) \par 
It has the following local description (\cite{borne1}, \cite{borne2}, 
see also \cite[complement 1]{kato}): 
Assume $\ol{X} = \Spec R$ is affine and assume that each $Z_i \,(1 \leq 
i \leq r)$ is equal to the zero locus of some element $t_i \in R$. 
Let us put $R' := R[s_1,...,s_r]/(s_1^n - t_1,...,s_r^n - t_r)$, 
$\ol{X}' := \Spec R'$, let $Z'_i \,(1 \leq i \leq r)$ 
be the divisor of $\ol{X}'$ defined by $s_i$ and let 
$M_{\ol{X}'}$ be the log structure associated to the simple normal 
crossing divisor 
$\bigcup_{i=1}^rZ'_i$. Then we have the diagram 
\begin{equation}\label{logdiag1}
\begin{CD}
\N^r @>{\gamma}>> \ol{M_{\ol{X}}} @= \bigoplus_{i=1}^r 
\N_{Z_i}\\ 
@V{n}VV @VVV @V{n}VV\\ 
\N^r @>{\gamma}>> \ol{M_{\ol{X}'}} @= \bigoplus_{i=1}^r 
\N_{Z_{i}'} \
\end{CD}
\end{equation}
(where $\gamma$'s are defined in the same way as before and $n$ denotes 
the multiplication-by-$n$ maps) 
and it induces the morphism 
\begin{equation}\label{atlas1}
\ol{X}' \lra (\ol{X},Z)^{1/n}. 
\end{equation}
Moreover, $\ol{X}'$ admits the canonical 
action of $\mu^r_n(:=$ the product of 
$r$ copies of the group scheme of $n$-th roots of unity) defined as 
the action on $s_i$'s. 
Since the lower horizontal line of \eqref{logdiag1} is invariant by 
the action of $\mu_n^r$, we see that the morphism \eqref{atlas1} is 
stable under the action of $\mu_n^r$ and induces 
the morphism 
\begin{equation}\label{borneisom}
[\ol{X}'/\mu_n^r] \os{=}{\lra} (\ol{X},Z)^{1/n}
\end{equation}
which is known to be an isomorphism (\cite{borne1}, \cite{borne2}). 
So, if we define $\ol{X}'_m \,(m=0,1,2)$ as the $(m+1)$-fold 
fiber product of $\ol{X}'$ over $[\ol{X}'/\mu_n^r] \cong (\ol{X},Z)^{1/n}$, 
$\ol{X}'_{\b}$ forms a $2$-truncated simplicial scheme over 
$(\ol{X},Z)^{1/n}$ and we have the equivalences 
\begin{align}
& \Isoc((\ol{X},Z)^{1/n}) \os{=}{\ra} 
\Isoc(\ol{X}'_{\b}), \,\,\,\,  
\FIsoc((\ol{X},Z)^{1/n}) \os{=}{\ra} 
\FIsoc(\ol{X}'_{\b}), \label{eqatlas} \\ 
& \FIsoc((\ol{X},Z)^{1/n})^{\circ} \os{=}{\ra} 
\FIsoc(\ol{X}'_{\b})^{\circ}. \nonumber 
\end{align}

We need a globalized version of (the log version of) the equivalences 
\eqref{eqatlas}. To describe it, we introduce several new notions. \par 
Let $X \hra \ol{X}, \ol{X} \setminus X =: Z = \bigcup_{i=1}^r Z_i$, 
$M_{\ol{X}}$ be as above and fix a positive integer $n$ prime to $p$. 
In this subsection, a chart for $(\ol{X},Z)$ is defined to be a pair 
$(\ol{Y}, \{t_i\}_{i=1}^r)$ consisting of an affine scheme $\ol{Y}$ 
endowed with a surjective etale 
morphism $\ol{Y} \lra \ol{X}$ and sections $t_i \in \Gamma(\ol{Y}, 
\cO_{\ol{Y}}) \,(1 \leq i \leq r)$ with $\ol{Y} \times_{\ol{X}} Z_i = 
\{t_i=0\}$. (This notion of a chart is not so different from that of 
a chart of $M_{\ol{X}}$.) As a variant of it, we define the notion of 
a multichart for $(\ol{X},Z)$ to be a triple 
$(\ol{Y},J,\{t_{ij}\}_{1 \leq i \leq r, j \in J})$ consisting of 
an affine scheme $\ol{Y}$ 
endowed with a surjective etale 
morphism $\ol{Y} \lra \ol{X}$, a non-empty finite set $J$ and sections 
$t_{ij} \in \Gamma(\ol{Y}, 
\cO_{\ol{Y}}) \,(1 \leq i \leq r, j \in J)$ with 
$\ol{Y} \times_{\ol{X}} Z_i = \{t_{ij}=0\}$ for any $j \in J$. 
A morphism $(\ol{Y},J,\{t_{ij}\}) \lra (\ol{Y}',J',\{t'_{ij}\})$ 
of multicharts for $(\ol{X},Z)$ is defined to be a pair 
$(\varphi, \varphi^{\sharp})$ consisting of 
$\varphi: \ol{Y} \lra \ol{Y}'$ and $\varphi^{\sharp}: J' \lra J$ 
with $\varphi^* t'_{ij} = t_{i\varphi^{\sharp}(j)} 
\,(1 \leq i \leq r, j \in J')$. \par 
Now let us take a multichart $(\ol{Y},J,\{t_{ij}\})$ and put 
$\ol{Y} =: \Spec R$. Let us denote $M_{\ol{X}} |_{\ol{Y}}$ simply by 
$M_{\ol{Y}}$. 
For $j \in J$, let us  put 
$$ R'_j := R[s_{ij}]_{1 \leq i \leq r}/(s_{ij}^n - t_{ij}), 
\,\,\,\, \ol{Y}'_j := \Spec R'_j. $$
Let $M_{\ol{Y}'_j}$ be the log structure on $\ol{Y}'_j$ 
associated to the homomorphism $\alpha'_j: 
\N^r \lra R'_j; e_i \mapsto s_{ij}$. 
Let $\varphi_j: (\ol{Y}'_j, M_{\ol{Y}'_j}) 
\lra (\ol{Y}, M_{\ol{Y}})$ be the morphism associated to the diagram 
\begin{equation*}
\begin{CD}
R'_j @<<< R \\ 
@A{\alpha'_j}AA @A{\alpha^{\circ}_j}AA \\ 
\N^r @<n<< \N^r, 
\end{CD}
\end{equation*}
where the upper horizontal arrow is the natural inclusion, 
$\alpha^{\circ}_j: \N^r \lra R$ is the map sending $e_i$ to $t_{ij}$ and 
$n:\N^r \lra \N^r$ is the multiplication by $n$. Then $\varphi_j$ 
is a finite Kummer log etale morphism. 
We have the natural action of $\mu_n^r$ (action on $s_{ij}$'s) on 
$(\ol{Y}'_j,M_{\ol{Y}'_j})$ such that $\varphi_j$ is 
$\mu_n^r$-stable. \par 
Now let $(\ol{Y}_0, M_{\ol{Y}_0})$ be the fiber product of 
$(\ol{Y}'_j, M_{\ol{Y}'_j})$'s $(j\in J)$ over $(\ol{Y},M_{\ol{Y}})$ 
in the category of fs log schemes and for $m=0,1,2$, 
let  $(\ol{Y}_m, M_{\ol{Y}_m})$ be the $(m+1)$-fold fiber product of 
$(\ol{Y}_0, M_{\ol{Y}_0})$ over $(\ol{Y},M_{\ol{Y}})$ in the category 
of fs log schemes. If we put $c := |J|$, $(\ol{Y}_m, M_{\ol{Y}_m})$ 
admits a natural sction of $\mu_n^{rc(m+1)}$. Then we have the following 
properties: 

\begin{proposition}\label{key_sr}
Let the notations be as above. Then$:$ 
\begin{enumerate}
\item 
There exists an isomorphism $\ol{Y}_m \cong \ol{Y}_0 \times \mu_n^{rcm}$ 
such that the morphism $\ol{Y}_0 \times \mu_n^{rc\b} \cong \ol{Y}_{\b} 
\lra [\ol{Y}_0/\mu_n^{rc}]$ is the $2$-truncated etale \v{C}ech 
hypercovering 
associated to the quotient map $\ol{Y}_0 \lra [\ol{Y}_0/\mu_n^{rc}]$. 
\item 
There exists an isomorphism 
\begin{equation}\label{multidiag}
[\ol{Y}_0/\mu_n^{rc}] \os{=}{\lra} \ol{Y} \times_{\ol{X}} (\ol{X},Z)^{1/n} 
\end{equation}
such that the log structure associated to the composite 
$$ 
\ol{Y}_m \os{\text{\rm any proj.}}{\lra} \ol{Y}_0 \lra 
[\ol{Y}_0/\mu_n^{rc}] \os{\text{\eqref{multidiag}}}{\lra} 
\ol{Y} \times_{\ol{X}} (\ol{X},Z)^{1/n} \lra [\Af^r_k/\G_{m,k}^r]
$$ 
via {\rm \cite[5.13]{olsson}} is equal to $M_{\ol{Y}_m}$. 
\end{enumerate}
\end{proposition}

\begin{proof}
In this proof, we put $J := \{1,2,...,c\}$. Let us put 
$u_{ij} := t_{ij}t_{i1}^{-1} \in R^{\times} \,(1 \leq i \leq r, j \in J)$ 
and put 
\begin{align*}
R_0 := R[s_{i1}\,(1 \leq i \leq r), s'_{ij} \, & 
(1 \leq i \leq r, 2 \leq j \leq c)]\\ & 
/(s_{i1}^n - t_{i1} \,(1 \leq i \leq r), 
{s'_{ij}}^n- u_{ij}\, (1 \leq i \leq r, 2 \leq j \leq c)). 
\end{align*}
Let us define the log structure $N_0$ on $\Spec R_0$ as the one 
associated to the homomorphism 
$\alpha_1: \N^r \lra R_0; \, e_i \mapsto s_{i1}$. 
Note that, since $s'_{ij}$'s are invertible, $N_0$ is associated also to 
the homomorphism 
$\alpha_j: \N^r \lra R_0; \, e_i \mapsto s_{i1}s'_{ij}$ for $2 \leq j \leq c$. 
Let $\varphi: (\Spec R_0, N_0) 
\lra (\ol{Y}, M_{\ol{Y}})$ be the morphism associated to the diagram 
\begin{equation*}
\begin{CD}
R_0 @<<< R \\ 
@A{\alpha_1}AA @A{\alpha^{\circ}_1}AA \\ 
\N^r @<n<< \N^r, 
\end{CD}
\end{equation*}
where the upper horizontal arrow is the natural inclusion, and 
let $\psi'_j: (\Spec R_0,N_0) \allowbreak 
\lra (\ol{Y}'_j, M_{\ol{Y}'_j})$ be the morphism associated to the diagram 
\begin{equation*}
\begin{CD}
R_0 @<<< R'_j \\ 
@A{\alpha_j}AA @A{\alpha'_j}AA \\ 
\N^r @<{=}<< \N^r, 
\end{CD}
\end{equation*}
where the upper horizontal arrow is the homomorphism over $R$ with 
$s_{i1} \mapsto s_{s1}$ when $j=1$, $s_{ij} \mapsto s_{i1}s'_{ij}$ 
when $2 \leq j \leq c$. Then we have $\varphi_j \circ \psi'_j = \varphi$ 
for all $j \in J$. So $\psi'_j$'s induce the morphism 
$\psi_0: (\Spec R_0,N_0) \lra (\ol{Y}_0,M_{\ol{Y}_0})$ over 
$(\ol{Y},M_{\ol{Y}})$. \par 
Let us prove that $\psi_0$ is an isomorphism. 
To do so, we can work etale locally on $\ol{Y}$ and so 
we may assume that $R$ contains $u_{ij}^{1/n}$ for 
$1 \leq i \leq r, 2 \leq j \leq c$. Also, we put $u_{i1}^{1/n}=1$.
In this situation, $\varphi_j$ is associated also to the diagram 
\begin{equation*}
\begin{CD}
R'_j @<<< R \\ 
@A{\alpha''_j}AA @A{\alpha_1^{\circ}}AA \\ 
\N^r @<n<< \N^r, 
\end{CD}
\end{equation*}
where $\alpha''_j$ is the homomorphism $\N^r \lra R'_j; \,\, 
e_i \mapsto s_{ij}u_{ij}^{1/n}$. Let us define 
$Q$ to be the push-out of homomorphisms of monoids 
$$ \N^r \os{n \circ \text{($j$-th incl.)}}{\lra} 
\bigoplus_{1 \leq j \leq c} \N^r = 
\bigoplus_{
\scriptstyle 1 \leq i \leq r \atop \scriptstyle 1 \leq j \leq c} \N 
\quad (1 \leq j \leq c) $$
and let $Q^{\sat}$ be the saturation of 
$Q$. Then we have the isomorphism 
\begin{equation}\label{satisom}
\N^r \oplus (\bigoplus_{
\scriptstyle 1 \leq i \leq r \atop \scriptstyle 2 \leq j \leq c} 
\Z/n\Z) \os{=}{\lra} Q^{\sat}
\end{equation}
characterized by $(e_i,0) \mapsto e_{i1}, (e_i,-e_{ij}) \mapsto e_{ij}$. 
Let 
$\alpha'':\Z[Q] \lra \bigotimes_{j=1}^c R'_j$ be the homomorphism induced by 
$\alpha''_j$'s. 
Then we can calculate $\Gamma(\ol{Y}_0,\cO_{\ol{Y}_0})$ by the 
equalities 
\begin{align*}
& \Gamma(\ol{Y}_0,\cO_{\ol{Y}_0}) \\ 
= \, & (\bigotimes_{j=1}^c R'_j) \otimes_{\alpha'',\Z[Q]} \Z[Q^{\sat}] \\ 
= \, & (R[s_{ij}]_{1 \leq i \leq r, 1 \leq j \leq c}/(s_{ij}^n - t_{ij})) 
\otimes_{\alpha'',\Z[Q]} \Z[Q^{\sat}] \\ 
= \, & (R[s_{ij}]_{1 \leq i \leq r, 1 \leq j \leq c}/((s_{ij}u_{ij}^{1/n})^n 
- t_{i1})) 
\otimes_{\alpha'',\Z[Q]} \Z[Q^{\sat}] \\ 
= \, & R[\{s_{i1}\}_{1 \leq i \leq r}, \{s''_{ij}\}_{1 \leq i \leq r, 
2 \leq j \leq c}]/(s_{i1}^n - t_{i1}, (s''_{ij})^n -1) \,\,\,\, (s''_{ij} = 
s_{i1}(s_{ij}u_{ij}^{1/n})^{-1}) \\ 
= \, & R[\{s_{i1}\}_{1 \leq i \leq r}, \{s'_{ij}\}_{1 \leq i \leq r, 
2 \leq j \leq c}]/(s_{i1}^n - t_{i1}, (s'_{ij})^n -u_{ij}) \,\,\,\, (s'_{ij} = 
u_{ij}^{1/n}s''_{ij}) \\
= \, & R_0 
\end{align*}
which is equal to the ring homomorphism induced by $\psi_0$. 
Moreover, it is easy to see that 
$\psi_0^*M_{\ol{Y}_0}$, being equal to the log structure associated to 
$Q^{\sat} \lra R_0$ induced by the above diagram, is equal to 
the log structure associated to 
$\N^r \hra Q^{\sat} \lra R_0$, that is, the log structure $N_0$. 
So $\psi_0$ is an isomorphism. 
Note also that 
we have the natural action of $\mu_n^r \times \mu_n^{r(c-1)}$ 
(action on $s_{i1}$'s and $s'_{ij}$'s) on $(\Spec R_0,N_0)$. 
Then, by definition, we can see that the isomorphism $\psi_0$ is 
equivariant with the group isomomorphism 
$$ \mu_n^r \times \mu_n^{r(c-1)} \os{=}{\lra} \mu_n^{rc} $$ 
defined by $(\zeta,1) \mapsto (\zeta, ..., \zeta) \,(\zeta \in \mu_n^r)$, 
$(1,\eta) \mapsto (1,\eta) \,(\eta \in \mu_n^{r(c-1)})$. \par 
Let us put $Z_{0,i} := \psi_0(\{s_{i1}=0\}) \subseteq \ol{Y}_0$. 
Then we see from the isomorphism $\psi_0$ that the log structure 
$M_{\ol{Y}_0}$ is associated to $\bigcup_{i=1}^r Z_{0,i}$. 
Then the commutative diagram 
\begin{equation}\label{logdiag2}
\begin{CD}
\N^r @>>> \ol{M_{\ol{Y}}} @= \bigoplus_{i=1}^r 
\N_{\ol{Y} \times_{\ol{X}} Z_i}\\ 
@V{n}VV @VVV @V{n}VV\\ 
\N^r @>>> \ol{M_{\ol{Y}_0}} @= \bigoplus_{i=1}^r 
\N_{Z_{0,i}} \
\end{CD}
\end{equation}
induces the morphism 
$\ol{Y}_0 \lra \ol{Y} \times_{\ol{X}} (\ol{X},Z)^{1/n}$, and since 
the lower horizontal line of \eqref{logdiag2} is invariant under the action of 
$\mu_n^{rc}$, this diagram further induces the morphism 
$[\ol{Y}_0/\mu_n^{rc}] \lra \ol{Y} \times_{\ol{X}} (\ol{X},Z)^{1/n}$, 
which is the definition of the morphism \eqref{multidiag}. \par 
Let us prove that this is an isomorphism. If we start the constuction 
of \eqref{multidiag} from the multichart $(\ol{Y}, \{1\}, \{t_{i1}\})$ 
instead of $(\ol{Y},J,\{t_{ij}\})$, we obtain the morphism 
\begin{equation}\label{multidiag2}
[\ol{Y}'_1/\mu_n^r] \lra \ol{Y} \times_{\ol{X}} (\ol{X},Z)^{1/n}, 
\end{equation}
and this is an isomorphism because the construction of it is the same as 
that of the isomorphism \eqref{borneisom}. Moreover, the 
morphism of multichart $\iota : (\ol{Y},J,\{t_{ij}\}) \lra (\ol{Y},\{1\}, 
\{t_{i1}\})$ defined by the identity $\ol{Y} \os{=}{\lra} \ol{Y}$ and the 
inclusion $\{1\} \hra J$ induces the factorization 
$$ [\ol{Y}_0/\mu_n^{rc}] \os{\iota'}{\lra} [\ol{Y}'_1/\mu_n^r] 
\os{\text{\eqref{multidiag2}}}{\lra} 
\ol{Y} \times_{\ol{X}} (\ol{X},Z)^{1/n} $$
of \eqref{multidiag}, where $\iota'$ is the map induced by $\iota$. 
So it suffices to prove that $\iota'$ above is an isomorphism, and one 
can see it because $\iota'$ is factorized to the following sequence of 
isomorphisms: 
\begin{align*}
& [\ol{Y}_0/\mu_n^{rc}] \\ \os{\psi_0^{-1}}{\lra} & 
[\Spec R_0/(\mu_n^r \times \mu_n^{r(c-1)})] \\ 
\os{=}{\lra} & 
[\Spec R'_1/\mu_n^r] \times_{\Spec R} 
[(\Spec R[s'_{ij}]_{1 \leq i \leq r, 2 \leq j \leq c}/
({s'_{ij}}^n-u_{ij})_{1 \leq i \leq r, 2 \leq j \leq c})/\mu_n^{r(c-1)}] \\ 
\os{\text{proj.}}{\lra} & [\Spec R'_1/\mu_n^r] = [\ol{Y}'_1/\mu_n^r].
\end{align*}
So we have shown that \eqref{multidiag} is an isomorphism. \par 
Next we prove the assertion (1). For $m=1,2$, we put 
$$R_m := 
R_0[s''_{ijl}]_{1 \leq i \leq r, 1 \leq j \leq c, 1 \leq l \leq m}
/({s''_{ijl}}^n -1)$$ 
and let $N_m$ be the log structure on $\Spec R_m$ associated to the 
homomorphism $\N^r \os{\alpha_1}{\lra} R_0 \os{\subset}{\lra} R_m$. 
Let $\pi_0: (\Spec R_1,N_1) \lra (\Spec R_0,N_0)$ be 
the morphism over $(\ol{Y},M_{\ol{Y}})$ defined by the natural inclusion 
$R_0 \os{\subset}{\lra} R_1$, and let 
$\pi_1: (\Spec R_1,N_1) \lra (\Spec R_0,N_0)$ be 
the morphism over $(\ol{Y},M_{\ol{Y}})$ defined by 
$$R_0 \lra R_1; \, s_{i1} \mapsto s_{i1}s''_{i11}, 
s'_{ij} \mapsto s'_{ij}s''_{ij1}.$$ 
Let $\pi_{01}: (\Spec R_2,N_2) \lra (\Spec R_1,N_1)$ 
over $(\ol{Y},M_{\ol{Y}})$ defined by the natural inclusion 
$R_1 \os{\subset}{\lra} R_2$, and let $\pi_{02}$ 
(resp. $\pi_{12}$) be the morphism 
$(\Spec R_2,N_2) \lra (\Spec R_1,N_1)$ 
over $(\ol{Y},M_{\ol{Y}})$ defined by 
\begin{align*}
& R_1 \lra R_2; \, s_{i1} \mapsto s_{i1}, 
s'_{ij} \mapsto s'_{ij}, s''_{ij1} \mapsto s''_{ij1}s''_{ij2}. \\
& \text{(resp.} \,\, 
R_1 \lra R_2; \, s_{i1} \mapsto s_{i1}s''_{i11}, 
s'_{ij} \mapsto s'_{ij}s''_{ij1}, s''_{ij1} \mapsto s''_{ij2}.) 
\end{align*}
Then we have $\pi_0 \circ \pi_{01} = \pi_0 \circ \pi_{02} (=: \varpi_0)$, 
 $\pi_0 \circ \pi_{12} = \pi_1 \circ \pi_{01} (=: \varpi_1)$, 
 $\pi_1 \circ \pi_{02} = \pi_1 \circ \pi_{12} (=: \varpi_2)$. \par 
Let us prove that $(\Spec R_m,N_m)$ is the $(m+1)$-fold fiber product 
of $(\Spec R_0,N_0)$ over $(\ol{Y},M_{\ol{Y}})$ for $m=1,2$. Let us put 
$Q_1 := \N^r \oplus_{n,\N^r,n} \N^r$, $Q_2 := \N^r \oplus_{n,\N,n} \N^r 
\oplus_{n,\N^r,n} \N^r$ and denote by $Q_1^{\sat}, Q_2^{\sat}$ their 
saturation. Then we have the isomorphism $\N^r \oplus (\Z/n\Z)^r \lra 
Q_1^{\sat}$ as a special case of \eqref{satisom} and the isomorphism 
\begin{align*}
Q_2^{\sat} & = (\N^r \oplus_{n,\N^r,n} Q_1^{\sat})^{\sat} 
= (\N^r \oplus_{n,\N^r,n} \N^r \oplus (\Z/n\Z)^r)^{\sat} \\ 
& = Q_1^{\sat} \oplus (\Z/n\Z)^r = \N^r \oplus (\Z/n\Z)^r \oplus (\Z/n\Z)^r. 
\end{align*}
Let $\alpha_{11}: \Z[Q_1] \lra R_0 \otimes_{R} R_0, \alpha_{12}: \Z[Q_2] \lra 
R_0 \otimes_R R_0 \otimes_R R_0$ be the homomorphism induced by $\alpha_1$. 
Then we can calculate the ring of global sections 
of the $2$-fold fiber product $(\Spec R_0,N_0)$ over $(\ol{Y},M_{\ol{Y}})$ as 
follows: 
\begin{align*}
& (R_0 \otimes_R R_0) \otimes_{\alpha_{11},\Z[Q_1]} \Z[Q_1^{\sat}] \\ 
= \, & (R[s'''_{ij1}]_{1 \leq i \leq r, 1 \leq j \leq c}/
((s'''_{ij1})^n - t_{ij}) 
\otimes_{\alpha_{11},\Z[Q_1]} \Z[Q_1^{\sat}] \\ 
= \, & 
(R[s''_{ij1}]_{1 \leq i \leq r, 1 \leq j \leq c}/
((s''_{ij1})^n - 1) \,\,\,\, (s''_{ij1} = s'''_{ij1}s_{ij}^{-1}) \\ 
= \, & R_1.  
\end{align*}
Moreover, it is easy to see that the log structure on 
the $2$-fold fiber product 
of $(\Spec R_0,N_0)$ over $(\ol{Y},M_{\ol{Y}})$, 
being equal to the log structure associated to 
$Q^{\sat}_1 \lra R_1$ induced by the above diagram, is equal to 
the log structure associated to 
$\N^r \hra Q^{\sat}_1 \lra R_1$, that is, the log structure $N_1$. 
So $(\Spec R_1, N_1)$ is the desired $2$-fold fiber product. 
Similarly, we can calculate the ring of global sections 
of the $3$-fold fiber product $(\Spec R_0,N_0)$ over $(\ol{Y},M_{\ol{Y}})$ as 
follows: 
\begin{align*}
& (R_0 \otimes_R R_0 \otimes_R R_0) 
\otimes_{\alpha_{12},\Z[Q_2]} \Z[Q_2^{\sat}] \\ 
= \, & (R[s'''_{ijl}]_{1 \leq i \leq r, 1 \leq j \leq c, l=1,2}/
((s'''_{ijl})^n - t_{ij}) 
\otimes_{\alpha_{11},\Z[Q_2]} \Z[Q_2^{\sat}] \\ 
= \, & 
(R[s''_{ijl}]_{1 \leq i \leq r, 1 \leq j \leq c,l=1,2}/
((s''_{ijl})^n - 1) \,\,\,\, (s''_{ij1} = s'''_{ij1}s_{ij}^{-1}, 
s''_{ij2} = s'''_{ij2}{s'''_{ij1}}^{-1}) \\ 
= \, & R_2.  
\end{align*}
Moreover, it is easy to see that the log structure on 
the $3$-fold fiber product 
of $(\Spec R_0,N_0)$ over $(\ol{Y},M_{\ol{Y}})$, 
being equal to the log structure associated to 
$Q^{\sat}_2 \lra R_2$ induced by the above diagram, is equal to 
the log structure associated to 
$\N^r \hra Q^{\sat}_2 \lra R_2$, that is, the log structure $N_2$. 
So $(\Spec R_1, N_2)$ is the desired $3$-fold fiber product. 
Hence, 
by using $\psi_0^{-1} \circ \pi_0$ and $\psi_0^{-1} \circ \pi_1$, 
we can define 
the isomorphism $\psi_1: (\Spec R_1, N_1) \os{=}{\lra} 
(\ol{Y}_1,M_{\ol{Y}_1})$ 
over $(\ol{Y},M_{\ol{Y}})$, and 
by using $\psi_0^{-1} \circ \varpi_0, \psi_0^{-1} \circ \varpi_1$ and 
$\psi_0^{-1} \circ \varpi_2$, we can define the isomorphism 
$\psi_2: (\Spec R_1, N_1) \os{=}{\lra} (\ol{Y}_1,M_{\ol{Y}_1})$ 
over $(\ol{Y},M_{\ol{Y}})$. We can also check from the above 
concrete descriptions that, 
by the identification via $\psi_m$'s, 
the morphisms $\pi_{i} \,(i=0,1)$, $\pi_{ij} \,(0 \leq i < j \leq 2)$ 
correspond to the projections between $(\ol{Y}_m, M_{\ol{Y}_m})$'s. 
Recall that $(\Spec R_0,N_0)$ admits the action of 
$\mu_n^r \times \mu_n^{r(c-1)} = \mu_n^{rc}$ (the action on $s_{i1}$'s and 
$s'_{ij}$'s). By the above definition of $R_1$ and $R_2$, we see 
the isomorphisms $\Spec R_m \cong \Spec R_0 \times \mu_n^{rcm}$, and 
via this isomorphisms, the morphisms 
$\pi_{i} \,(i=0,1)$, $\pi_{ij} \,(0 \leq i < j \leq 2)$ are described by 
\begin{align*}
& \pi_0: (y,\eta) \mapsto y, \,\,\,\, \pi_1: (y,\eta) \mapsto y^{\eta}, \\ 
& \pi_{01}: (y,\eta,\zeta) \mapsto (y,\eta), \,\,\, 
\pi_{02}: (y,\eta,\zeta) \mapsto (y,\eta\zeta), \,\,\, 
\pi_{12}: (y,\eta,\zeta) \mapsto (y^{\eta},\zeta), 
\end{align*}
where $y \in \Spec R_0, \eta,\zeta \in \mu_n^{rc}$ and 
the action of $\eta$ is denoted by 
$y \mapsto y^{\eta}$. By this description, we see that 
the diagram 
$\Spec R_{\b} \lra [\Spec R_0/\mu_n^{rc}]$ is the $2$-truncated 
etale \v{C}ech 
hypercovering associated to the quotient map $\Spec R_0 
\lra [\Spec R_0/\mu_n^{rc}]$. Using the identification by $\psi_m$'s, we 
conclude that the diagram $\ol{Y}_{\b} \lra [\ol{Y}_0/\mu_n^{rc}]$ 
 is the $2$-truncated etale \v{C}ech 
hypercovering associated to the quotient map 
$\ol{Y}_0 \lra [\ol{Y}_0/\mu_n^{rc}]$, that is, we have shown the 
assertion (1). \par 
Finally we prove that the morphism \eqref{multidiag} satisfies the assertion
 in (2). For $m=0$, this follows from the definition 
of the morphism \eqref{multidiag} (see the diagram \eqref{logdiag2}). 
For $m=1,2$, this follows from the fact that 
the projections $(\ol{Y}_m, M_{\ol{Y}_m}) \lra (\ol{Y}_0,M_{\ol{Y}_0})$ 
are strict (which follows from the same assertion for 
$\pi_i$'s and $\pi_{ij}$'s which can be seen easily) and the assertion 
in the case $m=0$. So we are done. 
\end{proof}

In view of Proposition \ref{key_sr}, we make the following definition. 

\begin{definition}
Let $(\ol{X},Z)$, $n$ and $(\ol{Y},J,\{t_{ij}\})$ be as above. 
Then we call the $2$-trucated 
simplicial log scheme $(\ol{Y}_{\b},M_{\ol{Y}_{\b}})$ contructed above 
a simplicial resolution of $\ol{Y} \times_{\ol{X}} (\ol{X},Z)^{1/n}$ 
associated to the multichart $(\ol{Y},J,\{t_{ij}\})$. 
\end{definition}

Using this, we define the notion of a bisimplicial resolution as follows: 

\begin{definition}\label{bisimp}
Let $(\ol{X},Z)$ and $n$ be as above and let $(\ol{X}_0, \{t_i\}_{i=1}^r)$ 
be a chart for $(\ol{X},Z)$. Then$:$ 
\begin{enumerate}
\item 
We define the 
$2$-truncated simplicial multichart 
$(\ol{X}_{\b},J_{\b},\{t_{ij}^{(\b)}\})$ associated to 
$(\ol{X}_0, \{t_i\}_{i=1}^r)$ as follows: For $m=0,1,2$, 
let $\ol{X}_m$ be the $(m+1)$-fold fiber product of $\ol{X}_0$ over $\ol{X}$, 
and let us define $J_m := \{0,...,m\}$. 
If we denote the projections $\ol{X}_m \lra 
\ol{X}_0$ by $\pi_a \,(0 \leq a \leq m)$, $t_{ij}^{(a)}$ is defined to be 
$\pi_a^*(t_i)$. We call the $2$-truncated simplicial log scheme 
$(\ol{X}_{\b},M_{\ol{X}_{\b}}) := (\ol{X}_{\b},M_{\ol{X}}|_{\ol{X}_{\b}})$ 
the simplicial semi-resolution of $(\ol{X},Z)^{1/n}$ associated to 
the chart 
$(\ol{X}_0, \{t_i\}_{i=1}^r)$. 
\item 
Keep the notation of $(1)$. 
For $m=0,1,2$, we have the simplicial resolution 
$(\ol{X}_{m\b},M_{\ol{X}_{m\b}})$ of 
$\ol{X}_{m} \times_{\ol{X}} (\ol{X},Z)^{1/n}$ associated to 
the multichart $(\ol{X}_{m},J_{m},\{t_{ij}^{(m)}\})$, and by 
functoriality, they form a $(2,2)$-truncated bisimplicial log scheme 
$(\ol{X}_{\b\b}, \allowbreak M_{\ol{X}_{\b\b}})$. We call it 
the bisimplicial resolution of $(\ol{X},Z)^{1/n}$ associated to 
the chart $(\ol{X}_0, \{t_i\}_{i=1}^r)$. 
\end{enumerate}
\end{definition}

Let the notation be as in Defintion \ref{bisimp}. Then, by Proposition 
\ref{key_sr}, we have $2$-truncated etale \v{C}ech hypercoverings 
$\ol{X}_{m\b} \lra \ol{X}_m \times_{\ol{X}} (\ol{X},Z)^{1/n}$ for 
$m=0,1,2$ and so we have the diagram 
\begin{equation}\label{bisimpdiag}
\ol{X}_{\b\b} \lra \ol{X}_{\b} \times_{\ol{X}} (\ol{X},Z)^{1/n} 
\lra 
(\ol{X},Z)^{1/n}, 
\end{equation}
where the first and the second morphisms are both $2$-truncated 
etale \v{C}ech 
hypercoverings. So we have equivalences 
\begin{equation}\label{bisimpeq1} 
\Isoc((\ol{X},Z)^{1/n}) \os{=}{\lra} 
\Isoc(\ol{X}_{\b} \times_{\ol{X}} (\ol{X},Z)^{1/n}) 
\os{=}{\lra}
\Isoc(\ol{X}_{\b\b}). 
\end{equation} 
We also have equivalences 
\begin{align}
& \FIsoc((\ol{X},Z)^{1/n}) \os{=}{\lra} 
\FIsoc(\ol{X}_{\b} \times_{\ol{X}} (\ol{X},Z)^{1/n}) 
\os{=}{\lra}
\FIsoc(\ol{X}_{\b\b}), 
\label{bisimpeq2} \\ 
& \FIsoc((\ol{X},Z)^{1/n})^{\circ} \os{=}{\lra} 
\FIsoc(\ol{X}_{\b} \times_{\ol{X}} (\ol{X},Z)^{1/n})^{\circ}  
\os{=}{\lra}
\FIsoc(\ol{X}_{\b\b})^{\circ}. \label{bisimpeq3}
\end{align}
We also have the log version: Note that we have the fine log 
stucture $M_{(\ol{X},Z)^{1/n}}$ on $(\ol{X},Z)^{1/n}$ defined by the 
projection $(\ol{X},Z)^{1/n} \lra [\Af^r_k/\G_{m,k}^r]$ 
via \cite[5.13]{olsson}. 
Then, by Proposition \ref{key_sr}(2), we can 
endow the diagram \eqref{bisimpdiag} with log structures and form the 
diagram 
\begin{equation}\label{bisimpdiaglog}
(\ol{X}_{\b\b},M_{\ol{X}_{\b\b}}) \lra 
\ol{X}_{\b} \times_{\ol{X}} ((\ol{X},Z)^{1/n},M_{(\ol{X},Z)^{1/n}}) 
\lra 
((\ol{X},Z)^{1/n}, M_{(\ol{X},Z)^{1/n}}),
\end{equation}
where the first and the second morphisms are both $2$-truncated 
strict etale \v{C}ech 
hypercoverings. So we have equivalences 
\begin{align}
\Isocl((\ol{X},Z)^{1/n},M_{(\ol{X},Z)^{1/n}}) & \os{=}{\lra} 
\Isocl(\ol{X}_{\b} \times_{\ol{X}} (\ol{X},Z)^{1/n}) \label{bisimpeq1log} \\ 
& \os{=}{\lra}
\Isocl(\ol{X}_{\b\b},M_{\ol{X}_{\b\b}}), \nonumber \\ 
\FIsocl((\ol{X},Z)^{1/n},M_{(\ol{X},Z)^{1/n}}) 
& \os{=}{\lra} 
\FIsocl(\ol{X}_{\b} \times_{\ol{X}} (\ol{X},Z)^{1/n}) \label{bisimpeq2log} \\ 
& \os{=}{\lra}
\FIsocl(\ol{X}_{\b\b},M_{\ol{X}_{\b\b}}). \nonumber 
\end{align}
We also have the log version with condition on exponents: 
Note that, for any $1 \leq i \leq r$, the morphism 
$$ (\ol{X},Z)^{1/n} \lra [\Af^r_k/\G_{m,k}^r] \os{\text{$i$-th proj.}}{\lra} 
[\Af_k/\G_{m,k}]$$ induces a fine log structure on $(\ol{X},Z)^{1/n}$ which 
we denote by 
$M_{(\ol{X},Z)^{1/n},i}$. In the notation of Proposition \ref{key_sr} and 
its proof, the pull-back of $M_{(\ol{X},Z)^{1/n}}$ and 
$M_{(\ol{X},Z)^{1/n},i}$ by the morphism 
$$ \ol{Y}_0 \lra [\ol{Y}/\mu_n^{rc}] \os{\text{\eqref{multidiag}}}{\lra} 
\ol{Y} \times_{\ol{X}} (\ol{X},Z)^{1/n} \lra (\ol{X},Z)^{1/n} $$ 
is associated to the simple normal crossing divisor 
$\bigcup_{i=1}^r Z_{0,i}$ and its subdivisor $Z_{0,i}$, respectively (see 
the diagram \eqref{logdiag2}). So the log structure $M_{(\ol{X},Z)^{1/n}}$ 
satisfies the condition $(*)$ in Section 2.1 and 
$\{M_{(\ol{X},Z)^{1/n},i}\}_{i=1}^r$ is a decomposition of 
$M_{(\ol{X},Z)^{1/n}}$, which also induces 
the decomposition 
$\{M_{\ol{X}_{\b} \times_{\ol{X}}(\ol{X},Z)^{1/n},i}\}_i$ of 
$M_{\ol{X}_{\b} \times_{\ol{X}}(\ol{X},Z)^{1/n}} := \allowbreak 
M_{(\ol{X},Z)^{1/n}} \allowbreak 
|_{\ol{X}_{\b} \times_{\ol{X}}(\ol{X},Z)^{1/n}}$ and 
the decomposition 
$\{M_{\ol{X}_{\b\b},i}\}_i$ of 
$M_{\ol{X}_{\b\b}}$. 
Hence we have the equivalence
\begin{align}
& \Isocl((\ol{X},Z)^{1/n},M_{(\ol{X},Z)^{1/n}})_{\Sigma(\ss)} 
\label{bisimpeq1sigma} \\ 
\os{=}{\lra} \,\, & 
\Isocl(\ol{X}_{\b} \times_{\ol{X}} (\ol{X},Z)^{1/n}, 
M_{\ol{X}_{\b} \times_{\ol{X}} (\ol{X},Z)^{1/n}})_{\Sigma(\ss)} 
\nonumber \\
\os{=}{\lra} \,\, & 
\Isocl(\ol{X}_{\b\b},M_{\ol{X}_{\b\b}})_{\Sigma(\ss)} \nonumber 
\end{align}
for $\Sigma = \prod_{i=1}^r \Sigma_i \subseteq \Z_p^r$, by 
Corollary \ref{starlogcor}. \par 
Finally, we explain the functoriality of 
the categories of isocrystals on 
the stack of roots $(\ol{X},Z)^{1/n}$ 
with respect to $n$. By definition, 
we have the morphism 
\begin{align}
(\ol{X},Z)^{1/n'} & =
\ol{X} \times_{[\Af_k^r/\G_{m,k}^r],n'} 
[\Af_k^r/\G_{m,k}^r] \label{tra} 
\\ & \os{\id \times (n'/n)}{\lra} 
\ol{X} \times_{[\Af_k^r/\G_{m,k}^r],n} 
[\Af_k^r/\G_{m,k}^r] = (\ol{X},Z)^{1/n} \nonumber 
\end{align} 
for $n,n'$ with $n \,|\, n'$, where $(n'/n)$ denotes the 
$(n'/n)$-th power map. Using the morphisms \eqref{tra} for 
all $n,n'$ with $n \,|\, n'$, we can form the limit 
$$ \varinjlim_{(n,p)=1} \Isoc((\ol{X},Z)^{1/n}), \,\,
\varinjlim_{(n,p)=1} \FIsoc((\ol{X},Z)^{1/n}), \,\,
\varinjlim_{(n,p)=1} \FIsoc((\ol{X},Z)^{1/n})^{\circ}. $$
We also have the log version: For a positive integer $a$, 
the $a$-th power morphism $[\Af_k^r/\G_{m,k}^r] \lra [\Af_k^r/\G_{m,k}^r]$ 
naturally induces the morphism of fine log algebraic stacks 
$([\Af_k^r/\G_{m,k}^r], M_{[\Af_k^r/\G_{m,k}^r]}) \lra 
([\Af_k^r/\G_{m,k}^r], M_{[\Af_k^r/\G_{m,k}^r]})$ 
by \cite[pp.780 -- 781]{olsson}
and so the morphism 
\eqref{tra} is enriched to a morphism of fine log algebraic stacks 
\begin{equation}\label{traol}
((\ol{X},Z)^{1/n'}, M_{(\ol{X},Z)^{1/n'}}) \lra 
((\ol{X},Z)^{1/n}, M_{(\ol{X},Z)^{1/n}}). 
\end{equation}
So we have also the limit 
$$ \varinjlim_{(n,p)=1} \Isocl((\ol{X},Z)^{1/n},M_{(\ol{X},Z)^{1/n}}), 
\,\,\, 
\varinjlim_{(n,p)=1} \FIsocl((\ol{X},Z)^{1/n},M_{(\ol{X},Z)^{1/n}}). 
$$ 
We also have the log version with exponent condition: 
Let us take $\Sigma := \prod_{i=1}^r \Sigma_i \subseteq \Z_p^r$. 
Fix for the moment two positive integers $n,n'$ with $n \,|\, n'$. 
Let us fix a multichart $(\ol{Y}, J, \{t_{ij}\})$ with $J = \{1\}$ and 
construct $(\ol{Y}_0, M_{\ol{Y}_0}) \cong (\Spec R_0, N_0), 
s_{i1} \in R_0, Z_{0,i} = \{s_{i1}=0\} \,(1 \leq i \leq r)$ as in 
the proof of Proposition \ref{key_sr} for $n$ and $n'$: 
We denote 
$(\ol{Y}_0, M_{\ol{Y}_0}) \cong (\Spec R_0, N_0), s_{i1}, Z_{0,i}$ 
for $n$ (resp. $n'$) 
by 
$(\ol{Y}^{(n)}_0, M_{\ol{Y}^{(n)}_0}) \cong (\Spec R^{(n)}_0, N^{(n)}_0), 
s_{i1}^{(n)}, Z^{(n)}_{0,i}$
(resp. 
$(\ol{Y}^{(n')}_0, M_{\ol{Y}^{(n')}_0}) \cong (\Spec R^{(n')}_0, 
N^{(n')}_0), s_{i1}^{(n')}, Z^{(n')}_{0,i}$). Then we can define the 
log etale morphism 
\begin{equation}\label{nn}
(\ol{Y}^{(n')}_0, M_{\ol{Y}^{(n')}_0}) \lra 
(\ol{Y}^{(n)}_0, M_{\ol{Y}^{(n)}_0})
\end{equation}
fitting into the diagram 
\begin{equation*}
\begin{CD}
(\ol{Y}^{(n')}_0, M_{\ol{Y}^{(n')}_0}) @>{\text{\eqref{nn}}}>> 
(\ol{Y}^{(n)}_0, M_{\ol{Y}^{(n)}_0}) \\ 
@VVV @VVV \\ 
((\ol{X},Z)^{1/n'}, M_{(\ol{X},Z)^{1/n'}}) @>{\text{\eqref{traol}}}>> 
((\ol{X},Z)^{1/n}, M_{(\ol{X},Z)^{1/n}})
\end{CD}
\end{equation*}
by $s^{(n)}_{i1} \mapsto {s^{(n')}_{i1}}^{n'/n}$. 
Then, by Proposition \ref{sig_pre2}, the upper horizontal arrow 
induces the functor 
$$ 
\Isocl(\ol{Y}^{(n)},M_{\ol{Y}^{(n)}})_{n\Sigma(\ss)} 
\lra 
\Isocl(\ol{Y}^{(n')},M_{\ol{Y}^{(n)}})_{n'\Sigma(\ss)}, 
$$ 
where the categories are defined with respect to 
the decomposition $\{Z^{(n)}_{0,i}\}_i$ of 
$\bigcup_i Z^{(n)}_{0,i}$ and the decomposition $\{Z^{(n')}_{0,i}\}_i$ of 
$\bigcup_i Z^{(n')}_{0,i}$. From this, we see the existence of the 
canonical functor 
$$ \Isocl((\ol{X},Z)^{1/n}, M_{(\ol{X},Z)^{1/n}})_{n\Sigma(\ss)} \lra 
\Isocl((\ol{X},Z)^{1/n'}, M_{(\ol{X},Z)^{1/n'}})_{n'\Sigma(\ss)}, $$
and this functor for all $n,n'$ with $n \,|\, n'$ induces the limit 
$$\varinjlim_{(n,p)=1} 
\Isocl((\ol{X},Z)^{1/n}, M_{(\ol{X},Z)^{1/n}})_{n\Sigma(\ss)}. $$

\subsection{Second stacky equivalence}

In this subsection, we prove the equivalence \eqref{seq3} and related 
equivalences. In this subsection, 
let 
$X \hra \ol{X}$ be an open immersion of connected smooth varieties over 
$k$ with $\ol{X} \setminus X =: Z = \bigcup_{i=1}^r Z_i$ a 
simple normal crossing divisor (each $Z_i$ being irreducible). 
Let $v_i \,(1 \leq i \leq r)$ be the discrete valuation 
$k(X)$ corresponding to the generic point of $Z_i$. \par
First we define a functor relating the right hand sides of 
\eqref{seq2} and \eqref{seq3} in a slightly generalized form. 

\begin{proposition}\label{root0}
Let $(X,\ol{X}), Z$ be as above. Let $\cG_X^t$ be the category of 
finite etale Galois tame covering $($tamely ramified at $v_i \,(1 \leq i 
\leq r))$ and for $Y \ra X$ in $\cG_X^t$, 
let $G_Y := \Aut(Y/X)$ and let $\ol{Y}^{\sm}$ be the smooth locus of 
the normalization $\ol{Y}$ of $\ol{X}$ in $k(Y)$. 
Then we have the canonical functors
\begin{align}
& \varinjlim_{Y\to X \in \cG_X^t} \FIsoc([\ol{Y}^{\sm}/G_Y]) \lra 
\varinjlim_{(n,p)=1} \FIsoc((\ol{X},Z)^{1/n}) \label{tmt1} \\ 
& \varinjlim_{(n,p)=1} \Isoc((\ol{X},Z)^{1/n}) \lra 
\Isocd(X,\ol{X}) \label{tmt10} 
\end{align}
which make the following diagram commutative: 
\begin{equation}\label{commu1}
\begin{CD}
\varinjlim_{Y\to X \in \cG_X^t} \FIsoc([\ol{Y}^{\sm}/G_Y]) 
@>{\text{\eqref{tmt1}}}>> 
\varinjlim_{(n,p)=1} \FIsoc((\ol{X},Z)^{1/n}) \\ 
@V{\text{\eqref{eqeq3t}}}VV @V{F\text{-\eqref{tmt10}}}VV \\ 
\FIsocd(X,\ol{X}) @= \FIsocd(X,\ol{X}). 
\end{CD}
\end{equation}
When $X$ is a curve, we have also the functor 
\begin{equation}
\varinjlim_{Y\to X \in \cG_X^t} \Isoc([\ol{Y}^{\sm}/G_Y]) \lra 
\varinjlim_{(n,p)=1} \Isoc((\ol{X},Z)^{1/n}) \label{tmtc1} 
\end{equation}
with $F\text{-\eqref{tmtc1}} = \text{\eqref{tmt1}}$ which makes 
the following diagram commutative: 
\begin{equation}\label{commu1'}
\begin{CD}
\varinjlim_{Y\to X \in \cG_X^t} \Isoc([\ol{Y}^{\sm}/G_Y]) 
@>{\text{\eqref{tmtc1}}}>> 
\varinjlim_{(n,p)=1} \Isoc((\ol{X},Z)^{1/n}) \\ 
@V{\text{\eqref{eqeq3t'}}}VV @V{\text{\eqref{tmt10}}}VV \\ 
\Isocd(X,\ol{X}) @= \Isocd(X,\ol{X}). 
\end{CD}
\end{equation}
\end{proposition}

\begin{proof}
First let us take an object $Y \lra X$ in $\cG_X^t$ and construct 
a functor 
\begin{equation}\label{namba}
\FIsoc([\ol{Y}^{\sm}/G_Y]) \lra \FIsoc((\ol{X},Z)^{1/n})
\end{equation}
for some 
$n$. Let $M_{\ol{X}}, M_{\ol{Y}}$ be the log structure on 
$\ol{X},\ol{Y}$ defined by $Z, \ol{Y} \setminus Y$, respectively. 
Then, by \cite[4.7(c), 7.6]{illusie}, the morphism $Y \lra X$ 
naturally induces a finite Kummer log etale morphism 
$f: (\ol{Y},M_{\ol{Y}}) \lra (\ol{X},M_{\ol{X}})$ of fs log schemes. \par 
In this proof, we follow the convention that 
fiber products of fs log schemes are always taken in the category of 
fs log schemes. For $m =0,1,2$, let 
$(\ol{Y}_{m}, M_{\ol{Y}_{m}})$ be the $(m+1)$-fold fiber product of 
$(\ol{Y},M_{\ol{Y}})$ over $(\ol{X},M_{\ol{X}})$. Then we have 
$(\ol{Y}_{m}, M_{\ol{Y}_{m}}) \cong (\ol{Y},M_{\ol{Y}}) \times G_Y^{m}$ 
(one can see it by using \cite[7.6]{illusie} and noting the isomorphism 
$Y \times_X Y \cong Y \times G$) and so we have the 
equivalence 
\begin{equation}\label{stiso0}
\Isoc([\ol{Y}^{\sm}/G_Y]) \cong \Isoc(\ol{Y}_{\b}^{\sm}), 
\end{equation}
where $\ol{Y}_{m}^{\sm}$ is the smooth locus of $\ol{Y}_m$ and 
$\ol{Y}_{\b}^{\sm}$ is the $2$-truncated simplicial scheme formed by 
 $\ol{Y}_{m}^{\sm} \,(m=0,1,2)$. (Note that $\ol{Y}_{\b}^{\sm}$ here is 
isomorphic to $\ol{Y}_{\b}^{\sm}$ in the proof of Theorem \ref{thm1}.) \par 
By \cite[2.2, 2.6]{cech}, $f$ is of $n$-Kummer type for some 
$n \in \N$ which is prime to $p$ in the terminology of \cite{cech}. 
Fix one such $n$. Take a chart $(\ol{X}_0, \{t_i\}_{1 \leq i \leq r})$ 
for $(\ol{X},Z)$ in the sense of Section 2.3, and let 
$(\ol{X}_{\b},M_{\ol{X}_{\b}})$, 
$(\ol{X}_{\b\b},M_{\ol{X}_{\b\b}})$ be 
the simplicial semi-resolution, 
the bisimplicial resolution 
of $(\ol{X},Z)^{1/n}$ associated to $(\ol{X}_0, \{t_i\}_{1 \leq i \leq r})$ 
respectively 
(see Definition \ref{bisimp}). Then we have the equivalence of 
categories \eqref{bisimpeq2}. 
Let $\ol{X}'$ be the image of $\ol{Y}^{\sm}$ in $\ol{X}$ and let us put 
$(\ol{X}'_{lm}, M_{\ol{X}'_{lm}}) := \ol{X}' \times_{\ol{X}} 
(\ol{X}_{lm}, M_{\ol{X}_{lm}})$. For $k,l,m \in \{0,1,2\}$, let us put 
$(\ol{U}_{klm},M_{\ol{U}_{klm}}) 
:= (\ol{Y}^{\sm}_{k}, M_{\ol{Y}_k}|_{\ol{Y}^{\sm}_k}) 
\times_{(\ol{X}',M_{\ol{X}'})} (\ol{X}'_{lm}, \allowbreak 
M_{\ol{X}'_{lm}})$ and 
let $g_{klm}: (\ol{U}_{klm},M_{\ol{U}_{klm}}) \lra 
(\ol{X}'_{lm}, M_{\ol{X}'_{lm}})$,  
$h_{klm}: (\ol{U}_{klm},M_{\ol{U}_{klm}}) \lra \allowbreak 
(\ol{Y}^{\sm}_{k}, M_{\ol{Y}_k}|_{\ol{Y}^{\sm}_k})$ 
be the projections. Then, since $f$ is of $n$-Kummer type, 
$g_{0lm}$ is a strict finite etale Galois morphism with Galois group $G_Y$ 
 by \cite[2.5]{cech} and so $g_{\b lm}$ is a $2$-truncated 
strict etale \v{C}ech hypercovering. 
Let $\ol{U}_{\b\b\b}$ be the $(2,2,2)$-truncated trisimplicial scheme 
formed by $\ol{U}_{klm}$'s. 
Then, by etale descent and \cite[3.1]{purity}, we have 
the equivalences 
\begin{align}
& \Isoc(\ol{X}'_{\b\b}) \os{=}{\lra} \Isoc(\ol{U}_{\b\b\b}), \label{stiso5} \\ 
& \FIsoc(\ol{X}_{\b\b}) \os{=}{\lra} \FIsoc(\ol{X}'_{\b\b}). 
\label{stiso6}
\end{align}
Now we define the functor \eqref{namba} 
as the composite 
\begin{align}
\FIsoc([\ol{Y}^{\sm}/G_Y]) & \os{F\text{-\eqref{stiso0}}}{\lra} 
\FIsoc(\ol{Y}^{\sm}_{\b}) \label{star} \\ 
& \os{h_{\b\b\b}^*}{\lra} \FIsoc(\ol{U}_{\b\b\b}) \nonumber \\ 
& \os{\text{$F$-\eqref{stiso5}}^{-1}}{\lra} 
\FIsoc(\ol{X}'_{\b\b}) \nonumber \\ 
& \os{\text{\eqref{stiso6}}^{-1}}{\lra} \FIsoc(\ol{X}_{\b\b}) \nonumber \\ 
& \os{\text{\eqref{bisimpeq2}}^{-1}}{\lra} 
\FIsoc((\ol{X},Z)^{1/n}). \nonumber 
\end{align}
Then we see that this functor induces the desired functor 
 \eqref{tmt1}. When $X$ is a curve, we can define the functor 
\eqref{tmtc1} with $F\text{-\eqref{tmtc1}} = \text{\eqref{tmt1}}$ 
in the same way as above: The only problem without 
Frobenius structure is that we do not know the analogue of the equivalence 
\eqref{stiso6} in general, 
but this does not cause any problem when $\ol{X}$ is 
a curve because we have $\ol{X} = \ol{X}'$ in this case. \par 
Next we would like to define the functor \eqref{tmt10} as the one 
induced by the composite 
\begin{align}
\Isoc((\ol{X},Z)^{1/n}) & 
\os{\text{\eqref{bisimpeq1}}}{\lra} 
\Isoc(\ol{X}_{\b\b})  \label{compos} \\ 
& \lra \Isocd(X_{\b\b}, \ol{X}_{\b\b}) \quad (X_{\b\b} := X \times_{\ol{X}} 
\ol{X}_{\b\b}) \nonumber \\ 
& \os{=}{\lla} \Isocd(X_{\b}, \ol{X}_{\b})
\quad (X_{\b} := X \times_{\ol{X}} 
\ol{X}_{\b}) \nonumber \nonumber \\ 
& \os{=}{\lla} \Isocd(X, \ol{X}). \nonumber 
\end{align} 
In order that the functor is well-defined, we should prove that 
the third arrow in the above composite is an equivalence. 
(It is rather easy to see that the fourth arrow is an equivalence because 
$\ol{X}_{\b} \lra \ol{X}$ is a $2$-truncated etale \v{C}ech hypercovering. 
See \cite[5.1]{relativeI} for example.) 
This can be shown in the following way. 
(The proof here is analogous to the proof of \eqref{2barbox} in 
the proof of Theorem \ref{thm1}.) 
We may enlarge $k$ in order that $k$ contains a primitive 
$n$-th root of unity and it suffices to prove the equivalence of the 
restriction functor 
\begin{equation}\label{a}
\Isocd(X_{l}, \ol{X}_{l}) {\lra} \Isocd(X_{l\b}, \ol{X}_{l\b}). 
\end{equation}
In this case, we have $\ol{X}_{lm} \cong \ol{X}_{l0} \times 
\mu_n(k)^{r(l+1)m}$ by Proposition \eqref{key_sr}(1). 
So the right hand side is the category of objects in 
$\Isocd(X_{l0}, \ol{X}_{l0})$ endowed with equivariant 
$\mu_n(k)^{r(l+1)}$-action. 
Then, if we denote the projection $\ol{X}_{l0} \lra \ol{X}_l$ by $\pi$, 
we have the functor 
\begin{equation}\label{aa}
\Isocd(X_{l\b}, \ol{X}_{l\b}) \lra \Isocd(X_{l}, \ol{X}_{l}); 
\quad \cE \mapsto (\pi_*\cE)^{\mu_n(k)^{r(l+1)}}, 
\end{equation}
where $\pi_*$ is the push-out functor defined by Tsuzuki \cite{tsuzuki}. 
By \cite{tsuzuki} and \cite[2.6.8]{kedlayaI}, we have the 
trace morphisms $(\pi_*\pi^*\cE)^{\mu_n(k)^{r(l+1)}} \lra \cE$, 
$\pi^*((\pi_*\cE)^{\mu_n(k)^{r(l+1)}}) \lra \cE$, and they are 
isomorphic in $\Isoc(X_l)$, $\Isoc(X_{l\b})$ by etale descent. Since 
$\Isocd(X_l,\ol{X}_l) \lra \Isoc(X_l), 
\Isocd(X_{l\b},\ol{X}_{l\b}) \lra \Isoc(X_{l\b})$ 
are exact and faithful, 
they are actually isomorphic. Hence \eqref{aa} is a quasi-inverse of 
\eqref{a} and so \eqref{a} is an equivalence. 
So the functor \eqref{tmt10} is defined. \par 
We prove the commutativity of the diagram \eqref{commu1}. By definition, 
\eqref{eqeq3t} is defined as the composite 
\begin{align}
\FIsoc([\ol{Y}^{\sm}/G_Y]) & \lra 
\FIsoc(\ol{Y}^{\sm}_{\b}) \label{maru3} \\ 
& \lra \FIsocd(Y_{\b},\ol{Y}^{\sm}_{\b}) \nonumber \\
& \os{\text{\eqref{2barbox}}}{\lla} 
\FIsocd(X,\ol{X}') \os{\text{\cite[3.1]{purity}}}{\lla} 
\FIsocd(X,\ol{X}). \nonumber 
\end{align}
Noting the commutative square of equivalences 
\begin{equation*}
\begin{CD}
\FIsocd(X,\ol{X}) @>>> \FIsocd(X_{\b\b},\ol{X}_{\b\b}) \\ 
@VVV @VVV \\ 
\FIsocd(X,\ol{X}') @>>> \FIsocd(X_{\b\b},\ol{X}'_{\b\b}) 
\end{CD}
\end{equation*}
and the functoriality of restriction functors, we see that the 
third line in the diagram \eqref{maru3} is rewritten as follows: 
\begin{align*}
\FIsocd(Y_{\b},\ol{Y}^{\sm}_{\b}) 
& \os{h^*_{\b\b\b}}{\lra} 
\FIsocd(U_{\b\b\b},\ol{U}_{\b\b\b}) \\ 
& \os{=}{\lla} \FIsocd(X_{\b\b},\ol{X}'_{\b\b}) \\ 
& \os{=}{\lla} \FIsocd(X_{\b\b},\ol{X}_{\b\b}) 
\os{=}{\lla} \FIsocd(X,\ol{X}). 
\end{align*}
(Here $U_{\b\b\b} := X \times_{\ol{X}} \ol{U}_{\b\b\b}$. 
The equivalence of the second arrow follows from the etale descent for 
the category of overconvergent isocrystals \cite[5.1]{relativeI}.) 
Using this description and the functoriality of restriction functors, 
we see that \eqref{eqeq3t} is rewritten as 
\begin{align*}
\FIsoc([\ol{Y}^{\sm}/G_Y]) & \lra 
\FIsoc(\ol{Y}^{\sm}_{\b}) \\ 
& \os{h^*_{\b\b\b}}{\lra} \FIsoc(\ol{U}_{\b\b\b}) \\ 
& \lra \FIsoc(\ol{X}'_{\b\b}) \\
& \os{=}{\lla} \FIsoc(\ol{X}_{\b\b}) \\ 
& \lra \FIsocd(X_{\b\b},\ol{X}_{\b\b}) \os{=}{\lla} \FIsocd(X,\ol{X}), 
\end{align*}
and we see from definition that this is equal to the composite 
$F\text{-\eqref{tmt10}} \circ \text{\eqref{tmt1}}$. So we have proved 
the commutativity of the diagram \eqref{commu1}. 
When $X$ is a curve, we can prove the commutativity of the diagram 
\eqref{commu1'} in the same way. (We do not have to use \cite[3.1]{purity}
 in this case because we have $\ol{X} = \ol{X}', \ol{X}_{\b\b} = \ol{X}'_{\b\b}$ when $X$ is a curve.) So we are done. 
\end{proof}

\begin{remark}\label{iiylogrem}
In the notation of Proposition \ref{root0}, we have the log 
version (with exponent condition) of the functor \eqref{compos} 
(for any $\Sigma = \prod_{i=1}^r \Sigma_i \subseteq \Z_p^r$): 
\begin{align}
\Isocl((\ol{X},Z)^{1/n},M_{(\ol{X},Z)^{1/n}})_{(\Sigma(\ss))} & 
\os{\text{\eqref{bisimpeq1log}, \eqref{bisimpeq1sigma}}}{\lra} 
\Isocl(\ol{X}_{\b\b},M_{\ol{X}_{\b\b}})_{(\Sigma(\ss))}  
\label{composlog} \\ 
& \lra \Isocd(X_{\b\b}, \ol{X}_{\b\b}) 
\quad \nonumber \\ 
& \os{=}{\lla} \Isocd(X_{\b}, \ol{X}_{\b}) \nonumber \\ 
& \os{=}{\lla} \Isocd(X, \ol{X}). \nonumber 
\end{align}
So we have the functor 
\begin{equation}\label{tmt10log}
\varinjlim_{(n,p)=1} \Isocl((\ol{X},Z)^{1/n})_{(n\Sigma(\ss))}
 \lra  \Isocd(X, \ol{X}), 
\end{equation}
which is the log version (with exponent condition) of \eqref{tmt10}. 
When $\Sigma$ is $\NID$ and $\NLD$, the functor \eqref{composlog} 
(the version with the subscript $\Sigma$ or $\Sigma\ss$) is 
fully faithful since the first arrow \eqref{bisimpeq1sigma} 
is an equivalence and the second arrow is fully faithful by Theorem 
\ref{sigmamain}. (see also Definition \ref{nl}.) Therefore, 
when $\Sigma$ is $\NRD$ and $\SNLD$, the functor \eqref{tmt10log} 
(the version with the subscript $n\Sigma$ or $n\Sigma\ss$) is 
fully faithful by Lemma \ref{lem1.4}. 
\end{remark}

\begin{remark}\label{iimlogrem}
In this remark, we give a construction of the log version (with 
exponent condition) of \eqref{tmtc1} when $X$ is a curve. \par 
In the following, we follow the notation in the proof of Proposition 
\ref{root0}. Then, in the same way as \eqref{star}, we can define the 
functor 
\begin{align}
\Isocl([\ol{Y}/G_Y],M_{[\ol{Y}/G_Y]}) & {\lra} 
\Isocl(\ol{Y}_{\b},M_{\ol{Y}_{\b}}) \label{starlog} \\ 
& \os{h_{\b\b\b}^*}{\lra} \Isocl(\ol{U}_{\b\b\b}, 
M_{\ol{U}_{\b\b\b}}) \nonumber \\ 
& \os{\text{\eqref{stiso5}}^{-1}}{\lra} 
\Isocl(\ol{X}_{\b\b},M_{\ol{X}_{\b\b}}) \nonumber \\ 
& \os{\text{\eqref{bisimpeq1log}}^{-1}}{\lra} 
\Isoc((\ol{X},Z)^{1/n},M_{(\ol{X},Z)^{1/n}}). \nonumber 
\end{align}
(We have $\ol{Y}^{\sm} = \ol{Y}, \ol{X}_{\b\b} = \ol{X}'_{\b\b}$ because 
$X$ is a curve.) 
Also, let us note that, for a morphism $f:Y' \lra Y$ in $\cG_X^t$, 
we have a morphism of log schemes $(\ol{Y}',M_{\ol{Y}'}) \lra 
(\ol{Y},M_{\ol{Y}})$ which is equivariant with  
$G_{Y'} \lra G_Y$ (\cite{illusie}). So the morphism 
$[\ol{Y}'/G_{Y'}] \lra [\ol{Y}/G_Y]$ induced by $f$ 
is enriched to a morphism of fine log algebraic stacks 
$([\ol{Y}'/G_{Y'}], \allowbreak M_{[\ol{Y}'/G_{Y'}]}) 
\lra ([\ol{Y}/G_Y], M_{[\ol{Y}/G_Y]})$, and by the functoriality, 
the functor \eqref{starlog} induces the functor 
\begin{equation}\label{iimlog}
\varinjlim_{Y \ra X \in \cG_X^t} \Isocl([\ol{Y}/G_Y],M_{[\ol{Y}/G_Y]}) 
\lra 
\varinjlim_{(n,p)=1} \Isocl((\ol{X},Z)^{1/n}, M_{(\ol{X},Z)^{1/n}}), 
\end{equation}
which is the log version of the functor \eqref{tmtc1}. \par 
Next let us consider the version with exponent condition. 
let us put $\ol{X} \setminus X =: \{z_1,...,z_r\}$ and take 
$\Sigma = \prod_{i=1}^r \Sigma_i \subseteq \Z_p^r$. 
Then we have the limit $$\varinjlim_{(n,p)=1}\Isocl((\ol{X},Z)^{1/n}, 
M_{(\ol{X},Z)^{1/n}})_{n\Sigma(\ss)}$$ 
as we have seen in Section 2.3. \par 
On the other hand, 
For $a_Y: Y \lra X \in \cG_X^t$, let $a_{\ol{Y}}: \ol{Y} \lra \ol{X}$ 
be the induced morphism. 
For $1 \leq i \leq r$, let 
$e_{\ol{Y},i}$ be the ramification index of $a_{\ol{Y}}$ at a point 
$z$ in $a_{\ol{Y}}^{-1}(z_i)$ (which is independent of the choice of 
$z$) and put $e_{\ol{Y}} := (e_{\ol{Y},i})_{i=1}^r$. Then, we can define the 
decomposition $\{M_{[\ol{Y}/G_Y],i}\}_{i=1}^r$ as in Example \ref{exam2} and 
so we can define the category 
$\Isocl([\ol{Y}/G_Y],M_{[\ol{Y}/G_Y]})_{e_{\ol{Y}}\Sigma(\ss)}$ 
(see Definition \ref{defstacksigma}). 
We also have the category $\Isocl(\ol{Y},M_{\ol{Y}})_{e_{\ol{Y}}\Sigma(\ss)}$, 
using the decomposition $\{M_{[\ol{Y}/G_Y],i}|_{\ol{Y}}\}_{i=1}^r$ of 
$M_{\ol{Y}}$. (Note that this decomposition corresponds to the 
decomposition of the simple normal crossing divisor 
$(\ol{Y} \setminus Y)_{\red}$ into the subdivisors 
$\{a_{\ol{Y}}^{-1}(z_i)_{\red}\}_{i=1}^r$.) Let $f:Y' \lra Y$ be a 
morphism in $\cG_X^t$ and let $\ol{f}: (\ol{Y}',M_{\ol{Y}'}) \lra 
(\ol{Y},M_{\ol{Y}})$ be the 
induced morphism which is finite Kummer log etale. 
Let us take $z \in a_{\ol{Y}}^{-1}(z_i)$ and 
$z' \in \ol{f}^{-1}(z)$, and put $e := e_{\ol{Y}',i}/e_{\ol{Y},i}$. 
Then we have $\ol{f}^*z = ez'$ etale locally on $z$ and $z'$. 
So, by Proposition \ref{sig_pre2}, there exists 
 the canonical restriction functor 
$$ \Isocl(\ol{Y},M_{\ol{Y}})_{e_{\ol{Y}}\Sigma(\ss)} \lra 
 \Isocl(\ol{Y}',M_{\ol{Y}'})_{e_{\ol{Y}'}\Sigma(\ss)}$$ 
and it induces the canonical restriction fucntor 
$$ \Isocl([\ol{Y}/G_Y],M_{[\ol{Y}/G_Y]})_{e_{\ol{Y}}\Sigma(\ss)} \lra 
 \Isocl([\ol{Y}'/G_{Y'}],M_{[\ol{Y}'/G_{Y'}]})_{e_{\ol{Y}'}\Sigma(\ss)}. $$
So we have the limit 
$\varinjlim_{Y\ra X \in \cG_X^t} 
\Isocl([\ol{Y}/G_Y],M_{[\ol{Y}/G_Y]})_{e_{\ol{Y}}\Sigma(\ss)}.$ \par 
In the following, for a Kummer log etale morphism 
$\varphi: (S,M_S) \lra (\ol{X},Z)$ from an fs log scheme $(S,M_S)$,  
we regard that the log structure $M_S$ is endowed with 
the decomposition $\{M_{S,i}\}_{i=1}^r$, where 
$M_{S,i}$ is the log structure associated to $(\varphi^*z_i)_{\red}$. 
(Note that $S$ is necessarily a smooth curve and $M_S$ is 
necessarily equal to 
the log structure associated to $(\varphi^*Z)_{\red}$.) 
Then, by Proposition \ref{ex_ex}, the first functor 
in \eqref{starlog} induces the functor 
\begin{equation}\label{logstarsigma1} 
\Isocl([\ol{Y}/G_Y],M_{[\ol{Y}/G_Y]})_{e_{\ol{Y}}\Sigma(\ss)} {\lra} 
\Isocl(\ol{Y}_{\b},M_{\ol{Y}_{\b}})_{e_{\ol{Y}}\Sigma(\ss)}
\end{equation}
and the third and the fourth functor in 
\eqref{starlog} induces the functor 
\begin{align}
 \Isocl(U_{\b\b\b},M_{U_{\b\b\b}})_{n\Sigma(\ss)} & \lra 
\Isocl(\ol{X}_{\b\b},M_{\ol{X}_{\b\b}})_{n\Sigma(\ss)} 
\label{logstarsigma2} \\ 
& \lra 
\Isoc((\ol{X},Z)^{1/n},M_{(\ol{X},Z)^{1/n}})_{n\Sigma(\ss)} \nonumber 
\end{align}
because the morphisms 
$(\ol{Y}_k,M_{\ol{Y}_K}) \lra ([\ol{Y}/G_Y],M_{[\ol{Y}/G_Y]})$, 
$(U_{klm},M_{U_{klm}}) \lra (\ol{X}_{lm}, \allowbreak M_{\ol{X}_{lm}}) \lra 
((\ol{X},Z)^{1/n}, M_{(\ol{X},Z)^{1/n}})$ are strict etale. 
Let us consider the second functor in \eqref{starlog}. 
By definition, the square 
\begin{equation}\label{sq-uxy}
\begin{CD}
(U_{klm},M_{U_{klm}}) @>>> (\ol{X}_{lm},M_{\ol{X}_{lm}}) \\
@VVV @VVV \\ 
(\ol{Y}_k, M_{\ol{Y}_k}) @>>> (\ol{X},Z)
\end{CD}
\end{equation}
is Cartesian in the category of fs log schemes. 
Let us take points in the schemes in \eqref{sq-uxy} 
over $z_i$ 
\[\xymatrix{
z'' \ar@{|->}[r]  \ar@{|->}[d]  & z' \ar@{|->}[d]  \\ 
z  \ar@{|->}[r]  & z_i. 
}\]
Then, etale locally around $z,z_i$ and $z'$, the diagram 
$(\ol{Y}_k, M_{\ol{Y}_k}) \lra (\ol{X},Z) \lla 
(\ol{X}_{lm},M_{\ol{X}_{lm}})$ admits a chart of the following form: 
\begin{align*}
\begin{CD}
\cO_{\ol{Y}_k} @<<< \cO_{\ol{X}} @>>> \cO_{\ol{X}_{lm}} \\ 
@AAA @AAA @AAA \\ 
\N @<{e_{\ol{Y},i}}<< \N @>n>> \N. 
\end{CD}
\end{align*}
Note that, since $(\ol{Y},M_{\ol{Y}}) \lra (\ol{X},M_{\ol{X}})$ is 
of $n$-Kummer type by definition of $n$ in the proof of 
Proposition \ref{root0}, we have $e_{\ol{Y},i}\,|\, n$.  
Then we have 
the isomorphism of monoids 
$$ (\N \oplus_{e_{\ol{Y},i},\N,n} \N )^{\rm sat} \os{=}{\lra} \N \oplus 
\Z/e_{\ol{Y},i}\Z, \,\,\,\, (1,0) \mapsto (n/e_{\ol{Y},i}, 1), 
(0,1) \mapsto (1,0), $$
where $(-)^{\rm sat}$ denotes the saturation. 
So, etale locally around $z$ and $z''$, the left vertical 
arrow of \eqref{sq-uxy} admits a chart 
\begin{equation*}
\begin{CD}
\cO_{\ol{Y}_k} @>>> \cO_{U_{klm}} \\ 
@AAA @AAA \\ 
\N @>{n/e_{\ol{Y},i}}>> \N. 
\end{CD}
\end{equation*}
From this diagram, we see by Proposition \ref{sig_pre2} that the functor 
$h^*_{\b\b\b}$ in \eqref{starlog} induces the functor 
\begin{equation}\label{logstarsigma3}
\Isocl(\ol{Y}_{\b},M_{\ol{Y}_{\b}})_{e_{\ol{Y}}\Sigma(\ss)} \lra 
 \Isocl(U_{\b\b\b},M_{U_{\b\b\b}})_{n\Sigma(\ss)}. 
\end{equation}
By \eqref{logstarsigma1}, \eqref{logstarsigma2} and \eqref{logstarsigma3}, 
we obtain the functor 
\begin{equation*}
\Isocl([\ol{Y}/G_Y],M_{[\ol{Y}/G_Y]})_{e_Y\Sigma(\ss)} 
\lra 
\Isocl((\ol{X},Z)^{1/n}, M_{(\ol{X},Z)^{1/n}})_{n\Sigma(\ss)} 
\end{equation*}
and by functoriality, it induces the functor 
\begin{align}
\varinjlim_{Y \ra X \in \cG_X^t} 
\Isocl([\ol{Y}/G_Y], & \,\, M_{[\ol{Y}/G_Y]})_{e_Y\Sigma(\ss)} 
\lra \label{iimsigma} \\ 
& \varinjlim_{(n,p)=1} 
\Isocl((\ol{X},Z)^{1/n}, M_{(\ol{X},Z)^{1/n}})_{n\Sigma(\ss)}, \nonumber
\end{align}
which is the log version with exponent condition 
of the functor \eqref{tmtc1}. \par 
Keep the assumption that $X$ is a curve. Then 
we can define the functor 
$$ \Isocl([\ol{Y}/G_Y], M_{\ol{Y}/G_Y})_{(\Sigma(\ss))} 
\lra \Isocd(X,\ol{X}) $$ 
in the same way as \eqref{eqeq3t'} as the composite 
\begin{align*}
\Isocl([\ol{Y}/G_Y], M_{\ol{Y}/G_Y})_{(\Sigma(\ss))} & \os{=}{\lra} 
\Isocl(\ol{Y}_{\b}, M_{\ol{Y}_{\b}})_{(\Sigma(\ss))} \\ 
& \lra 
\Isocd(Y_{\b},\ol{Y}_{\b}) \os{=}{\lla} \Isocd(X,\ol{X}). 
\end{align*}
Hence we have the functor 
\begin{equation}\label{std}
\varinjlim_{Y\to X \in \cG_X^t} 
\Isocl([\ol{Y}/G_Y], M_{\ol{Y}/G_Y})_{(e_{\ol{Y}}\Sigma(\ss))} 
\lra \Isocd(X,\ol{X}). 
\end{equation}
We can check the commutativity of the diagram 
\begin{equation*}
\begin{CD}
\varinjlim_{Y\to X \in \cG_X^t} 
\Isocl([\ol{Y}/G_Y], M_{\ol{Y}/G_Y})_{(e_{\ol{Y}}\Sigma(\ss))} 
@>{\text{\eqref{std}}}>> \Isocd(X,\ol{X}) \\ 
@V{\text{\eqref{iimlog}, \eqref{iimsigma}}}VV @| \\ 
\varinjlim_{(n,p)=1} 
\Isocl((\ol{X},Z)^{1/n}, M_{(\ol{X},Z)^{1/n}})_{(n\Sigma(\ss))} 
 @>{\text{\eqref{composlog}}}>> 
\Isocd(X,\ol{X})
\end{CD}
\end{equation*}
in the same way as that of \eqref{commu1}, \eqref{commu1'}. 
\end{remark}

Now we compare the categories $\Rep_{K^{\sigma}}(\pi_1^t(X))$ and 
$\varinjlim_{(n,p)=1}\FIsoc((\ol{X},Z)^{1/n})^{\circ}$: 

\begin{theorem}\label{thm2}
The composite 
\begin{align}
\Rep_{K^{\sigma}}(\pi_1^t(X)) & \os{\text{\eqref{eqeq2t}}}{\lra} 
\varinjlim_{Y\ra X \in \cG_X^t} \FIsoc ([\ol{Y}^{\sm}/G_Y])^{\circ} 
\label{thm2iso} \\ 
& \os{\text{\eqref{tmt1}}^{\circ}}{\lra} 
\varinjlim_{(n,p)=1} \FIsoc((\ol{X},Z)^{1/n})^{\circ} \nonumber 
\end{align}
is an equivalence of categories. 
$($In particular, the functor $\text{\eqref{tmt1}}^{\circ}$ is 
an equivalence.$)$ 
\end{theorem}

The equivalence \eqref{thm2iso} is nothing but \eqref{seq3}, which is 
a $p$-adic version of \eqref{eq3}. 

\begin{proof}
Note that a part of the functors \eqref{star}${}^{\circ}$
$$ 
\FIsoc(\ol{Y}_{\b}^{\sm})^{\circ} \lra 
\FIsoc(\ol{U}_{\b\b\b})^{\circ} \os{=}{\lla} 
\FIsoc(\ol{X}'_{\b\b}) \os{=}{\lla} 
\FIsoc(\ol{X}_{\b\b})^{\circ}
$$ 
is rewritten via the equivalence \eqref{neweq} in the following way: 
\begin{equation}\label{331/1}
\Sm_{K^{\sigma}}(\ol{Y}_{\b}^{\sm}) \lra 
\Sm_{K^{\sigma}}(\ol{U}_{\b\b\b}) \os{=}{\lla} 
\Sm_{K^{\sigma}}(\ol{X}'_{\b\b}) \os{=}{\lla} 
\Sm_{K^{\sigma}}(\ol{X}_{\b\b}).
\end{equation}
(Here $\ol{Y}_{\b}^{\sm}, \ol{U}_{\b\b\b}, \ol{X}'_{\b\b}, 
\ol{X}_{\b\b}$ are as in the proof of Proposition \ref{root0}.) 
Since $\ol{X}_{\b\b}$, being a part of the data of 
bisimplicial resolution of $(\ol{X},Z)^{1/n}$, depends on $n$, let us denote 
it by $\ol{X}^{(n)}_{\b\b}$ in the sequel in this proof. 
Then \eqref{331/1} induces the funtor 
\begin{equation}\label{331/2}
\varinjlim_{Y \ra X \in \cG_X^t} \Sm_{K^{\sigma}}(\ol{Y}_{\b}^{\sm}) 
\lra \varinjlim_{(n,p)=1} \Sm_{K^{\sigma}}(\ol{X}^{(n)}_{\b\b})
\end{equation}
and using this, we can rewrite the 
functor \eqref{thm2iso} as the composite 
\begin{align}
\Rep_{K^{\sigma}}(\pi_1^t(X)) & 
\os{=}{\lra} \varinjlim_{Y\to X \in \cG} \Sm_{K^{\sigma}}(\ol{Y}^{\sm}_{\b})
 \os{\text{\eqref{331/2}}}{\lra} 
\varinjlim_{(n,p)=1} \Sm_{K^{\sigma}}(\ol{X}^{(n)}_{\b\b})  
\label{331/3} \\  
& \os{\text{\eqref{neweq}},=}{\lra}
\varinjlim_{(n,p)=1} \FIsoc(\ol{X}^{(n)}_{\b\b})
\os{=}{\lla} \varinjlim_{(n,p)=1} \FIsoc((\ol{X},Z)^{1/n}). \nonumber 
\end{align}
To prove the theorem, it suffices to prove that the first line 
\begin{equation}\label{331/4}
\Rep_{K^{\sigma}}(\pi_1^t(X))
\os{=}{\lra} \varinjlim_{Y\to X \in \cG} \Sm_{K^{\sigma}}(\ol{Y}^{\sm}_{\b})
 \os{\text{\eqref{331/2}}}{\lra} 
\varinjlim_{(n,p)=1} \Sm_{K^{\sigma}}(\ol{X}^{(n)}_{\b\b})  
\end{equation} 
of the functor \eqref{331/3} is an equivalence. 
By construction, the composite of the functor \eqref{331/4} and the 
restriction functor 
\begin{equation}\label{331/5}
\varinjlim_{(n,p)=1} \Sm_{K^{\sigma}}(\ol{X}^{(n)}_{\b\b})  
\lra \Sm_{K^{\sigma}}(X_{\b\b})  
\end{equation}
($X_{\b\b}$ are as in the proof of Proposition \ref{root0}) is 
equal to the composite 
$\Rep_{K^{\sigma}}(\pi_1^t(X)) \allowbreak 
\os{\subset}{\lra} \Rep_{K^{\sigma}}(\pi_1(X)) 
\os{=}{\lra} \Sm_{K^{\sigma}}(X_{\b\b})$, which is fully faithful. 
Also, it is easy to see that \eqref{331/5} is faithful. So 
\eqref{331/4} is fully faithful. Also, it is obvious that 
any object $\rho$ in $\Rep_{K^{\sigma}}(\pi_1(X))$ which is sent to 
an object in $\varinjlim_{(n,p)=1} \Sm_{K^{\sigma}}(\ol{X}^{(n)}_{\b\b})  
\subset \Sm_{K^{\sigma}}(X_{\b\b})$ is tamely ramified along $Z$. 
So the functor \eqref{331/4} is an equivalence, as desired. 
\end{proof}

We have seen 
above that the functor $\text{\eqref{tmt1}}^{\circ}$ is an equivalence. 
So it is natural to ask the following question. 

\begin{question}\label{doubt}
Are the functors \eqref{tmt1}, \eqref{tmtc1}, \eqref{iimlog}, 
\eqref{iimsigma} equivalences? 
\end{question}

Here first we prove that the functors \eqref{tmt1} and \eqref{tmtc1}
are equivalences for curves, although we postpone a part of the 
proof to the next section. 

\begin{theorem}\label{curve}
Let $X \hra \ol{X}$ be an open immersion of connected smooth curves such that 
$\ol{X} \setminus X =: Z$ is a simple normal crossing divisor $(=$ disjoint 
union of closed points$)$. 
Then, the functor \eqref{tmtc1} 
$$ \varinjlim_{Y\to X \in \cG_X^t} \Isoc([\ol{Y}/G_Y]) \lra 
\varinjlim_{(n,p)=1} \Isoc((\ol{X},Z)^{1/n})$$  
is an equivalence of categories $($hence so is \eqref{tmt1}$)$. 
\end{theorem}

Bofore the proof, we introduce one terminology: 
In the following, a smooth connected curve $C$ over $k$ is called 
a $(g,l)$-curve when the smooth compactification $\ol{C}$ of $C$ 
has genus $g$ and $(\ol{C} \setminus C) \otimes_{k} \ol{k}$ consists of 
$l$ points. 

\begin{proof}
Assume that $\ol{X}$ is a $(g,l)$-curve and $X$ is a $(g,l')$-curve 
$(l \leq l')$. First we prove the theorem in the case 
$(g,l,l') \not= (0,0,1)$. Fix a positive integer $n$ prime to $p$ 
for the moment, and take a chart $(\ol{X}_0,\{t_i\}_{i=1}^r)$ for 
$(\ol{X},Z)$ in the sense of Section 2.3. Let 
$(\ol{X}_{\b},M_{\ol{X}_{\b}})$, $(\ol{X}_{\b\b},M_{\ol{X}_{\b\b}})$ be
 the simplicial semi-resolution, the bisimplicial resolution of 
$(\ol{X},Z)^{1/n}$ associated to the chart $(\ol{X}_0,\{t_i\}_{i=1}^r)$, 
respectively. 
Let us consider the following claim: \\
\quad \\
{\bf claim 1.} \,\,\, 
Assume $(g,l,l') \not= (0,0,1)$. 
For any $n$, there exists a finite Kummer log etale Galois morphism 
$(\ol{Y}, M_{\ol{Y}}) \lra (\ol{X},M_{\ol{X}})$ such that 
if we put $(\ti{Y}, M_{\ti{Y}}) := 
(\ol{Y}, M_{\ol{Y}}) \times_{(\ol{X},M_{\ol{X}})} 
(\ol{X}_{00},M_{\ol{X}_{00}})$ (the fiber product in the category of 
fs log schemes), the projection 
$(\ti{Y}, M_{\ti{Y}}) \lra (\ol{Y}, M_{\ol{Y}})$ is strict 
etale. \\
\quad \\
First we prove that the claim implies the theorem 
 in the case 
$(g,l,l') \not= (0,0,1)$. Take a positive integer $n$ prime to $p$, 
take a finite Kummer log etale Galois morphism 
$(\ol{Y}, M_{\ol{Y}}) \lra (\ol{X},M_{\ol{X}})$ as in the claim and denote 
its Galois group by $G_Y$. 
For $m =0,1,2$, let $(\ol{Y}_{m}, M_{\ol{Y}_{m}})$ be the 
$(m+1)$-fold fiber product of $(\ol{Y}, M_{\ol{Y}})$ over 
$(\ol{X},M_{\ol{X}})$ and for $k,l,m \in \{0,1,2\}$, 
let us put 
$(\ti{Y}_{kl}, M_{\ti{Y}_{kl}}) := 
(\ol{Y}_{k}, M_{\ol{Y}_{k}}) \times_{(\ol{X},M_{\ol{X}})} 
(\ol{X}_{l},M_{\ol{X}_{l}})$, 
$(\ti{Y}_{klm}, M_{\ti{Y}_{klm}}) := 
(\ol{Y}_{k}, M_{\ol{Y}_{k}}) \times_{(\ol{X},M_{\ol{X}})} 
(\ol{X}_{lm},M_{\ol{X}_{lm}})$. Let us prove the following claim: \\
\quad \\
{\bf claim 2.} \,\, With the above notation, we have the following: 
\begin{enumerate}
\item 
The morphism $(\ti{Y}_{k\b}, M_{\ti{Y}_{k\b}}) \lra (\ol{Y}_k,M_{\ol{Y}_k})$ 
induced by the canonical 
projection is a $2$-truncated strict etale \v{C}ech hypercovering. 
\item 
The morphism $(\ti{Y}_{kl0},M_{\ti{Y}_{kl0}}) \lra 
(\ti{Y}_{kl},M_{\ti{Y}_{kl}})$ is strict etale and it induces the 
$2$-truncated strict etale \v{C}ech hypercovering 
 $(\ti{Y}_{kl\b},M_{\ti{Y}_{kl\b}}) \lra 
(\ti{Y}_{kl},M_{\ti{Y}_{kl}})$. 
\end{enumerate} 
\quad \\
Let us prove claim 2 (assuming claim 1). The assertion (1) is easy 
because the morphism in question is the base change of the 
$2$-truncated strict etale \v{C}ech hypercovering 
$(\ol{X}_{\b}, M_{\ol{X}_{\b}}) \lra (\ol{X},M_{\ol{X}})$ by 
the morphism $(\ol{Y}_k,M_{\ol{Y}_k}) \lra (\ol{X},M_{\ol{X}})$. 
Let us prove the former 
assertion of (2). First, by claim 1, the morphism 
$$ 
(\ti{Y},M_{\ti{Y}}) = (\ol{Y},M_{\ol{Y}}) \times_{(\ol{X},M_{\ol{X}})} 
(\ol{X}_{00},M_{\ol{X}_{00}}) \lra (\ol{Y},M_{\ol{Y}}) 
$$ 
is strict etale. By pulling it back by the morphism 
$$ 
(\ti{Y}_{kl},M_{\ti{Y}_{kl}}) = 
(\ol{Y}_k,M_{\ol{Y}_k}) \times_{(\ol{X},M_{\ol{X}})} 
(\ol{X}_l,M_{\ol{X}_l}) \lra (\ol{Y}_k,M_{\ol{Y}_k}) \lra (\ol{Y},M_{\ol{Y}}), 
$$ 
we see that the morphism
\begin{equation}\label{maru1}
(\ol{Y}_k,M_{\ol{Y}_k}) \times_{(\ol{X},M_{\ol{X}})} 
(\ol{X}_l,M_{\ol{X}_l}) \times_{(\ol{X},M_{\ol{X}})}
(\ol{X}_{00},M_{\ol{X}_{00}}) \lra 
(\ti{Y}_{kl},M_{\ti{Y}_{kl}}) 
\end{equation}
is strict etale. Now note that the log structure on 
$(\ol{X}_l,M_{\ol{X}_l}) \times_{(\ol{X},M_{\ol{X}})}
(\ol{X}_{00},M_{\ol{X}_{00}})$, being equal to the 
pull back of the log structure $M_{\ol{X}_{00}}$, 
is isomorphic to the pull back of the log structure 
$M_{(\ol{X},Z)^{1/n}}$ 
on $(\ol{X},Z)^{1/n}$ by the etale morphism 
$\ol{X}_l \times_{\ol{X}} \ol{X}_{00} \lra \ol{X}_l \times_{\ol{X}} 
(\ol{X},Z)^{1/n} \lra (\ol{X},Z)^{1/n}$. On the other hand, 
by definition, the log structure $M_{\ol{X}_{l0}}$ is also isomorphic 
to the pull back of the log structure 
$M_{(\ol{X},Z)^{1/n}}$ 
on $(\ol{X},Z)^{1/n}$ by the etale morphism 
$\ol{X}_{l0} \lra \ol{X}_l \times_{\ol{X}} 
(\ol{X},Z)^{1/n} \lra (\ol{X},Z)^{1/n}$. So the canonical morphism 
\begin{equation}\label{l0}
(\ol{X}_{l0},M_{\ol{X}_{l0}}) \lra 
(\ol{X}_l,M_{\ol{X}_l}) \times_{(\ol{X},M_{\ol{X}})} 
(\ol{X}_{00},M_{\ol{X}_{00}})
\end{equation}
is strict etale. So the composite 
\begin{align*}
& (\ti{Y}_{kl0},M_{\ti{Y}_{kl0}}) = 
(\ol{Y}_k,M_{\ol{Y}_k}) \times_{(\ol{X},M_{\ol{X}})} 
(\ol{X}_{l0},M_{\ol{X}_{l0}}) \\ 
\os{\id \times \text{\eqref{l0}}}{\lra} \,\, & 
(\ol{Y}_k,M_{\ol{Y}_k}) \times_{(\ol{X},M_{\ol{X}})} 
(\ol{X}_l,M_{\ol{X}_l}) \times_{(\ol{X},M_{\ol{X}})} 
(\ol{X}_{00},M_{\ol{X}_{00}})
 \os{\text{\eqref{maru1}}}{\lra} 
(\ti{Y}_{kl},M_{\ti{Y}_{kl}})
\end{align*}
is strict etale, as desired. \par 
Let us prove the latter assertion of (2). Since 
the morphism 
 $(\ti{Y}_{kl\b},M_{\ti{Y}_{kl\b}}) \lra 
(\ti{Y}_{kl},M_{\ti{Y}_{kl}})$
is the pull back of the $2$-truncated log etale \v{C}ech 
hypercovering $(\ol{X}_{l\b},M_{\ol{X}_{l\b}}) \lra (\ol{X}_l,M_{\ol{X}_l})$, 
it is also a $2$-truncated log etale \v{C}ech 
hypercovering. Then, 
$(\ti{Y}_{kl0},M_{\ti{Y}_{kl0}}) \lra 
(\ti{Y}_{kl},M_{\ti{Y}_{kl}})$ is strict etale by the former assertion 
of (2), we can conclude that it is a $2$-truncated strict etale \v{C}ech 
hypercovering. So we have proved the assertion (2). \par 

By claim 2, we can define the functor 
\begin{align}
\Isoc((\ol{X},Z)^{1/n}) & \os{=}{\lra} \Isoc(\ol{X}_{\b\b}) 
\label{hanamaru}\\ 
& \lra \Isoc(\ti{Y}_{\b\b\b}) \nonumber \\ 
& \os{=}{\lra} \Isoc(\ti{Y}_{\b\b}) \nonumber \\ 
& \os{=}{\lra} \Isoc(\ol{Y}_{\b}) = \Isoc([\ol{Y}/G_Y]). \nonumber  
\end{align}
Varying $n$, we obtain the functor 
$$ \varinjlim_{(n,p)=1} \Isoc((\ol{X},Z)^{1/n}) \lra 
\varinjlim_{Y\to X \in \cG_X^t} \Isoc([\ol{Y}/G_Y]), $$
which gives the inverse of 
\eqref{tmtc1} by construction. So it suffices to prove the claim 1
for the theorem in the case $(g,l,l') \not= (0,0,1)$. \par 
Let us prove the claim 1. To prove the claim, we may replace $k$ by 
its algebraic closure because the scalar extension is ind-etale. 
So we may assume that $k$ is algebraically closed. 
First we prove the claim in the case $l'-l \geq 2$. 
Let us put 
$\ol{X} \setminus X = \{z_1,...,z_r\}$, let $I_i$ be the inertia group 
at $z_i$ and let $\iota_i: I_i \lra \pi_1(X)$ be the canonical map 
(defined up to conjugate). 
Then, to prove the claim, it suffices to show the following: 
When we are given open normal subgroups $J_i \lhd I_i \,(1 \leq i \leq r)$ 
containing wild inertia subgroup, there exists an open normal 
subgroup $N \lhd \pi_1(X)$ such that, for any $1 \leq i \leq r$, 
$\iota_i^{-1}(N)$ is contained in $J_i$ and contains wild inertia subgroup. 
(In fact, we obtain claim 1 if we put $J_i$'s so that $I_i/J_i \cong \Z/n\Z$ 
and if we define $(\ol{Y},M_{\ol{Y}}) \lra (\ol{X},M_{\ol{X}})$ 
to be the finite Kummer log etale Galois covering corresponding to 
$N$.) 
Let us consider the 
following maps 
$$ I_{i} \os{\psi_i}{\lra} 
\pi_1(\Spec \Frac (\cO^h_{\ol{X},z_i})) \os{\psi'_i}{\lra} 
\pi_1(X) \os{\pi}{\lra} \pi_1(\ol{X} \setminus \{z_i,z_{i+1}\})(\text{tame at 
$z_{i+1}$}), $$
where we put $z_{r+1} := z_1$. (Note that $\iota_i = \psi'_i \circ 
\psi_i$.) Then, by Katz \cite{katzpi1}, there exists a morphism 
$\alpha: \pi_1(\ol{X} \setminus \{z_i,z_{i+1}\})(\text{tame at 
$z_{i+1}$}) \lra \pi_1(\Spec \Frac (\cO^h_{\ol{X},z_i}))$ 
satisfying $\alpha \circ \pi \circ \psi'_i = \id$. Let us take 
an open normal subgroup $J'_i \lhd 
\pi_1(\Spec \Frac (\cO^h_{\ol{X},z_i}))$ with 
$\psi_i^{-1}(J'_i) = J_i$. (It is possible because the tame inertia 
group of the henselian field ${\rm Frac}(\cO_{\ol{X},z_i}^h)$ is equal to 
the tame inertia group of its completion.) Then, if we put 
$N_i := \pi^{-1}\alpha^{-1}(J'_i)$, we have $\iota_i^{-1}(N_i) = J_i$. 
Moreover, by construction, $\iota_j^{-1}(N_i)$ contains the wild inertia 
subgroup of $I_j$ even for $j \not= i$. Therefore, if we define 
$N$ to be the maximal normal subgroup of $\pi_1(X)$ contained in 
$\bigcap_{i=1}^r N_i$, this $N$ satisfies the required property. \par 
In the case $l'=l$, we have $X = \ol{X}$ and the claim 1 is obviuosly true 
because all the log structures appearing in the statement of 
claim 1 are trivial in this case. 
So the case $l'-l=1$ remains unproved. In this case, we have 
$(g,l) \not= (0,0)$ because we assumed $(g,l,l') \not= (0,0,1)$ for the 
moment. Then we have a non-trivial finite etale covering
 $\ol{X}' \lra \ol{X}$, and if we put $X' := X \times_{\ol{X}} \ol{X}'$, 
we have $|(\ol{X}' \setminus X')(k)| \geq 2$. 
Then the claim 1 for $X \subseteq \ol{X}$ is reduced to the claim 1 for 
$X' \subseteq \ol{X}'$ and this is true by the previous argument. So 
we have proved the claim 1 and so the proof of the theorem is finished 
when $(g,l,l') \not= (0,0,1)$. \par 
In the case $(g,l,l') = (0,0,1)$, 
we have $\ol{X} = \P^1_k$. So 
the category $\cG_X^t$ contains only the 
trivial covering and we have 
$$ \varinjlim_{Y\to X \in \cG_X^t} \Isoc([\ol{Y}/G_Y]) = \Isoc(\ol{X}) = 
\Isoc(\P^1_k) = 
\{\text{constant objects}\} \os{=}{\lra} \Vect_K, $$
where a constant object means a finite direct sum of 
the structure convergent isocrystal $\cO$, $\Vect_K$ means the 
category of finite dimensional vector spaces over $K$ and the 
last functor is defined by $\cE \mapsto \Hom(\cO,\cE)$. 
(The third equality follows from \cite[4.4]{ogus} and the 
fourth equality follows from the equality 
$\Hom(\cO,\cO) = K$.) 
On the other hand, we shall see later (Proposition \ref{atode}) 
that the category 
$\varinjlim_{(n,p)=1} \Isoc((\ol{X},Z)^{1/n})$ is also naturally 
equivalent to $\Vect_K$. 
So we have proved the theorem also in this case 
modulo Proposition \ref{atode}. 
\end{proof}

Secondly we answer (under certain assumption on $\Sigma$) 
the question for the functors \eqref{iimlog} and 
\eqref{iimsigma}. (We postpone a part of the 
proof to the next section.)

\begin{theorem}\label{curvelog}
Let $X \hra \ol{X}$ be an open immersion of connected smooth curves such that 
$\ol{X} \setminus X =: Z$ is a simple normal crossing divisor $(=$ disjoint 
union of closed points$)$. Assume moreover that $\ol{X}$ is a 
$(g,l)$-curve and that $X$ is a $(g,l')$-curve
 $($so $l \leq l')$. Then$:$ 
\begin{enumerate}
\item 
If $(g,l,l')\not=(0,0,1)$, the functors 
\eqref{iimlog} and \eqref{iimsigma} are equivalences. 
\item 
If $(g,l,l') = (0,0,1)$, the functor \eqref{iimlog} is not 
an equivalence. 
\item 
Assume that 
$\Sigma \subseteq \Z_p$ is $\NRD$, $\SNLD$ and 
that $(g,l,l') = (0,0,1)$. 
Then the functor \eqref{iimsigma} is an equivalence if and only 
if $\Sigma \cap \Z = \Sigma \cap \Z_{(p)}$. 
\end{enumerate}
\end{theorem}

\begin{proof}
In this proof, we follow the notation in the proof of Theorem 
\ref{curve}. First let us prove the assertion (1). 
In the situation in (1), 
by claim 2 in the proof of Theorem \ref{curve}, we have 
the log version of the diagram \eqref{hanamaru}
\begin{align}
\Isocl((\ol{X},Z)^{1/n}, M_{(\ol{X},Z)^{1/n}}) & 
\os{=}{\lra} \Isocl(\ol{X}_{\b\b},M_{\ol{X}_{\b\b}})  
\label{hanamarulog} \\ 
& \lra \Isocl(\ti{Y}_{\b\b\b},M_{\ti{Y}_{\b\b\b}}) \nonumber \\ 
& \os{=}{\lra} \Isocl(\ti{Y}_{\b\b},M_{\ti{Y}_{\b\b}}) \nonumber \\ 
& \os{=}{\lra} \Isocl(\ol{Y}_{\b},M_{\ol{Y}_{\b}}) \nonumber \\ & = 
\Isocl([\ol{Y}/G_Y], M_{[\ol{Y}/G_Y]}) \nonumber 
\end{align}
and by definition, it is easy to see that the induced functor 
$$ \varinjlim_{(n,p)=1} \Isocl((\ol{X},Z)^{1/n}, M_{(\ol{X},Z)^{1/n}}) \lra 
\varinjlim_{Y\to X \in \cG_X^t} 
\Isocl([\ol{Y}/G_Y], M_{[\ol{Y}/G_Y]}) $$
is the inverse of the functor \eqref{iimlog}. \par 
Next let us consider the log 
version with exponent condition. Let us put 
$\ol{X} \setminus X = \{z_1,...,z_r\}$ and let us take 
$\Sigma = \prod_{i=1}^r \Sigma_i \subseteq \Z_p^r$. 
Let us fix a positive integer $n$ prime to $p$ for the moment and choose 
a morphism $(\ol{Y},M_{\ol{Y}}) \lra (\ol{X},M_{\ol{X}})$ as in claim 1. 
Let us define the ramification index 
$e_{\ol{Y}} := (e_{\ol{Y},i})_{i=1}^r$ 
as in Remark \ref{iimlogrem}. Then, in the notation of the proof of 
claim 1, we have 
$I_i/\iota_i^{-1}(N) \cong \Z/e_{\ol{Y},i}\Z, 
I_i/J_i \cong \Z/n\Z, \iota^{-1}(N) \subseteq J_i$. Hence we have 
$n \,|\, e_{\ol{Y},i}$. Then, etale locally around any points 
in the inverse image of $z_i$, the diagram 
$$ (\ol{Y}_k,M_{\ol{Y}_k}) \lra (\ol{X},M_{\ol{X}}) \lla 
(\ol{X}_{lm},M_{\ol{X}_{lm}})$$ 
admits a chart of the following form: 
\begin{equation*}
\begin{CD}
\ol{Y}_k @>>> \ol{X} @<<< \ol{X}_{lm} \\ 
@AAA @AAA @AAA \\ 
\N @<{e_{\ol{Y},i}}<< \N @>n>> \N. 
\end{CD}
\end{equation*}
Using this, we see (by the same argument as in Remark \ref{iimlogrem}) 
that the projection 
$$(\ti{Y}_{klm},M_{\ti{Y}_{klm}}) = 
(\ol{Y}_k,M_{\ol{Y}_k}) \times_{(\ol{X},M_{\ol{X}})} 
(\ol{X}_{lm},M_{\ol{X}_{lm}}) \lra (\ol{X}_{lm},M_{\ol{X}_{lm}})$$ 
admits a chart 
\begin{equation*}
\begin{CD}
\cO_{\ol{X}_{lm}} @>>> \cO_{\ti{Y}_{klm}} \\ 
@AAA @AAA \\ 
\N @>{e_{\ol{Y},i}/n}>> \N 
\end{CD}
\end{equation*}
and from this diagram, we see by Proposition \ref{sig_pre2}(2) 
that the second functor in \eqref{hanamarulog}
induces the functor 
$$ 
\Isocl(\ol{X}_{\b\b},M_{\ol{X}_{\b\b}})_{n\Sigma} \lra 
\Isocl(\ti{Y}_{\b\b\b},M_{\b\b\b})_{e_{\ol{Y}}\Sigma}. 
$$
Since the other functors 
are induced by strict etale morphisms, we can conclude that 
\eqref{hanamarulog} induces the diagram 
\begin{align*}
\Isocl((\ol{X},Z)^{1/n}, M_{(\ol{X},Z)^{1/n}})_{n\Sigma} & 
\os{=}{\lra} \Isocl(\ol{X}_{\b\b},M_{\ol{X}_{\b\b}})_{n\Sigma}  \\ 
& \lra 
\Isocl(\ti{Y}_{\b\b\b},M_{\b\b\b})_{e_{\ol{Y}}\Sigma} \\ 
& \os{=}{\lra} 
\Isocl(\ti{Y}_{\b\b},M_{\ti{Y}_{\b\b}})_{e_{\ol{Y}}\Sigma} \\ 
& \os{=}{\lra} 
\Isocl(\ol{Y}_{\b},M_{\ol{Y}_{\b}})_{e_{\ol{Y}}\Sigma} \\ 
& = 
\Isocl([\ol{Y}/G_Y], M_{[\ol{Y}/G_Y]})_{e_{\ol{Y}}\Sigma}, 
\end{align*}
which induces the functor 
$$ \varinjlim_{(n,p)=1} \Isocl((\ol{X},Z)^{1/n}, 
M_{(\ol{X},Z)^{1/n}})_{n\Sigma} \lra 
\varinjlim_{Y\to X \in \cG_X^t} 
\Isocl([\ol{Y}/G_Y], M_{[\ol{Y}/G_Y]})_{e_{\ol{Y}}\Sigma} $$
giving the inverse to the functor \eqref{iimsigma}. Hence we have proved 
the assertion (1). \par 
Next we prove the assertion (2), by showing that the functor 
\eqref{iimlog} is not essentially surjective in this case. 
First, note that, since the category $\cG_X^t$ contains only the 
trivial covering, we have 
$$ \varinjlim_{Y\to X \in \cG_X^t} \Isocl([\ol{Y}/G_Y],M_{[\ol{Y}/G_Y]}) = 
\Isocl(\ol{X},Z). $$
Let us calculate the category on the other hand side. 
Let us put $Z =: \{z\}$, and take a $k$-rational point $z'$ of $\ol{X}$ other 
than $z$. (It is possible because $\ol{X} \cong \P^1_k$.) Let us put 
$U := \ol{X} \setminus \{z'\}, V := X \cap U$ and for a positive integer $n$ 
prime to $p$, let $\varphi^{(n)}: U^{(n)} \lra U$ be the morphism 
$$ U^{(n)} = \Af^1_k \lra \Af^1_k \cong U $$
induced by $k[t] \lra k[t]; \,\, t \mapsto t^n$ and put 
$V^{(n)} := \varphi^{(n),-1}(V)$. Let $M_{U}$ (resp. $M_{U^{(n)}}$) 
be the log structure on $U$ (resp. $U^{(n)}$) associated to 
$\{z\}$ (resp. $\{0\} \subseteq \Af^1_K = U^{(n)}$). Note that 
$U^{(n)}$ admits the canonical action (action on the coodinate $t$) 
of $\mu_n$ over $U$, and on $[U^{(n)}/\mu_n]$ we have the log structure 
$M_{[U^{(n)}/\mu_n]}$ induced by $M_{U^{(n)}}$. Then we have 
\begin{align}
& \Isocl((\ol{X},Z)^{1/n},M_{(\ol{X},Z)^{1/n}}) \label{hos} \\ 
=\, & 
 \Isocl(X \times_{\ol{X}} ((\ol{X},Z)^{1/n},M_{(\ol{X},Z)^{1/n}})) 
\times_{\Isocl(V \times_{\ol{X}} ((\ol{X},Z)^{1/n},M_{(\ol{X},Z)^{1/n}}))} 
\nonumber \\ 
& \hspace{6cm} \Isocl(U \times_{\ol{X}} ((\ol{X},Z)^{1/n},M_{(\ol{X},Z)^{1/n}})) \nonumber \\ =\, & 
\Isoc(X) \times_{\Isoc(V)} 
\Isocl([U^{(n)}/\mu_n],M_{[U^{(n)}/\mu_n]}) \nonumber \\ 
=\, & 
\Isoc(X) \times_{\Isoc([V^{(n)}/\mu_n])} 
\Isocl([U^{(n)}/\mu_n],M_{[U^{(n)}/\mu_n]}). \nonumber 
\end{align}
Let us take a diagram (with the square Cartesian)
\begin{equation}\label{p1}
\begin{CD}
\cX @<{\supset}<< \cV @>{\subset}>> \cU \\ 
@. @AAA  @A{\ti{\varphi}}AA \\
@. \cV^{(n)} @>{\subset}>> \cU^{(n)}
\end{CD}
\end{equation}
smooth over $\Spf O_K$ lifting 
\begin{equation*}
\begin{CD}
X @<{\supset}<< V @>{\subset}>> U \\ 
@. @AAA  @A{\varphi}AA \\
@. V^{(n)} @>{\subset}>> U^{(n)}
\end{CD}
\end{equation*}
such that $\cX, \cU, \cU^{(n)}$ are isomorphic to $\fAf_{O_K}^1$, 
$\ti{\varphi}$ is induced by $O_K\{t\} \lra O_K\{t\}; \, t \mapsto t^n$ and 
$\cV, \cV^{(n)} \cong \fG_{m,O_K}$. Note that $\cU^{(n)}, \cV^{(n)}$ admits 
the canonical action of $\mu_n$ which lifts the action of $\mu_n$ on 
$U^{(n)}, V^{(n)}$. 
Let $\mu_n\text{-}\LNM_{\cU^{(n)}_K}, \mu_n\text{-}\LNM_{\cV^{(n)}_K}$ be 
the category of log-$\nabla$-modules on $\cU^{(n)}_K, \cV^{(n)}_K$ with respect  to $t(:=$ the coodinate of $\fAf^1_{O_K} = \cU^{(n)}$) with equivariant 
$\mu_n$-action. Then we have the canonical fully faithful functors 
\begin{equation}\label{ffff1}
\Isoc(X) \lra \NM_{\cX_K}, 
\end{equation}
\begin{align}\label{ffff2}
& \Isocl([U^{(n)}/\mu_n],M_{[U^{(n)}/\mu_n]}) \\ \os{=}{\lra} \,\, & 
\{\text{objects in $\Isocl(U^{(n)},M_{U^{(n)}})$ 
with equivariant $\mu_n$-action}\} \nonumber \\ 
\lra \,\, & 
\mu_n\text{-}\LNM_{\cU^{(n)}_K}, \nonumber 
\end{align}
\begin{align}\label{ffff3}
& \Isocl([V^{(n)}/\mu_n],M_{[V^{(n)}/\mu_n]}) \\ \os{=}{\lra} \,\,
& \{\text{objects in $\Isocl(V^{(n)},M_{V^{(n)}})$ 
with equivariant $\mu_n$-action}\} \nonumber \\ 
\lra \,\,
& \mu_n\text{-}\LNM_{\cV^{(n)}_K}. \nonumber 
\end{align}
Let us fix $n \geq 2$ prime to $p$. 
We define an object $\cE := (\cE_0,\cE_1,\iota) \in 
\Isocl((\ol{X},Z)^{1/n}, \allowbreak M_{(\ol{X},Z)^{1/n}})
$ (where 
$\cE_0 \in \Isoc(X)$, $\cE_1 \in \Isocl([U^{(n)}/\mu], M_{[U^{(n)}/\mu]})$ 
and $\iota$ is the isomorphism between the restriction of $\cE_1$ and $\cE_0$ 
in the category $\Isoc([V^{(n)}/\mu_n])$) as follows: 
Let $\cE_0$ be the structure convergent isocrystal on 
$X$, which is sent to $(\cO_{\cX_K},d)$ by \eqref{ffff1}. 
Let $\cE_1$ be the unique object in 
$\Isocl([U^{(n)}/\mu], M_{[U^{(n)}/\mu]})$ which is sent by 
\eqref{ffff2} to $(t\cO_{\cU^{(n)}_K}, d |_{t\cO_{\cU^{(n)}_K}})$ with 
natural action (the action with $\zeta \cdot t = \zeta t$ for 
$\zeta \in \mu_n$). (Note that the log-$\nabla$-module 
$(t\cO_{\cU^{(n)}_K}, d |_{t\cO_{\cU^{(n)}_K}})$ actually comes from a 
log convergent isocrystal because the restriction of it to 
a strict neighborhood of $\cV^{(n)}_K$ in $\cU^{(n)}_K$ comes 
from the structure overconvergent isocrystal on $(V^{(n)},U^{(n)})$. 
See \cite[6.4.1]{kedlayaI}.) $\iota$ is defined to be 
the isomorphism from the restriction of  $\cE_1$ to that of $\cE_0$ 
defined by 
$t\cO_{\cU^{(n)},K} |_{\cV^{(n)}_K} = t\cO_{\cV^{(n)}_K} \hra 
\cO_{\cV^{(n)}_K} = \cO_{\cX_K} |_{\cV^{(n)}_K}$. 
(This is an isomorphism since $t$ is invertible on $\cV^{(n)}$.) 
We denote the induced object in the limit 
$\varinjlim_{(m,p)=1} \Isocl((\ol{X},Z)^{1/m},M_{(\ol{X},Z)^{1/m}})$ 
also by $\cE$. \par 
We prove that the above $\cE$ is not contained in the essential image of 
\eqref{iimlog}. Assume the contrary. Then $\cE \in 
\Isocl((\ol{X},Z)^{1/m},M_{(\ol{X},Z)^{1/m}})$ comes from 
some object $\cF$ in $\Isocl(\ol{X},Z)$ for some $m$ dividing $n$. 
Then, by the commutative diagram 
{\small{\begin{equation*}
\begin{CD}
\Isocl(\ol{X},Z) @>>> \Isocl(U,M_{U}) @>{\subset}>>
\LNM_{\cU_K} \\ 
@V{\text{\eqref{iimlog}}}VV @. @V{\ti{\varphi}_K^*}VV \\ 
\Isocl((\ol{X},Z)^{1/m},M_{(\ol{X},Z)^{1/m}}) 
@>>> \Isocl([U^{(m)}/\mu_n], M_{[U^{(m)}/\mu_m]}) 
@>{\subset}>> \mu_n\text{-}\LNM_{\cU^{(m)}_K} 
\end{CD}
\end{equation*}}}
(where $\ti{\varphi}_K^*$ is the pull back by 
the morphism $\ti{\varphi}_K: \cU^{(m)}_K \lra \cU$ induced by 
$\ti{\varphi}$), there exists a locally free module of finite rank 
$F$ on $\cU_K$ and a $\mu_m$-equivariant isomorphism 
$\ti{\varphi}_K^*F \os{=}{\lra} t^{m/n}\cO_{\cU^{(m)}_K}$. 
(Here $t$ is the coordinate of $\cU^{(m)}_K$.) 
But this is impossible since $t^{m/n}\cO_{\cU^{(m)}_K}$ is not 
generated by $\mu_m$-invariant sections. So we have a contradiction and so 
\eqref{iimlog} is not essentially surjective in this case. Hence we 
have proved the assertion (2). \par 
Let us prove the assertion (3). As in (2), we see that 
$\varinjlim_{Y\to X \in \cG_X^t} \Isocl([\ol{Y}/G_Y], \allowbreak 
M_{[\ol{Y}/G_Y]})_{e_{\ol{Y}}\Sigma(\ss)} = 
\Isoc(\ol{X},Z)_{\Sigma(\ss)}.$ 
For a ring $A$, 
Let $\bX_A \hra \ol{\bX}_A \hookleftarrow \bZ_A$ be 
the diagram $\Af^1_A \hra \P^1_A \hookleftarrow \{\infty\}$ over $\Spec A$. 
When $A=O_K$, this is a lift of the diagram 
$X \hra \ol{X} \hookleftarrow Z$. We denote the $p$-adic completion of this 
diagram in the case $A=O_K$ by 
$\cX \hra \ol{\cX} \hookleftarrow \cZ$. Note that we have the 
canonical fully faithful functor 
\begin{equation}\label{diagsigma}
\Isocl(\ol{X},Z)_{\Sigma(\ss)} \hra \LNM_{(\ol{\cX},\cZ),\Sigma} 
\cong \LNM_{(\ol{\bX}_K,\bZ_K),\Sigma}, 
\end{equation}
where, for a field $A$, $\LNM_{(\ol{\bX}_A,\bZ_A),\Sigma}$ denotes 
the category whose object is a locally free module $E$ 
of finite rank over $\ol{\bX}_A$ endowed with an integrable log connection 
$\nabla: E \lra E \otimes_{\cO_{\ol{\bX}_A}} 
\Omega^1_{\ol{\bX}_A/A}(\log \bZ_A)$ 
relative to $A$ whose exponents along $\bZ_A$ are contained in $\Sigma$ 
in algebraic sense. (Note that 
the latter equivalence in \eqref{diagsigma} follows from GAGA theorem.) 
Note that, since $\Sigma$ is $\NRD$, we have $|\Sigma \cap \Z| \leq 1$.  
We prove the following claim: \\ 
\quad \\
{\bf claim.} \,\, When $\Sigma \cap \Z$ is empty, the catgory 
$\LNM_{(\ol{\bX}_K,\bZ_K),\Sigma}$ is empty. If 
$\Sigma \cap \Z = \{N\}$ is nonempty, any object in 
$\LNM_{(\ol{\bX}_K,\bZ_K),\Sigma}$ is a finite direct sum of the 
object $(\cO_{\ol{\bX}_{K}}(-N\bZ_{K}), d)$. \\
\quad \\
Let us take $(E,\nabla) \in \LNM_{(\ol{\bX}_K,\bZ_K),\Sigma}$. 
Then there exists 
a countable subfield $K_0 \subseteq K$ and $(E_0,\nabla_0) \in 
\LNM_{(\ol{\bX}_{K_0},\bZ_{K_0})}$ such that 
$(E,\nabla) = (E_0 \otimes_{K_0} K,\nabla_0 \otimes_{K_0} K)$. 
Then take an inclusion $K_0 \subseteq \C$. Then we obtain 
$(E_0 \otimes_{K_0} \C, \nabla_0 \otimes_{K_0} \C) \in 
\LNM_{(\ol{\bX}_{\C},\bZ_{\C}),\Sigma}$ such that the restriction of it to 
$\bX_{\C}^{\an} = \Af^{1,\an}_{\C}$ is trivial since 
$\pi_1(\Af^{1,\an}_{\C})$ is trivial. 
If $\Sigma \cap \Z$ is empty, there does not exist such 
$(E_0 \otimes_{K_0} \C, \nabla_0 \otimes_{K_0} \C)$ by monodromy reason. 
So the category $\LNM_{(\ol{\bX}_K,\bZ_K),\Sigma}$ is empty in this case. 
If $\Sigma \cap \Z = \{N\}$, such $(E_0 \otimes_{K_0} \C, 
\nabla_0 \otimes_{K_0} \C)$ is necessarily isomorphic to 
a finite direct sum of $(\cO_{\ol{\bX}_{\C}}(-N\bZ_{\C}), d)$ by 
\cite[II 5.4]{deligne}. 
So we have 
\begin{align*}
& \dim_{K_0} \Hom((\cO_{\ol{\bX}_{K_0}}(-N\bZ_{K_0}), d),(E_0,\nabla_0)) \\ 
= \, 
& \dim_{\C} \Hom((\cO_{\ol{\bX}_{\C}}(-N\bZ_{\C}), d),
(E_0 \otimes_{K_0} \C, 
\nabla_0 \otimes_{K_0} \C)) = \rk E_0
\end{align*} 
and hence $(E_0,\nabla_0)$ is isomorphic to 
a finite direct sum of $(\cO_{\ol{\bX}_{K_0}}(-N\bZ_{K_0}), d)$. 
Therefore $(E,\nabla)$ is isomorphic to 
a finite direct sum of $(\cO_{\ol{\bX}_{K}}(-N\bZ_{K}), d)$, as desired. 
So the claim is proved. \par  
By claim, 
$\Isoc(\ol{X},Z)_{\Sigma(\ss)}$ is empty if $\Sigma \cap \Z$ is empty. 
Let us consider the case $\Sigma \cap \Z = \{N\}$. In this case, 
the object $(\cO_{\ol{\bX}_{K}}(-N\bZ_{K}), d)$
comes from 
an object in $\Isocl(\ol{X},Z)_{\Sigma(\ss)}$ since the restriction of it 
to a strict neighborhood of $\cX_K$ in 
$\ol{\cX}_K$ comes from the structure overconvergent isocrystal on 
$(X,\ol{X})$. So, in this case, \eqref{diagsigma} induces the equivalence 
$$ 
\Isocl(\ol{X},Z)_{\Sigma(\ss)} \os{=}{\lra} \LNM_{(\ol{\cX},\cZ),\Sigma} 
\cong \LNM_{(\bX_K,\bZ_K),\Sigma} \os{=}{\lra} \Vect_K, 
$$ 
where the last functor is defined by 
$(E,\nabla) \mapsto \Hom((\cO_{\ol{\bX}_{K}}(-N\bZ_{K}), d),(E,\nabla))$. \par 
On the other hand, we will prove later (Proposition \ref{atode}) that 
the category 
$\varinjlim_{(n,p)=1} \Isocl((\ol{X},Z)^{1/n}, M_{(\ol{X},Z)^{1/n}})$ 
is empty when $\Sigma \cap \Z_{(p)}$ is empty and equivalent to 
$\Vect_K$ in compatible way as the 
above equivalence when $\Sigma \cap \Z_{(p)}$ 
consists of one element. (Note that we have $|\Sigma \cap \Z_{(p)}| \leq 1$ 
because $\Sigma$ is $\NRD$.) So, in the situation of (3), 
the functor \eqref{iimsigma} is an equivalence if and only if 
$\Sigma \cap \Z = \Sigma \cap \Z_{(p)}$. So we have finished the proof of the 
theorem modulo Proposition \ref{atode}. 
\end{proof}

\section{Parabolic log convergent isocrystals}

Let $X \hra \ol{X}$ be an open immersion of smooth $k$-varieties 
such that $Z :=\ol{X} \setminus X$ is a simple normal crossing divisor and let 
$Z=\bigcup_{i=1}^rZ_i$ be the decomposition of $Z$ into irreducible 
components. We regard $\{Z_i\}_{i=1}^r$ as the fixed decomposition 
of $Z$ in the sense of Definition \ref{defsigexpo2}. 
In this section, we introduce the category 
of (semisimply adjusted) parabolic (unit-root) log convergent 
($F$-)isocrystals 
and prove the equivalence \eqref{seq4}. In the course of the proof, 
we prove the equivalence of the variants of right hand sides of 
\eqref{seq3} and \eqref{seq4} without Frobenius structures, with log 
structures and with exponent conditions. \par 
Before the defintion of 
parabolic log convergent isocrystals on $(\ol{X},Z)$, first we prove 
the existence of certain objects in $\Isocl(\ol{X},Z)$. 

\begin{proposition}\label{twist}
Let $X,\ol{X},Z = \bigcup_{i=1}^rZ_i$ be as above. 
There exists a unique inductive system 
$(\cO(\sum_i\alpha_iZ_i))_{\alpha = (\alpha_i) \in \Z^{r}}$ of objects in 
$\Isocl(\ol{X},Z)$ $($we denote the transition map by 
$\iota^0_{\alpha\beta}: \cO(\sum_i\alpha_iZ_i) \lra \cO(\sum_i\beta_iZ_i)$ 
for $\alpha = (\alpha_i), \beta = (\beta_i) \in \Z^r$ with 
$\alpha_i \leq \beta_i \,(\forall i))$ satisfying the following 
conditions$:$ 
\begin{enumerate}
\item 
$\cO(\sum_i\alpha_iZ_i)$ has exponents in $\{-\alpha\} = 
\prod_{i=1}^r\{-\alpha_i\}$ with semisimple residues. 
\item 
The restriction 
$((j^{\dagger}\cO(\sum_{i}\alpha_iZ_i))_{\alpha}, 
(j^{\dagger}\iota^0_{\alpha\beta})_{\alpha,\beta})$ of 
$((\cO(\sum_{i}\alpha_iZ_i))_{\alpha}, 
(\iota^0_{\alpha\beta})_{\alpha,\beta})$ to an inductive system in 
$\Isocd(X,\ol{X})$ is equal to the constant object 
$((j^{\dagger}\cO),(\id))$, where $j^{\dagger}\cO$ denotes 
the structure overconvegent isocrystal on $(X,\ol{X})$. 
\end{enumerate}
Moreover, it has the following property. 
\begin{enumerate}
\item[$(3)$] 
For any open subscheme $\ol{U} \hra \ol{X}$ and a charted standard small frame 
$((U,\ol{U},\ol{\cX}, \allowbreak i,j),t_1,...,t_r)$ 
enclosing $(U:=X\cap \ol{U}, \ol{U})$ with 
$\cZ = \bigcup_{i=1}^r \cZ_i$ the lift of $Z \cap \ol{U}$, 
the inductive system of log-$\nabla$-module 
$(E_{\alpha},\nabla_{\alpha})_{\alpha}$ 
on 
$(\ol{\cX},\cZ)$ induced by $(\cO(\sum_i\alpha_iZ_i))_{\alpha}$ 
has the form 
\begin{equation}\label{conn}
(E_{\alpha},\nabla_{\alpha}) = 
(\cO_{\ol{\cX}_K}(\sum_i\alpha_i\cZ_{i,K}), d), 
\end{equation}
with $\iota^0_{\alpha\beta}$ equal to the canonical inclusion. 
\end{enumerate}
\end{proposition}
 
\begin{proof}
Let $\tau': (\Z_p/\Z) \setminus \{\ol{0}\} \lra \Z_p$ 
(where $\ol{0}$ is the class of $0$ in $\Z_p/\Z$) 
be any section of 
the projection $\Z_p \lra \Z_p/\Z$, $\tau_n: \Z_p/\Z \lra \Z_p$ be 
the section of the projection extending $\tau'$ with $\tau(\ol{0})=-n$ and 
for $\alpha \in \Z^r$, let us put 
$\tau_{\alpha}:=\prod_{i=1}^r \tau_{\alpha_i}$. 
Then we have the equivalence of categories 
\begin{equation}\label{eq--}
j^{\dagger}: \Isocl(\ol{X},Z)_{\tau_{\alpha}(\0)\text{-}{\rm ss}} 
\os{=}{\lra} 
\Isocd(X,\ol{X})_{\0\text{-}{\rm ss}} 
\end{equation}
by Theorem \ref{sigmamain}. 
So there exists a unique object $\cO(\sum_i\alpha_iZ_i)$ in 
$\Isocl(\ol{X},Z)_{\tau_{\alpha}(\0)\text{-}{\rm ss}} \allowbreak = 
\Isocl(\ol{X},Z)_{\{-\alpha\}\ss}$ with 
$j^{\dagger}(\cO(\sum_i\alpha_iZ_i)) = j^{\dagger}\cO$. 
Moreover, by Proposition \ref{uniqueinj}, we have a unique morphism 
$\iota^0_{\alpha\beta}: \cO(\sum_i\alpha_iZ_i) \lra 
\cO(\sum_i\beta_iZ_i)$ for any 
$\alpha = (\alpha_i), \beta = (\beta_i) \in \Z^r$ with 
$\alpha_i \leq \beta_i \,(\forall i)$ satisfying 
$j^{\dagger}\iota^0_{\alpha\beta} = \id_{j^{\dagger}\cO}$. 
Then the resulting inductive system 
$(\cO(\sum_i\alpha_iZ_i))_{\alpha}$ satisfies 
the conditions (1), (2). The uniqueness is also clear from 
the construction. \par 
Let us prove that 
$(\cO(\sum_i\alpha_iZ_i))_{\alpha}$ satisfies the condition (3). 
In the situation of (3), 
the equivalence \eqref{eq--} holds also for $(U,\ol{U})$. 
Also, note that 
the log-$\nabla$-module \eqref{conn} on $(\ol{\cX}_K,\cZ_K)$ 
is 
restricted to the trivial $\nabla$-module on a strict neighborhood of 
$(\ol{\cX} \setminus \cZ)_K$ in $\ol{\cX}_K$. Hence it defines an object in 
$\Isocl(\ol{U},Z \cap \ol{U})$ by \cite[6.4.1]{kedlayaI}
and it is easy to see that it has 
exponents in $\tau_{\alpha}(\0) = \{-\alpha\}$ 
with semisimple residues (with respect to the decomposition 
$\{Z_i \cap \ol{U}\}_i$ of $Z \cap \ol{U}$). 
Hence \eqref{conn} defines an object in 
$\Isocl(U,\ol{U})_{\tau_{\alpha}(\0)\text{-}{\rm ss}}$. 
Moreover, the transition maps defined in (3) give morphisms in 
$\Isocl(U,\ol{U})$ which reduce to the identity in $\Isocd(U,\ol{U})$. 
So, by the uniqueness in the construction of 
$(\cO(\sum_i\alpha_iZ_i))_{\alpha}$
(this remains true even when we replace $(X,\ol{X})$ by $(U,\ol{U})$), 
we see that the inductive system of log-$\nabla$-modules given in (3) 
is induced from $(\cO(\sum_i\alpha_iZ_i))_{\alpha}$. 
So the property (3) is also proved. 
\end{proof}

For $\cE \in \Isocl(\ol{X},Z)$ and $\alpha = (\alpha_i)_i \in \Z^r$, 
we define $\cE(\sum_i \alpha_iZ_i)$ by 
$\cE(\sum_i \alpha_iZ_i) := \cE \otimes \cO(\sum_i\alpha_iZ_i)$. 
We define the notion of parabolic log convergent isocrystals on 
$(\ol{X},Z)$ as follows: 

\begin{definition}\label{defpar}
Let $X \hra \ol{X}, Z=\bigcup_{i=1}^rZ_i$ be as above. Then a parabolic 
log convergent isocrystal on $(\ol{X},Z)$ is an inductive system 
$(\cE_{\alpha})_{\alpha \in \Z_{(p)}^r}$ of objects in 
$\Isocl(\ol{X},Z)$ $($we denote the transition map by 
$\iota_{\alpha\beta}: \cE_{\alpha} \lra \cE_{\beta}$ for 
$\alpha = (\alpha_i), \beta = (\beta_i) \in \Z_{(p)}^r$ with 
$\alpha_i \leq \beta_i \,(\forall i))$ satisfying the following 
conditions$:$ 
\begin{enumerate}
\item 
For any $1 \leq i \leq r$, there is an isomorphism as inductive systems 
$$ 
((\cE_{\alpha+e_i})_{\alpha}, (\iota_{\alpha+e_i,\beta+e_i})_{\alpha,\beta}) 
\cong 
((\cE_{\alpha}(Z_i))_{\alpha}, 
(\iota_{\alpha\beta} \otimes \id)_{\alpha,\beta}) $$
via which the morphism 
$(\iota_{\alpha,\alpha+e_i})_{\alpha} : 
(\cE_{\alpha})_{\alpha} \lra (\cE_{\alpha+e_i})_{\alpha}$ is identified 
with the morphism 
$(\id \otimes \iota_{\alpha,\alpha+e_i}^0)_{\alpha} : 
(\cE_{\alpha})_{\alpha} \lra (\cE_{\alpha}(Z_i))_{\alpha}$. 
\item 
There exists a positive integer $n$ prime to $p$ satisfying the following 
condition$:$ For any $\alpha = (\alpha_i)_i$, $\iota_{\alpha'\alpha}$ is 
an isomorphism if we put $\alpha' = ([n\alpha_i]/n)_i$. 
\end{enumerate}
We denote by $\PIsoc(\ol{X},Z)$ the category of parabolic 
log convergent isocrystals on $(\ol{X},Z)$. 
\end{definition}

For a parabolic log convergent isocrystal 
$(\cE_{\alpha})_{\alpha}$ on $(\ol{X},Z)$, the transition map 
$\iota_{\alpha\beta}: \cE_{\alpha} \lra \cE_{\beta}$ is always 
injective, because we have the diagram 
$$ \cE_{\alpha} \os{\iota_{\alpha\beta}}{\lra} \cE_{\beta} 
\os{\iota_{\beta,\alpha+N}}{\lra} \cE_{\alpha +N} \cong \cE(\sum_iN_iZ_i) $$ 
for some $N = (N_i)_i \in \Z^r$ and the composite is injective. 
Then, applying $j^{\dagger}$ to the above diagram and noting that 
$j^{\dagger}\cE_{\alpha} \lra j^{\dagger}\cE(\sum_iN_iZ_i)$ is an isomorphism, 
we see that the inductive system $(j^{\dagger}\cE_{\alpha})_{\alpha}$ of 
objects in $\Isocd(X,\ol{X})$ is constant, that is, 
$j^{\dagger}\cE_{\alpha} \in \Isocd(X,\ol{X})$ is independent of 
$\alpha$. \par 
Next we give a definition of 
parabolic (unit-root) log convergent $F$-isocrystals. 

\begin{definition}\label{deff}
A parabolic log convergent $F$-isocrystal on $(\ol{X},Z)$ is a parabolic 
log convergent isocrystal $(\cE_{\alpha})_{\alpha\in\Z_{(p)}^r}$ 
endowed with an isomorphism $\Psi: \varinjlim 
(F^*\cE_{\alpha})_{\alpha\in\Z_{(p)}^r} \allowbreak \os{=}{\lra} \varinjlim 
(\cE_{\alpha})_{\alpha\in\Z_{(p)}^r}$ as ind-objects. It is called 
unit-root if the object $(j^{\dagger}\cE_{\alpha},j^{\dagger}\Psi)$ 
in $\FIsocd \allowbreak (X,\ol{X})$ 
induced by $((\cE_{\alpha})_{\alpha},\Psi)$ is unit-root. 
A morphism $f: ((\cE_{\alpha})_{\alpha},\Psi) \lra 
((\cE'_{\alpha})_{\alpha},\Psi')$ between 
parabolic log convergent $F$-isocrystals is defined to be a 
map of inductive system of log convergent isocrystals 
$(f_{\alpha})_{\alpha}: 
(\cE_{\alpha})_{\alpha} \lra (\cE'_{\alpha})_{\alpha}$ such that 
$(\varinjlim_{\alpha} f_{\alpha}) \circ \Psi = 
\Psi' \circ (\varinjlim_{\alpha} f_{\alpha})$ as morphism of 
ind-objects. 
\end{definition}

We define the notion of `(semisimple) $\Sigma$-adjustedness' for 
a parabolic log convergent isocrystals as follows: 

\begin{definition}\label{defadj}
Let $\Sigma = \prod_{i=1}^r \Sigma_i$ be a subset of $\Z_p^r$ which is 
$\NRD$ and $\SNLD$. Then a parabolic log convergent isocrystal 
$\cE := (\cE_{\alpha})_{\alpha}$ is called $\Sigma$-adjusted 
$($resp. semisimply $\Sigma$-adjusted$)$ if, for any $\alpha = (\alpha_i)_i 
\in \Z_{(p)}^r$, $\cE_{\alpha}$ has exponents in 
$\prod_{i=1}^r (\Sigma_i + ([-\alpha_i, -\alpha_i+1)\cap\Z_{(p)}))$. 
$($resp. $\cE_{\alpha}$ has exponents in 
$\prod_{i=1}^r (\Sigma_i + ([-\alpha_i, -\alpha_i+1)\cap\Z_{(p)}))$ 
with semisimple residues.$)$ When $\Sigma = \0$, we call it 
simply by adjusted $($resp. semisimply adjusted$)$. 
A parabolic log convergent $F$-isocrystal $((\cE_{\alpha})_{\alpha},\Psi)$ 
is called adjusted 
$($resp. semisimply adjusted$)$ 
if so is $(\cE_{\alpha})_{\alpha}$. \par 
We denote the category of $\Sigma$-adjusted $($resp. semisimply 
$\Sigma$-adjusted$)$ parabolic log convergent isocrystals on 
$(\ol{X},Z)$ by 
$\PIsoc(\ol{X},Z)_{\Sigma}$ $($resp. $\PIsoc(\ol{X},Z)_{\Sigmass})$.  
Also, we denote the category of adjusted $($resp. semisimply adjusted$)$ 
parabolic log convergent $F$-isocrystals on $(\ol{X},Z)$ by 
$\PFIsoc(\ol{X},Z)_{\0}$ $($resp. $\PFIsoc(\ol{X}, \allowbreak 
Z)_{\0\ss})$ and 
the category of semisimply adjusted parabolic unit-root log convergent 
$F$-isocrystals on $(\ol{X},Z)$ by 
$\PFIsoc(\ol{X},Z)^{\circ}_{\0\ss}$. 
\end{definition}

\begin{remark}
Let $X \hra \ol{X}, Z= \bigcup_{i=1}^rZ_i$ be as above. Let $Z_{\sing}$ be
 the set of singular points of $Z$ and assume given a subset 
$\Sigma = \prod_{i=1}^r \Sigma_i$ of $\Z_p^r$ which is 
$\NRD$ and $\SNLD$. In this remark, we prove that an object 
$\cE := ((\cE_{\alpha})_{\alpha}, (\iota_{\alpha\beta})_{\alpha,\beta})$ 
in $\PIsoc(\ol{X},Z)$ is contained in 
$\PIsoc(\ol{X},Z)_{\Sigma\ss}$ if and only if it satisfies 
the following condition 
$(**)$: \\
\quad \\
$(**)$ \,\, For any $1 \leq i \leq r$, for any 
open subscheme $\ol{U} \subseteq \ol{X} \setminus Z_{\sing}$ 
containing the generic point of $Z_i$ and any 
charted smooth standard small frame with generic point 
$((U,\ol{U},\ol{\cX}),t,L)$ enclosing $(U,\ol{U})$ (where 
$U := X \cap \ol{U}$), the inductive system of log-$\nabla$-modules 
$(E_{\cE,L,\alpha}, \nabla_{\cE,L,\alpha})$ on $A^1_L[0,1)$ 
induced by $\cE$ has the form 
\begin{equation}\label{un}
(E_{\cE,L,\alpha},\nabla_{\cE,L,\alpha}) = 
\bigoplus_{j=1}^{\mu} (\cO_{A_L^1[0,1)}, 
d + (\gamma_j + \lfloor b_j \rfloor_{-\alpha_i})\dlog t)
\end{equation}
for some $\mu \in \N, \gamma_j \in \Sigma_i, b_j \in [0,1) \cap \Z_{(p)} \, 
(1 \leq j \leq \mu)$ with $\iota_{\alpha\beta}$ equal to the 
multiplication by 
$t^{\lfloor b_j \rfloor_{-\alpha_i}-\lfloor b_j \rfloor_{-\beta_i}}$, 
where, for $a,b \in \Z_{(p)}$, $\lfloor b \rfloor_a$ denotes the 
unique element in $[a,a+1) \cap (b+\Z)$. \\
\quad \\
The `if' part is easy because, if we have the equality \eqref{un}, 
$\cE_{\alpha}$ has exponents in $\prod_{i=1}^r(\Sigma_i+([-\alpha_i,-\alpha_i 
+1) \cap \Z_{(p)}))$ with semisimple residues 
by Lemma \ref{wdlemma}(2). Let us prove the `only if' 
part. By the argument in the proof of Theorem \ref{sigmamain}, 
we see that $(E_{\cE,L,0},\nabla_{\cE,L,0})$ is 
$(\Sigma_i + ([0,1)\cap \Z_{(p)}))$-semisimple. So there exists 
some $\mu \in \N, \gamma_j \in \Sigma_i, b_j \in [0,1) \cap \Z_{(p)} \, 
(1 \leq j \leq \mu)$ such that the equality \eqref{un} holds for 
$\alpha=0$. Let us note now that the inductive system 
$(\bigoplus_{j=1}^{\mu} (\cO_{A_L^1[0,1)}, 
d + (\gamma_j + \lfloor b_j \rfloor_{-\alpha_i})\dlog t))_{\alpha}$ 
above has the following properties: 
\begin{enumerate}
\item 
$
\bigoplus_{j=1}^{\mu} (\cO_{A_L^1[0,1)}, 
d + (\gamma_j + \lfloor b_j \rfloor_{-\alpha_i})\dlog t)
$ is $(\Sigma_i + ([-\alpha_i,-\alpha_i+1)\cap \Z_{(p)}))$-semisimple. 
\item 
Transition maps are isomorphism on $A_L^1[\lambda,1)$ for some $\lambda 
\in (0,1) \cap \Gamma^*$. 
\end{enumerate}
Note also that the inductive system 
$(E_{\cE,L,\alpha},\nabla_{\cE,L,\alpha})_{\alpha}$ also has the 
same properties. So, by Proposition \ref{uniqueinj0}, 
the equality \eqref{un} for $\alpha =0$ extends to the isomorphism 
\eqref{un} as inductive systems. So we have proved the desired claim. 
\end{remark}

Note that, by the argument after Definition \ref{defpar}, we have the 
functor
\begin{equation}\label{par-to-d}
\PIsoc(\ol{X},Z)_{\Sigma(\ss)} 
\os{\subset}{\lra} \PIsoc(\ol{X},Z) \lra \Isocd(X,\ol{X})
\end{equation}
defined by $(\cE_{\alpha})_{\alpha} \mapsto j^{\dagger}\cE_{\alpha}$. 
Then the key result in this section is given as follows: 

\begin{theorem}\label{mainpar}
Let $X,\ol{X},Z= \bigcup_{i=1}^r Z_i$ be as above and let $\Sigma = 
\prod_{i=1}^r \Sigma_i$ be a subset of $\Z_p^r$ which is 
$\NRD$ and $\SNLD$. Then there exists 
the canonical equivalence of categories 
\begin{equation}\label{root-to-par}
\varinjlim_{(n,p)=1} \Isocl((\ol{X},Z)^{1/n}, 
M_{(\ol{X},Z)^{1/n}})_{n\Sigma(\ss)} \os{=}{\lra} 
\PIsoc(\ol{X},Z)_{\Sigma(\ss)}  
\end{equation}
which makes the following diagram commutative$:$
\begin{equation}\label{commutat}
\begin{CD}
\varinjlim_{(n,p)=1} \Isocl((\ol{X},Z)^{1/n},
 M_{(\ol{X},Z)^{1/n}})_{n\Sigma(\ss)} 
@>{\text{\eqref{root-to-par}}}>> 
\PIsoc(\ol{X},Z)_{\Sigma(\ss)} \\
@V{\text{\eqref{tmt10log}}}VV @V{\text{\eqref{par-to-d}}}VV \\
\Isocd(X,\ol{X}) @= \Isocd(X,\ol{X}). 
\end{CD}
\end{equation}
\end{theorem}

\begin{proof}
First we prove that the functor \eqref{tmt10log} induces the 
equivalence of categories 
\begin{equation}\label{step1}
\varinjlim_{(n,p)=1} \Isocl((\ol{X},Z)^{1/n}, 
M_{(\ol{X},Z)^{1/n}})_{n\Sigma(\ss)} \os{=}{\lra} 
\Isocd(X,\ol{X})_{\ol{\Sigma}(\ss)},  
\end{equation}
where $\ol{\Sigma} := \im (\Sigma + \Z_{(p)}^r \hra \Z_p^r \lra 
\Z_p^r/\Z^r)$. 
To see this, it suffices to prove that the functor \eqref{composlog}
induces the equivalence 
\begin{equation}\label{step11}
\Isocl((\ol{X},Z)^{1/n},M_{(\ol{X},Z)^{1/n}})_{n\Sigma(\ss)} 
\os{=}{\lra} 
\Isocd(X,\ol{X})_{\ol{\Sigma}_n(\ss)},  
\end{equation}
where $\ol{\Sigma}_n := \im (\Sigma + (\frac{1}{n}\Z)^r \hra \Z_p^r \lra 
\Z_p^r/\Z^r)$. 
We have already seen in Remark \ref{iiylogrem} 
that the functor 
\begin{equation}\label{kyo}
\Isocl((\ol{X},Z)^{1/n},M_{(\ol{X},Z)^{1/n}})_{n\Sigma(\ss)} \lra 
\Isocd(X,\ol{X}) 
\end{equation} 
defined in \eqref{composlog} 
is fully faithful. So 
it suffices to prove that essential image of this functor is equal to 
$\Isocd(X,\ol{X})_{\ol{\Sigma}_n(\ss)}$. 
Let $\cE$ be an object in the essential image of \eqref{kyo}.
To see that it is in 
$\Isocd(X,\ol{X})_{\ol{\Sigma}_n(\ss)}$, 
it suffices to check it locally around generic points of $Z$. 
So we may assume that $\ol{X} = \Spec \ol{R}$ is affine, 
$Z$ is smooth and defined as the zero locus 
of some element $\ol{t} \in \ol{R}$. Let $M_{\ol{X}}$ be the log structure on 
$\ol{X}$ associated to $Z$, let $(\ol{X}_{0},M_{\ol{X}_0})$ be the log 
scheme defined by $\ol{X}_0 := \Spec \ol{R}[\ol{s}]/(\ol{s}^n-\ol{t})$, 
$M_{\ol{X}_0} := \text{log structure associated to $\{\ol{s}=0\}$}$ and 
let $(\ol{X}_{m},M_{\ol{X}_m}) \,(m=0,1,2)$ be the $(m+1)$-fold 
fiber product of $(\ol{X}_{0},M_{\ol{X}_0})$ over $(\ol{X},M_{\ol{X}})$. 
Also, let us put $X_{\b} := X \times_{\ol{X}} \ol{X}_{\b}$. 
Then the functor $\text{\eqref{kyo}} = \text{\eqref{composlog}}$ 
in this case is equal to the composite 
\begin{align*}
\Isocl((\ol{X},Z)^{1/n},M_{(\ol{X},Z)^{1/n}})_{n\Sigma(\ss)} 
& \os{=}{\lra} \Isocl(\ol{X}_{\b},M_{\ol{X}_{\b}})_{n\Sigma(\ss)} \\ 
& \lra \Isocd(X_{\b},\ol{X}_{\b}) \os{=}{\lla} \Isocd(X,\ol{X}). 
\end{align*}
Hence the restriction of $\cE$ to $\Isocd(X_0,\ol{X}_0)$ extends to 
an object in the category 
$\Isocl(\ol{X}_0,M_{\ol{X}_0})_{n\Sigma(\ss)}$. Now let us take a 
charted smooth standard small frame with genetic point 
$(\bX,t,L) := ((X,\ol{X},\ol{\cX}), t, L)$ enclosing $(X,\ol{X})$ 
such that $t$ is a lift of $\ol{t}$. 
Then we have a 
charted smooth standard small frame with genetic point of the form 
$(\bX_0,s,L) := ((X_0,\ol{X}_0,\ol{\cX}_0), s, L)$ and a morphism 
$\psi: (\bX_0,s,L) \lra (\bX,t,L)$ with $\psi^*(t) = s^n$ lifting the morphism 
$\ol{X}_0 \lra \ol{X}$, by Lemma \ref{lifting} and Remark \ref{liftingrem}. 
It 
naturally induces the map $\psi_n: A_L^1[0,1) \lra A_L^1[0,1)$ 
between `discs at generic points' defined by 
$t \mapsto t^n$. Let $(E,\nabla)$ be the $\nabla$-module on 
$A_L^1[\lambda,1)$ (for some $\lambda \in (0,1)\cap \Gamma^*$) which 
is induced by $\cE$. Then, since 
the restriction of $\cE$ to $\Isocd(X_0,\ol{X}_0)$ extends to 
an object (which we denote by $\ti{\cE}$) 
in $\Isocl(\ol{X}_0,M_{\ol{X}_0})_{n\Sigma(\ss)}$, 
we see that $\psi_n^*(E,\nabla)$ 
extends to a log-$\nabla$-module on $A_L^1[0,1)$ which is induced by 
$\ti{\cE}$. Then, by Theorem \ref{sigmamain}, $\psi_n^*(E,\nabla)$ is 
$n\Sigma$-unipotent ($n\Sigma$-semisimple). 
Then $\psi_{n,*}\psi_n^*(E,\nabla)$ 
is written as a successive extension by the objects of the form 
(as a direct sum of the objects of the form)
$$ 
\psi_{n,*}(M_{n\xi},\nabla_{M_{n\xi}}) \cong 
\bigoplus_{i=0}^{n-1}(M_{\xi+\frac{i}{n}},\nabla_{M_{\xi+\frac{i}{n}}}) 
\,\,\,\, (\xi \in \Sigma). $$
In particular, it is 
$(\Sigma + \{\frac{i}{n}\,|\,0 \leq i \leq n-1\})$-unipotent 
($(\Sigma + \{\frac{i}{n}\,|\,0 \leq i \leq n-1\})$-semisimple). 
So $(E,\nabla)$, being a direct summand of $\psi_{n,*}\psi_n^*(E,\nabla)$, is 
also 
$(\Sigma + \{\frac{i}{n}\,|\,0 \leq i \leq n-1\})$-unipotent 
($(\Sigma + \{\frac{i}{n}\,|\,0 \leq i \leq n-1\})$-semisimple). 
So we have shown that $\cE$ 
has $\ol{\Sigma}_n$-unipotent 
($\ol{\Sigma}_n$-semisimple) generic monodromy. \par 
Conversely, let $\cE$ be an object in 
$\Isocd(X,\ol{X})_{\ol{\Sigma}_n(\ss)}$ and prove that 
it is in the essential image of \eqref{kyo}. 
Since \eqref{kyo} is fully faithful, it suffices to prove it 
Zariski locally. 
So let $\ol{X}_{\b}, L, \psi_n, ...$ as in the previous paragraph. 
Then the $\nabla$-module $(E,\nabla)$ on $A_L^1[\lambda,1)$ induced by 
$\cE$ is $(\Sigma + \{\frac{i}{n}\,|\,0 \leq i \leq n-1\})$-unipotent 
($(\Sigma + \{\frac{i}{n}\,|\,0 \leq i \leq n-1\})$-semisimple). 
Hence one can see by easy calculation that 
$\psi_n^*(E,\nabla)$ on $A^1_L[\lambda^{1/n},1)$ is $n\Sigma$-unipotent 
($n\Sigma$-semisimple). 
So 
the restriction of $\cE$ to $\Isocd(X_0,\ol{X}_0)$ 
has $n\Sigma$-unipotent ($n\Sigma$-semisimple) generic monodromy and 
by Theorem \ref{sigmamain}, 
it extends uniquely to 
an object (which we denote by $\ti{\cE}_0$) 
in $\Isocl(\ol{X}_0,M_{\ol{X}_0})_{n\Sigma(\ss)}$. 
Then,  since a projection 
$(\ol{X}_l,M_{\ol{X}_l}) \lra (\ol{X}_0,M_{\ol{X}_0})$ is 
strict etale (which follows from Proposition 
\ref{key_sr}), the pull-back of $\ti{\cE}_0$ to 
$(\ol{X}_l,M_{\ol{X}_l})$ by this projection is in 
$\Isocl(\ol{X}_l,M_{\ol{X}_l})_{n\Sigma(\ss)}$ by 
Proposition \ref{sig_pre2}. Hence 
the restriction of $\cE$ to $\Isocd(X_l,\ol{X}_l)$ 
has $n\Sigma$-unipotent ($n\Sigma$-semisimple)
generic monodromy. 
So it extends uniquely to 
an object (which we denote by $\ti{\cE}_l$) 
in $\Isocl(\ol{X}_l,M_{\ol{X}_l})_{n\Sigma(\ss)}$. 
The uniqueness implies that $\ti{\cE}_{\b}$ defines an object in 
$\Isocl(\ol{X}_{\b},M_{\ol{X}_{\b}})_{n\Sigma(\ss)} 
= \Isocl((\ol{X},Z)^{1/n},M_{(\ol{X},Z)^{1/n}})_{n\Sigma(\ss)}$. 
Therefore $\cE$ 
is in the essential image of \eqref{kyo}. So we have shown the 
equivalence \eqref{step11}, hence the equivalence \eqref{step1}. \par 
Next we prove that the functor \eqref{par-to-d} induces the 
equivalence of categories 
\begin{equation}\label{step2}
\PIsoc(\ol{X},Z)_{\Sigma(\ss)} \os{=}{\lra} 
\Isocd(X,\ol{X})_{\ol{\Sigma}(\ss)}. 
\end{equation}
To do so, first we prove the well-definedness and 
the full faithfulness of the functor \eqref{step2}. 
Let $\tau'_i: (\Z_p/\Z) \setminus ((\Sigma_i+\Z_{(p)})/\Z) \lra \Z_p$ 
be any section of 
the projection $\Z_p \lra \Z_p/\Z$ 
and for $a \in \Z_{(p)}$, 
let $\tau''_{a,i}: (\Sigma_i + \Z_{(p)})/\Z \lra \Z_{p}$ be the map 
defined as $\tau''_{a,i}(\ol{\gamma+b}) := \gamma + \lfloor b \rfloor_{-a} \, 
(\gamma \in \Sigma_i, b \in \Z_{(p)})$, where $\lfloor b \rfloor_{-a}$ is 
the unique element in $(b+\Z) \cap [-a,-a+1)$. 
Let $\tau_{a,i}: \Z_p/\Z \lra \Z_p$ be the map 
defined as $\tau_{a,i} |_{(\Z_p/\Z) \setminus ((\Sigma_i+\Z_{(p)})/\Z)} = 
\tau'_i, 
\tau_{a,i} |_{(\Sigma_i + \Z_{(p)})/\Z} = \tau''_{a,i}$ and 
for $\alpha = (\alpha_i)_i \in 
\Z_{(p)}^r$, let us put $\tau_{\alpha} := \prod_{i=1}^r\tau_{\alpha_i,i}$.
Let us take 
$(\cE_{\alpha})_{\alpha} \in \PIsoc(\ol{X},Z)_{\Sigma(\ss)}$. 
Then, by (semisimple) $\Sigma$-adjustedness of 
$(\cE_{\alpha})_{\alpha}$, $\cE_{\alpha}$ 
has exponents in $\tau_{\alpha}(\ol{\Sigma})$ (with 
semisimple residues). Hence $j^{\dagger}\cE_{\alpha}$ has 
$\ol{\Sigma}$-unipotent ($\ol{\Sigma}$-semisimple) 
generic monodromy. So the functor 
\eqref{par-to-d} naturally induces the functor \eqref{step2}, that is, 
\eqref{step2} is well-defined. Next, let us take 
$((\cE_{\alpha})_{\alpha}, (\iota_{\alpha\beta})_{\alpha,\beta}), 
((\cE'_{\alpha})_{\alpha}, (\iota'_{\alpha\beta})_{\alpha,\beta}) 
\in \PIsoc(\ol{X},Z)$. 
Then, since $\cE_{\alpha}, \cE'_{\alpha}$ have exponents in 
$\tau_{\alpha}(\ol{\Sigma})$ (with semisimple residues) for any $\alpha$, 
we have the isomorphism 
\begin{equation}\label{ff-}
\Hom(\cE_{\alpha},\cE'_{\beta}) \os{=}{\lra} \Hom(j^{\dagger}\cE_{\alpha}, 
j^{\dagger}\cE'_{\beta})
\end{equation}
for any $\alpha = (\alpha_i), \beta = (\beta_i) \in 
\Z_{(p)}$ with $\alpha_i \leq \beta_i \,(1 \leq i \leq r)$, 
by Proposition \ref{uniqueinj}. Using \eqref{ff-} in the case 
$\alpha = \beta$ for $\alpha \in \Z_{(p)}^r$, we see that the 
functor \eqref{step2} is faithful. On the other hand, if we are given 
an element 
$\varphi \in \Hom(j^{\dagger}\cE_{\alpha}, j^{\dagger}\cE'_{\alpha})$ 
for some $\alpha$ (note that $\Hom(j^{\dagger}\cE_{\alpha}, 
j^{\dagger}\cE'_{\alpha})$ is independent of $\alpha$), 
it induces for any $\alpha$ the unique element $\varphi_{\alpha} \in 
\Hom(\cE_{\alpha},\cE'_{\alpha})$ which is sent to $\varphi$ by 
\eqref{ff-} (for $\alpha = \beta$). These $\varphi_{\alpha}$'s satisfy 
the equality 
$\iota'_{\alpha\beta}\circ \varphi_{\alpha} = \varphi_{\beta} \circ 
\iota_{\alpha\beta}$ because both are sent to $\varphi$ by 
\eqref{ff-}. Hence $(\varphi_{\alpha})_{\alpha}$ defines a morphism 
$(\cE_{\alpha})_{\alpha} \lra (\cE'_{\alpha})_{\alpha}$ which is 
sent to $\varphi$ by the functor \eqref{step2}. Hence \eqref{step2} is 
full, and so we have shown the full faithfulness. \par 
We prove the essential surjectivity of the functor \eqref{step2}. 
Let us take an object $\cE$ in $\Isocd(X,\ol{X})_{\ol{\Sigma}(\ss)}$ 
and let $\tau_{\alpha}$ be as in the previous paragraph. 
Then there exists the unique object $\cE_{\alpha} \in 
\Isocl(\ol{X},Z)_{\tau_{\alpha}(\ol{\Sigma})(\ss)}$ 
with $j^{\dagger}\cE_{\alpha} = \cE$ by Theorem \ref{sigmamain} and for 
$\alpha = (\alpha_i)_i, \beta = (\beta_i)_i$ with 
$\alpha_i \leq \beta_i$, we have the unique morphism 
$\iota_{\alpha\beta}: \cE_{\alpha} \lra \cE_{\beta}$ with 
$j^{\dagger}\iota_{\alpha\beta} = \id_{\cE}$ by Proposition 
\ref{uniqueinj}. Then $((\cE_{\alpha})_{\alpha}, 
(\iota_{\alpha\beta})_{\alpha,\beta})$ forms an inductive system of 
objects in $\Isocl(\ol{X},Z)$, and one can check that it satisfies 
the property (1) in Definition \ref{defpar} by the above-mentioned 
uniqueness. Also, note that there exists some $n$ prime to $p$ such that 
$\cE$ is actually contained in 
$\Isocd(X,\ol{X})_{\ol{\Sigma}_n(\ss)}$. 
Then, for any $\alpha = (\alpha_i)_i$, 
$\cE_{\alpha}$ is in fact the unique object in 
$\Isocl(\ol{X},Z)_{\tau_{\alpha}(\ol{\Sigma}_n)(\ss)}$ 
with $j^{\dagger}\cE_{\alpha} = \cE$. Since 
we have $\tau_{\alpha}(\ol{\Sigma}_n) = 
\tau_{\alpha'}(\ol{\Sigma}_n)$ when 
$\alpha' = ([n\alpha_i]/n)_i$, we see that 
$\iota_{\alpha'\alpha}$ is an isomorphism by uniqueness. So 
we have checked that $(\cE_{\alpha})_{\alpha}$ satisfies the property (2) 
in Definition \ref{defpar}. Also, we see by definition that 
$\tau_{\alpha}(\ol{\Sigma}) = \prod_{i=1}^r 
(\Sigma_i + ([-\alpha_i, -\alpha_i+1) \cap \Z_{(p)}))$. 
Hence $(\cE_{\alpha})_{\alpha}$ is (semisimply) $\Sigma$-adjusted and so 
we have shown the essential surjectivity of 
the functor \eqref{step2}. Hence the functor \eqref{step2} is an 
equivalence. \par 
By composing the equivalences \eqref{step1} and $\text{\eqref{step2}}^{-1}$, 
we can construct the equivalence \eqref{root-to-par} which 
makes the diagram \eqref{commutat} commutative. So we are done. 
\end{proof}

\begin{corollary}\label{repparcor}
Let $X,\ol{X},Z= \bigcup_{i=1}^r Z_i$ be as above. Then there exist 
the canonical equivalences of categories 
\begin{align}
& \varinjlim_{(n,p)=1} \Isocl((\ol{X},Z)^{1/n},M_{(\ol{X},Z)^{1/n}})_{\0} 
\os{=}{\lra} 
\PIsoc(\ol{X},Z)_{\0}, \label{root-to-par-0} \\ 
& \varinjlim_{(n,p)=1} \Isoc((\ol{X},Z)^{1/n}) \os{=}{\lra} 
\PIsoc(\ol{X},Z)_{\0\ss}, \label{root-to-par-0ss} \\ 
& \varinjlim_{(n,p)=1} \FIsocl((\ol{X},Z)^{1/n},M_{(\ol{X},Z)^{1/n}}) 
\os{=}{\lra} 
\PFIsoc(\ol{X},Z)_{\0}, \label{root-to-par-f} \\ 
& \varinjlim_{(n,p)=1} \FIsoc((\ol{X},Z)^{1/n}) \os{=}{\lra} 
\PFIsoc(\ol{X},Z)_{\0\ss}, \label{root-to-par-fss} \\ 
& \varinjlim_{(n,p)=1} \FIsoc((\ol{X},Z)^{1/n})^{\circ} \os{=}{\lra} 
\PFIsoc(\ol{X},Z)^{\circ}_{\0\ss} \label{root-to-par-urf}. 
\end{align}
In particular, we have the canonical equivalence
\begin{equation}\label{rep-par}
\Rep_{K^{\sigma}}(\pi_1(X)) \os{=}{\lra} 
\varinjlim_{(n,p)=1} \FIsoc((\ol{X},Z)^{1/n})^{\circ} \os{=}{\lra} 
\PFIsoc(\ol{X},Z)^{\circ}_{\0\ss}
\end{equation}
when $X$ is connected. 
\end{corollary}

The equivalence \eqref{rep-par} is nothing but 
\eqref{seq4}, which is a $p$-adic version of \eqref{eq2}. 

\begin{proof}
The equivalences \eqref{root-to-par-0} and \eqref{root-to-par-0ss} are 
special cases of Theorem \ref{mainpar} (note that we have 
$\Isocl((\ol{X},Z)^{1/n},M_{(\ol{X},Z)^{1/n}})_{\0\ss} = 
\Isoc((\ol{X},Z)^{1/n})$). Next, 
let $\FIsocd(X,\ol{X})_{(\Z_{(p)}/\Z)^r(\ss)}$ 
be the category of pairs $(\cE,\Psi)$ consisting of 
$\cE \in \Isocd(X,\ol{X})_{(\Z_{(p)}/\Z)^r(\ss)}$ 
endowed with the isomorphism $\Psi: F^*\cE \os{=}{\lra} \cE$ 
and let $\FIsocd(X,\ol{X})^{\circ}_{(\Z_{(p)}/\Z)^r\text{-}{\rm ss}}$ be 
the category of pairs $(\cE,\Psi) \in 
\FIsocd(X,\ol{X})_{(\Z_{(p)}/\Z)^r\text{-}{\rm ss}}$ which is unit-root 
as an object in $\FIsocd(X,\ol{X})$. Then it suffices to prove the 
equivalences 
\begin{align}
& \varinjlim_{(n,p)=1} \FIsocl((\ol{X},Z)^{1/n},
 M_{(\ol{X},Z)^{1/n}}) \os{=}{\lra} 
\FIsocd(X,\ol{X})_{(\Z_{(p)}/\Z)^r}, \label{b1} \\ 
& \varinjlim_{(n,p)=1} \FIsoc((\ol{X},Z)^{1/n}) \os{=}{\lra} 
\FIsocd(X,\ol{X})_{(\Z_{(p)}/\Z)^r\ss}, \label{b1bis} \\ 
& \varinjlim_{(n,p)=1} \FIsoc((\ol{X},Z)^{1/n})^{\circ} \os{=}{\lra} 
\FIsocd(X,\ol{X})_{(\Z_{(p)}/\Z)^r\text{-}{\rm ss}}^{\circ}, \label{b2} \\ 
& \PFIsoc(\ol{X},Z)_{\0} \os{=}{\lra} 
\FIsocd(X,\ol{X})_{(\Z_{(p)}/\Z)^r}, \label{b3} \\ 
& \PFIsoc(\ol{X},Z)_{\0\ss} \os{=}{\lra} 
\FIsocd(X,\ol{X})_{(\Z_{(p)}/\Z)^r\ss}, \label{b3bis} \\ 
& \PFIsoc(\ol{X},Z)^{\circ} \os{=}{\lra} 
\FIsocd(X,\ol{X})_{(\Z_{(p)}/\Z)^r\text{-}{\rm ss}}^{\circ}. \label{b4}
\end{align}

As a direct consequence of the equivalence 
\eqref{step1} (and the functoriality of it with respect to Frobenius), 
we obtain the equivalence 
$$\varinjlim_{(n,p)=1} \FIsocl((\ol{X},Z)^{1/n},
 M_{(\ol{X},Z)^{1/n}})_{\0} \os{=}{\lra} 
\FIsocd(X,\ol{X})_{(\Z_{(p)}/\Z)^r} $$
and the equivalence \eqref{b1bis} (note that $\FIsocl((\ol{X},Z)^{1/n},
 M_{(\ol{X},Z)^{1/n}})_{\0} = \FIsoc ((\ol{X},Z)^{1/n})$.) 
To prove the equivalence \eqref{b1}, it suffice to prove the 
equivalence 
\begin{equation}\label{0-non-0}
\FIsocl((\ol{X},Z)^{1/n},
 M_{(\ol{X},Z)^{1/n}})_{\0} \os{=}{\lra} 
\FIsocl((\ol{X},Z)^{1/n},
 M_{(\ol{X},Z)^{1/n}}), 
\end{equation}
that is, any object $(\cE,\Psi)$ in $\FIsocl((\ol{X},Z)^{1/n},
 M_{(\ol{X},Z)^{1/n}})$ has exponents in $\0$. By definition 
of exponents given in Definition \ref{defstacksigma}, 
we can check it by pulling $(\cE,\Psi)$ back to 
some fine log scheme $(Y,M_Y)$ such that $Y$ is smooth over $k$ and 
that $M_Y$ is associated to some simple normal crossing divisor 
$Z$ in $Y$. Moreover, by Lemma \ref{wdlemma}(2), we may assume that 
$Z$ is a non-empty smooth divisor and we may shrink $Y$ if necessary 
(as long as $Z$ is non-empty). So we may assume that there exists a 
charted smooth standard small frame $((Y \setminus Z, Y, \cY),t)$ 
enclosing $(Y, Y \setminus Z)$, and that, if we put 
$\cZ := \{t=0\}$, there exists a $\sigma^*$-linear endomorphism 
$\varphi: (\cY,\cZ) \lra (\cY,\cZ)$ as fine log formal scheme 
lifting the $q$-th power map $F: (Y,Z) \lra (Y,Z)$ with 
$\varphi^*t = t^q$. Then $\cE$ gives rise to a log-$\nabla$-module 
$(E,\nabla)$ on $\cY_K$ with respect to $t$ and the isomorphism 
$\Psi: F^*\cE \os{=}{\lra} \cE$ gives rise to 
an isomorphism 
\begin{equation}\label{inducedpsi}
\Psi: \varphi^*(E,\nabla) \os{=}{\lra} (E,\nabla). 
\end{equation}
Let us denote the residue of $(E,\nabla)$ along $\cZ_K = \{t=0\}$ by 
$\res$ and let $P(x) \in K[x]$ be the minimal polynomial of 
$\res$. Then it suffices to prove that all the roots of $P(x)$ are $0$. 
Let $a \geq 0$ be the maximum of the absolute values of the roots of $P(x)$. 
By the isomophism \eqref{inducedpsi}, we have the equality 
$p \Psi(\varphi^*(\res)) = \res$, and so we have 
$P^{\sigma}(\res/p) = 0$. Hence $P(x)$ divides $P^{\sigma}(x/p)$, and so 
we have $a \leq q^{-1}a$, that is, $a=0$. 
Hence all the roots of $P(x)$ are $0$ as desired and so we have shown 
the equivalence \eqref{0-non-0}, thus the equivalence \eqref{b1}. \par 
To prove the equivalence \eqref{b2}, it suffices to prove that 
an object $(\cE,\Psi)$ in $\FIsoc((\ol{X},Z)^{1/n})$ which is 
unit-root in $\FIsocd(X,\ol{X})$ is necessarily unit-root. 
To prove this, we may work locally. So 
we can take a surjective etale morphism 
$\ol{X}_0 \lra (\ol{X},Z)^{1/n}$ with $\ol{X}_0 \in \Sch$. 
To prove the unit-rootness on $(\cE,\Psi)$, we need to prove that, 
for any $Y \lra (\ol{X},Z)^{1/n}$ with $Y \in \Sch$ and 
for any perfect-field-valued point $y \lra Y$ of $Y$, the 
Newton polygon of the 
$F$-isocrystal 
$(\cE_y, \Psi_y)$ 
on $K_y = K \otimes_{W(k)} W(k(y))$ induced by 
$(\cE,\Psi)$ has pure slope $0$. 
Since we may check it etale locally on 
$y$, we may replace $Y$ by $\ol{X}_0 \times_X Y$. So it suffices to 
check it for perfect-field-valued points $y \lra \ol{X}_0$ of 
$\ol{X}_0$. Moreover, it suffices to check it for 
perfect-field-valued points $y \lra \ol{X}'_0$ of 
$\ol{X}'_0$ which admits a surjective morphism of finite type 
$\ol{X}'_0 \lra \ol{X}_0$. 
So it suffices to prove that the restriction of $(\cE,\Psi)$ to 
$\FIsoc(\ol{X}'_0)$ for such $\ol{X}'_0$ is unit-root. 
Since the restriction of $(\cE,\Psi)$ to 
$\FIsocd(X,\ol{X})$ is unit-root, so is the 
the restriction of $(\cE,\Psi)$ to 
$\FIsocd(X_0,\ol{X}_0)$. Then, by \cite{tsuzuki},  
there exists an alteration 
$\ol{X}'_0 \lra \ol{X}_0$ such that the restriction of $(\cE,\Psi)$ to 
$\FIsocd(X'_0,\ol{X}'_0)$ (where $X'_0 := X \times_{\ol{X}_0} \ol{X}'_0$) 
extends to a unit-root object $(\ti{\cE}, \ti{\Psi})$ 
in $\FIsoc(\ol{X}'_0)$. By the full faithfulness of 
$\FIsoc(\ol{X}'_0) \lra \FIsocd(X'_0,\ol{X}'_0)$, 
$(\ti{\cE}, \ti{\Psi})$ coincides with the restriction of $(\cE,\Psi)$ to 
$\FIsoc(\ol{X}'_0)$. So the restriction of $(\cE,\Psi)$ to 
$\FIsoc(\ol{X}'_0)$ is unit-root and so we have shown 
the equivalence \eqref{b2}. \par 
Let us prove the equivalences \eqref{b3}, \eqref{b3bis}. 
The faithfulness of them 
follows from the equivalence \eqref{step2}. To prove the fullness, 
let us take $((\cE_{\alpha})_{\alpha}, \Psi), 
((\cE'_{\alpha})_{\alpha}, \Psi') \in \PFIsoc\allowbreak
(\ol{X},Z)_{\0(\ss)}$ and assume 
we are given a morphism $f: (j^{\dagger}\cE_{\alpha}, j^{\dagger}\Psi) 
\lra (j^{\dagger}\cE'_{\alpha}, j^{\dagger}\Psi')$. 
Then, by the equivalence \eqref{step2}, we have the unique morphism 
$$ 
\ti{f} = (\ti{f}_{\alpha})_{\alpha}: (\cE_{\alpha})_{\alpha} \lra 
(\cE'_{\alpha})_{\alpha}
$$ 
lifting $f$. Let us see that $\ti{f}$ is compatible with 
$\Psi$ and $\Psi'$. Take any $\alpha = (\alpha_i) \in \Z_{(p)}^r$ and  
take $\beta = (\beta_i)$ with $q\alpha_i \leq \beta_i$ 
such that $\Psi, \Psi'$ induce morphisms 
$\Psi_{\alpha}: 
F^*\cE_{\alpha} \lra \cE_{\beta}, \Psi'_{\alpha}: 
F^*\cE'_{\alpha} \lra \cE'_{\beta}$. Under this situation, it suffices to 
prove the equality $\ti{f}_{\beta} \circ \Psi_{\alpha} = 
\Psi'_{\alpha} \circ (F^*\ti{f}_{\alpha})$. This follows from 
the equality 
$$ 
j^{\dagger}(\ti{f}_{\beta} \circ \Psi_{\alpha}) = 
f_{\beta} \circ j^{\dagger}\Psi_{\alpha} \os{\text{assumption}}{=} 
j^{\dagger}\Psi'_{\alpha} \circ (F^*f_{\alpha})
= 
j^{\dagger}(\Psi'_{\alpha} \circ (F^*\ti{f}_{\alpha}))$$ 
and Proposition \ref{uniqueinj}, since 
the exponents of $F^*\cE_{\alpha}$ is contained in 
$q\tau_{\alpha}((\Z_{(p)}/\Z)^r)$, where $\tau_{\alpha}$ is 
as in the proof of 
Theorem \ref{mainpar}. So we have shown that 
\eqref{b3}, \eqref{b3bis} are fully faithful. 
To prove the essential surjectivity, 
let us take an object $(\cE,\Psi)$ in 
$\FIsocd(X,\ol{X})_{(\Z_{(p)}/\Z)^r(\ss)}$ and define 
$(\cE_{\alpha})_{\alpha} \in \PIsoc(\ol{X},Z)_{\0(\ss)}$ 
which is sent to $\cE$ as in the proof of Theorem \ref{mainpar}. 
Then, for any $\alpha \in \Z_{(p)}^r$, we have the unique morphism 
$\Psi_{\alpha}: F^*\cE_{\alpha} \lra \cE_{q\alpha}$ extending 
$\Psi$, by Proposition \ref{uniqueinj}. On the other hand, 
for any $\alpha = (\alpha_i)_i \in \Z_{(p)}^r$, we have the unique morphism 
$\Psi^{-1}_{\alpha}: \cE_{\alpha} \lra F^*\cE_{\beta}$ extending 
$\Psi^{-1}$ when $\beta = (\beta_i)_i \in \Z_{(p)}^r$ satisfies 
$\beta_i \geq (\alpha_i/q) + 1$, again by Proposition \ref{uniqueinj}. 
Again by using  Proposition \ref{uniqueinj}, we see that 
$\ti{\Psi} = (\Psi_{\alpha})_{\alpha}, \ti{\Psi}^{-1} = 
(\Psi^{-1}_{\alpha})_{\alpha}$ define morphisms as 
ind-objects which are the inverse of each other. So 
$((\cE_{\alpha}),\ti{\Psi})$ defines an object in 
$\PFIsoc(\ol{X},Z)_{\0(\ss)}$ 
which is sent to $(\cE,\Psi)$. So we have shown that 
\eqref{b3}, \eqref{b3bis} are essentially surjective and so 
they are equivalences. 
The equivalence \eqref{b4} is the immediate consequence of the 
equivalence \eqref{b3bis}. 
\end{proof}

Finally in this section, we prove a proposition which we need to finish 
the proof of Propositions \ref{curve}, \ref{curvelog}. 

\begin{proposition}\label{atode}
Let $\ol{X}$ be a connected proper smooth curve of genus $0$ over $k$, 
let $Z \subseteq \ol{X}$ be a $k$-rational point and 
let $X = \ol{X} \setminus Z$. Let $\Sigma$ be a subset of $\Z_p$ which is 
$\NRD$ and $\SNLD$. Then, if $\Sigma \cap \Z_{(p)}$ is empty, 
the category $\varinjlim_{(n,p)=1} \allowbreak \Isocl((\ol{X},Z)^{1/n}, 
M_{(\ol{X},Z)^{1/n}})_{n\Sigma}$ is empty. 
If $\Sigma \cap \Z_{(p)}$ consists of one element, the functor 
$$\varinjlim_{(n,p)=1}\Isocl((\ol{X},Z)^{1/n}, M_{(\ol{X},Z)^{1/n}})_{n\Sigma} 
\os{\text{\eqref{step1}}}{\lra} \Isocd(X,\ol{X})_{\ol{\Sigma}} \lra \Vect_K $$
$($where $\ol{\Sigma} := \im (\Sigma + \Z_{(p)} \hra \Z_p 
\lra \Z_p/\Z)$ and the last 
functor is defined by $\cE \mapsto \Hom(j^{\dagger}\cO, \cE)$ with 
$j^{\d}\cO$ the structure overconvergent isocrystal on $(X,\ol{X})$ over $K)$ 
is an equivalence. 
\end{proposition}

\begin{proof}
By Theorem \ref{mainpar}, it suffices to prove the same property for 
the category $\PIsoc(\ol{X},Z)_{\Sigma}$ (the functor \eqref{step1} replaced 
by \eqref{step2}). Let $\bX_K \hra \ol{\bX}_K \hla \bZ_K, \cX \hra \ol{\cX} \hla \cZ, \LNM_{(\ol{\bX}_K,\bZ_K),?} \,(? \subseteq \Z_p)$ 
be as in the proof of Proposition \ref{curvelog}. 
Let us take an object $\cE = (\cE_{\alpha})_{\alpha} \in 
\PIsoc(\ol{X},Z)_{\Sigma}$. Then $\cE_0$, which is an 
object in $\Isocl(\ol{X},Z)_{\Sigma + ([0,1)\cap \Z_{(p)})}$, 
induces naturally an object $(E,\nabla)$ in 
$\LNM_{(\ol{\cX},\cZ),\Sigma + ([0,1)\cap \Z_{(p)})} \allowbreak \cong 
\allowbreak 
\LNM_{(\ol{\bX}_K,\bZ_K),\Sigma + ([0,1)\cap \Z_{(p)})}$. Then, by 
claim in the proof of Proposition \ref{curvelog}, there is no such object if 
$(\Sigma + ([0,1)\cap \Z_{(p)})) \cap \Z$ is empty, that is, 
if $\Sigma \cap \Z_{(p)}$ is empty. So we can conclude that 
$\PIsoc(\ol{X},Z)_{\Sigma}$ is empty and we are done in this case. 
If $\Sigma \cap \Z_{(p)}$ consists of one element, 
$(\Sigma + ([0,1)\cap \Z_{(p)})) \cap \Z = \{N\}$ also consists of 
one element. Then, by claim in the proof of Proposition \ref{curvelog}, 
$(E,\nabla)$ is a finite direct sum of 
$(\cO_{\ol{\bX}_{K}}(-N\bZ_{K}), d)$. So the restriction of 
$(E,\nabla)$ (regarded as an object in 
$\LNM_{(\ol{\cX},\cZ),\Sigma + ([0,1)\cap \Z_{(p)})}$) 
to a strict neighborhood of $\cX_K$ in 
$\ol{\cX}_K$ is a finite direct sum of 
the structure overconvergent isocrystal on $(X,\ol{X})$. Hence 
the object $\cE$ is sent by \eqref{step2} to a finite direct sum of 
the structure overconvergent isocrystal on $(X,\ol{X})$. 
So the functor 
$$ \PIsoc(\ol{X},Z)_{\Sigma}
\os{\text{\eqref{step2}}}{\lra} \Isocd(X,\ol{X})_{\ol{\Sigma}} \lra \Vect_K $$
(the last 
functor is defined by $\cE \mapsto \Hom(j^{\dagger}\cO, \cE)$) 
is an equivalence in this case. So the proof is finished. 
\end{proof}

\section{Unit-root $F$-lattices}

In this section, we will prove the equivalences \eqref{seq2-}, \eqref{seq3-} 
and \eqref{seq4-}. To do so, first we recall the definition of 
unit-root $F$-lattices and recall how the equivalence of Crew \eqref{eq5} 
is factorized in liftable case. \par 
Let $\cX_{\c}$ be a $p$-adic formal scheme smooth and separated of finite type 
over $\Spf W(k)$ 
endowed with a lift $F_{\c}: \cX_{\c} \lra \cX_{\c}$ of the 
$q$-th power Frobenius endomorphism on the special fiber $\cX_{\c} 
\otimes_{W(k)} k$ which is 
compatible with $(\sigma|_{W(k)})^*:\Spf W(k) \lra \Spf W(k)$, 
and let us put 
$\cX := \cX_{\c} \otimes_{W(k)} O_K, 
F := F_{\c} \otimes_{(\sigma|_{W(k)})^*} \sigma^*: \cX \lra \cX$. 
Then an $F$-lattice 
(resp. a unit-root $F$-lattice) on $\cX$ is defined to 
be a pair $(\cE,\phi)$ consisting of 
a locally free $\cO_{\cX}$-module $\cE$ of finite rank and an 
isomorphism 
$\phi: (F^*\cE)_{\Q} \os{=}{\lra} \cE_{\Q}$ in the $\Q$-linearization 
$\Coh(\cX)_{\Q}$ of the category $\Coh(\cX)$ of coherent $\cO_{\cX}$-modules. 
(resp. an isomorphism $\phi: F^*\cE \os{=}{\lra} \cE$ in 
 the category $\Coh(\cX)$ of coherent $\cO_{\cX}$-modules.) 
(An unit-root $F$-lattice on $\cX$ 
is called a unit-root $F$-lattice on $\cX_{\c}/(O_K,\sigma)$ 
in \cite{crew}.)  Let us denote 
the category of $F$-lattices (resp. unit-root $F$-lattices) on $\cX$ by 
$\FLatt(\cX)$ (resp. $\FLatt(\cX)^{\circ}$). 
(Note that the same definition is possible even when 
$\cX_{\c}$ is a diagram of $p$-adic formal schemes smooth and 
separated of finite type over $\Spf W(k)$ 
endowed with an endomorphism $F_{\c}: \cX_{\c} \lra \cX_{\c}$ as above. 
Note also that only the unit-root $F$-lattices are treated in this section 
and the $F$-lattices will be treated in the next section.) \par 
Let $\cX_{\c}, \cX, F$ be as above and let us put 
$X := \cX_{\c} \otimes_{W(k)} k$. 
Then, the equivalence \eqref{eq5} of Crew 
is written as the composite of 
the equivalence of Katz (\cite[4.1.1]{katzcrew}, \cite[2.2]{crew}) 
\begin{equation}\label{katzeq}
\Rep_{O_K^{\sigma}}(\pi_1(X)) \os{=}{\lra} 
\FLatt(\cX)^{\c} 
\end{equation}
and the equivalence 
\begin{equation}\label{crewisog}
\FLatt(\cX)^{\c}_{\Q} \os{=}{\lra} \FIsoc(X)^{\circ}
\end{equation}
proved in \cite{crew}, 
where, for an additive category $\cC$, $\cC_{\Q}$ denotes the 
$\Q$-linearization of it. \par 
Let us recall the definition of \eqref{katzeq} 
(cf. Section 2.2). 
Let $\rho$ be an object in $\Rep_{O_K^{\sigma}}(\pi_1(X))$ and let 
$\cF$ be the corresponding object in 
$\Sm_{O_K^{\sigma}}(X)$. Let $\ti{\cF}$ be the object in 
$\Sm_{O_K^{\sigma}}(\cX)$ corresponding to $\cF$ via 
the equivalence $\Sm_{O_K^{\sigma}}(\cX) \cong 
\Sm_{O_K^{\sigma}}(X)$. Then $F:\cX \lra \cX$ 
induces the equivalence 
 $F^{-1}: \Sm_{O^{\sigma}_K}(\cX) \os{=}{\lra} 
\Sm_{O^{\sigma}_K}(\cX)$ 
with $F^{-1}(\ti{\cF}) \cong \ti{\cF}$. Then, if we define $\cE$ and $\phi$ by 
$\cE := \ti{\cF} \otimes_{O^{\sigma}_K} \cO_{\cX}$, 
$$ \phi: F^*\cE = F^{-1}\ti{\cF} \otimes_{O^{\sigma}_K} \cO_{\cX} 
\os{=}{\lra} \ti{\cF} \otimes_{O^{\sigma}_K} \cO_{\cX} = \cE, $$
the functor \eqref{katzeq} is defined as $\rho \mapsto (\cE,\phi)$. 
In view of this, we see that \eqref{katzeq} is 
written as the equivalence 
\begin{equation}\label{katzeq'}
\Sm_{O_K^{\sigma}}(X) \os{=}{\lra} \FLatt(\cX)^{\c}. 
\end{equation}
(The choice of a base point in the definition of $\pi_1(X)$ 
is not essential.) Also, we see easily the functoriality of 
\eqref{katzeq}, \eqref{katzeq'}. \par 
Now we fix the notation to prove the equivalences 
\eqref{seq2-}, \eqref{seq3-} and \eqref{seq4-}. 
In the following in this section, let 
$X \hra \ol{X}$ be an open immersion of conntected 
smooth $k$-varieties such that 
$\ol{X} \setminus X =: Z = \bigcup_{i=1}^r Z_i$ 
is a simple normal crossing divisor (with each $Z_i$ irreducible), and 
assume that we have an open immersion $\cX_{\c} \hra \ol{\cX}_{\c}$ of 
$p$-adic formal schemes smooth and separated of finite type 
over $\Spf W(k)$ such that 
there exists a relative simple normal 
crossing divisor $\cZ_{\c} := \bigcup_{i=1}^r \cZ_{\c,i} \hra \ol{\cX}_{\c}$ 
with 
$\ol{\cX}_{\c} \setminus \cZ_{\c} = \cX_{\c}$, 
$\ol{\cX}_{\c} \otimes_{W(k)} k = \ol{X}$, 
$\cZ_{\c,i} \otimes_{W(k)} k = Z_i$. 
Assume moreover that 
we have a lift $F_{\c}: (\ol{\cX}_{\c},\cZ_{\c}) \lra 
(\ol{\cX}_{\c},\cZ_{\c})$ 
(endomorphism as log schemes) of the 
$q$-th power Frobenius endomorphism on $(\ol{X},Z)$ which is compatible 
with $(\sigma|_{W(k)})^*:\Spf W(k) \lra \Spf W(k)$. 
Let us put $\cX := \cX_{\c} \otimes_{W(k)} O_K, 
\ol{\cX} := \ol{\cX}_{\c} \otimes_{W(k)} O_K, 
\cZ := \cZ_{\c} \otimes_{W(k)} O_K, \cZ_i := \cZ_{\c,i} \otimes_{W(k)} O_K$. 
Then $(\ol{\cX},\cZ)$ admits 
the lift $F := F_{\c} \otimes_{(\sigma|_{W(k)})^*} \sigma^*: 
(\ol{\cX},\cZ) \lra (\ol{\cX},\cZ)$ of the 
$q$-th power Frobenius endomorphism on $(\ol{X},Z)$ which is compatible 
with $\sigma^*:\Spf O_K \lra \Spf O_K$. \par 
For $a \in \N$, let us put 
$(\ol{\cX}_{\c},\cZ_{\c})_a := (\ol{\cX}_{\c,a},\cZ_{\c,a}) := 
(\ol{\cX},\cZ) \otimes_{W(k)} W_a(k)$, 
$\cZ_{\c,i,a} := \cZ_{\c,i} \otimes_{W(k)} W_a(k)$, 
$(\ol{\cX},\cZ)_a := (\ol{\cX}_a,\cZ_a) := 
(\ol{\cX},\cZ) \otimes_{W(k)} W_a(k)$. 
As in Section 2, 
let $\cG_X^t$ be the category of finite etale Galois tame covering 
of $X$ (tamely ramified at generic points of $Z$) and  
for an object $Y \ra X$ in $\cG^t_X$, let 
$\ol{Y}$ be the normalization of $\ol{X}$ in $k(Y)$, let 
$\ol{Y}^{\sm}$ be the smooth locus of $\ol{Y}$ and let 
$G_Y := \Aut (Y/X)$. Then $\ol{Y}$ admits naturally the log structure 
$M_{\ol{Y}}$ associated to $\ol{Y} \setminus Y$ such that 
the morphism $(\ol{Y},M_{\ol{Y}}) \lra (\ol{X},Z)$ is finite Kummer 
log etale (\cite{illusie}). 
Then this morphism admits a 
unique finite Kummer log etale lifts 
$(\ol{\cY}_{\c,a},M_{\ol{\cY}_{\c,a}}) \lra (\ol{\cX}_{\c},\cZ_{\c})_a$ 
for $a \in \N$ 
(hence the lift 
$(\ol{\cY}_{\c},M_{\ol{\cY}_{\c}}) := \varinjlim_a 
(\ol{\cY}_{\c,a},M_{\ol{\cY}_{\c,a}}) \lra (\ol{\cX}_{\c},\cZ_{\c})$) 
by \cite[2.8, 3.10, 3.11]{illusie}. 
Let us put $(\ol{\cY},M_{\ol{\cY}}) := 
(\ol{\cY}_{\c}, M_{\ol{\cY}_{\c}}) \otimes_{W(k)} O_K$ and 
let $\ol{\cY}^{\sm}_{\c} \subseteq \ol{\cY}_{\c}, 
\ol{\cY}^{\sm} \subseteq \ol{\cY}$ be the smooth loci. 
Then the endomorphism 
$F_{\c}: (\ol{\cX}_{\c},\cZ_{\c}) \lra (\ol{\cX}_{\c},\cZ_{\c})$ 
uniquely lifts to 
an endomorphism $F_{\c}$ on $(\ol{\cY}_{\c}, M_{\ol{\cY}_{\c}})$ and on 
$(\ol{\cY}_{\c}^{\sm}, M_{\ol{\cY}_{\c}^{\sm}} := 
M_{\ol{\cY}_{\c}}|_{\ol{\cY}^{\sm}_{\c}})$, and it induces 
the endomorphism $F := F_{\c} \otimes \sigma^*$ on 
$(\ol{\cY}, M_{\ol{\cY}})$ and on 
$(\ol{\cY}^{\sm}, M_{\ol{\cY}^{\sm}} := 
M_{\ol{\cY}}|_{\ol{\cY}^{\sm}})$. 
So the category $\FLatt(\ol{\cY}^{\sm})^{\c}$ of unit-root $F$-lattices on 
$\ol{\cY}^{\sm}$ 
is defined. \par 
Let us put $(\ol{\cY}_a,M_{\ol{\cY}_a}) := 
(\ol{\cY},M_{\ol{\cY}}) \otimes_{W(k)} W_a(k)$, 
$(\ol{\cY}^{\sm}_a,M_{\ol{\cY}^{\sm}_a}) := 
(\ol{\cY}^{\sm},M_{\ol{\cY}^{\sm}}) \otimes_{W(k)} W_a(k)$. 
Then $G_Y$ naturally acts on them and so 
we can define 
the fine log Deligne-Mumford stack $([\ol{\cY}_a/G_Y], 
M_{[\ol{\cY}_a/G_Y]})$ as in Example \ref{exam1}. 
As an open 
substack of it, we have the fine log Deligne-Mumford stack 
$([\ol{\cY}^{\sm}_a/G_Y], M_{[\ol{\cY}^{\sm}_a/G_Y]})$ for 
$a \in \N$ and they induce the 
ind fine log algebraic stack $([\ol{\cY}^{\sm}/G_Y], \allowbreak 
M_{[\ol{\cY}^{\sm}/G_Y]}) 
:= \varinjlim_a ([\ol{\cY}^{\sm}_a/G_Y], \allowbreak M_{[\ol{\cY}^{\sm}_a/G_Y]})$. 
Note also that the endomorphism 
$F$ on $(\ol{\cY}^{\sm},M_{\ol{\cY}^{\sm}})$ induces the 
endomorphism on $([\ol{\cY}^{\sm}/G_Y], M_{[\ol{\cY}^{\sm}/G_Y]})$,
which we denote also by $F$. Then we define the category 
$\FLatt([\ol{\cY}^{\sm}/G_Y])^{\c}$ of 
unit-root $F$-lattices on $[\ol{\cY}^{\sm}/G_Y]$ 
as the category of pairs $(\cE,\phi)$ consisting of 
a locally free module $\cE$ of finite rank 
on $[\ol{\cY}^{\sm}/G_Y]$ 
($=$ a compatible family of locally free modules of finite rank 
on $[\ol{\cY}^{\sm}_a/G_Y]$ for $a \in \N$) and an isomorphism 
$\phi: F^*\cE \os{=}{\lra} \cE$. It is equivalent to the category 
of objects in $\FLatt(\ol{\cY}^{\sm})^{\c}$ endowed with equivariant 
action of $G_Y$. By Katz' equivalence \eqref{katzeq} for $\ol{\cY}^{\sm}$, 
we have the equivalence 
\begin{equation}\label{katzeq2}
\Rep_{O_K^{\sigma}}(\pi_1(\ol{Y}^{\sm})) \os{=}{\lra} 
\FLatt(\ol{\cY}^{\sm})^{\c}, 
\end{equation}
and the left hand side is unchanged if we replace $\ol{Y}^{\sm}$ by 
its open formal subscheme ${\ol{Y}^{\sm}}'$ with 
$\codim (\ol{Y}^{\sm} \setminus {\ol{Y}^{\sm}}', \ol{Y}^{\sm}) \geq 2$. 
Hence the right hand side is also unchanged if we replace $\ol{\cY}^{\sm}$ by 
its open formal subscheme ${\ol{\cY}^{\sm}}'$ with 
$\codim (\ol{\cY}^{\sm} \setminus {\ol{\cY}^{\sm}}', \ol{\cY}^{\sm}) 
\geq 2$. Using this, we see by the argument in the proof of 
Proposition \ref{wd} that the limit 
$\varinjlim_{Y\ra X \in \cG_X^t} \FLatt ([\ol{\cY}^{\sm}/G_Y])^{\c}$ 
is defined. Also, let $\Vect ([\ol{\cY}^{\sm}/G_Y])$ be 
the category of locally free modules of finite rank 
on $[\ol{\cY}^{\sm}/G_Y]$. Then the limit 
$\varinjlim_{Y\ra X \in \cG_X^t} \Vect ([\ol{\cY}^{\sm}/G_Y])$ is defined 
when $X$ is a curve. \par 
On the other hand, we have the morphism 
$\ol{\cX}_{\c,a} \lra [\Af_{W_a(k)}^r/\G_{m,W_a(k)}^r]$ defined by 
the log structure $M_{\ol{\cX}_a}$ on $\ol{\cX}_{\c,a}$ 
associated to $\cZ_{\c,a}$ and 
the natural homomorphism 
$\N^r \lra \ol{M_{\ol{\cX}_{\c,a}}} \cong 
\bigoplus_{i=1}^r \N_{\cZ_{\c,i,a}}$,
where $\N_{\cZ_{\c,i,a}}$ is the direct image to $\ol{\cX}_{\c,a}$ of the 
constant sheaf on $\cZ_{\c,i,a}$ with fiber $\N$. 
Then we put 
$(\ol{\cX}_{\c},\cZ_{\c})_{a}^{1/n} := \ol{\cX}_{\c,a} 
\times_{[\Af_{W_a(k)}^r/\G_{m,W_a(k)}^r],n} 
[\Af_{W_a(k)}^r/\G_{m,W_a(k)}^r]$ for $n \in \N$ prime to $p$, 
and put $(\ol{\cX}_{\c},\cZ_{\c})^{1/n} := 
\varinjlim_a (\ol{\cX}_{\c},\cZ_{\c})_a^{1/n}$. 
Then they admit a canonical endomorphism $F_0$ which lifts 
the $q$-th power map on $(\ol{X},Z)^{1/n}$ and which is compatible with 
$(\sigma|_{W(k)})^*$. Let us put 
$(\ol{\cX},\cZ)_{a}^{1/n} := (\ol{\cX},\cZ)_{\c,a}^{1/n} \otimes_{W(k)} O_K$, 
$(\ol{\cX},\cZ)^{1/n} := \varinjlim_a (\ol{\cX},\cZ)_{a}^{1/n}$. 
Then it admits the endomorphism $F := F_{\c} \otimes \sigma^*$. 
Let 
$\Vect((\ol{\cX},\cZ)^{1/n})$ be 
the category of 
locally free modules of finite rank 
on $(\ol{\cX},\cZ)^{1/n}$ 
($=$ a compatible family of locally free modules of finite rank 
on $(\ol{\cX},\cZ)_a^{1/n}$) and 
let $\FLatt((\ol{\cX},\cZ)^{1/n})^{\c}$ be 
the category of unit-root $F$-lattices on $(\ol{\cX},\cZ)^{1/n}$, 
that is, the category of pairs $(\cE,\phi)$ consisting of 
a locally free module $\cE$ of finite rank 
on $(\ol{\cX},\cZ)^{1/n}$ and an isomorphism 
$\phi: F^*\cE \os{=}{\lra} \cE$. Then we have the limits 
$$ \varinjlim_{(n,p)=1} \Vect((\ol{\cX},\cZ)^{1/n}), \,\,\,\,
\varinjlim_{(n,p)=1} \FLatt((\ol{\cX},\cZ)^{1/n})^{\c}. $$
Then we have the following theorem: 

\begin{theorem}\label{flatttt}
Let the notations be as above. 
\begin{enumerate}
\item 
There exists an equivalence 
\begin{equation}\label{flatt1}
\Rep_{O_K^{\sigma}}(\pi_1^t(X)) \os{=}{\lra} 
\varinjlim_{Y\ra X \in \cG_X^t} \FLatt([\ol{\cY}^{\sm}/G_Y])^{\c}
\end{equation} 
$($where $\Rep_{O_K^{\sigma}}(\pi_1^t(\ol{X}))$ 
denotes the category of 
continuous representations of the tame fundamental group 
$\pi_1^t(\ol{X})$ $($tamely ramified at generic points of $Z)$ 
to free $O_K^{\sigma}$-modules of finite rank$)$ 
and an equivalence 
\begin{equation}\label{flatt2}
\varinjlim_{Y\ra X \in \cG_X^t} \FLatt([\ol{\cY}^{\sm}/G_Y])^{\c}_{\Q} 
\os{=}{\lra}
\varinjlim_{Y\ra X \in \cG_X^t} \FIsoc([\ol{Y}^{\sm}/G_Y])^{\circ} 
\end{equation}
such that the following diagram is commutative$:$ 
\begin{equation}\label{flatt3}
\begin{CD}
\Rep_{O_K^{\sigma}}(\pi_1^t(X))_{\Q} 
@>{\text{\eqref{flatt1}}_{\Q}}>> 
\varinjlim_{Y\ra X \in \cG_X^t} \FLatt([\ol{\cY}^{\sm}/G_Y])^{\c}_{\Q} \\ 
@\vert @V{\text{\eqref{flatt2}}}VV \\ 
\Rep_{K^{\sigma}}(\pi_1^t(X)) 
@>{\text{\eqref{eqeq2t}}}>> 
\varinjlim_{Y\ra X \in \cG_X^t} \FIsoc([\ol{Y}^{\sm}/G_Y])^{\circ}. 
\end{CD}
\end{equation}
\item 
There exist equivalences 
\begin{equation}\label{flatt4}
\varinjlim_{Y\ra X \in \cG_X^t} \FLatt([\ol{\cY}^{\sm}/G_Y])^{\c}
 \os{=}{\lra} \varinjlim_{(n,p)=1} \FLatt((\ol{\cX},\cZ)^{1/n})^{\c}, 
\end{equation}
\begin{equation}\label{flatt5}
\varinjlim_{(n,p)=1} \FLatt((\ol{\cX},\cZ)^{1/n})^{\c}_{\Q} 
\os{=}{\lra} 
\varinjlim_{(n,p)=1} \FIsoc((\ol{X},Z)^{1/n})^{\circ}
\end{equation}
which makes the following diagram commutative$:$ 
\begin{equation}\label{flatt6}
\begin{CD}
\varinjlim_{Y\ra X \in \cG_X^t} \FLatt([\ol{\cY}^{\sm}/G_Y])^{\c}_{\Q}
@>{\text{\eqref{flatt4}}_{\Q}}>> 
\varinjlim_{(n,p)=1} \FLatt((\ol{\cX},\cZ)^{1/n})^{\c}_{\Q} \\ 
@V{\text{\eqref{flatt2}}}VV @V{\text{\eqref{flatt5}}}VV \\ 
\varinjlim_{Y\ra X \in \cG_X^t} \FIsoc([\ol{Y}^{\sm}/G_Y])^{\circ}
@>{\text{\eqref{tmt1}}^{\circ}}>> 
\varinjlim_{(n,p)=1} \FIsoc((\ol{X},Z)^{1/n})^{\circ}. 
\end{CD}
\end{equation}
In particular, we have the equivalence 
\begin{equation}\label{flatt7}
\Rep_{O_K^{\sigma}}(\pi_1^t(X)) \os{=}{\lra} 
\varinjlim_{(n,p)=1} \FLatt((\ol{\cX},\cZ)^{1/n})^{\c}
\end{equation}
which is defined as the composite $\text{\eqref{flatt4}}\circ
\text{\eqref{flatt1}}$. 
When $X$ is a curve, we have also a functor 
\begin{equation}\label{flatt8}
\varinjlim_{Y\ra X \in \cG_X^t} \Vect([\ol{\cY}^{\sm}/G_Y])
 {\lra} \varinjlim_{(n,p)=1} \Vect((\ol{\cX},\cZ)^{1/n}). 
\end{equation}
satisfying $F\text{-\eqref{flatt8}}^{\circ}=\text{\eqref{flatt4}}$. 
\end{enumerate}
\end{theorem}

\begin{proof}
The method of the proof is similar to that of 
Theorem \ref{cor1}, Proposition \ref{root0} and Theorem \ref{thm2}, 
as we explain below. \par 
Let $G_Y\text{-}\Sm_{O_K^{\sigma}}(\ol{Y}^{\sm})$ be the category of 
smooth $O_K^{\sigma}$-sheaves on $\ol{Y}^{\sm}$ endowed with 
equivariant $G_Y$-action. Then we have the equivalence 
\begin{equation*}
\Rep_{O_K^{\sigma}}(\pi_1^t(X)) \os{=}{\lra} 
\varinjlim_{Y \ra X \in \cG_X^t} 
G_Y\text{-}\Sm_{O_K^{\sigma}}(\ol{Y}^{\sm}) 
\end{equation*}
(which is the integral version of \eqref{eqeqeq1} and can be proved exactly 
in the same way). For $m=0,1,2$, let 
$\ol{Y}^{\sm}_m$, $\ol{\cY}^{\sm}_m$ be the $(m+1)$-fold fiber 
product of $\ol{Y}^{\sm}$, $\ol{\cY}^{\sm}$ over $[\ol{Y}^{\sm}/G_Y], 
[\ol{\cY}^{\sm}/G_Y]$ respectively. Then we have $2$-truncated simplicial 
scheme $\ol{Y}^{\sm}_{\b}$, $2$-truncated simplicial formal 
scheme $\ol{\cY}^{\sm}_{\b}$ and we have the equivalence 
\begin{align*}
\varinjlim_{Y \ra X \in \cG_X^t} 
G_Y\text{-}\Sm_{O_K^{\sigma}}(\ol{Y}^{\sm}) 
& \os{=}{\lra} 
\Sm_{O_K^{\sigma}}(\ol{Y}^{\sm}_{\b}) \\ 
& \hspace{-15pt} \os{\text{\eqref{katzeq'} for $\ol{\cY}_{\b}^{\sm}$}}{\lra} 
\varinjlim_{Y\ra X \in \cG_X^t} \FLatt(\ol{\cY}^{\sm}_{\b})^{\c} \\ 
& \os{=}{\lla} 
\varinjlim_{Y\ra X \in \cG_X^t} \FLatt([\ol{\cY}^{\sm}/G_Y])^{\c}, 
\end{align*}
where the first and the third equivalence follow from the etale descent. 
By composing these, we obtain the equivalence \eqref{flatt1}. \par 
Let us denote the category of compatible family of objects in 
$\FLatt(\ol{\cY}^{\sm}_m)^{\c}_{\Q}$ \,($m=0,1,2$) by 
$\{\FLatt(\ol{\cY}^{\sm}_m)^{\c}_{\Q}\}_{m=0,1,2}$. 
(Note that this is not a priori equal to the $\Q$-linearization 
$\FLatt(\ol{\cY}^{\sm}_{\b})^{\c}_{\Q}$ of the category 
$\FLatt(\ol{\cY}^{\sm}_{\b})^{\c}$ of unit-root $F$-lattices on 
$\ol{\cY}^{\sm}_{\b}$.) Then 
the functor 
\eqref{flatt2} is defined as the composite of equivalences 
\begin{align}
\varinjlim_{Y\ra X \in \cG_X^t} \FLatt([\ol{\cY}^{\sm}/G_Y])^{\c}_{\Q} 
& \os{=}{\lra} 
\varinjlim_{Y\ra X \in \cG_X^t} \FLatt(\ol{\cY}^{\sm}_{\b})^{\c}_{\Q} 
\label{nu0} \\ 
& \lra 
\varinjlim_{Y\ra X \in \cG_X^t}
\{\FLatt(\ol{\cY}^{\sm}_m)^{\c}_{\Q}\}_{m=0,1,2} \nonumber \\ 
& \hspace{-15pt} 
\os{\text{\eqref{crewisog} for $\ol{\cY}_{\b}^{\sm}$},=}{\lra} 
\varinjlim_{Y\ra X \in \cG_X^t} \FIsoc(\ol{Y}^{\sm}_{\b})^{\circ} \nonumber \\ 
& \os{=}{\lla} 
\varinjlim_{Y\ra X \in \cG_X^t} \FIsoc([\ol{Y}^{\sm}/G_Y])^{\circ}. \nonumber 
\end{align}
The commutativity 
of the diagram \eqref{flatt3} is the immediate consequence of 
the construction of the Crew's equivalence $G$ (as the composite 
$\text{\eqref{crewisog}} \circ \text{\eqref{katzeq}}$) for 
$\ol{Y}^{\sm}_{\b}$ and $\ol{\cY}^{\sm}_{\b}$. 
Since \eqref{flatt1}${}_{\Q}$ and \eqref{eqeq2t} are equivalences, 
we see that \eqref{flatt2} is also an equivalence. So we see that the 
the functor 
\begin{equation}\label{nu}
\varinjlim_{Y\ra X \in \cG_X^t} \FLatt(\ol{\cY}^{\sm}_{\b})^{\c}_{\Q} 
\lra 
\varinjlim_{Y\ra X \in \cG_X^t}
\{\FLatt(\ol{\cY}^{\sm}_m)^{\c}_{\Q}\}_{m=0,1,2}
\end{equation}
of the second line of \eqref{nu0} is an equivalence. \par 
Next, let us recall that 
the functor $\FIsoc([\ol{Y}^{\sm}/G_Y]) \lra 
\FIsoc((\ol{X},Z)^{1/n})$ is defined as 
the composite 
\begin{align*}
\FIsoc([\ol{Y}^{\sm}/G_Y]) & \os{=}{\lra} 
\FIsoc(\ol{Y}^{\sm}_{\b}) \\ 
& \,\,{\lra}\,\, \FIsoc(\ol{U}_{\b\b\b}) \\ 
& \os{=}{\lra} \FIsoc(\ol{X}'_{\b\b}) \\ 
& \os{=}{\lra} \FIsoc(\ol{X}_{\b\b}) \\ 
& \os{=}{\lra} \FIsoc((\ol{X},Z)^{1/n}), 
\end{align*}
where $\ol{Y}^{\sm}_{\b}, \ol{U}_{\b\b\b}, \ol{X}'_{\b\b}$ are as in 
the proof of Proposition \ref{root0}. Since they are finite Kummer log etale 
over its image in $(\ol{X},Z)$ 
with respect to suitable log structures associated to certain 
simple normal crossing divisors, we have the canonical lifting 
$\ol{\cY}^{\sm}_{\c,\b}, 
\ol{\cU}_{\c,\b\b\b}, \ol{\cX}'_{\c,\b\b}$ of them to $p$-adic formal schemes 
smooth over $\Spf W(k)$ and 
they admit the endomorphism 
$F_{\c}$ which lifts the $q$-th power map on the special fiber and which are 
compatible with $F_{\c}$ on $(\ol{\cX}_{\c},\cZ_{\c})$. 
If we put 
$\ol{\cY}^{\sm}_{\b} := \ol{\cY}^{\sm}_{\c,\b} \otimes_{W(k)} O_K, 
\ol{\cU}_{\b\b\b} := \ol{\cU}_{\c,\b\b\b} \otimes_{W(k)} O_K$ and 
$\ol{\cX}'_{\b\b} := \ol{\cX}'_{\c,\b\b} \otimes_{W(k)} O_K$, 
they are smooth over $\Spf O_K$ and they admit 
the endomorphism 
$F := F_{\c} \otimes \sigma^*$ 
which lifts the $q$-th power map on the special fiber and which are 
compatible with $F$ on $(\ol{\cX},\cZ)$. 
Hence we can define the sequence of functors 
\begin{align}
\FLatt([\ol{\cY}^{\sm}/G_Y])^{\c} & \os{=}{\lra} 
\FLatt(\ol{\cY}^{\sm}_{\b})^{\c} \label{43/0} \\ 
& \,\, {\lra} \,\, \FLatt(\ol{\cU}_{\b\b\b})^{\c} \nonumber \\ 
& \os{=}{\lra} \FLatt(\ol{\cX}'_{\b\b})^{\c} \nonumber \\ 
& \os{=}{\lra} \FLatt(\ol{\cX}_{\b\b})^{\c} \nonumber \\ 
& \os{=}{\lra} \FLatt((\ol{\cX},\cZ)^{1/n})^{\c} \nonumber 
\end{align}
(where the equality follows from descent property and the fact that 
$\codim (\ol{\cX}_{\b\b} \setminus \ol{\cX}'_{\b\b}, \ol{\cX}_{\b\b}) 
\geq 2$). So we have defined the functor \eqref{flatt4}. 
When $X$ is a curve, we can define the functor 
\eqref{flatt8} in the same way as above, noting the equality 
$\ol{\cX} = \ol{\cX}'$ in this case. \par 
In the following in this proof, we denote $\ol{X}_{\b\b}, \ol{\cX}_{\b\b}$ 
defined from $(\ol{X},Z)^{1/n}, (\ol{\cX},\cZ)^{1/n}$ by 
$\ol{X}^{(n)}_{\b\b}, \ol{\cX}^{(n)}_{\b\b}$, because they depend on $n$. 
Also, let us denote the category of compatible family of objects in 
$\FLatt(\ol{\cX}^{(n)}_{kl})^{\c}_{\Q}$ \,($k,l=0,1,2$) by 
$\{\FLatt(\ol{\cX}^{(n)}_{kl})^{\c}_{\Q}\}_{k,l=0,1,2}$. 
Then 
the functor \eqref{flatt5} is defined as the composite 
\begin{align}
\varinjlim_{(n,p)=1} \FLatt((\ol{\cX},\cZ)^{1/n})^{\c}_{\Q} & \os{=}{\lra} 
\varinjlim_{(n,p)=1} \FLatt(\ol{\cX}_{\b\b}^{(n)})^{\c}_{\Q} \label{nunu0} 
\\ & 
\lra 
\varinjlim_{(n,p)=1} \{\FLatt(\ol{\cX}_{kl}^{(n)})^{\c}_{\Q}\}_{k,l=0,1,2} 
\nonumber \\ 
& \hspace{-15pt}
\os{\text{\eqref{crewisog} for $\ol{\cX}_{\b\b}^{(n)}$},=}{\lra}
\varinjlim_{(n,p)=1} 
\FIsoc(\ol{X}_{\b\b}^{(n)})^{\c} \nonumber \\ & \os{=}{\lra} 
\varinjlim_{(n,p)=1} 
\FIsoc((\ol{X},Z)^{1/n})^{\c}. \nonumber 
\end{align}
Then one can prove the commutativity \eqref{flatt6} easily, using 
the functoriality of \eqref{crewisog}. \par 
Finally we prove that the functor \eqref{flatt4} is an equivalence. 
To prove this, it suffices to prove the equivalence 
\eqref{flatt7}. Note that a part of the functors \eqref{43/0}
$$ 
\FLatt(\ol{\cY}^{\sm}_{\b})^{\c} 
\lra \FLatt(\ol{\cU}_{\b\b\b})^{\c} 
\os{=}{\lra} \FLatt(\ol{\cX}'_{\b\b})^{\c} 
\os{=}{\lra} \FLatt(\ol{\cX}_{\b\b})^{\c} 
$$ 
is rewritten via the equivalence \eqref{katzeq'} in the following way: 
\begin{equation}\label{43/1}
\Sm_{O_K^{\sigma}}(\ol{Y}_{\b}^{\sm}) \lra 
\Sm_{O_K^{\sigma}}(\ol{U}_{\b\b\b}) \os{=}{\lla} 
\Sm_{O_K^{\sigma}}(\ol{X}'_{\b\b}) \os{=}{\lla} 
\Sm_{O_K^{\sigma}}(\ol{X}_{\b\b}).
\end{equation}
So it induces the funtor 
\begin{equation}\label{43/2}
\varinjlim_{Y \ra X \in \cG_X^t} \Sm_{O_K^{\sigma}}(\ol{Y}_{\b}^{\sm}) 
\lra \varinjlim_{(n,p)=1} \Sm_{O_K^{\sigma}}(\ol{X}^{(n)}_{\b\b})
\end{equation}
and using this, we can rewrite the 
functor \eqref{flatt7} as the composite 
\begin{align}
\Rep_{O_K^{\sigma}}(\pi_1^t(X)) & 
\os{=}{\lra} \varinjlim_{Y\to X \in \cG} \Sm_{O_K^{\sigma}}(\ol{Y}^{\sm}_{\b})
 \os{\text{\eqref{43/2}}}{\lra} 
\varinjlim_{(n,p)=1} \Sm_{O_K^{\sigma}}(\ol{X}^{(n)}_{\b\b})  
\label{43/3} \\  
& \os{\text{\eqref{katzeq'}},=}{\lra}
\varinjlim_{(n,p)=1} \FLatt(\ol{\cX}^{(n)}_{\b\b})^{\c}
\os{=}{\lla} \varinjlim_{(n,p)=1} \FLatt((\ol{\cX},\cZ)^{1/n})^{\c}. \nonumber 
\end{align}
So it suffices to prove that the first line 
\begin{equation}\label{43/4}
\Rep_{O_K^{\sigma}}(\pi_1^t(X))
\os{=}{\lra} \varinjlim_{Y\to X \in \cG} \Sm_{O_K^{\sigma}}(\ol{Y}^{\sm}_{\b})
 \os{\text{\eqref{43/2}}}{\lra} 
\varinjlim_{(n,p)=1} \Sm_{O_K^{\sigma}}(\ol{X}^{(n)}_{\b\b})  
\end{equation} 
of the functor \eqref{43/3} is an equivalence. 
By construction, the composite of the functor \eqref{43/4} and the 
restriction functor 
\begin{equation}\label{43/5}
\varinjlim_{(n,p)=1} \Sm_{O_K^{\sigma}}(\ol{X}^{(n)}_{\b\b})  
\lra \Sm_{O_K^{\sigma}}(X_{\b\b})  
\end{equation}
($X_{\b\b} := X \times_{\ol{X}} \ol{X}^{(n)}_{\b\b}$) 
is equal to the composite 
$\Rep_{O_K^{\sigma}}(\pi_1^t(X)) \allowbreak 
\os{\subset}{\lra} \Rep_{O_K^{\sigma}}(\pi_1(X)) 
\os{=}{\lra} \Sm_{O_K^{\sigma}}(X_{\b\b})$, which is fully faithful. 
Also, it is easy to see that \eqref{43/5} is faithful. So 
\eqref{331/4} is fully faithful. Also, it is obvious that 
any object $\rho$ in $\Rep_{O_K^{\sigma}}(\pi_1(X))$ which is sent to 
an object in $\varinjlim_{(n,p)=1} \Sm_{O_K^{\sigma}}(\ol{X}^{(n)}_{\b\b})  
\subset \Sm_{K^{\sigma}}(X_{\b\b})$ is tamely ramified along $Z$. 
So the functor \eqref{43/4} is an equivalence as desired and we are done. 
(As a corollary, we see that the functor 
\begin{equation}\label{nunu}
\varinjlim_{(n,p)=1} \FLatt(\ol{\cX}_{\b\b}^{(n)})^{\c}_{\Q} 
\lra 
\varinjlim_{(n,p)=1} \{\FLatt(\ol{\cX}_{kl}^{(n)})^{\c}_{\Q}\}_{k,l=0,1,2} 
\end{equation}
of the second line in \eqref{nunu0} is an equivalence.) 
\end{proof}

As for the functor \eqref{flatt8}, we have the following result: 

\begin{proposition}
Let the notations be as above and assume that $\ol{X}$ is a 
$(g,l)$-curve, $X$ is a $(g,l')$-curve $\,(l' \geq l)$. 
Then the functor \eqref{flatt8} is an equivalence if 
$(g,l,l') \not= (0,0,1)$ and it is not an equivalence if 
$(g,l,l') = (0,0,1)$. 
\end{proposition}

\begin{proof}
Recall that, in the proof of Theorem \ref{curve}, 
we have defined the functor 
\begin{align*}
\Isoc((\ol{X},Z)^{1/n}) & \os{=}{\lra} \Isoc(\ol{X}_{\b\b}) \\ 
& \lra \Isoc(\ti{Y}_{\b\b\b}) \\ 
& \os{=}{\lra} \Isoc(\ol{Y}_{\b\b}) \\ 
& \os{=}{\lra} \Isoc(\ol{Y}_{\b}) = \Isoc([\ol{Y}/G_Y]) 
\end{align*}
in the case $(g,l,l') \not= (0,0,1)$. 
(The notations are as in there.) In the situation here, we have 
the natural lifts 
$\ol{\cX}_{\b\b}, \ti{\cY}_{\b\b\b}, \ol{\cY}_{\b\b}, \ol{\cY}_{\b}$ of 
$\ol{X}_{\b\b}, \ti{Y}_{\b\b\b}, \ol{Y}_{\b\b}, \ol{Y}_{\b}$ to 
$p$-adic formal schemes smooth over $\Spf O_K$ (as in the proof of 
Theorem \ref{flatttt}) and we can define 
in the same way the functor 
\begin{align*}
\Vect((\ol{\cX},\cZ)^{1/n}) & \os{=}{\lra} \Vect(\ol{\cX}_{\b\b}) \\ 
& \lra \Vect(\ti{\cY}_{\b\b\b}) \\ 
& \os{=}{\lra} \Vect(\ol{\cY}_{\b\b}) \\ 
& \os{=}{\lra} \Vect(\ol{\cY}_{\b}) = \Vect([\ol{\cY}/G_Y]) 
\end{align*}
inducing the functor 
$$ \varinjlim_{(n,p)=1} \Vect((\ol{\cX},\cZ)^{1/n}) \lra 
\varinjlim_{Y\to X \in \cG_X^t} \Vect([\ol{\cY}/G_Y]) = 
\varinjlim_{Y\to X \in \cG_X^t} \Vect([\ol{\cY}^{\sm}/G_Y]), $$
which gives the inverse of \eqref{flatt8}. \par 
Let us consider the case $(g,l,l') = (0,0,1)$. In this case, we have 
$$
\varinjlim_{Y\ra X \in \cG_X^t} \Vect([\ol{\cY}^{\sm}/G_Y])
= \Vect(\ol{\cX}).$$ On the other hand, 
if we define the diagram 
\begin{equation*}
\begin{CD}
\cX @<{\supset}<< \cV @>{\subset}>> \cU \\ 
@. @AAA  @A{\ti{\varphi}}AA \\
@. \cV^{(n)} @>{\subset}>> \cU^{(n)}
\end{CD}
\end{equation*}
and $t$ 
as in \eqref{p1} with $n \geq 2$, we have 
$$ \Vect((\ol{\cX},\cZ)^{1/n}) = 
\Vect(\cX) \times_{\Vect([\cV^{(n)}/\mu_n])} \Vect([\cU^{(n)}/\mu_n]) $$
(see \cite{borne1}) 
and this contains an object of the form $\cE = (\cE_0,\cE_1,\iota)$, where 
$\cE_0 = \cO_{\cX}, \cE_1 = t\cO_{\cU^{(n)}}$ endowed with the action 
$\zeta \cdot t = \zeta t$ \,($\zeta \in \mu_n$) of $\mu_n$ 
and $\iota$ is the $\mu_n$-equivariant isomorphism 
$\cE_1 |_{\cV^{(n)}} = t\cO_{\cV^{(n)}} \hra \cO_{\cV^{(n)}} = 
\cE_0 |_{\cV^{(n)}}$. Then the image of $\cE$ in the category 
$\Vect((\ol{\cX},\cZ)^{1/m})$ for any $n \,|\, m$ does not come from 
an object in $\Vect(\ol{\cX})$ because $\cE_1 |_{\cU^{(m)}} = 
t^{m/n}\cO_{\cU^{(m)}}$ is not locally generated by $\mu_m$-invariant 
sections. Hence the functor \eqref{flatt8} is not essentially surjective 
in this case. 
\end{proof}

Next we define the notion of parabolic vector bundles and 
parabolic (unit-root) $F$-lattices on $(\ol{\cX},\cZ)$. 

\begin{definition}\label{defp}
Let $(\ol{\cX},\cZ)$ be as above. 
\begin{enumerate}
\item 
A parabolic vector bundle on $(\ol{\cX},\cZ)$ is an inductive system 
$(\cE_{\alpha})_{\alpha \in \Z_{(p)}^r}$ of vector bundles on $\ol{\cX}$ 
$($we denote the transition map by 
$\iota_{\alpha\beta}: \cE_{\alpha} \lra \cE_{\beta}$ for 
$\alpha = (\alpha_i), \beta = (\beta_i) \in \Z_{(p)}^r$ with 
$\alpha_i \leq \beta_i \,(\forall i))$ satisfying the following 
conditions$:$ 
\begin{enumerate}
\item 
For any $1 \leq i \leq r$, there is an isomorphism as inductive systems 
$$ 
((\cE_{\alpha+e_i})_{\alpha}, (\iota_{\alpha+e_i,\beta+e_i})_{\alpha,\beta}) 
\cong 
((\cE_{\alpha}(\cZ_i))_{\alpha}, 
(\iota_{\alpha\beta} \otimes \id)_{\alpha,\beta}) $$
via which the morphism 
$(\iota_{\alpha,\alpha+e_i})_{\alpha} : 
(\cE_{\alpha})_{\alpha} \lra (\cE_{\alpha+e_i})_{\alpha}$ is identified 
with the morphism 
$\id \otimes \iota_{0,e_i}^0: 
(\cE_{\alpha})_{\alpha} \lra (\cE_{\alpha}(\cZ_i))_{\alpha}$, 
where $\iota_{0,e_i}^0: \cO_{\ol{\cX}} \hra \cO_{\ol{\cX}}(\cZ_i)$ 
denotes the natural inclusion. 
\item 
There exists a positive integer $n$ prime to $p$ which satisfies the following 
condition$:$ For any $\alpha = (\alpha_i)_i$, $\iota_{\alpha'\alpha}$ is 
an isomorphism if we put $\alpha' = ([n\alpha_i]/n)_i$. 
\end{enumerate}
\item 
A parabolic $F$-lattice 
$($resp. a parabolic unit-root $F$-lattice$)$ 
on $(\ol{\cX},\cZ)$ is a pair 
$((\cE_{\alpha})_{\alpha \in \Z_{(p)}^r}, \Psi := 
(\Psi_{\alpha})_{\alpha \in \Z_{(p)}^r})$ 
consisting of a parabolic vector bundle 
$(\cE_{\alpha})_{\alpha \in \Z_{(p)}^r}$ on $(\ol{\cX},\cZ)$ 
endowed with morphisms 
$\Psi_{\alpha}: (F^*\cE_{\alpha})_{\Q} \lra 
\cE_{q\alpha,\Q}$ in the category $\Coh(\cX)_{\Q}$ 
$($resp. $\Psi_{\alpha}: F^*\cE_{\alpha} \lra 
\cE_{q\alpha}$ in the category $\Coh(\cX))$ 
such that $\varinjlim_{\alpha} \Psi_{\alpha}: 
\varinjlim_{\alpha} (F^*\cE_{\alpha})_{\Q} 
\lra \varinjlim_{\alpha} \cE_{\alpha,\Q}$ 
$($resp.  $\varinjlim_{\alpha} \Psi_{\alpha}: 
\varinjlim_{\alpha} F^*\cE_{\alpha} 
\lra \varinjlim_{\alpha} \cE_{\alpha})$ 
is isomorphic as ind-objects. 
\end{enumerate}
\end{definition}

For $\alpha := (\alpha_i)_i \in \Z_{(p)}^r$, let 
$\cO_{\ol{\cX}}(\sum_i\alpha_i\cZ_i) := 
(\cO_{\ol{\cX}}(\sum_i\alpha_i\cZ_i)_{\beta})_{\beta}$ 
be the parabolic vector bundle on $(\ol{\cX},\cZ)$ defined by 
$\cO_{\ol{\cX}}(\sum_i\alpha_i\cZ_i)_{\beta} := 
\cO_{\ol{\cX}}(\sum_i[\alpha_i + \beta_i] \cZ_i)$, where 
$\beta = (\beta_i)_i$. Using this, we define the notion of locally 
abelian parabolic vector bundles and locally 
abelian parabolic (unit-root) $F$-lattices on $(\ol{\cX},\cZ)$. 
(This terminology is essentially due to Iyer-Simpson \cite{is}.)

\begin{definition}\label{defla}
A parabolic vector bundle $(\cE_{\alpha})_{\alpha}$ 
on $(\ol{\cX},\cZ)$ is called locally abelian if there exists a 
positive integer $n$ prime to $p$ such that, Zariski locally on 
$\ol{\cX}_n:= \ol{\cX} \otimes_{O_K} O_K[\mu_n]$, 
$(\cE_{\alpha})_{\alpha} |_{\ol{\cX}_n}$ has the form 
$\bigoplus_{j=1}^{\mu} \cO_{\ol{\cX}}(\sum_i \alpha_{ij}\cZ_i) 
|_{\ol{\cX}_n}$ 
for some $\alpha_{ij} \in \frac{1}{n}\Z 
\,(1 \leq i \leq r, 1 \leq j \leq \mu)$. 
A parabolic $($unit-root$)$ $F$-lattice $((\cE_{\alpha})_{\alpha},\Psi)$ is 
called locally abelian if so is $(\cE_{\alpha})_{\alpha}$. \par 
We denote the category of locally abelian parabolic vector bundles 
on $(\ol{\cX},\cZ)$ by $\PVect(\ol{\cX},\cZ)$, 
 the category of locally abelian parabolic $F$-lattices 
on $(\ol{\cX},\cZ)$ by $\PFLatt(\ol{\cX},\cZ)$ and 
 the category of locally abelian parabolic unit-root $F$-lattices 
on $(\ol{\cX},\cZ)$ by $\PFLatt(\ol{\cX},\cZ)^{\c}$. 
\end{definition}

Then we have the following theorem. 

\begin{theorem}\label{flattthm}
Let $(\ol{\cX},\cZ)$ be as above. Then there exist equivalences of 
categories 
\begin{equation}\label{vec+}
\ba: \varinjlim_{(n,p)=1}\Vect((\ol{\cX},\cZ)^{1/n}) \os{=}{\lra} 
\PVect(\ol{\cX},\cZ), 
\end{equation}
\begin{equation}\label{flatt+}
\varinjlim_{(n,p)=1}\FLatt((\ol{\cX},\cZ)^{1/n})^{\c} \os{=}{\lra} 
\PFLatt(\ol{\cX},\cZ)^{\c},  
\end{equation}
\begin{equation}\label{flatt+--}
\varinjlim_{(n,p)=1}\FLatt((\ol{\cX},\cZ)^{1/n}) \os{=}{\lra} 
\PFLatt(\ol{\cX},\cZ). 
\end{equation}
In particular, we have the equivalence {\rm \eqref{seq4-}} defined 
as the composite 
$$ \Rep_{O_K^{\sigma}}(\pi_1^t(X)) 
\os{\text{\eqref{flatt7}}}{\lra} 
\varinjlim_{(n,p)=1} \FLatt ((\ol{\cX},\cZ)^{1/n})^{\c} 
\os{\text{\eqref{flatt+}}}{\lra} \PFLatt(\ol{\cX},\cZ)^{\c}. $$
\end{theorem}

\begin{proof}
The proof of the equivalence \eqref{vec+} we give below is 
essentially due to Iyer-Simpson \cite{is}. (See also 
Borne \cite{borne1}, \cite{borne2}.) \par 
Let us fix $n \in \N$ which is prime to $p$. Then, 
we have the inductive system of 
line bundles $$(\cO(\sum_i\alpha_i\cZ_i))_{\alpha = (\alpha_i)_i \in 
(\frac{1}{n} \Z / \Z)^r}$$ 
on $(\ol{\cX},\cZ)^{1/n}$ (see \cite[p.353]{is}). 
(Here we give a definition using log structure: 
It suffices to define the inductive system of 
line bundles $(\cO(\sum_i\alpha_i\cZ_i)_a)_{\alpha = (\alpha_i)_i \in 
(\frac{1}{n} \Z / \Z)^r}$ on 
$(\ol{\cX},\cZ)^{1/n}_a \,(a \in \N)$ 
which are compatible with respect to $a$. 
Let 
$Y \os{g}{\lra} (\ol{\cX},\cZ)^{1/n}_a$ 
be a surjective etale morphism from 
some scheme $Y$. Then the composite 
$Y \os{g}{\lra} (\ol{\cX},\cZ)^{1/n}_a \os{\text{proj.}}{\lra} 
(\ol{\cX}_{\c},\cZ_{\c})^{1/n}_a \os{\text{proj.}}{\lra} 
[\Af_{W_a(k)}^r/\G_{m,W_a(k)}^r]$ corresponds to 
a log structure $M \os{\psi}{\ra} \cO_Y$ 
on $Y$ and a morphism $\gamma: \N^r \lra 
\ol{M} := M/\cO_Y^{\times}$ which is liftable to a chart etale 
locally. Take $Y$ etale local enough so that 
$\gamma$ is lifted to a chart $\ti{\gamma}: \N^r \lra M$. 
Let us put $Y' := Y \times_{(\ol{\cX},\cZ)^{1/n}_a} Y$ and 
denote the $j$-th projection $Y' \lra Y$ by $\pi_j$. 
Then, for $\alpha := (\alpha_i)_i \in (\frac{1}{n} \Z / \Z)^r$, 
take $\alpha_+ := (\alpha_{+,i})_i, \alpha_- := (\alpha_{-,i})_i$ 
in $(\frac{1}{n} \Z / \Z)^r$ with $\alpha = \alpha_+ - \alpha_-, 
\alpha_{+,i}, \alpha_{-,i} \geq 0 \,(\forall i)$ and 
define the line bundle $\cO(\sum_i \alpha_i \cZ_i)_a$ by patching 
the trivial line bundle $\cO_Y$ on $Y$ by the isomorphism 
$\pi_2^*\cO_Y \os{=}{\lra} \pi_1^*\cO_Y$ defined by 
the section $u \in \Gamma(Y',\cO_{Y'}^{\times})$ satisfying 
$u\pi_1^*(\ti{\gamma}(n\alpha_-))\pi_2^*(\ti{\gamma}(n\alpha_+)) = 
\pi_1^*(\ti{\gamma}(n\alpha_+))\pi_2^*(\ti{\gamma}(n\alpha_-))$. 
For $\alpha = (\alpha_i)_i, \beta := (\beta_i)_i$ in 
$(\frac{1}{n} \Z / \Z)^r$ with $\alpha_i \leq \beta_i \,(\forall i)$, 
the homomorphism 
$\cO(\sum_i\alpha_i\cZ_i)_a \lra 
\cO(\sum_i\beta_i\cZ_i)_a$ is defined as the multiplication 
by $\psi \circ \ti{\gamma}(n(\beta - \alpha))$ on $Y$.) \par 
For an object $\cE$ in $\Vect((\ol{\cX},\cZ)^{1/n})$ and $\alpha = 
(\alpha_i)_i \in 
\Z_{(p)}^r$, let us define $\ba(\cE)_{\alpha} := 
\pi_*(\cE \otimes \cO_{\ol{\cX}}(\sum_i ([n\alpha_i]/n) \cZ_i))$, 
where $\pi$ denotes the morphism $(\ol{\cX},\cZ)^{1/n} \lra \ol{\cX}$. 
Then $\ba(\cE) := (\ba(\cE)_{\alpha})_{\alpha}$ forms an 
inductive system of sheaves on $\ol{\cX}$. \par 
Let us examine the local structure of $\ba(\cE)$. 
Let us take an affine open formal subscheme 
$\cU = \Spf R \subseteq \ol{\cX}_n := \ol{\cX} \otimes_{O_K} O_K[\mu_n]$ 
such that $\cU \times_{\ol{\cX}} \cZ_i$ is defined as 
$\{t_i=0\}$ for some $t_i$ \,$(1 \leq i \leq r)$, and let us put 
$\cU^{(n)} := \Spf R[s_i]_{1 \leq i \leq r}/(s_i^n - t_i)_{1 \leq i \leq r}$, 
$\cZ_i^{(n)} := \{s_i = 0\}$ and let us denote the natural morphism 
$\cU^{(n)} \lra \cU$ by $\pi^{(n)}$. Then $G = \mu_n^r \cong 
(\Z/n\Z)^r$ naturally acts on $\cU^{(n)}$ (as the action on $s_i$'s). 
Let us take a closed point $x$ of $\cU$ and put $x^{(n)} := 
(\cU^{(n)} \times_{\cU} x)_{\rm red}$. Then $G$ acts naturally on $x^{(n)}$. 
For a character $\xi: G \lra \mu_n$, let $\cO_{x^{(n)}}(\xi)$ be 
the structure sheaf $\cO_{x^{(n)}}$ on $x^{(n)}$ endowed with the 
equivariant $G$-action by which $g \in G$ acts as 
$g^*\cO_{x^{(n)}} \os{g^*}{\lra} \cO_{x^{(n)}} \os{\xi(g)}{\lra} 
\cO_{x^{(n)}}$. Then it is easy to see that the restriction 
$\cE |_{x^{(n)}}$ of $\cE$ to $x^{(n)}$ has the form 
$\bigoplus_{j=1}^{\mu} \cO_{x^{(n)}}(\xi_j)$ for some 
characters $\xi_j: G \lra \mu_n \,(1 \leq j \leq \mu)$. 
Note that, for a character $\xi: G \lra \mu_n$, 
we can define the sheaf $\cO_{\cU^{(n)}}(\xi)$ on $\cU^{(n)}$ 
endowed with an equivariant $G$-action in the same way as 
$\cO_{x^{(n)}}(\xi)$. Now let us put 
$\cF := \oplus_{j=1}^{\mu} \cO_{\cU^{(n)}}(\xi_j)$
and let 
$u_0: \cF \lra \cE |_{\cU^{(n)}}$ be any $\cO_{\cU^{(n)}}$-linear 
homomorphism which lifts the canonical isomorphism 
$\cF |_{\ol{x}^{(n)}} \cong \cE |_{\ol{x}^{(n)}}$. 
Then $u := \sum_{g \in G} g^{-1}u_0g$ gives a $G$-equivariant 
$\cO_{\cU^{(n)}}$-linear homomorphism lifting the isomorphism 
$\cF |_{\ol{x}^{(n)}} \cong \cE |_{\ol{x}^{(n)}}$. 
Hence there exists an element $f_x \in \Gamma(\cU,\cO_{\cU})$ with 
$x \in \{f_x \not= 0\} =: \cU_x \subseteq \cU$ such that 
$u$ is isomorphic on $\cU^{(n)}_x := \pi^{(n),-1}(\cU_x)$. 
Then we have 
\begin{align*}
& (\ba(\cE)_{\alpha})_{\alpha} |_{\cU_x} \\ := \,\, & 
(\pi_*(\cE \otimes \cO(\sum_i [n\alpha_i]/n \cZ_i)))_{\alpha} 
|_{\cU_x}
= 
(\pi^{(n)}_*(\cE_{\cU^{(n)}}\otimes \cO_{\cU^{(n)}}
(\sum_i [n\alpha_i] \cZ^{(n)}_i))^G)_{\alpha} |_{\cU_x} \\  
\cong \,\, & 
(\pi^{(n)}_*(\cF \otimes \cO_{\cU^{(n)}}
(\sum_i [n\alpha_i]\cZ_i^{(n)}))^G )_{\alpha} |_{\cU_x}
= 
\oplus_{j=1}^{\mu} (\pi^{(n)}_* 
(\cO_{\cU^{(n)}}(\xi_j)(\sum_i [n\alpha_i] \cZ_i^{(n)})))^G)_{\alpha} 
|_{\cU_x}
\end{align*}
and we can check that there exists some $a_j = (a_{ji})_i \in 
(\frac{1}{n}\Z)^r \,(1 \leq j \leq \mu)$ 
such that the above inductive system is isomorphic to 
$$ \oplus_{j=1}^{\mu} 
(\cO_{\ol{\cX}}(\sum_i [a_{ji} + [n\alpha_i]/n] \cZ_i))_{\alpha} |_{\cU_x} = 
\oplus_{j=1}^{\mu} 
(\cO_{\ol{\cX}}(\sum_i [a_{ji} + \alpha_i] \cZ_i))_{\alpha} |_{\cU_x}. $$
Hence $(\ba(\cE)_{\alpha})_{\alpha}$ is a locally abelian 
parabolic vector bundle on $(\ol{\cX},\cZ)$. 
The full faithfulness of the functor $\ba$ can be checked locally and 
hence we may check it for the objects of the form 
$\cF = \oplus_{j=1}^{\mu} \cO_{\cU^{(n)}}(\xi_j)$, and it is easy in 
this case. The essential surjectivity can be also checked locally and so 
it is enough to check it for the objects of the form 
$\bigoplus_{j=1}^{\mu} \cO_{\ol{\cX}}(\sum_i \alpha_{ij}\cZ_i)$, which can 
be easily seen. So we have proved the equivalence \eqref{vec+}. \par 
We give an explicit quasi-inverse functor 
$$ \bb: \PVect(\ol{\cX},\cZ) \lra 
\varinjlim_{(n,p)=1}\Vect((\ol{\cX},\cZ)^{1/n})$$ 
in the following 
way, as in \cite{borne1}, \cite{borne2}: 
For an object $\cE := (\cE_{\alpha})_{\alpha}$ in 
$\PVect(\ol{\cX},\cZ)$, take $n \in \N$ prime to $p$ as in 
Definition \ref{defp} (1)(b) and let $\bb(\cE)$ be the coend 
of the family $\{\cO(\sum_i -a_i\cZ_i) \otimes \pi^*\cE_{b}\}_{a=(a_i)_i,b \in 
(\frac{1}{n}\Z)^r}$, that is, the universal object 
in $\Vect((\ol{\cX},\cZ)^{1/n})$ 
which admits morphisms 
$$ f_a: \cO(\sum_i -a_i\cZ_i) \otimes \pi^*\cE_{a} \lra \bb(\cE) $$ 
for any $a = (a_i)_i \in (\frac{1}{n}\Z)^r$ making the diagram 
\begin{equation*}
\begin{CD}
\cO(\sum_i -a_i\cZ_i) \otimes \pi^*\cE_{b} @>>> 
\cO(\sum_i -a_i\cZ_i) \otimes \pi^*\cE_{a} \\ 
@VVV @V{f_a}VV \\ 
\cO(\sum_i -b_i\cZ_i) \otimes \pi^*\cE_{b} @>{f_b}>> \bb(\cE)
\end{CD}
\end{equation*}
commutative for any $a=(a_i)_i, b=(b_i)_i \in (\frac{1}{n}\Z)^r$ 
with $a_i \geq b_i \,(\forall i)$: 
The existence of such object and the base change property 
for flat morphism are checked rather easily for 
the objects of the form 
$\bigoplus_{j=1}^{\mu} \cO_{\ol{\cX}}(\sum_i \alpha_{ij}Z_i)$ 
(see \cite[3.17]{borne1}), and 
this implies the existence in general case by descent. 
We can construct the morphism $\ba \circ \bb \lra \id$ in the same way as 
\cite[3.18]{borne1} and check that this is an isomorphism by looking 
at locally and arguing as \cite[3.17]{borne1}. 
By looking at the proof in \cite[3.17]{borne1}, we see the 
following two 
facts: $\bb(\cE)$ is in fact the coend of the family 
$\{\cO(\sum_i -a_i\cZ_i) \otimes 
\pi^*\cE_{b}\}_{a=(a_i)_i,b \in (\frac{1}{n}\Z \cap [0,1))^r}$, since 
so is it when $\cE$ has the form 
$\bigoplus_{j=1}^{\mu} \cO_{\ol{\cX}}(\sum_i \alpha_{ij}\cZ_i)$. Also, 
we see that, when $\cE$ has the form $\cO_{\ol{\cX}}(\sum_i \alpha_{i}\cZ_i)$, 
there exists some $a \in (\frac{1}{n}\Z \cap [0,1))^r$ such that 
$f_a$ is an isomorphism. By the former fact, there exists the canonical 
map 
\begin{equation}\label{injinj}
\bb(\cE) \lra \cO \otimes \pi^*\cE_1 = \pi^*\cE_1 
\end{equation}
and it is injective because, when $\cE$ has the form 
$\cO_{\ol{\cX}}(\sum_i \alpha_{i}Z_i)$, the first map in \eqref{injinj} 
is identified with 
the map $\cO(-\sum_ia_i\cZ_i) \otimes \pi^*\cE_a \hra 
\cO \otimes \pi^*\cE_1$ for some $a \in (\frac{1}{n}\Z \cap [0,1))^r$ 
by the latter fact above. \par 
Now we define the functors \eqref{flatt+}, 
\eqref{flatt+--}. Let $(\cE,\Psi)$ be an object in 
$\FLatt((\ol{\cX},\cZ)^{1/n})^{\c}$. Then we define the morphisms 
$\Psi_{\alpha}: F^*\ba(\cE)_{\alpha} \lra \ba(\cE)_{q\alpha}$ as the 
composite 
\begin{align}
F^*\ba(\cE)_{\alpha} & = 
F^*\pi_*(\cE \otimes \cO_{\ol{\cX}}(\sum_i ([n\alpha_i]/n) \cZ_i)) 
\label{f-1} \\ 
& \lra 
\pi_*F^*(\cE \otimes \cO_{\ol{\cX}}(\sum_i ([n\alpha_i]/n) \cZ_i)) 
\nonumber \\ 
& \os{\Psi}{\lra} 
\pi_*(\cE \otimes \cO_{\ol{\cX}}(\sum_i (q[n\alpha_i]/n) \cZ_i)) 
\nonumber \\ 
& \lra 
\pi_*(\cE \otimes \cO_{\ol{\cX}}(\sum_i ([nq\alpha_i]/n) \cZ_i)) 
= \ba(\cE)_{q\alpha}. \nonumber 
\end{align}
When $(\cE,\Psi)$ is an object in 
$\FLatt((\ol{\cX},\cZ)^{1/n})$, we can define the 
morphisms 
$\Psi_{\alpha}: (F^*\ba(\cE)_{\alpha})_{\Q} \lra (\ba(\cE)_{q\alpha})_{\Q}$ 
in the same way. \par 
Let $\cU$, $t_i$ be as above. 
We prove that the map $\Psi_{\alpha} |_{\cU}$ is injective and 
the cokernel of it is killed by some power of $t^{\1} := \prod_{i=1}^r t_i$. 
To see this, it suffices to prove the same property for the 
first arrow in \eqref{f-1}, and we are reduced to showing 
the same property for the map 
$$ 
F^* \pi^{(n)}_* 
(\cO_{\cU^{(n)}}(\sum_i [n\alpha_i] \cZ_i^{(n)})) \lra 
\pi^{(n)}_* F^*(\cO_{\cU^{(n)}}(\sum_i [n\alpha_i] \cZ_i^{(n)})). 
$$ 
Let us put $A := 
R[s_i]_{1 \leq i \leq r}/(s_i^n - t_i)_{1 \leq i \leq r}$, 
$c := \prod_{i=1}^r s_i^{-[n\alpha_i]} \in {\rm Frac}\,A$. 
Then the above map is rewritten as 
\begin{equation}\label{ring}
R \otimes_{F^*,R} cA \lra F^*(c)A; \,\,\,\, r \otimes x \mapsto rF^*(x), 
\end{equation}
where $F^*: R \lra R, F^*:A \lra A$ are the homomorphisms induced by 
$F$ on $(\cU, \cZ \cap \cU)$ and $(\cU^{(n)},\cup_{i=1}^r\cZ^{(n)}_i)$. 
Let us first prove the injectivity of the map \eqref{ring}. 
By assumption on $F$, we can write 
$F^*(s_i) = s_i^qu_i$ for some $u_i \in 1 + \fm_K A \,(1 \leq i \leq r)$. 
$1 \otimes cs^m$ for $m = (m_i)_{1 \leq i \leq r} 
\in \{0,...,n-1\}^r$ forms a basis 
of $R \otimes_{F^*,R} cA$ as $R$-module and they are sent to 
$F^*(c)s^{qm}u^m$. Noting the fact that $qm$'s are mutually different 
modulo $n$ and the fact $u^m \in 1 + \fm_KA$, we see that $F^*(c)s^{qm}u^m$'s 
are linearly independent over $R$. Hence the map \eqref{ring} is 
injective. Prove now that the cokernel of \eqref{ring} is killed by 
a power of $t^{\1}$. For $j = (j_i) \in \{0,...,q-1\}^r$, 
choose $N(j) = (N(j)_i)_i, M(j) = (M(j)_i)_i \in \N^r$ such that 
$j + nN(j) = qM(j)$. Then $\{F^*(c)s^ju^{M(j)}\}_j$ generates 
$F^*(c)A$ over $RF^*(A)$, where $RF^*(A)$ denotes the $R$-subalgebra of $A$ 
generated by $F^*(A)$. So it suffices to prove that the image of each 
$F^*(c)s^ju^{M(j)}$ in $\Coker \text{\eqref{ring}}$ 
is killed by a power of $t^{\1}$. Noting that 
$t^{N(j)}F^*(c)s^ju^{M(j)} = F^*(c)(s^{qM(j)}u^{M(j)}) = F^*(cs^{M(j)})$ is 
in the image of \eqref{ring}, we see that the image of each 
$F^*(c)s^ju^{M(j)}$ in $\Coker \text{\eqref{ring}}$ 
is killed by $(t^{\1})^{\max_{j,i} N(j)_i}$ and we have the 
desired property. \par 
Now we prove the claim 
that $\varinjlim_{\alpha} \Psi_{\alpha}$ is an isomorphism 
as morphism between ind-objects. Here we will work only in the case 
$\cE \in \varinjlim_{(n,p)=1} \FLatt((\ol{\cX},\cZ)^{1/n})^{\c}$, 
because the proof in the case 
$\cE \in \varinjlim_{(n,p)=1} \FLatt((\ol{\cX},\cZ)^{1/n})$ can be done 
exactly in the same way. 
To prove the claim, it suffices to define the 
morphism $\Psi': \varinjlim_{\alpha} \ba(\cE)_{\alpha} \lra 
\varinjlim_{\alpha} F^*
\ba(\cE)_{\alpha}$ which is inverse to 
$\varinjlim_{\alpha} \Psi_{\alpha}$. By the injectivity of 
$\Psi_{\alpha}$'s proven above, it suffices to work locally. 
Let us take $x \in \ba (\cE)_{\alpha}$ with 
$\alpha = (\alpha_i)_i$ with $\alpha_i \geq 0$. Then there exists some 
$N \in \N^r$ (which depends only on $\alpha$) 
such that $(t^{\1})^Nx$, regarded as an element of $\ba (\cE)_{q\alpha}$, 
is in the image of $\Psi_{\alpha}$. 
If we take $M,L \in \N$ such that $qM - N = L$, we see that 
the element $u(t^{\1})^{-L}x$ in $\ba(\cE)_{q(\alpha + M\1)}$ for some 
$u \in 1 + \fm_K \cO_{\ol{\cX}}$ can be written as $\Psi_{\alpha + M\1}(y)$ 
for some $y \in F^*\ba(\cE)_{\alpha + M\1}$. Then we can define $\Psi'$ by 
$\Psi'(x) := u^{-1}t^L \otimes y \in F^*\ba(\cE)_{\alpha + M\1} 
\hra \varinjlim_{\alpha}F^*\ba (\cE)_{\alpha}$. 
It is easy to check that this $\Psi'$ is 
the inverse of $\varinjlim_{\alpha} \Psi_{\alpha}$. Therefore 
$\varinjlim_{\alpha} \Psi_{\alpha}$ is an isomorphism 
as morphism between ind-objects, as desired. 
So 
$(\ba(\cE),(\Psi_{\alpha})_{\alpha})$ defines an object 
in $\PFLatt(\ol{\cX},\cZ)^{(\c)}$ and hence we have defined the functors 
\eqref{flatt+}, \eqref{flatt+--}. \par 
Next we define the functors 
\begin{equation}\label{flatt+-}
\PFLatt(\ol{\cX},\cZ)^{\c} \lra 
\varinjlim_{(n,p)=1} \FLatt((\ol{\cX},\cZ)^{1/n})^{\c}, 
\end{equation}
\begin{equation}\label{flatt+---}
\PFLatt(\ol{\cX},\cZ) \lra 
\varinjlim_{(n,p)=1} \FLatt((\ol{\cX},\cZ)^{1/n}), 
\end{equation}
of the converse direction. Since the construction is the same, 
we explain only the construction of \eqref{flatt+-}. 
Let $(\cE := (\cE_{\alpha})_{\alpha}, 
(\Psi_{\alpha})_{\alpha})$ be
 an object in $\PFLatt(\ol{\cX},\cZ)^{\c}$. Then the maps 
\begin{align}
F^*\cO(\sum_i -a_i\cZ_i) \otimes F^*\pi^*\cE_{b} & \lra 
\cO(\sum_i -qa_i\cZ_i) \otimes \pi^*F^*\cE_{b} \label{coendd} \\ 
& \lra 
\cO(\sum_i -qa_i\cZ_i) \otimes \pi^*\cE_{qb} \nonumber 
\end{align} 
induced by $\Psi_b$ induces, by taking coend, the morphism 
$\Psi: F^*\bb(\cE) \lra \bb(\cE)$. \eqref{injinj} induces the 
commutative diagram 
\begin{equation*}
\begin{CD}
F^* \bb(\cE) @>{F^*\text{\eqref{injinj}}}>> 
F^*\pi^*\cE_1 \\ 
@V{\Psi}VV @V{\Psi_1}VV \\ 
\bb(\cE) @>>> \pi^*\cE_q 
\end{CD}
\end{equation*}
(where the lower horizontal arrow is the composite 
$\bb(\cE) \os{\text{\eqref{injinj}}}{\lra} \pi^*\cE_1 
\os{\subset}{\lra} \pi^*\cE_q$) and since the horizontal arrows and 
$\Psi_1$ are injective, $\Psi$ is also injective. 
Let us prove the surjectivity of $\Psi$. To do so, we may work locally.  
Let us take 
any $x_0 \in \bb (\cE)$. When $\cE$ has the form $\cO(\sum_i\alpha_i\cZ_i)$, 
$x_0$ is a sum of elements of the form $f_c(h \otimes \pi^*x)$ for some 
fixed 
$c \in (\frac{1}{n}\Z/\Z)^r$ since we can take $f_c$ to be an isomorphism, 
and since we have $f_{c}(h \otimes \pi^*x) = f_{c+m}(ht^m \otimes 
\pi^*t^{-m}x)$ for $m \in \Z^r$, we can change $c$ in order that $c$ has 
the form $qa$. In general case, we see from this observation that 
$x_0$ is a sum of elements of the form 
$f_{qa}(h \otimes \pi^*x)$ ($a$ also varies this time) etale locally. 
So, to prove the surjectivity, we may assume that 
$x_0 = f_{qa}(h \otimes \pi^*x)$. 
Then there exists some 
$N \in \N^r$ 
such that $(t^{\1})^Nx$ is in the image of $\Psi_{a}$. 
If we take $M,L \in \N$ such that $pM - N = L$, we see that 
the element $u(t^{\1})^{-L}x$ in $\cE_{q(a+M\1)}$ for some 
$u \in 1 + \fm_K\cO_{\ol{\cX}}$ can be written as $\Psi_{a + M\1}(y)$ 
for some $y \in F^*\cE_{a + M\1}$. Then $x_0$ is equal to the 
image of $u^{-1}(t^{\1})^Lh \otimes \pi^*(u(t^{\1})^{-L}x) \in 
 \cO(\sum_i -q(a_i+M)\cZ_i) \otimes \pi^*\cE_{q(a+M\1)}$ 
 which is in the image of $\Psi$, by definition of it. 
So $\Psi$ is an isomorphism and so $(\bb(\cE), \Psi)$ defines an 
object in $\FLatt((\ol{\cX},\cZ)^{1/n})^{\c}$. \par 
We see that the functor \eqref{flatt+-} 
(resp. \eqref{flatt+---}) is the inverse of the 
functor \eqref{flatt+} (resp. \eqref{flatt+--}) 
using the fact that 
the (parabolic unit-root) $F$-lattice structure 
is determined by the underlying structure on (parabolic) vector 
bundle structure and the morphism $\Psi$ on $\cX = \ol{\cX} \setminus \cZ$. 
So we have shown that 
the functors \eqref{flatt+}, \eqref{flatt+--} are equivalent 
and hence we are done. 
\end{proof}

\begin{remark}
By Corollary \ref{repparcor}, Theorem \ref{flatttt} and 
Theorem \ref{flattthm}, we have the equivalence 
\begin{align}
\PFLatt(\ol{\cX},\cZ)^{\c}_{\Q} 
& \os{=}{\lra} 
\varinjlim_{(n,p)=1} 
\FLatt((\ol{\cX},\cZ)^{1/n})^{\c}_{\Q} \label{boxstar} \\ 
& \os{=}{\lra} 
\varinjlim_{(n,p)=1} 
\FIsoc((\ol{X},Z)^{1/n})^{\circ} \nonumber \\ 
& \os{=}{\lra} 
\PFIsoc(\ol{X},Z)^{\circ}_{\0\ss}, \nonumber 
\end{align}
which is a parabolic version of \eqref{crewisog}. Therefore, 
an object $((\cE_{\alpha})_{\alpha},(\Psi_{\alpha})_{\alpha}
)_{\Q}$ in the category  
$\PFLatt(\ol{\cX},\cZ)^{\c}_{\Q}$ 
(here, for an object $A$ in an additive category $\cC$, 
$A_{\Q}$ denotes the object $A$ regarded as an object in 
$\cC_{\Q}$) is sent to an object in 
$\PFIsoc(\ol{X},Z)^{\circ}_{\0\ss}$, which induces an inductive system of 
log-$\nabla$-modules $(\cE'_{\alpha},\nabla'_{\alpha})_{\alpha}$ on 
$(\ol{\cX}_K,\cZ_K)$ endowed with horizontal isomorphism 
$\Psi': \varinjlim_{\alpha} F^*\cE'_{\alpha} \lra 
\varinjlim_{\alpha} \cE'_{\alpha}$ of ind-objects. 
Note that the $\Q$-linearization of the category of 
coherent sheaves on $\ol{\cX}$ is equivalent to 
the category of coherent sheaves on $\ol{\cX}_K$. 
We prove in this remark that, in these equivalent categories, 
$(\cE_{\alpha})_{\alpha,\Q}$ is equal to 
$(\cE'_{\alpha})_{\alpha}$ as inductive systems and that 
$\varinjlim_{\alpha} \Psi_{\alpha,\Q} = \Psi'$. \par 
Suppose first that we have shown the equality 
$(\cE_{\alpha})_{\alpha,\Q} = (\cE'_{\alpha})_{\alpha}$. 
Then, for any $\alpha \in \Z_{(p)}^r$, there exists some 
$\beta \in \Z_{(p)}^r$ such that both $\Psi_{\alpha,\Q}, \Psi'$ define 
morphism of the form $F^*\cE_{\alpha,\Q} \lra \cE_{\beta,\Q}$. 
Then we have $\Psi_{\alpha,\Q}|_{\cX} = \Psi' |_{\cX}$ because 
they are equal when $Z$ is empty (which follows from the definition of 
\eqref{crewisog}), and this equality implies the equality 
 $\Psi_{\alpha,\Q} = \Psi'$ as morphisms 
$F^*\cE_{\alpha,\Q} \lra \cE_{\beta,\Q}$. Hence we have 
$\varinjlim_{\alpha} \Psi_{\alpha,\Q} = \Psi'$. So, to prove the claims 
in the previous paragraph, it suffices to prove the equality 
$(\cE_{\alpha})_{\alpha,\Q} = (\cE'_{\alpha})_{\alpha}$ as inductive 
systems. \par 
Take a chart $(\ol{X}_0,\{t_i\}_{i=1}^r)$ for $(\ol{X},Z)$ in the sense 
of Section 2.3 which is smooth over $k[\mu_n]$ 
and let $(\ol{X}_{\b}, M_{\ol{X}_{\b}})$, 
 $(\ol{X}_{\b\b}, M_{\ol{X}_{\b\b}})$
be the simplicial semi-resolution, 
the bisimplicial resolution of $(\ol{X},Z)^{1/n}$, respectively. 
Let $(\ol{\cX}_{\c,\b}, M_{\ol{\cX}_{\c,\b}})$, 
 $(\ol{\cX}_{\c,\b\b}, M_{\ol{\cX}_{\c,\b\b}})$ be the log etale lifts of 
$(\ol{X}_{\b}, M_{\ol{X}_{\b}})$, 
 $(\ol{X}_{\b\b}, M_{\ol{X}_{\b\b}})$ over $(\ol{\cX}_{\c},\cZ_{\c})$
and let us put $(\ol{\cX}_{\b}, M_{\ol{\cX}_{\b}}) \allowbreak := 
(\ol{\cX}_{\c,\b}, M_{\ol{\cX}_{\c,\b}}) \otimes_{W(k)} O_K$, 
$(\ol{\cX}_{\b\b}, M_{\ol{\cX}_{\b\b}}) := 
(\ol{\cX}_{\c,\b\b}, M_{\ol{\cX}_{\c,\b\b}}) \otimes_{W(k)} O_K$.  
Then $M_{\ol{\cX}_{\b\b}}$ is associated to some $(2,2)$-truncated 
bisimplicial relative simple normal crossing divisor 
$\cZ_{\b\b} = \bigcup_{i=1}^r \cZ_{\b\b,i}$ compatible with 
transition maps. (Here $\cZ_{\b\b,i}$ is characterized as 
the smooth subdivisor of $\cZ_{\b\b}$ 
which is homeomorphic to the inverse image of $\cZ_i$.) 
Let us put 
$\cX_{\b} := \cX \times_{\ol{\cX}} \ol{\cX}_{\b}, 
\cX_{\b\b} := \cX \times_{\ol{\cX}} \ol{\cX}_{\b\b}$. 
Let us denote by $\NM^{\d}_{\ol{\cX}_K}$ (resp. 
$\NM^{\d}_{\ol{\cX}_{\b,K}}$, $\NM^{\d}_{\ol{\cX}_{\b\b, K}}$) 
the category of locally free $j^{\dagger}\cO_{\ol{\cX}_K}$-modules 
(resp. $j^{\dagger}\cO_{\ol{\cX}_{\b,K}}$-modules, 
$j^{\dagger}\cO_{\ol{\cX}_{\b\b,K}}$-modules) of finite rank 
endowed with integrable connections, where $j$ denotes the morphism 
$\cX_K \hra \ol{\cX}_K$ (resp. $\cX_{\b,K} \hra \ol{\cX}_{\b,K}$, 
$\cX_{\b\b,K} \hra \ol{\cX}_{\b\b,K}$). Let us denote the 
restriction of $(\cE_{\alpha})_{\alpha}$ to $\ol{\cX}_{\b}$ by 
$(\cE_{\alpha,\b})_{\alpha}$ and the restriction of 
$(\cE'_{\alpha}, \nabla'_{\alpha})_{\alpha}$ to $\ol{\cX}_{\b,K}$ by 
$(\cE'_{\alpha,\b}, \nabla'_{\alpha,\b})_{\alpha}$. 
By rigid analytic faithfully flat descent, 
it suffices to prove the following: Via the equivalence between 
the $\Q$-linearization of the category of coherent sheaves on 
$\ol{\cX}_{\b}$ and 
the category of coherent sheaves on $\ol{\cX}_{\b,K}$, 
we have the equality 
$(\cE_{\alpha,\b})_{\alpha,\Q} = (\cE'_{\alpha,\b})_{\alpha}$. \par 
Suppose that 
$((\cE_{\alpha})_{\alpha},\Psi)$ is sent to 
$(\cE'',\Psi'')_{\Q} \in \FLatt((\ol{\cX},\cZ)^{1/n})^{\c}_{\Q}$ by 
the first functor of \eqref{boxstar}. 
Let us consider the diagram 
\begin{align}
\FLatt((\ol{\cX},\cZ)^{1/n})^{\c}_{\Q} & \os{=}{\lra} 
\FIsoc((\ol{X},Z)^{1/n})^{\circ} \lra 
\PFIsoc(\ol{X},Z)_{\0\ss} \label{ichi} \\ 
& \os{\text{induced by 
\eqref{b4}}}{\lra} \Isocd(X,\ol{X}) \os{\subset}{\lra} 
\NM^{\d}_{\cX_K} \lra \NM^{\d}_{\cX_{\b,K}}, \nonumber 
\end{align}
where the first two functors are induced by the second and the third ones 
in \eqref{boxstar}. Then, by definition, $(\cE'', \Psi'')$ is sent by 
\eqref{ichi} to 
$j^{\d}(\cE'_{\alpha,\b},\nabla'_{\alpha,\b})$. 
By definition of the functors in \eqref{ichi}, it is rewritten as 
\begin{align}
& \FLatt((\ol{\cX},\cZ)^{1/n})^{\c}_{\Q} \label{ni} \\ 
\os{=}{\lra}\,\, & 
\FLatt(\ol{\cX}_{\b} \times_{\ol{\cX}} (\ol{\cX},\cZ)^{1/n})^{\c} 
\os{=}{\lra}
\FLatt(\ol{\cX}_{\b\b})_{\Q}^{\c} 
\os{\text{\eqref{crewisog} for $\ol{\cX}_{\b\b}$}}{\lra} 
\FIsoc(\ol{X}_{\b\b})^{\c} \nonumber \\ 
\os{\subset}{\lra}\,\,& \Isoc(\ol{X}_{\b\b}) 
\lra \Isocd(X_{\b\b}, \ol{X}_{\b\b}) 
\os{\text{\eqref{a}}^{-1}}{\lra} 
\Isocd(X_{\b}, \ol{X}_{\b}) \os{\subset}{\lra} 
\NM^{\d}_{\cX_{\b,K}}. \nonumber
\end{align}
Furthermore, by the definition of the inverse of \eqref{a} 
given in \eqref{aa}, the composite 
$\Isoc(\ol{X}_{\b\b}) \lra \NM^{\d}_{\cX_{\b,K}}$ 
in \eqref{ni} is rewritten as 
$$ 
\Isoc(\ol{X}_{\b\b}) \os{\subset}{\lra} 
\NM_{\ol{\cX}_{\b\b,K}} \os{j^{\d}}{\lra} 
\NM^{\d}_{\ol{\cX}_{\b\b,K}} \os{g}{\lra} 
\NM^{\d}_{\ol{\cX}_{\b,K}}, 
$$ 
where the last functor $g$ is defined as follows: An object $E_{\b\b}$ in 
$\NM^{\d}_{\ol{\cX}_{\b\b,K}}$, regarded as an object $E_{\b 0}$ in 
$\NM^{\d}_{\ol{\cX}_{\b 0,K}}$ endowed with an equivariant 
$\mu_n^{r(\b + 1)}$-action, is sent to 
$(\pi_*E_{\b 0})^{\mu_n^{r(\b + 1)}}$, where $\pi$ is the morphism 
$\ol{\cX}_{\b 0} \lra \ol{\cX}_{\b}$. Then, for any $\alpha \in 
\Z_{(p)}^r$, this is rewritten as 
\begin{align}
\Isoc(\ol{X}_{\b\b}) & \os{\subset}{\lra} 
\NM_{\ol{\cX}_{\b\b,K}} \os{f_{\alpha}}{\lra} 
\LNM_{\ol{\cX}_{\b\b,K}} \label{san} \\ & \os{g'}{\lra} 
\LNM_{\ol{\cX}_{\b,K}} 
\os{j^{\dagger}}{\lra} \NM^{\d}_{\ol{\cX}_{\b,K}}, \nonumber 
\end{align}
where $f_{\alpha}$ is the functor $- \otimes_{\cO_{\ol{\cX}_{\b\b}}} 
\cO_{\ol{\cX}_{\b\b}}(\sum_{i=1}^r[n\alpha_i]\cZ_{\b\b,i})$ 
(endowed with canonical extension or restriction of the connection) 
and $g'$ is defined as follows, as in the case of $g$: 
An object $E_{\b\b}$ in 
$\LNM_{\ol{\cX}_{\b\b,K}}$, regarded as an object $E_{\b 0}$ in 
$\LNM_{\ol{\cX}_{\b 0,K}}$ endowed with an equivariant 
$\mu_n^{r(\b + 1)}$-action, is sent to 
$(\pi_*E_{\b 0})^{\mu_n^{r(\b + 1)}}$. \par 
Now let us assume that the object 
$(\cE'',\Psi'')_{\Q} \in \FLatt((\ol{\cX},\cZ)^{1/n})^{\c}_{\Q}$ 
is sent by the first two functors in \eqref{ni} as 
$(\cE'',\Psi'')_{\Q} \mapsto 
(\cE''_{\b},\Psi''_{\b})_{\Q} \mapsto 
(\cE''_{\b\b},\Psi''_{\b\b})_{\Q}$. First, by definition of 
$\cE''$, we have $\cE_{\alpha} = \ba (\cE'')_{\alpha}$ and 
the functor $\ba$ is compatible with etale localization. So we have 
\begin{equation}\label{char_a}
\cE_{\alpha,\b} = \ba (\cE''_{\b})_{\alpha} 
= \pi_*(\cE''_{\b0} \otimes \cO_{\ol{\cX}_{\b0}}
(\sum_i [n\alpha_i] \cZ_{\b0,i}))^{\mu_n^{r(\b+1)}}. 
\end{equation}
The image of $(\cE''_{\b\b}, \Psi''_{\b\b})_{\Q}$ by the functor 
$$ 
\FLatt(\ol{\cX}_{\b\b}) \lra \FIsoc(\ol{X}_{\b\b}) 
\lra \Isoc(\ol{X}_{\b\b}) \lra \NM_{\ol{\cX}_{\b\b,K}}
$$ 
(where the first two functors are as in \eqref{ni} and the last 
functor is as in \eqref{san}) has the form 
$(\cE''_{\b\b,\Q},\nabla_{\b\b})$, by definition of Crew's functor 
\eqref{crewisog} which is given in \cite{crew}. 
If we apply the functor $g' \circ f_{\alpha}$, we obtain an 
object in $\LNM_{\ol{\cX}_{\b,K}}$ of the form 
$((\cE_{\alpha,\b})_{\Q}, \nabla_{\alpha,\b})$ by definition of 
the functors $f_{\alpha}, g'$ and \eqref{char_a}. Hence we have 
\begin{equation}\label{char_i}
j^{\d}((\cE_{\alpha,\b})_{\Q}, \nabla_{\alpha,\b}) = 
j^{\d}(\cE'_{\alpha,\b},\nabla'_{\alpha,\b}). 
\end{equation}
By \eqref{char_i}, we see that 
$j^{\d}((\cE_{\alpha,\b})_{\Q}, \nabla_{\alpha,\b})$ comes from an 
overconvergent isocrystal. So 
$((\cE_{\alpha,\b})_{\Q}, \nabla_{\alpha,\b})$ defines an 
object in $\Isocl(\ol{X}_{\b} \times_{\ol{X}} (\ol{X},Z))$. 
Since there exists the canonical morphism of functors 
$f_{\alpha} \lra f_{\beta}$ for $\alpha = (\alpha_i)_i, 
\beta = (\beta_i)_i$ with $\alpha_i \leq \beta_i \, (\forall i)$, 
we see that $((\cE_{\alpha,\b})_{\Q}, \nabla_{\alpha,\b})$ for 
$\alpha \in \Z_{(p)}$ form an inductive system, and it is easy to 
see from the definition of $f_{\alpha}$'s that it is a 
parabolic log convergent isocrystal. Moreover, 
Since 
$(\cE''_{\b\b,\Q},\nabla_{\b\b})$ has exponents in $\0$ with 
semisimple residues, $f_{\alpha}(\cE''_{\b\b,\Q},\nabla_{\b\b})$ 
has exponents in 
$\prod_{i=1}^r \{- [n\alpha_i]\}$ with semisimple 
residues, and then we see that 
$((\cE_{\alpha,\b})_{\Q}, \nabla_{\alpha,\b})$ has exponents 
in 
$\prod_{i=1}^r \{- \frac{[n\alpha_i]}{n} + \frac{j}{n} 
\,|\, 0 \leq j \leq n-1\} \subseteq 
\prod_{i=1}^r ([-\alpha_i,-\alpha_i+1) \cap \Z_{(p)}) 
$ with semisimple residues. 
Hence $((\cE_{\alpha,\b})_{\Q}, \nabla_{\alpha,\b})_{\alpha}$ 
forms an object in $\PIsoc(\ol{X}_{\b} \times_{\ol{X}} (\ol{X},Z))_{\0\ss}$. 
Since so is $(\cE'_{\alpha,\b},\nabla'_{\alpha,\b})$ and 
we have the isomorphism \eqref{char_i}, they are isomorphic by the equivalence 
\eqref{step2}. So we have the desired isomorphism 
$(\cE_{\alpha,\b})_{\alpha,\Q} = 
(\cE'_{\alpha,\b})_{\alpha}$. 
\end{remark}

\section{Unit-rootness and generic semistability}

In this section, we prove several propositions which give 
interpretations of unit-rootness in terms of certain semistability 
called generic semistability, and prove the equivalences 
\eqref{sseq1}--\eqref{sseq4+}. 
The proof uses the results 
in \cite{katzslope}, \cite{crewsp} and \cite{crew} and the constructions 
up to the previous section. \par 
First let us give a review on some results proven in 
\cite{katzslope}, \cite{crewsp} and \cite{crew}. (We also give a 
slight generalization of them which we need in this paper.) 
In this section, let $\pi$ be a fixed uniformizer of $O_K$. 
For a perfect 
field $l$ containing $k$, we put $O(l) := W(l) \otimes_{W(k)} O_K$, 
$K(l) := \Frac O(l)$. We denote the endomorphism $F_{\circ} 
\otimes \sigma$ (where $F_{\circ}$ 
is the endomorphism on $W(k)$ lifting the $q$-th power map on $l$) 
on $O(l)$ by $F$, and denote the induced endomorphism on $K(l)$ by 
the same letter. 
An $F$-isocrystal on $l$ is defined to be 
a pair $(E,\Psi)$ consisting of a finite dimensional $K(l)$-vector 
space $E$ and a $F$-linear endomorphism $\Psi$. Then, by 
Dieudonn\'e-Manin classification theorem, any $F$-isocrystal 
$(E,\Psi)$ has the decomposition $(E,\Psi) := \oplus_{\lambda \in \Q} 
(E_{\lambda}, \Psi_{\lambda})$ as $F$-isocrystals such that, for 
each $\lambda = a/b$, $E \otimes_{K(l)} K(l^{\rm alg})$ has 
a basis consising of elements $e$ with $\Psi^b(e) = \pi^a$. 
The Newton polygon of $(E,\Psi)$ is the convex polygon 
defined as the convex closure of the points 
$(\sum_{\lambda \leq \nu} \dim E_{\lambda}, \sum_{\lambda \leq \nu} 
\lambda \dim E_{\lambda}) \,(\nu \in \Q)$ in the plane $\R^2$. 
(So it has the endpoint 
$(\dim E, \sum_{\lambda} \lambda \dim E_{\lambda})$). 
For an $F$-isocrystal $(E,\Psi) \not= 0$, 
we put $\mu(E) := 
\dfrac{\sum_{\lambda} \lambda \dim E_{\lambda}}{\dim E}$. 
We say that the Newton polygon of $(E,\Psi)$ has pure slope $\lambda$ 
when it is the straight line connecting $(0,0)$ and 
$(\dim E, \lambda \dim E)$. \par 
For a fine log scheme $(X,M_X)$ separated of finite type over $k$, 
an object $(\cE,\Psi) \in \FIsoc(X,M_X)$ and a perfect valued 
point $x = \Spec l \lra X$, we can pull back $(\cE,\Psi)$ by $x$ to an 
object of the category of convergent $F$-isocrystal on 
$(x,M_X|_x)$ over $(\Spf O(l), W(M_X|_x)|_{\Spec O(l)})$ 
(where $W(M_X|_x)$ denotes the canonical lift of the log structure 
$M_X|_x$ on $x = \Spec l$ into $\Spf W(l)$), which is canonically 
equivalent to the category of $F$-isocrystals on $l$. We denote 
this object by $x^*(\cE,\Phi)$ or $(x^*\cE,x^*\Psi)$. 
We call the Newton polygon of 
$x^*(\cE,\Psi)$ the Newton polygon of $(\cE,\Psi)$ at $x$ and 
we put $\mu_x(\cE) := \mu(x^*\cE)$. \par 
Next, let $X$ be a smooth scheme separated of finite type over $k$ and 
assume that it is liftable to a $p$-adic formal scheme 
$\cX_{\circ}$ separated smooth of finite type over $\Spf W(k)$ 
which is endowed with a lift $F_{\circ}: \cX_{\circ} \lra \cX_{\circ}$ 
of the $q$-th power Frobenius endomorphism on $X$ compatible with 
$(\sigma|_{W(k)})^*: \Spf W(k) \lra \Spf W(k)$. Let us put 
$\cX := \cX_{\circ} \otimes O_K$, $F := F_{\circ} \otimes \sigma^*: 
\cX \lra \cX$. Then the category $\FLatt(\cX)$ of $F$-lattices on 
$\cX$ was defined in the beginning of the previous section. 
Also, we define the notion of 
$F$-vector bundle on $\cX_K$ as a pair $(\cE,\Psi)$ of a locally free 
$\cO_{\cX_K}$-module $\cE$ endowed with an isomorphism 
$\Psi: F^*\cE \os{=}{\lra} \cE$. (It is called a rigid $F$-bundle 
in \cite{weng}.) We denote the category of 
$F$-vector bundles on $\cX_K$ by $\FVect(\cX_K)$. 
Note that, via the equivalence 
$\Coh(\cX)_{\Q} \os{=}{\lra} \Coh(\cX_K)$, we have the canonical 
inclusion $\FLatt(\cX)_{\Q} \hra \FVect(\cX_K)$. 
(We do not know whether they are equal or not.) 
When we are given an object $(\cE,\Psi) \in \FVect(\cX_K)$ and 
a perfect valued point $x = \Spec l \lra X$, 
we have the canonical lift $O(x) := \Spf O(l) \lra \cX$ and 
so we have the map $\Spm K(x) \lra \cX_K$. If we pull back $(\cE,\Psi)$ 
by this map, we obtain an $F$-isocrystal on $l$, which we denote by 
$x^*(\cE,\Psi)$ or $(x^*\cE,x^*\Psi)$. Also in this case, 
we call the Newton polygon of 
$x^*(\cE,\Psi)$ the Newton polygon of $(\cE,\Psi)$ at $x$ and 
we put $\mu_x(\cE) := \mu(x^*\cE)$. \par 
Assume moreover that there exists a fine log structure 
$M_X$ on $X$, $M_{\cX}$ on $\cX$ on 
$\cX$ with $M_{\cX}|_{X} = M_X$ and 
that $F$ induces the endomorphism 
$F: (\cX,M_{\cX}) \lra (\cX,M_{\cX})$ lifting the $q$-th 
power Frobenius on $(X,M_X)$. Then, by definition, 
we have the commutative diagram 
\begin{equation*}
\begin{CD}
\FIsoc(X,M_X) @>>> \FVect(\cX_K) \\ 
@V{x^*}VV @V{x^*}VV \\ 
\text{($F$-isocrystals on $l$)} @= 
\text{($F$-isocrystals on $l$)},
\end{CD}
\end{equation*}
where the top horizontal arrow is the functor 
of realization at $(\cX,M_{\cX})$. It is easy to see that, 
for a fixed object $(\cE,\Psi)$ in $\FIsoc(X)$ or $\FVect(\cX_K)$, 
the Newton polygon of $(\cE,\Psi)$ at $x$ and $\mu_x(\cE)$ 
for a perfect field valued point $x \lra X$ 
depends only on the image of $x$ in $X$. Hence we can speak of the 
Newton polygon of $(\cE,\Psi)$ at $x$ and the value $\mu_x(\cE)$ 
at any point $x$ of $X$. \par 
Now let us recall several results on $F$-lattices which are due to 
Grothendieck, Katz and Crew. The first one is Grothendieck's 
specialization theorem (\cite[2.3]{katzslope}, \cite[1.6]{crewsp}: 

\begin{proposition}\label{grothsp}
let $X$ be a smooth scheme separated of finite type over $k$ and 
assume that it is liftable to a $p$-adic formal scheme 
$\cX_{\circ}$ separated of finite type over $\Spf W(k)$ 
which is endowed with a lift $F_{\circ}: \cX_{\circ} \lra \cX_{\circ}$ 
of the $q$-th power Frobenius endomorphism on $X$ compatible with 
$(\sigma|_{W(k)})^*: \Spf W(k) \lra \Spf W(k)$. Let us put 
$\cX := \cX_{\circ} \otimes O_K$, $F := F_{\circ} \otimes \sigma^*: 
\cX \lra \cX$, let $(\cE,\Psi)$ be an object in $\FLatt(\cX)_{\Q}$ 
and let $P$ be a convex polygon. Then the set of points of 
$X$ at which the Newton polygon of $(\cE,\Psi)$ lies on or above $P$ 
is Zariski closed. 
\end{proposition}

The next one is the constance of $\mu_x(\cE)$ \cite[1.7]{crewsp}: 

\begin{proposition}\label{endpoint}
Let $X, \cX, F$ be as above and suppose that $X$ is connected. 
Then, for a non-zero object $(\cE,\Psi)$ in $\FLatt(\cX)_{\Q}$, 
$\mu_x(\cE)$ is the same for all points $x$ of $X$. 
$($In this case, we put $\mu(\cE) := \mu_x(\cE).)$
\end{proposition}

The next one, which is a weaker form of \cite[2.6.2]{katzslope}, 
is a `generic' Newton filtration theorem: 

\begin{proposition}\label{newfil}
Let $X, \cX, F$ be as above and let $(\cE,\Psi)$ be an object 
in $\FLatt(\cX)_{\Q}$ such that its Newton polygon at $x \in X$ 
is not a straight line $($has a break point$)$ 
and independent of $x \in X$. Then, on an dense open formal 
subscheme $\cU$ of $\cX$, $(\cE,\Psi)$ admits a 
non-trivial saturated 
subobject $(\cE',\Psi')$ with $\mu(\cE') < \mu(\cE)$. 
$($Where a subobject $(\cE',\Psi')$ of  $(\cE,\Psi)$ is called 
saturated if the quotient $(\cE/\cE', \ol{\Psi})$ $(\ol{\Psi}$ is 
the morphism induced by $\Psi)$ is again an object in 
$\FLatt(\cX)_{\Q}.)$ 
When $\cE_{\Q}$ is endowed with an integrable connection for which 
$\Psi$ is horizontal, $\nabla |_{\cE'_{\Q}}$ defines 
an integrable connection on $\cE'_{\Q}$ for which 
$\Psi'$ is horizontal. 
\end{proposition}

The next proposition, which is shown in \cite[2.5.1]{crew} 
and \cite[2.3]{crewsp}, fills a possible gap in the inclusion 
$\FLatt(\cX)_{\Q} \hra \FVect(\cX_K)$ generically: 

\begin{proposition}\label{lat}
Let $X, \cX, F$ be as above and assume that $\cX$ is affine. 
Then, for any object $(\cE,\Psi)$ in $\FVect(\cX_K)$, 
there exists an dense open formal subscheme $\cU$ of $\cX$ such that 
$(\cE,\Psi)|_{\cU} \in \FVect(\cX_K)$ is contained in 
$\FLatt(\cU)_{\Q}$. Moreover, in the case where $\dim X = 1$, 
we can take $\cU = \cX$. 
\end{proposition}

With more strong condition on the Newton polygons, we have 
the following result, which is proved in \cite[2.5.1--2.6]{crew}: 

\begin{proposition}\label{urlat}
Let $X, \cX, F$ be as above and assume that $\cX$ is affine. 
Let $(\cE,\Psi)$ be an object in $\FVect(\cX_K)$ such that, for 
any point $x \in X$, the Newton polygon of it at $x$ has pure slope 
$0$. Then there exists a unit-root $F$-lattice $(\cE_0,\Psi_0)$ 
on $\cX$ with $(\cE_0,\Psi_0)_{\Q} = (\cE,\Psi)$. 
\end{proposition}

Using the above results, we see the following proposition, which 
is a kind of specialization theorem for log convergent 
$F$-isocrystals. 

\begin{proposition}\label{grothprop}
Let $X \hra \ol{X}$ be an open immersion of connected smooth 
$k$-varieties such that $Z = \ol{X} \setminus X$ is a simple 
normal crossing divisor, and let $(\cE,\Psi)$ be an object in 
$\FIsocl(\ol{X},Z)$. Let $\eta$ be the generic point of $\ol{X}$. 
Then, for any $x \in \ol{X}$, the Newton polygon of 
$(\cE,\Psi)$ at $x$ lies on or above the Newton polygon of 
$(\cE,\Psi)$ at $\eta$ and we have the equality 
$\mu_x(\cE) = \mu_{\eta}(\cE)$. 
$($In particular, $\mu_x(\cE)$ is independent of $x$. So we put 
$\mu(\cE) := \mu_x(\cE)$ in this situation.$)$ 
\end{proposition}

Before the proof, we introduce one terminology. 
For a fine log scheme $(T,M_T)$ separated of finite type over 
$k$, a strong lift of $(T,M_T)$ is a fine log formal scheme 
$(\cT, M_{\cT})$ separated of finite type over $\Spf O_K$ endowed with 
an endomorphism $F: (\cT,M_{\cT}) \lra (\cT,M_{\cT})$ such that 
there exists a lift $(\cT_{\c}, M_{\cT_{\c}})$ of $(T,M_T)$ over 
$\Spf W(k)$ endowed with a $(\sigma|_{W(k)})^*$-linear endomorphism 
$F_{\c}: (\cT_{\c}, M_{\cT_{\c}}) \lra (\cT_{\c}, M_{\cT_{\c}})$ 
lifting the $q$-th power map on $(T,M_T)$ satisfying 
$(\cT,M_{\cT}) = (\cT_{\c},M_{\cT_{\c}}) \otimes_{W(k)} O_K, 
F = F_{\c} \otimes_{(\sigma|_{W(k)})^*} \sigma^*$. 
When $(T,M_T)$ is log smooth over $k$, $(\cT,M_{\cT})$ is called 
a strong smooth lift of $(T,M_T)$ if we can take 
$(\cT_{\c},M_{\cT_{\c}})$ to be log smooth over $\Spf W(k)$. 

\begin{proof}
Let $Z = \bigcup_{i=1}^r Z_i$ be the decomposition of $Z$ into 
irreducible components. Since we may work locally around $x$, 
we may suppose that each $Z_i$ is defined as the zero locus of 
some element $t_i$ in $\ol{X}$ and that 
$x \in \bigcap_{i=1}^r Z_i =:Y_r$. (So $r=r(x)$ is the number of 
irreducible components which contains $x$.) \par 
First we reduce the proof of the proposition to the case $r=r(x)=0$ 
by descending induction on $r$. Let us denote the log strucure on 
$\ol{X}$ associated to $Z$ by $M_{\ol{X}}$ and let us put 
$M_{Y_r} := M_{\ol{X}}|_{Y_r}$. Then $M_{Y_r}$ is equal to the log structure 
associated to zero map $\N^r \lra \cO_{Y_r}$. 
Let $V_x$ be a non-empty 
smooth open subscheme of the closure of $x$ in $Y_r$. 
Then, after shrinking $V_x$ properly, 
it admits a strong smooth lift $\cV_x$, and 
if we define $M_{\cV_x}$ to be the log structure associated to 
the zero map $\N^r \lra \cO_{\cV_x}$, $(\cV_x,M_{\cV_x})$ is a strong 
lift (no more smooth) of $(V_x,M_{Y_r}|_{V_x})$. 
Then $(\cE,\Psi)$ naturally 
induces an object $(\cE_{\cV_x},\Psi_{\cV_x}) \in \FVect(\cV_{x,K})$, 
and by shrinking $\cV_x$, we may assume that 
$(\cE_{\cV_x},\Psi_{\cV_x})$ belongs to $\FLatt(\cV_{x})_{\Q}$, 
by Proposition \ref{lat}. Then, by Propositions \ref{grothsp} and 
\ref{endpoint}, we have the following after shrinking $V_x$ 
further: For any $y \in V_x$, 
$\mu_y(\cE) = \mu_x(\cE)$ and that the Newton polygon of 
$\cE$ at $y$ is the same as the Newton polygon of $\cE$ at $x$. 
So, by replacing $x$ by a closed point of $V_x$, we may assume that 
$x$ is a closed point to prove the proposition. 
Next, let us put 
$Y_{r-1} := \bigcap_{i=1}^{r-1} Z_i$, 
$M_{Y_{r-1}} := M_X|_{Y_{r-1}}$ and let $M'_{Y_{r-1}}$ be the 
log structure on $Y_{r-1}$ associated to $Y_r$. 
Then $M_{Y_{r-1}}$ is equal to the direct sum (in the category 
of fine log structures) 
of $M'_{Y_{r-1}}$ and the 
log structure associated to zero map $\N^{r-1} \lra \cO_{Y_{r-1}}$. 
Let us take an affine smooth 
curve $C$ on $Y_{r-1}$ passing through $x$ which is 
transversal to $Y_r$. Then $(C, M'_{Y_{r-1}}|_C)$, being log smooth, 
admits a strong smooth lift $(\cC, M'_{\cC})$ 
when we shrink $C$ appropriately, 
keeping the condition $x \in C$. 
Then, if we define $M_{\cC}$ to be the direct sum of $M'_{\cC}$ and 
the log structure associated to 
the zero map $\N^{r-1} \lra \cO_{\cC}$, $(\cC,M_{\cC})$ is a strong lift 
(no more smooth) of $(C,M_{Y_{r-1}}|_{C})$. 
Then $(\cE,\Psi)$ naturally 
induces an object $(\cE_{\cC},\Psi_{\cC}) \in \FVect(\cC_{K})$ 
and by  Proposition \ref{lat}, it belongs to $\FLatt(\cC)_{\Q}$. 
Let $y$ be the generic point of $C$. 
Then, by Propositions \ref{grothsp} and 
\ref{endpoint}, we have 
$\mu_y(\cE) = \mu_x(\cE)$ and that the Newton polygon of 
$\cE$ at $x$ lies on or above the Newton polygon of $\cE$ at $y$. Hence
 the proposition for $x$ is reduced to the proposition for $y$. 
Since $y \in Y_{r-1} \setminus Y_r$, we have $r(y) = r(x)-1$. 
So, by descending induction, we can reduce the proof of the theorem 
to the case $r(x)=0$. \par 
In the case $r(x)=0$, the theorem is essentially due to 
Crew \cite[2.1]{crewsp}. Here we give a proof for the convenience of 
the reader, since the statement here and that in \cite[2.1]{crewsp} are 
slightly different. By the first argument in the previous paragraph, 
we may assume that $x$ is a closed point of $X$. 
Since we may work locally around $x$, we may assume that 
$X = \ol{X}$ and that this admits a strong smooth lift 
$\cX$. 
Then $(\cE,\Psi)$ naturally 
induces an object $(\cE_{\cX},\Psi_{\cX}) \in \FVect(\cX_{K})$. 
By Proposition \ref{lat}, we have a dense open formal 
subscheme $\cU \subseteq \cX$ such that 
$(\cE_{\cX},\Psi_{\cX})|_{\cU}$ belongs to $\FLatt(\cU)_{\Q}$. 
Then, by Propositions \ref{grothsp} and 
\ref{endpoint}, we have the following after shrinking $\cU$ 
further: For any $y \in U := \cU \otimes k$, 
$\mu_y(\cE) = \mu_{\eta}(\cE)$ and that the Newton polygon of 
$\cE$ at $y$ lies on or above the Newton polygon of $\cE$ at $\eta$. 
Next let us take an affine smooth curve $C$ on $X$ passing through $x$ 
with $U \cap C \not= \emptyset$ such that $C$ admits a 
strong smooth lift $\cC$. 
Then $(\cE,\Psi)$ naturally 
induces an object $(\cE_{\cC},\Psi_{\cC}) \in \FVect(\cC_{K})$ 
and by Proposition \ref{lat}, it belongs to $\FLatt(\cC)_{\Q}$. 
Let $y$ be the generic point of $C$, which belongs to $U \cap C$. 
Then, by Propositions \ref{grothsp} and 
\ref{endpoint}, we have 
$\mu_y(\cE) = \mu_x(\cE)$ and that the Newton polygon of 
$\cE$ at $x$ is on or above the Newton polygon of $\cE$ at $y$. 
Hence we have shown that 
$\mu_x(\cE) = \mu_{\eta}(\cE)$ and that the Newton polygon of 
$\cE$ at $x$ lies on or above the Newton polygon of $\cE$ at $\eta$, 
as desired. 
\end{proof}

\begin{corollary}\label{grothcor}
Let $(\ol{X},Z)$, $\eta$ be as in Proposition \ref{grothprop} and let 
let $(\cE,\Psi)$ be an object in 
$\FIsocl(\ol{X},Z)$ whose Newton polygon at $\eta$ has pure slope $s$. 
Then, for any $x \in \ol{X}$, the Newton polygon of 
$(\cE,\Psi)$ at $x$ has pure slope $s$. 
\end{corollary}

The following is the generic Newton filtration theorem for 
(log) convergent $F$-isocrystals. 

\begin{proposition}\label{newfilprop}
Let $(\ol{X},Z), X$, $\eta$ be as in Proposition \ref{grothprop} and let 
$(\cE,\Psi)$ be an object in 
$\FIsocl(\ol{X},Z)$ such that its Newton polygon at $\eta$ 
is not a straight line $($has a break point$)$ 
and independent of $x \in X$. Then, on an dense open 
subscheme of $X$, $(\cE,\Psi)$ admits a 
non-trivial subobject $(\cE',\Psi')$ with $\mu(\cE') < \mu(\cE)$. 
\end{proposition}

\begin{proof}
Since we may shrink $\ol{X}$, we may assume that $X = \ol{X}$ and that 
it admits a strong smooth lift $\cX$. Then 
$(\cE,\Psi)$ naturally induces an object $(\cE_{\cX},\Psi_{\cX})$ in 
$\FVect(\cX_K)$ endowed with an integrable connection 
$\nabla$ for which $\Psi$ is horizontal. 
Then we may assume that $(\cE_{\cX},\Psi_{\cX})$ belongs to 
$\FLatt(\cX)_{\Q}$ by Proposition \ref{lat} and by 
Proposition \ref{newfil}, $(\cE_{\cX},\Psi_{\cX})$ admits a 
non-trivial subobject $(\cE'_{\cX},\Psi'_{\cX})$ with $\mu(\cE') < \mu(\cE)$ 
such that $\nabla |_{\cE'_{\cX}}$ defines an integrable connection on 
$\cE'_{\cX}$ for which $\Psi'_{\cX}$ is horizontal. Then the triple 
$(\cE'_{\cX},\nabla |_{\cE'_{\cX}},\Psi'_{\cX})$ naturally 
induces a non-trivial subobject $(\cE',\Psi')$ of $(\cE,\Psi)$ 
with $\mu(\cE') < \mu(\cE)$. (The `convergence' of 
$\nabla |_{\cE'_{\cX}}$ follows from that of $\nabla$.) 
\end{proof} 

Now we give the definition of generic semistability for 
log convergent $F$-isocrystals. 

\begin{definition}
Let $X \hra \ol{X}$ be an open immersion of connected smooth 
$k$-varieties such that $Z = \ol{X} \setminus X$ is a simple 
normal crossing divisor, and let $(\cE,\Psi)$ be an object in 
$\FIsocl(\ol{X},Z)$. Then $(\cE,\Psi)$ is called 
generically semistable $($gss$)$ if, for any dense open subscheme 
$U \subseteq X$, $(\cE,\Psi)|_U \in \FIsoc(U)$ does not admit 
any non-trivial subobject $(\cE',\Psi')$ with $\mu(\cE') < \mu(\cE)$. 
We denote the full subcategory of $\FIsocl(\ol{X},Z)$ consisting 
of generically semistable objects $(\cE,\Psi)$ with $\mu(\cE)=0$ by 
$\FIsocl(\ol{X},Z)^{\gss, \mu=0}$. 
In the case where $\ol{X} = \coprod_i \ol{X}_i$ 
is not necessarily connected, 
we define the category $\FIsocl(\ol{X},Z)^{\gss, \mu=0}$ as 
the product of the categories 
$\FIsocl(\ol{X}_i,\ol{X}_i \cap Z)^{\gss, \mu=0}$. 
\end{definition}

Then we have the following: 

\begin{proposition}\label{gssmain1}
Let $X \hra \ol{X}$ be an open immersion of connected smooth 
$k$-varieties such that $Z = \ol{X} \setminus X$ is a simple 
normal crossing divisor. Then we have the canonical 
equivalence 
\begin{equation}\label{gss1} 
\FIsoc(\ol{X})^{\circ} \os{=}{\lra} \FIsocl(\ol{X},Z)^{\gss,\mu=0}
\end{equation}
\end{proposition}

\begin{proof}
Let $\eta$ be the generic point of $\ol{X}$. 
We see by definition that, for any object $(\cE,\Psi)$ in 
$\FIsoc(\ol{X})^{\circ}$, the Newton polygon of it at 
$\eta$ has pure slope $0$. Hence so is for any dense open 
$U \subseteq X$ and any subobject of 
 $(\cE,\Psi)|_U$ in $\FIsoc(U)$. 
Hence $(\cE,\Psi)$ is generically 
semistable with $\mu(\cE)=0$, that is, the functor \eqref{gss1} is 
well-defined as the canonical inclusion. \par 
It suffices to show the essential surjectivity of 
the functor \eqref{gss1} to prove the proposition. So let us take 
an object $(\cE,\Psi)$ in $\FIsocl(\ol{X},Z)$ generically 
semistable with $\mu(\cE)=0$. Then, by Propositions \ref{grothsp}, 
\ref{endpoint} and \ref{lat}, there exists an open dense subscheme 
$U \subseteq X$ such that the Newton polygon of $(\cE,\Psi)$ at $x$ 
is independent of $x$. So, 
if the Newton polygon 
of $(\cE,\Psi)$ at $\eta$ is not a straight line, $(\cE,\Psi)$ is not 
generically semistable by Proposition \ref{newfilprop}. 
So the Newton polygon 
of $(\cE,\Psi)$ at $\eta$ has pure slope $\mu(\cE)=0$. Then, 
by Corollary \ref{grothcor}, the Newton polygon 
of $(\cE,\Psi)$ at $x$ has pure slope $0$ at any point $x$ in 
$\ol{X}$. Therefore, to see that $(\cE,\Psi)$ is in 
$\FIsoc(\ol{X})^{\circ}$, it suffices to show that it is contained 
in $\FIsoc(\ol{X})$. To see this, we may work locally. 
So we can assume that $X, \ol{X}$ are affine and that 
there exists a smooth lift $\ol{\cX}$ of 
$\ol{X}$, a lift $\cZ = \bigcup_{i=1}^r \cZ_i$ 
of $Z$ which is a relative simple normal 
crossing divisor of $\ol{\cX}$ and a lift $F: (\ol{\cX},\cZ) \lra 
(\ol{\cX},\cZ)$ of the $q$-th power Frobenius endomorphism 
on $(\ol{X},Z)$. Moreover, we may assume that each $\cZ_i$ is defined 
as the zero locus of some element $t_i$ in $\Gamma(\ol{\cX},\cO_{\ol{\cX}})$. 
Then $(\cE,\Psi)$ naturally induces an object 
$(\cE_{\ol{\cX}},\Psi_{\ol{\cX}})$ in $\FVect(\ol{\cX}_{K})$ endowed with 
an integrable log connection $\nabla: \cE_{\ol{\cX}} \lra 
\cE_{\ol{\cX}} \otimes \Omega^1_{\ol{\cX}_{K}}(\log \cZ_K)$. 
Then, to see that $(\cE,\Psi)$ is contained 
in $\FIsoc(\ol{X})$, it suffices to see that the image of $\nabla$ is 
contained in 
$\cE_{\ol{\cX}} \otimes \Omega^1_{\ol{\cX}_{K}}$. 
By Proposition \ref{urlat}, there exists a unit-root $F$-lattice 
$(\cE_0,\Psi_0)$ with $(\cE_0,\Psi_0)_{\Q} := (\cE_{0,\Q}, \Psi_{0,\Q}) = 
(\cE_{\ol{\cX}}, \Psi_{\ol{\cX}})$. In the following, we identify 
$\cE_{\ol{\cX}}, \cE_0, \Omega^1_{\ol{\cX}_K}, 
\Omega^1_{\ol{\cX}_K}(\log \cZ_K)$ with the set of global sections of 
them. For $n \in \Z$, let us put 
$$ 
\Omega_n(\cE_0) := 
(\sum_{i=1}^r p^n \cE_0 \dlog t_i) + 
\cE_{\ol{\cX}} \otimes \Omega^1_{\ol{\cX}_K} \subseteq 
\cE_{\ol{\cX}} \otimes \Omega^1_{\ol{\cX}_K}(\dlog \cZ_K). $$
We prove the inclusion $\nabla(\cE_0) \subseteq \Omega_n(\cE_0) \,(n \in \Z)$ 
by induction on $n$. First, since $\cE_0$ is finitely generated as 
$\cO_{\ol{\cX}}$-module and we have $\bigcup_{n\in\Z} \Omega_n(\cE_0) = 
\cE_{\ol{\cX}} \otimes \Omega^1_{\ol{\cX}_K}(\dlog \cZ_K)$, we have 
$\nabla(\cE_0) \subseteq \Omega_n$ for sufficiently small $n$. 
Next, assume that we have $\nabla(\cE_0) \subseteq \Omega_n(\cE_0)$. 
Since we have 
$$ F^*\Omega_n(\cE_0) = 
(\sum_{i=1}^r p^{n+1} F^*\cE_0 \dlog t_i) + 
F^*\cE_{\ol{\cX}} \otimes \Omega^1_{\cX_K} =: \Omega_{n+1}(F^*\cE_0), $$
we have 
$F^*\nabla(F^*\cE_0) \subseteq \Omega_{n+1}(F^*\cE_0)$, 
where we denoted the pull-back by $F$ of $\nabla: \cE_{\ol{\cX}} \lra 
\cE_{\ol{\cX}} \otimes \Omega^1_{\ol{\cX}_{K}}(\log \cZ_K)$ by 
$F^*\nabla$. 
By sending this by $\Psi$, we obtain the inclusion 
$\nabla(\cE_0) \subseteq \Omega_{n+1}(\cE_0)$, as desired. 
Then we have  
$$ \nabla(\cE_0) \subseteq \bigcap_{n \in \Z} \Omega_n(\cE_0)  
= \cE_{\ol{\cX}} \otimes \Omega^1_{\ol{\cX}_K}, $$
and so we obtain the desired inclusion 
$\nabla(\cE_{\ol{\cX}}) \subseteq 
\cE_{\ol{\cX}} \otimes \Omega^1_{\ol{\cX}_{K}}$. 
So we are done. 
\end{proof}

\begin{remark}
Since the category $\FIsoc(\ol{X})$ is a full subcategory of 
$\FIsocl(\ol{X},Z)$, the equivalence \eqref{gss1} induces the 
equivalence 
\begin{equation*}
\FIsoc(\ol{X})^{\circ} \os{=}{\lra} \FIsoc(\ol{X})^{\gss,\mu=0}.  
\end{equation*}
So we have the interpretation of unit-rootness in terms of 
generic semistability. 
\end{remark}

Next we define the notion of generic semistability for 
certain stacky categories and give 
a stacky version of Proposition \ref{gssmain1}. 

\begin{definition}\label{defgss2}
Let $X \hra \ol{X}$ be an open immersion of connected smooth 
$k$-varieties with $\ol{X} \setminus X =: Z = \bigcup_{i=1}^r Z_i$ 
a simple normal crossing divisor $($each $Z_i$ being irreducible$)$. 
\begin{enumerate}
\item 
Let 
$\cG_X$ be the category of finite etale Galois covering of $X$. Let 
$\varphi_Y: 
Y \lra X$ be an object in $\cG_X$, let $G_Y := \Aut (Y/X)$ and let 
$\ol{Y}^{\sm}$ be the smooth locus of the normalization 
$\ol{Y}$ in $\ol{X}$ in $k(Y)$. $($Then we have the quotient stack 
$[\ol{Y}^{\sm}/G_Y]$ and the canonical log structure 
$M_{[\ol{Y}^{\sm}/G_Y]}$ which are defined in Section 2.1.$)$ 
Then an object 
$(\cE,\Psi)$ in 
$\FIsocl([\ol{Y}^{\sm}/G_Y], M_{[\ol{Y}^{\sm}/G_Y]})$
is called 
generically semistable $($gss$)$ if, for any dense open subscheme 
$U \subseteq X$, the image $(\cE|_U,\Psi|_U)$ of 
$(\cE,\Psi)$ by the restriction functor 
\begin{align*}
& 
\FIsocl([\ol{Y}^{\sm}/G_Y], M_{[\ol{Y}^{\sm}/G_Y]})
\lra \\ 
& \hspace{2cm}
\FIsocl([\varphi_Y^{-1}(U)/G_Y], 
M_{[\ol{Y}^{\sm}/G_Y]}|_{[\varphi_Y^{-1}(U)/G_Y]})
= \FIsoc(U)
\end{align*}
does not admit 
any non-trivial subobject $(\cE',\Psi')$ with $\mu(\cE') < \mu(\cE|_U)$. 
$($Note that, in the definition above, the quantity 
$\mu(\cE|_U)$ is independent of the choice of $U$. So we denote it 
simply by $\mu(\cE)$ in the sequel.$)$
We denote the full subcategory of 
$\FIsocl([\ol{Y}^{\sm}/G_Y], M_{[\ol{Y}^{\sm}/G_Y]})$
consisting of generically semistable objects 
$(\cE,\Psi)$ with $\mu(\cE)=0$ by 
$$
\FIsocl([\ol{Y}^{\sm}/G_Y], M_{[\ol{Y}^{\sm}/G_Y]})^{\gss,\mu=0}.$$
\item 
For $n \in \N$ with $(n,p)=1$, let 
$(\ol{X},Z)^{1/n}$ be the stack of $n$-th roots of $(\ol{X},Z)$. 
$($Then we have the canonical log structure 
$M_{(\ol{X},Z)^{1/n}}$, which is defined in Section 2.3.$)$ 
Then an object 
$(\cE,\Psi)$ in $
\FIsocl((\ol{X},Z)^{1/n}, M_{(\ol{X},Z)^{1/n}})$
is called 
generically semistable $($gss$)$ if, for any dense open subscheme 
$U \subseteq X$, the image $(\cE|_U,\Psi|_U)$ of 
$(\cE,\Psi)$ by the restriction functor 
\begin{align*}
& 
\FIsocl((\ol{X},Z)^{1/n}, M_{(\ol{X},Z)^{1/n}})
\lra \\ 
& \hspace{2cm}
\FIsocl(U \times_{\ol{X}} (\ol{X},Z)^{1/n}, 
M_{(\ol{X},Z)^{1/n}}|_{U \times_{\ol{X}} (\ol{X},Z)^{1/n}})
= \FIsoc(U)
\end{align*} 
does not admit 
any non-trivial subobject $(\cE',\Psi')$ with $\mu(\cE') < \mu(\cE|_U)$. 
$($Note that, in the definition above, the quantity 
$\mu(\cE|_U)$ is independent of the choice of $U$. So we denote it 
simply by $\mu(\cE)$ in the sequel.$)$
We denote the full subcategory of $
\FIsocl((\ol{X},Z)^{1/n}, M_{(\ol{X},Z)^{1/n}})$ consisting 
of generically semistable objects $(\cE,\Psi)$ with $\mu(\cE)=0$ by 
$
\FIsocl((\ol{X},Z)^{1/n}, M_{(\ol{X},Z)^{1/n}})^{\gss, \mu=0}$. 
\end{enumerate}

\end{definition}

\begin{proposition}\label{gssmain2}
Let the notations be as in Definition \ref{defgss2}. 
Then we have the canonical 
equivalences 
\begin{equation}\label{gss2w}
\varinjlim_{Y \ra X \in \cG_X} 
\FIsoc([\ol{Y}^{\sm}/G_Y])^{\circ} 
\os{=}{\lra} \hspace{-10pt}
\varinjlim_{Y \ra X \in \cG_X} 
\FIsocl([\ol{Y}^{\sm}/G_Y], M_{[\ol{Y}^{\sm}/G_Y]})^{\gss,\mu=0}, 
\end{equation}
\begin{equation}\label{gss2}
\varinjlim_{Y \ra X \in \cG_X^t} 
\FIsoc([\ol{Y}^{\sm}/G_Y])^{\circ} 
\os{=}{\lra} \hspace{-10pt}
\varinjlim_{Y \ra X \in \cG_X^t} 
\FIsocl([\ol{Y}^{\sm}/G_Y], M_{[\ol{Y}^{\sm}/G_Y]})^{\gss,\mu=0} 
\end{equation}
$($where $\cG_X^t$ denotes the category of finite etale Galois tame 
covering of $X$ $($tamely ramified along the discrete valuations 
$v_i\,(1 \leq i \leq r)$ corresponding to the generic point of 
$Z_i))$ and 
\begin{equation}\label{gss3}
\varinjlim_{(n,p)=1} 
\FIsoc((\ol{X},Z)^{1/n})^{\circ}
\os{=}{\lra}
\varinjlim_{(n,p)=1} 
\FIsocl((\ol{X},Z)^{1/n}, M_{(\ol{X},Z)^{1/n}})^{\gss,\mu=0}.
\end{equation}
\end{proposition}

\begin{proof}
For $m=0,1,2$, let $\ol{Y}^{\sm}_m$ be the $(m+1)$-fold fiber product of 
$\ol{Y}^{\sm}$ over $[\ol{Y}^{\sm}/G_Y]$ and denote the 
resulting $2$-truncated simplicial scheme by $\ol{Y}^{\sm}_{\b}$. 
Then we have the equivalences 
\begin{align*}
& \varinjlim_{Y \ra X \in \cG_X} 
\FIsoc([\ol{Y}^{\sm}/G_Y])^{\circ} 
\os{=}{\lra} \varinjlim_{Y \ra X \in \cG^t_X} 
\FIsoc(\ol{Y}^{\sm}_{\b})^{\circ} \\ 
& \os{\text{\eqref{gss1}}}{\lra}
\varinjlim_{Y \ra X \in \cG_X} 
\FIsocl(\ol{Y}^{\sm}_{\b}, M_{\ol{Y}^{\sm}_{\b}})^{\gss,\mu=0} \\
& \os{=}{\lra} 
\varinjlim_{Y \ra X \in \cG_X} 
\FIsocl([\ol{Y}^{\sm}/G_Y], M_{[\ol{Y}^{\sm}/G_Y]})^{\gss,\mu=0},
\end{align*}
where $\FIsocl(\ol{Y}^{\sm}_{\b}, M_{\ol{Y}^{\sm}_{\b}})^{\gss,\mu=0}$ 
denotes the full subcategory of 
$\FIsocl(\ol{Y}^{\sm}_{\b}, M_{\ol{Y}^{\sm}_{\b}})$ consisting of 
objects whose restriction to 
$\FIsocl(\ol{Y}^{\sm}_{m}, M_{\ol{Y}^{\sm}_{m}})$ are contained in 
$\FIsocl(\ol{Y}^{\sm}_{m}, M_{\ol{Y}^{\sm}_{m}})^{\gss,\mu=0}$ for 
$m=0,1,2$. So we have shown \eqref{gss2w}, and we can prove 
\eqref{gss2} exactly in the same way. 
Next, take a chart $(\ol{X}_0, \{t_i\}_{1 \leq i \leq r})$ of 
$(\ol{X},Z)$ in the sense of Section 2.3 and for $n \in \N$ with 
$(n,p)=1$, let 
$(\ol{X}_{\b\b}^{(n)}, M_{\ol{X}_{\b\b}^{(n)}})$ be the 
bisimplicial resultion of $(\ol{X},Z)^{1/n}$ associated to 
the chart $(\ol{X}_0, \{t_i\}_{1 \leq i \leq r})$. Then we have the 
equivalences 
\begin{align*}
& 
\varinjlim_{(n,p)=1} 
\FIsoc((\ol{X},Z)^{1/n})^{\circ}
\os{=}{\lra} 
\varinjlim_{(n,p)=1} 
\FIsoc(\ol{X}_{\b\b}^{(n)})^{\circ} \\ 
& \os{\text{\eqref{gss1}}}{\lra}
\varinjlim_{(n,p)=1} 
\FIsocl(\ol{X}^{(n)}_{\b\b}, M_{\ol{X}^{(n)}_{\b\b}})^{\gss,\mu=0} \\ 
& \os{=}{\lra} 
\varinjlim_{(n,p)=1} 
\FIsocl((\ol{X},Z)^{1/n}, M_{(\ol{X},Z)^{1/n}})^{\gss,\mu=0}. 
\end{align*}
So we are done. 
\end{proof}

Next we define the notion of generic semistability for 
adjusted parabolic log convergent $F$-isocrystals and give 
a parabolic version of Proposition \ref{gssmain1}. 

\begin{definition}\label{gssdef3}
Let $X \hra \ol{X}$ be an open immersion of connected smooth 
$k$-varieties with $\ol{X} \setminus X =: Z = \bigcup_{i=1}^r Z_i$ 
a simple normal crossing divisor $($each $Z_i$ being irreducible$)$. 
Then an object 
$((\cE_{\alpha})_{\alpha},\Psi)$ in 
$\PFIsoc(\ol{X},Z)_{\0}$
is called 
generically semistable $($gss$)$ if, for any dense open subscheme 
$U \subseteq X$, the image $(\cE|_U, \Psi|_U)$ of 
$((\cE_{\alpha})_{\alpha},\Psi)$ by the restriction functor 
$$ 
\PFIsoc(\ol{X},Z)_{\0}
\lra \PFIsoc(U,\emptyset)_{\0} = \FIsoc(U)
$$ 
does not admit 
any non-trivial subobject $(\cE',\Psi')$ with $\mu(\cE') < \mu(\cE|_U)$. 
$($Note that, in the definition above, the quantity 
$\mu(\cE|_U)$ is independent of the choice of $U$. So we denote it 
by $\mu((\cE_{\alpha})_{\alpha})$ in the sequel.$)$
In the following, we denote the full subcategory of 
$\PFIsoc(\ol{X},Z)_{\0}$ 
consisting of generically semistable objects 
$((\cE_{\alpha})_{\alpha},\Psi)$ with $\mu((\cE_{\alpha})_{\alpha})=0$ by 
$\PFIsoc(\ol{X},Z)_{\0}^{\gss,\mu=0}.$
\end{definition}

\begin{theorem}\label{gssmain3}
Let the notations be as in Definition \ref{gssdef3}. Then we have the 
canonical equivalence of categories 
\begin{equation}\label{gss6}
\PFIsoc(\ol{X},Z)^{\circ}_{\0\ss} \os{=}{\lra} 
\PFIsoc(\ol{X},Z)_{\0}^{\gss,\mu=0}.
\end{equation}
\end{theorem}

\begin{proof}
This is an immediate consequence of the equivalence 
\eqref{gss2} and the equivalences \eqref{root-to-par-f}, 
\eqref{root-to-par-urf}. 
\end{proof}

As an immediate consequence 
of Propositions \ref{gssmain2}, \ref{gssmain3}, we have 
the following, which is a $p$-adic analogues of \eqref{eq1} and 
\eqref{eq2} and \eqref{eq4} which includes the notion of `stability': 

\begin{corollary}\label{stability1}
Let the notations be as in Definition \ref{defgss2}. Then we have 
the equivalences 
\begin{align*}
& \Rep_{K^{\sigma}}^{\fin}(\pi_1(X)) \os{=}{\lra} 
\varinjlim_{Y \ra X \in \cG_X} 
\FIsocl([\ol{Y}^{\sm}/G_Y], M_{[\ol{Y}^{\sm}/G_Y]})^{\gss,\mu=0}, \\ 
& \Rep_{K^{\sigma}}(\pi_1^t(X)) \os{=}{\lra} 
\varinjlim_{Y \ra X \in \cG^t_X} 
\FIsocl([\ol{Y}^{\sm}/G_Y], M_{[\ol{Y}^{\sm}/G_Y]})^{\gss,\mu=0}, \\ 
& \Rep_{K^{\sigma}}(\pi_1^t(X)) \os{=}{\lra} 
\varinjlim_{(n,p)=1} 
\FIsocl((\ol{X},Z)^{1/n}, M_{(\ol{X},Z)^{1/n}})^{\gss, \mu=0}, \\ 
& \Rep_{K^{\sigma}}(\pi_1^t(X)) \os{=}{\lra} 
\PFIsoc(\ol{X},Z)^{\gss,\mu=0}. 
\end{align*}
\end{corollary}

Next we prove `$F$-lattice versions' of Propositions \ref{gssmain1}, 
\ref{gssmain2} and \ref{gssmain3}. 
Let $X$ be a connected 
smooth scheme separated of finite type over $k$ and 
assume that it is liftable to a $p$-adic formal scheme 
$\cX_{\circ}$ separated of finite type over $\Spf W(k)$ 
which is endowed with a lift $F_{\circ}: \cX_{\circ} \lra \cX_{\circ}$ 
of the $q$-th power Frobenius endomorphism on $X$ compatible with 
$(\sigma|_{W(k)})^*: \Spf W(k) \lra \Spf W(k)$. Let us put 
$\cX := \cX_{\circ} \otimes O_K$, $F := F_{\circ} \otimes \sigma^*: 
\cX \lra \cX$. Then an object $(\cE,\Psi)$ in 
$\FLatt(\cX)_{\Q}$ is said to be generically semistable (gss) 
if, for any open dense formal subscheme $\cU_{\circ} \lra \cX_{\circ}$, 
$(\cE|_{\cU}, \Psi|_{\cU})$ (where $\cU := \cU_{\circ} \otimes_{W(k)} O_K$) 
admits no non-trivial saturated subobject $(\cE',\Psi')$ with 
$\mu(\cE') < \mu(\cE)$. We denote the full subcategory of 
$\FLatt(\cX)_{\Q}$ consisting of generically semistable 
objects $(\cE,\Psi)$ with $\mu(\cE)=0$ by 
$(\FLatt(\cX)_{\Q})^{\gss,\mu=0}$. 
When $X$ is not necessarily connected, 
$(\FLatt(\cX)_{\Q})^{\gss,\mu=0}$ denotes the full subcategory of 
$\FLatt(\cX)_{\Q}$ which are generically semistable with $\mu=0$ 
on each connected component of $\cX$. Then 
the $F$-lattice version of Proposition \ref{gssmain1} 
is the following proposition (which is an immediate consequence of 
the results of Crew and Katz quoted in the beginning of this section): 

\begin{proposition}\label{gssmain2-1}
Let the notations be as above. Then we have the equivalence 
\begin{equation}\label{gss2-1}
\FLatt(\cX)^{\circ}_{\Q} = (\FLatt(\cX)_{\Q})^{\gss, \mu=0}
\end{equation}
\end{proposition}

\begin{proof}
It is easy to see that the left hand side is contained in the right hand 
side, and we may assume that $\cX$ is connected. 
Let $\eta$ be the generic point of $X$ and let us take 
an object $(\cE,\Psi)$ in $(\FLatt(\cX)_{\Q})^{\gss, \mu=0}$. 
Then, by Propositions \ref{grothsp} and \ref{endpoint}, 
there exists an dense open $\cU_{\circ} \lra \cX_{\circ}$ such that 
the Newton polygon of 
$(\cE|_{\cU}, \Psi|_{\cU})$ (where $\cU := \cU_{\circ} \otimes_{W(k)} O_K$) 
at any point is equal to that at $\eta$. Then, by Proposition 
\ref{newfil}, the generic semistability of $(\cE,\Psi)$ implies that 
the Newton polygon of $(\cE,\Psi)$ at $\eta$ has pure slope 
$\mu(\cE)=0$. Again by Propositions \ref{grothsp} and \ref{endpoint}, 
we see that the Newton polygon $(\cE,\Psi)$ at any point has 
pure slope $0$. Then, by Proposition \ref{urlat}, we see that 
$(\cE,\Psi)$ is contained in $\FLatt(\cX)^{\circ}_{\Q}$. So we are done. 
\end{proof}

Next we define the notion of generic semistability for 
$F$-lattices on stacks and give 
an $F$-lattice version of Theorem \ref{gssmain2}. \par 
Let the notations be 
as in Section 4, the paragraphs after 
\eqref{katzeq'}, before Theorem \ref{flatttt}. Then, for 
an object $Y \lra X$ in $\cG^t_X$, the ind fine log algebraic 
stack $([\ol{\cY}^{\sm}/G_Y], M_{[\ol{\cY}^{\sm}/G_Y]}) = 
\varinjlim_a ([\ol{\cY}^{\sm}_a/G_Y], M_{[\ol{\cY}^{\sm}_a/G_Y]})$, 
the endomorphism $F$ (lift of Frobenius) 
on it and 
the category 
$\FLatt([\ol{\cY}^{\sm}/G_Y])$ of $F$-lattices on 
$[\ol{\cY}^{\sm}/G_Y]$ are defined there. 
The ind algebraic stack 
$(\ol{\cX},\cZ)^{1/n} = \varinjlim_a (\ol{\cX},\cZ)^{1/n}_a$ endowed 
with the endomorphism $F$ (lift of Frobenius) and 
the category 
$\FLatt((\ol{\cX},\cZ)^{1/n})$ of $F$-lattices on 
$(\ol{\cX},\cZ)^{1/n}$ are also defined there. 
\par 
An object $(\cE,\Psi)$ in 
$\FLatt([\ol{\cY}^{\sm}/G_Y])_{\Q}$ 
(resp. $\FLatt((\ol{\cX},\cZ)^{1/n})_{\Q}$)
is called 
generically semistable (gss) if, for any dense open formal 
subscheme $\cU_{\circ} \subseteq \cX_{\circ}$, 
the image of $(\cE,\Psi)$ by the restriction functor 
\begin{align*}
&  \FLatt([\ol{\cY}^{\sm}/G_Y])_{\Q} \lra 
\FLatt(\cU \times_{\cX} [\ol{\cY}^{\sm}/G_Y])_{\Q} \os{=}{\lra} 
\FLatt(\cU)_{\Q} \\ 
& \text{(resp. } \,\,
\FLatt((\ol{\cX},\cZ)^{1/n})_{\Q} \lra 
\FLatt(\cU \times_{\cX}(\ol{\cX},\cZ)^{1/n})_{\Q} 
\os{=}{\lra} \FLatt(\cU)_{\Q} \,\,\text{)} 
\end{align*}
admits no non-trivial saturated subobject $(\cE',\Psi')$ with 
$\mu(\cE') < \mu(\cE|_{\cU})$. (Note that $\mu(\cE|_{\cU})$ does not 
depend on $\cE$. Hence we denote it by $\mu(\cE)$.) 
We denote the full subcategory of 
$\FLatt([\ol{\cY}^{\sm}/G_Y])_{\Q}$ 
(resp. $\FLatt((\ol{\cX},\cZ)^{1/n})_{\Q}$)
consisting of generically semistable 
objects $(\cE,\Psi)$ with $\mu(\cE)=0$ by 
$
\FLatt([\ol{\cY}^{\sm}/G_Y])_{\Q}^{\gss,\mu=0}$ 
(resp. $\FLatt((\ol{\cX},\cZ)^{1/n})_{\Q}^{\gss,\mu=0}$). 
Then we have the following proposition, which is 
the $F$-lattice version of Proposition \ref{gssmain2}: 

\begin{theorem}\label{gssmain2-2}
Let the notations be as above. 
Then we have the canonical 
equivalences 
\begin{equation}\label{gss2-2}
\varinjlim_{Y \ra X \in \cG^t_X} 
\FLatt([\ol{\cY}^{\sm}/G_Y])^{\circ}_{\Q} 
\os{=}{\lra} 
\varinjlim_{Y \ra X \in \cG^t_X} 
\FLatt([\ol{\cY}^{\sm}/G_Y])_{\Q}^{\gss,\mu=0}, 
\end{equation}
\begin{equation}\label{gss2-3}
\varinjlim_{(n,p)=1} 
\FLatt((\ol{\cX},\cZ)^{1/n})^{\circ}_{\Q}
\os{=}{\lra}
\varinjlim_{(n,p)=1} 
\FLatt((\ol{\cX},\cZ)^{1/n})_{\Q}^{\gss,\mu=0}. 
\end{equation}
\end{theorem}

\begin{proof}
Let $\ol{\cY}^{\sm}_{\b}$ be as in the proof of 
Theorem \ref{flatttt}. Then we have the functors 
\begin{align}
& \varinjlim_{Y \ra X \in \cG^t_X} 
\FLatt([\ol{\cY}^{\sm}/G_Y])^{\circ}_{\Q} 
\os{=}{\lra} \varinjlim_{Y \ra X \in \cG^t_X} 
\FLatt(\ol{\cY}^{\sm}_{\b})^{\circ}_{\Q} \label{tea} \\ 
& \lra 
\varinjlim_{Y \ra X \in \cG^t_X} 
\FLatt(\ol{\cY}^{\sm}_{\b})_{\Q}^{\gss,\mu=0} \os{=}{\lra} 
\varinjlim_{Y \ra X \in \cG^t_X} 
\FLatt([\ol{Y}^{\sm}/G_Y])_{\Q}^{\gss,\mu=0}, \nonumber 
\end{align}
where $\FLatt(\ol{\cY}^{\sm}_{\b})_{\Q}^{\gss,\mu=0}$ denotes the 
full subcategory of 
 $\FLatt(\ol{\cY}^{\sm}_{\b})_{\Q}$ consisting of objects 
whose restriction to $\FLatt(\ol{\cY}^{\sm}_{m})_{\Q}$ are 
contained in $\FLatt(\ol{\cY}^{\sm}_{m})_{\Q}^{\gss,\mu=0}$ for 
$m=0,1,2$. 
Let us prove that the second arrow in the above diagram is 
an equivalence. 
Let us denote the category of compatible system of objects in 
$\FLatt(\ol{\cY}^{\sm}_{m})^{\circ}_{\Q} \, (m=0,1,2)$ 
(resp. $\FLatt(\ol{\cY}^{\sm}_{m})_{\Q}^{\gss,\mu=0} \,(m=0,1,2)$)
by $\{\FLatt(\ol{\cY}^{\sm}_{m})^{\circ}_{\Q}\}_{m=0,1,2}$ 
(resp. $\{\FLatt(\ol{\cY}^{\sm}_{m})_{\Q}^{\gss,\mu=0}\}_{m=0,1,2}$.) 
Then we have the following commutative diagram 
\begin{equation*}
\begin{CD}
\varinjlim_{Y \ra X \in \cG^t_X} 
\FLatt(\ol{\cY}^{\sm}_{\b})^{\circ}_{\Q} @>>>
\varinjlim_{Y \ra X \in \cG^t_X} 
\{\FLatt(\ol{\cY}^{\sm}_{m})^{\circ}_{\Q}\}_{m=0,1,2} \\
@VVV @VVV \\ 
\varinjlim_{Y \ra X \in \cG^t_X} 
\FLatt(\ol{\cY}^{\sm}_{\b})_{\Q}^{\gss,\mu=0} @>>>
\varinjlim_{Y \ra X \in \cG^t_X} 
\{\FLatt(\ol{\cY}^{\sm}_{m})_{\Q}^{\gss,\mu=0}\}_{m=0,1,2}, 
\end{CD}
\end{equation*}
where the horizontal arrows are natural fully faithful inclusions, 
the left vertical arrow is the one in the diagram \eqref{tea} and 
the right vertical arrow is defined as in the left vertical one. 
Then, Proposition \ref{gssmain2-1} implies that the left vertical 
arrow is an equivalence. On the other hand, by the proof of 
Theorem \ref{flatttt}, we see that the top horizontal arrow, 
which is equal to \eqref{nu}, is 
an equivalence. Hence the right vertical arrow is also an 
equivalence. So \eqref{tea} is an equivalence, as desired. \par 
Next, let 
$\ol{\cX}_{\b\b}^{(n)}$ be also as in the proof of Theorem 
\ref{flatttt}. 
Then we have the functors 
\begin{align*}
& 
\varinjlim_{(n,p)=1} 
\FLatt((\ol{\cX},\cZ)^{1/n})^{\circ}_{\Q}
\os{=}{\lra} 
\varinjlim_{(n,p)=1} 
\FLatt(\ol{\cX}^{(n)}_{\b\b})^{\circ}_{\Q} \\ 
& \lra
\varinjlim_{(n,p)=1} 
\FLatt(\ol{\cX}^{(n)}_{\b\b})^{\gss,\mu=0}_{\Q} \os{=}{\lra} 
\varinjlim_{(n,p)=1} 
\FLatt((\ol{\cX},
\cZ)^{1/n})^{\gss,\mu=0}_{\Q} 
\end{align*}
(where $\FLatt(\ol{\cX}^{(n)}_{\b\b})_{\Q}^{\gss,\mu=0}$ denotes the 
full subcategory of 
 $\FLatt(\ol{\cX}^{(n)}_{\b\b})_{\Q}$ consisting of objects 
whose restriction to $\FLatt(\ol{\cX}^{(n)}_{kl})_{\Q}$ are 
contained in $\FLatt(\ol{\cX}^{(n)}_{kl})_{\Q}^{\gss,\mu=0}$ for 
$k,l=0,1,2$), and we can prove that the second arrow is an equivalence 
in the same way as 
in the previous paragraph, using 
the equivalence \eqref{nunu} instead of \eqref{nu}. So we are done. 
\end{proof}

Next we define the notion of generic semistability for 
locally abelian parabolic $F$-lattices and 
prove an $F$-lattice version of Proposition \ref{gssmain3}. \par 
Let the notations be as in Definitions \ref{defp}, \ref{defla}. 
Then an object $((\cE_{\alpha})_{\alpha},(\Psi_{\alpha})_{\alpha})$ in 
$\PFLatt(\ol{\cX},\cZ)_{\Q}$
is called 
generically semistable (gss) if, for any dense open formal 
subscheme $\cU_{\circ} \subseteq \cX_{\circ}$, 
the image $(\cE|_{\cU}, \Psi|_{\cU})$ 
of $((\cE_{\alpha})_{\alpha},(\Psi_{\alpha})_{\alpha})$
by the restriction functor 
\begin{align*}
&  \PFLatt(\ol{\cX},\cZ)_{\Q} \lra 
\PFLatt(\cU \times_{\ol{\cX}} (\ol{\cX},\cZ))_{\Q} \os{=}{\lra} 
\FLatt(\cU)_{\Q} \\ 
\end{align*}
(where $\cU := \cU_{\circ} \otimes_{W(k)} O_K$) 
admits no non-trivial saturated subobject $(\cE',\Psi')$ with 
$\mu(\cE') < \mu(\cE|_{\cU})$. (Note that $\mu(\cE|_{\cU})$ does not 
depend on $\cE$. Hence we denote it by $\mu(\cE)$.) 
We denote the full subcategory of 
$\PFLatt(\ol{\cX},\cZ)_{\Q}$
consisting of generically semistable 
objects $(\cE,\Psi)$ with $\mu(\cE)=0$ by 
$\PFLatt(\ol{\cX},\cZ)_{\Q}^{\gss,\mu=0}$. 
Then we have the following proposition, which is 
the $F$-lattice version of Proposition \ref{gssmain3}: 

\begin{proposition}\label{gssmain2-3}
Let the notations be as above. Then we have the 
canonical equivalence of categories 
\begin{equation}\label{gss6}
\PFLatt(\ol{\cX},\cZ)^{\circ}_{\Q} \os{=}{\lra} 
\PFLatt(\ol{\cX},\cZ)^{\gss,\mu=0}_{\Q}.
\end{equation}
\end{proposition}

\begin{proof}
Note that the equivalence \ref{flatt+--} 
\begin{equation*}
\varinjlim_{(n,p)=1}\FLatt((\ol{\cX},\cZ)^{1/n}) \os{=}{\lra} 
\PFLatt(\ol{\cX},\cZ) 
\end{equation*}
preserves the generic 
semistabilities and the value of $\mu$. 
Then the desired equivalence 
follows from the equivalence \eqref{gss2-3}. 
\end{proof}

As an immediate consequence of 
Propositions \ref{gssmain2-2}, \ref{gssmain2-3}, we have 
the following, which is a $p$-adic analogues ($F$-lattice 
version) of \eqref{eq1} and 
\eqref{eq2} and \eqref{eq4} which includes the notion of stability: 

\begin{theorem}\label{stability2}
Let the notations be as above. Then we have 
the equivalences 
\begin{align*}
& \Rep_{K^{\sigma}}(\pi_1^t(X)) \os{=}{\lra} 
\varinjlim_{Y \ra X \in \cG^t_X} 
\FLatt([\ol{\cY}^{\sm}/G_Y])^{\gss,\mu=0}_{\Q}. \\ 
& \Rep_{K^{\sigma}}(\pi_1^t(X)) \os{=}{\lra} 
\varinjlim_{(n,p)=1} 
\FLatt((\ol{\cX},\cZ)^{1/n})^{\gss, \mu=0}_{\Q}. \\ 
& \Rep_{K^{\sigma}}(\pi_1^t(X)) \os{=}{\lra} 
\PFLatt(\ol{\cX},\cZ)^{\gss,\mu=0}_{\Q}. 
\end{align*}
\end{theorem}

Roughly speaking, 
Theorem \ref{stability2} claims 
(when compared with Theorem \ref{stability1}) 
that we can forget the isocrystal structure (connection) 
in the categories on the right hand sides if we assume 
a strong liftability condition. But we have to put 
the lattice structure. Finally in this paper, we introduce 
the category of `$F$-vector bundles on certain rigid analytic stacks', 
`locally abelian 
parabolic $F$-vector bundles on certain log rigid analytic spaces' and 
prove a $p$-adic analogues ($F$-vector bundle version) 
of \eqref{eq1} and \eqref{eq2} and \eqref{eq4} in which 
neither isocrystal structure and lattice structure appear, 
in the case of curves satisfying a strong liftability condition. \par 
Let $X$ be a connected 
smooth scheme separated of finite type over $k$ and 
assume that it is liftable to a $p$-adic formal scheme 
$\cX_{\circ}$ separated of finite type over $\Spf W(k)$ 
which is endowed with a lift $F_{\circ}: \cX_{\circ} \lra \cX_{\circ}$ 
of the $q$-th power Frobenius endomorphism on $X$ compatible with 
$(\sigma|_{W(k)})^*: \Spf W(k) \lra \Spf W(k)$. Let us put 
$\cX := \cX_{\circ} \otimes O_K$, $F := F_{\circ} \otimes \sigma^*: 
\cX \lra \cX$. Then an object $(\cE,\Psi)$ in 
$\FVect(\cX_K)$ is said to be generically semistable (gss) 
if, for any open dense formal subscheme $\cU_{\circ} \lra \cX_{\circ}$, 
$(\cE|_{\cU_K}, \Psi|_{\cU_K})$ 
(where $\cU := \cU_{\circ} \otimes_{W(k)} O_K$) 
admits no non-trivial saturated subobject $(\cE',\Psi')$ with 
$\mu(\cE') < \mu(\cE)$, where a subobject $(\cE',\Psi')$ of 
$(\cE|_{\cU_K}, \Psi|_{\cU_K})$ is called saturated if 
the quotient $(\cE|_{\cU_K}/\cE', \ol{\Psi |_{\cU_K}})$ 
($\ol{\Psi |_{\cU_K}}$ is the morphism induced by $\Psi |_{\cU_K}$) 
is again an object in $\FVect(\cU_K)$. 
(Note that $\mu(\cE|_{\cU_K})$ does not 
depend on $\cE$. Hence we denote it by $\mu(\cE)$.) 
 We denote the full subcategory of 
$\FVect(\cX)_{\Q}$ consisting of generically semistable 
objects $(\cE,\Psi)$ with $\mu(\cE)=0$ by 
$\FVect(\cX_K)^{\gss,\mu=0}$. 
When $X$ is not necessarily connected, 
$\FVect(\cX_K)^{\gss,\mu=0}$ denotes the full subcategory 
consisting of objects which are generically semistable with $\mu=0$ 
on each connected component of $X$. Then we have the 
following proposition: 

\begin{proposition}\label{gssmain3-1}
Let the notations be as above and assume $\dim X=1$. 
Then we have the equivalence 
\begin{equation}\label{gss3-1}
\FLatt(\cX)^{\circ}_{\Q} \os{=}{\lra} \FVect(\cX_K)^{\gss, \mu=0}. 
\end{equation}
\end{proposition}

\begin{proof}
By Proposition \ref{lat}, 
$\FVect(\cX_K)^{\gss, \mu=0}$ is contained in 
$(\FLatt(\cX)_{\Q})^{\gss, \mu=0}$ in the case of curves. 
Hence the proposition is reduced to Proposition \ref{gssmain2-1}. 
\end{proof}

Recall that, for a $p$-adic formal scheme $\cS$ separated of finite 
type over $\Spf O_K$, we have the canonical equivalence 
$\Coh(\cS)_{\Q} \simeq \Coh(\cS_K)$ of the $\Q$-linearization of 
the catgory of coherent $\cO_{\cS}$-modules and 
the category of coherent $\cO_{\cS_K}$-modules. 
With this in mind, we give the following 
definitions: Let the notations be as in Section 4, the 
paragraphs after \eqref{katzeq'}, before Definition \ref{flatttt}. 
Then, for an object $Y \lra X \in \cG^t_Y$, the ind stack 
$[\ol{\cY}^{\sm}/G_Y] = \varinjlim_a [\ol{\cY}^{\sm}_a/G_Y]$ 
and the endomorphism $F$ on it is 
defined. We define the category $\FVect([\ol{\cY}^{\sm}/G_Y]_K)$ 
of `$F$-vector bundles on $[\ol{\cY}^{\sm}/G_Y]_K$' as the 
category of pairs $(\cE,\Psi)$, where 
$\cE$ is an object in 
$\Coh([\ol{\cY}^{\sm}/G_Y])_{\Q}$ ($:=$ the $\Q$-linearization 
of the category of compatible families of coherent 
$\cO_{[\ol{\cY}^{\sm}_a/G_Y]}$-modules) such that, for any 
morphism 
$\cS \lra [\ol{\cY}^{\sm}/G_Y]$ 
from a $p$-adic formal scheme separated of finite type 
over $\Spf \cO_K$, 
$\cE|_{\cS} \in \Coh(\cO_{\cS})_{\Q} \simeq \Coh(\cO_{\cS_K})$ 
is a locally free $\cO_{\cS_K}$-module of finite rank, and 
$\Psi$ is an isomorphism $F^*\cE \os{=}{\lra} \cE$. 
(Attention: We only defined 
the category  
of $F$-vector bundles on $[\ol{\cY}^{\sm}/G_Y]_K$, not 
the rigid stack $[\ol{\cY}^{\sm}/G_Y]_K$ itself.) 
In the same way, 
we define the category 
$\FVect((\ol{\cX},\cZ)^{1/n}_K)$ of `$F$-vector bundles on 
$(\ol{\cX},\cZ)^{1/n}_K$' as the category 
of pairs $(\cE,\Psi)$, where 
$\cE$ is an object in 
$\Coh((\ol{\cX},\cZ)^{1/n})_{\Q}$ ($:=$ the $\Q$-linearization 
of the category of compatible families of coherent 
$\cO_{(\ol{\cX},\cZ)^{1/n}_a}$-modules) such 
that, for any 
morphism 
$\cS \lra (\ol{\cX},\cZ)^{1/n})$ 
from a $p$-adic formal scheme separated of finite type 
over $\Spf \cO_K$, 
$\cE|_{\cS} \in \Coh(\cO_{\cS})_{\Q} \simeq 
\Coh(\cO_{\cS_K})$ is locally free of finite rank, and 
$\Psi$ is an isomorphism $F^*\cE \os{=}{\lra} \cE$. \par 
An object $(\cE,\Psi)$ in 
$\FVect([\ol{\cY}^{\sm}/G_Y]_K)$ 
(resp. $\FVect((\ol{\cX},\cZ)^{1/n}_K)$)
is called 
generically semistable (gss) if, for any dense open formal 
subscheme $\cU_{\circ} \subseteq \cX_{\circ}$, 
the image  $(\cE |_{\cU_K},\Psi |_{\cU_K})$
of $(\cE,\Psi)$ by the restriction functor 
\begin{align*}
&  \FVect([\ol{\cY}^{\sm}/G_Y]_K) \lra 
\FVect((\cU \times_{\cX} [\ol{\cY}^{\sm}/G_Y])_K) \os{=}{\lra} 
\FVect(\cU_K)_{\Q} \\ 
& \text{(resp. } \,\,
\FVect((\ol{\cX},\cZ)^{1/n}_K) \lra 
\FVect((\cU \times_{\cX}(\ol{\cX},\cZ)^{1/n})_K) 
\os{=}{\lra} \FVect(\cU_K) \,\,\text{)} 
\end{align*}
admits no non-trivial saturated subobject $(\cE',\Psi')$ with 
$\mu(\cE') < \mu(\cE|_{\cU_K})$. (Note that $\mu(\cE|_{\cU_K})$ does not 
depend on $\cU$. Hence we denote it by $\mu(\cE)$.) 
We denote the full subcategory of 
$\FVect([\ol{\cY}^{\sm}/G_Y]_K)$ 
(resp. $\FVect((\ol{\cX},\cZ)^{1/n}_K)$)
consisting of generically semistable 
objects $(\cE,\Psi)$ with $\mu(\cE)=0$ by 
$
\FVect([\ol{\cY}^{\sm}/G_Y]_K)^{\gss,\mu=0}$ 
(resp. $\FVect((\ol{\cX},\cZ)^{1/n}_K)^{\gss,\mu=0}$). 
Then we have the following proposition: 

\begin{proposition}\label{gssmain3-2}
Let the notations be as above and assume that $\dim X=1$. 
Then we have the canonical 
equivalences 
\begin{equation}\label{gss3-2}
\varinjlim_{Y \ra X \in \cG^t_X} 
\FLatt([\ol{\cY}^{\sm}/G_Y])^{\circ}_{\Q} 
\os{=}{\lra} 
\varinjlim_{Y \ra X \in \cG^t_X} 
\FVect([\ol{\cY}^{\sm}/G_Y]_K)^{\gss,\mu=0}, 
\end{equation}
\begin{equation}\label{gss3-3}
\varinjlim_{(n,p)=1} 
\FLatt((\ol{\cX},\cZ)^{1/n})^{\circ}_{\Q}
\os{=}{\lra}
\varinjlim_{(n,p)=1} 
\FVect((\ol{\cX},\cZ)^{1/n}_K)^{\gss,\mu=0}. 
\end{equation}
\end{proposition}

\begin{proof}
Let $\ol{\cY}^{\sm}_{\b}$ be as in the proof of 
Theorem \ref{flatttt}, and in the proof, we omit to write 
the superscript ${}^{\sm}$. (This is justified because we have 
$\ol{\cY} = \ol{\cY}^{\sm}$ in the case of curves.) 
Then we have the diagram 
\begin{align}
\varinjlim_{Y \ra X \in \cG^t_X} 
\FVect([\ol{\cY}/G_Y]_K)^{\gss,\mu=0} & \lra 
\varinjlim_{Y \ra X \in \cG^t_X} 
\FVect(\ol{\cY}_{\b})^{\gss,\mu=0} \label{1/28tea} \\ 
& {\lla} 
\varinjlim_{Y \ra X \in \cG^t_X} 
\{\FLatt(\ol{\cY}_m)_{\Q}^{\circ}\}_{m=0,1,2} \nonumber \\ 
& {\lla}
\varinjlim_{Y \ra X \in \cG^t_X} 
\FLatt(\ol{\cY}_{\b})^{\circ}_{\Q} \nonumber \\ 
& \os{=}{\lla} 
\FLatt([\ol{\cY}/G_Y])^{\circ}_{\Q}, \nonumber 
\end{align}
where $\FVect(\ol{\cY}_{\b})^{\gss,\mu=0}$ denotes the 
full subcategory of $\FVect(\ol{\cY}_{\b})$ consisting of 
objects whose restriction to $\FVect(\ol{\cY}_m)$ is contained in 
$\FVect(\ol{\cY}_{m})^{\gss,\mu=0}$ for each $m$. 
In the diagram \eqref{1/28tea}, the third arrow is an equivalence 
since we already proved it in the proof of Proposition \ref{gssmain2-2}, 
and the second arrow is an equivalence by Proposition \ref{gssmain3-1}. 
So, to show the equivalence 
\eqref{gss3-2}, it suffices to prove that the first arrow in 
\eqref{1/28tea} is an equivalence. Since it is easy to see 
that the condition `gss and $\mu=0$' is preserved by the natural functor 
\begin{equation}\label{1/282}
\FVect([\ol{\cY}/G_Y]_K) \lra \FVect(\ol{\cY}_{\b}), 
\end{equation}
it suffices to show that the functor \eqref{1/282} is an 
equivalence. 
Let $\{\Coh(\ol{\cY}_{m})_{\Q}\}_{m=0,1,2}$ be 
the category of compatible family of objects in 
$\Coh(\ol{\cY}_m)_{\Q} \,(m=0,1,2)$. 
Then, the functor \eqref{1/282} is 
induced by the functor 
\begin{equation}\label{1/283}
\Coh([\ol{\cY}/G_Y])_{\Q} \lra \{\Coh(\ol{\cY}_{m})_{\Q}\}_{m=0,1,2}, 
\end{equation}
as the locally free part of \eqref{1/283} with $F$-structure. 
Since local freeness for an object in $\Coh([\ol{\cY}/G_Y])_{\Q}$ 
can be checked in $\Coh(\ol{\cY}_{m})_{\Q}\,(m=0,1,2)$, it suffices 
to show the equivalence of the functor \eqref{1/283}. 
Note that it is factorized as 
\begin{equation*}
\Coh([\ol{\cY}/G_Y])_{\Q} \lra 
\Coh(\ol{\cY}_{\b})_{\Q} \lra 
\{\Coh(\ol{\cY}_{m})_{\Q}\}_{m=0,1,2}, 
\end{equation*}
in which the first arrow is an equivalence by usual faithfully 
flat descent. Moreover, we see by the same way as \cite[1.9]{ogus}
that the second arrow is also an equivalence. Hence we have shown the 
equivalence \eqref{gss3-2}. 
We can prove the equivalence \eqref{gss3-3} in the same way, 
by using $\ol{\cX}^{(n)}_{\b\b}$ in the proof of Theorem 
\ref{flatttt} instead of $\ol{\cY}_{\b}$. 
\end{proof}

Next, let the notations be as in Definition \ref{defp}. 
Then `a parabolic vector bundle on $(\ol{\cX},\cZ)_K$' is defined to 
be an inductive system $(\cE_{\alpha})_{\alpha \in \Z_{(p)}}$ 
of locally free $\cO_{\ol{\cX}_K}$-modules of finite rank satisfying 
the following conditions: 
\begin{enumerate}
\item[(a)]
For any $1 \leq i \leq r$, there is an isomorphism as inductive systems 
$$ 
((\cE_{\alpha+e_i})_{\alpha}, (\iota_{\alpha+e_i,\beta+e_i})_{\alpha,\beta}) 
\cong 
((\cE_{\alpha}(\cZ_{i,K}))_{\alpha}, 
(\iota_{\alpha\beta} \otimes \id)_{\alpha,\beta}) $$
via which the morphism 
$(\iota_{\alpha,\alpha+e_i})_{\alpha} : 
(\cE_{\alpha})_{\alpha} \lra (\cE_{\alpha+e_i})_{\alpha}$ is identified 
with the morphism 
$\id \otimes \iota_{0,e_i}^0: 
(\cE_{\alpha})_{\alpha} \lra (\cE_{\alpha}(\cZ_{i,K}))_{\alpha}$, 
where $\iota_{0,e_i}^0: \cO_{\ol{\cX}} \hra \cO_{\ol{\cX}}(\cZ_{i,K})$ 
denotes the natural inclusion. 
\item[(b)]
There exists a positive integer $n$ prime to $p$ which satisfies the following 
condition$:$ For any $\alpha = (\alpha_i)_i$, $\iota_{\alpha'\alpha}$ is 
an isomorphism if we put $\alpha' = ([n\alpha_i]/n)_i$. 
\end{enumerate}
A parabolic $F$-vector bundle on $(\ol{\cX},\cZ)_K$ is a 
pair $((\cE_{\alpha})_{\alpha}, (\Psi_{\alpha})_{\alpha})$
consisting of 
a parabolic vector bundle $(\cE_{\alpha})_{\alpha}$ on 
$(\ol{\cX},\cZ)_K$ endowed with morphisms 
$\Psi_{\alpha}: F^*\cE_{\alpha} \lra \cE_{q\alpha}$ in 
the category of $\cO_{\ol{\cX}_K}$-modules 
such that $\varinjlim_{\alpha} \Psi_{\alpha}: 
\varinjlim_{\alpha} F^*\cE_{\alpha} \lra 
\varinjlim_{\alpha} \cE_{q\alpha}$ is isomorphic as 
ind-objects. \par 
For 
$\alpha := (\alpha_i)_i \in \Z_{(p)}^r$, let 
$\cO_{\ol{\cX}_K}(\sum_i\alpha_i\cZ_{i,K}) := 
(\cO_{\ol{\cX}}(\sum_i\alpha_i\cZ_i)_{\beta})_{\beta}$ 
be the parabolic vector bundle on $(\ol{\cX},\cZ)_K$ defined by 
$\cO_{\ol{\cX}}(\sum_i\alpha_i\cZ_{i,K})_{\beta} := 
\cO_{\ol{\cX}}(\sum_i[\alpha_i + \beta_i] \cZ_{i,K})$ (where 
$\beta = (\beta_i)_i$), and we say that a 
parabolic $F$-vector bundle 
$((\cE_{\alpha})_{\alpha}, (\Psi_{\alpha})_{\alpha})$
on $(\ol{\cX},\cZ)_K$ 
is locally abelian if there exists 
some positive integer $n$ prime to $p$ and an admissible 
covering $\{\fX_{\lambda}\}_{\lambda}$ of $\cX_{n,K} := 
(\cX \otimes_{O_K} O_K[\mu_n])_K$ such that 
$(\cE_{\alpha})_{\alpha} |_{\fX_{\lambda}}$ has the form 
$\bigoplus_{j=1}^{\mu}\cO_{\ol{\cX}_K}(\sum_i \alpha_{ij} \cZ_{i,K}) 
|_{\fX_{\lambda}}$ for some $\alpha_{ij} \in \Z_{(p)}$, for each 
$\lambda$. 
We denote the category of locally abelian 
parabolic $F$-vector bundles on $(\ol{\cX},\cZ)_K$ by 
$\PFVect((\ol{\cX},\cZ)_K)$. \par 
An object $((\cE_{\alpha})_{\alpha},(\Psi_{\alpha})_{\alpha})$ in 
$\PFVect((\ol{\cX},\cZ)_K)$
is called 
generically semistable (gss) if, for any dense open formal 
subscheme $\cU_{\circ} \subseteq \cX_{\circ}$, 
the image $(\cE|_{\cU_K}, \Psi|_{\cU_K})$ 
of $((\cE_{\alpha})_{\alpha},(\Psi_{\alpha})_{\alpha})$
by the restriction functor 
\begin{align*}
&  \PFVect((\ol{\cX},\cZ)_K) \lra 
\PFVect((\cU \times_{\ol{\cX}} (\ol{\cX},\cZ))_K) \os{=}{\lra} 
\FVect(\cU_K) \\ 
\end{align*}
(where $\cU_K := (\cU_{\circ} \otimes_{W(k)} O_K)_K$) 
admits no non-trivial saturated subobject $(\cE',\Psi')$ with 
$\mu(\cE') < \mu(\cE|_{\cU_K})$. (Note that $\mu(\cE|_{\cU_K})$ does not 
depend on $\cE$. Hence we denote it by $\mu(\cE)$.) 
We denote the full subcategory of 
$\PFVect((\ol{\cX},\cZ)_K)$
consisting of generically semistable 
objects $(\cE,\Psi)$ with $\mu(\cE)=0$ by 
$\PFVect((\ol{\cX},\cZ)_K)^{\gss,\mu=0}$. 
Then we have the following proposition. 

\begin{proposition}\label{gssmain3-3}
Let the notations be as above and assume that $\dim X=1$. 
Then we have the canonical equivalence of categories 
\begin{equation}\label{gss6}
\PFLatt(\ol{\cX},\cZ)^{\circ}_{\Q} \os{=}{\lra} 
\PFVect((\ol{\cX},\cZ)_K)^{\gss,\mu=0}.
\end{equation}
\end{proposition}

\begin{proof}
We can prove the equivalence 
\begin{equation}\label{haru}
\ba: \varinjlim_{(n,p)=1}\FVect((\ol{\cX},\cZ)^{1/n}_K) \os{=}{\lra} 
\PFVect((\ol{\cX},\cZ)_K)  
\end{equation}
in the same way as Theorem \ref{flattthm} and the generic 
semistabilities and the values of $\mu$ coincide via the above 
equivalence. 
(In the proof of Theorem \ref{flattthm}, when we are given an object 
$\cE$ 
in $\Vect((\ol{\cX},\cZ)^{1/n})$, an open affine $\cU \subseteq \cX$ and 
a closed point $x$ of $\cU$, we constructed an open formal subscheme 
$\cU_x = \{f_x \not= 0\}$ of $\cU$ containing $x$ on which 
$\ba(\cE)$ has a simple shape. Here, for an object 
$\cE$ 
in $\Vect((\ol{\cX},\cZ)^{1/n}_K)$, an open affine $\cU \subseteq \cX$ and 
a point $x$ of $\cU_K$, we can construct in the same way 
an open rigid subspace 
$\fU_x = \{f_x \not= 0\}$ of $\cU_K$ containing $x$ on which 
$\ba(\cE)$ has a simple shape, and we see that the covering 
$\cX_K = \bigcup_{\cU: \text{affine}} \bigcup_{x \in \cU_K} \fU_x$ 
is an admissible covering.) 
Then the desired equivalence 
follows from this, \eqref{gss3-3} and Theorem \ref{flattthm}. 
\end{proof}

\begin{remark}
\begin{enumerate}
\item 
As we see from the above proof, 
the equivalence \eqref{haru} holds for $X$ of any dimension. 
\item 
When $\dim X=1$, any parabolic $F$-vector bundle on $(\ol{\cX},\cZ)_K$ 
is locally abelian. Hence we can drop the condition of locally abelianness 
from the definition of $\PFVect((\ol{\cX},\cZ)_K)$ when $\dim X=1$. 
\end{enumerate}
\end{remark}

As an immediate consequence of Propositions \ref{gssmain3-1}, 
\ref{gssmain3-2} and \ref{gssmain3-3}, we obtain the following, 
which is a $p$-adic analogues ($F$-vector bundle 
version) of \eqref{eq1} and 
\eqref{eq2} and \eqref{eq4} which includes the notion of stability and 
in which neither isocrystal structure and lattice structure appear, 
in the case of curves with strong liftability condition: 

\begin{corollary}\label{stability3}
Let the notations be as above and assume that $\dim X=1$. 
Then we have the equivalences 
\begin{align*}
& \Rep_{K^{\sigma}}(\pi_1^t(X)) \os{=}{\lra} 
\varinjlim_{Y \ra X \in \cG^t_X} 
\FVect([\ol{\cY}/G_Y]_K)^{\gss,\mu=0}. \\ 
& \Rep_{K^{\sigma}}(\pi_1^t(X)) \os{=}{\lra} 
\varinjlim_{(n,p)=1} 
\FVect((\ol{\cX},\cZ)^{1/n}_K)^{\gss, \mu=0}. \\ 
& \Rep_{K^{\sigma}}(\pi_1^t(X)) \os{=}{\lra} 
\PFVect((\ol{\cX},\cZ)_K)^{\gss,\mu=0}. 
\end{align*}
\end{corollary}

This is a precise form of the micro reciprocity law 
conjectued in \cite[48.6, 49.3]{weng}.


\begin{thebibliography}{99}

\bibitem{berthelot} 
  P. ~Berthelot, 
  {\it Cohomologie rigide et cohomologie rigide \`{a} supports propres
   \,\,\, premi\`{e}re partie}, pr\'{e}publication de 
   l'IRMAR 96-03.  
   Available at http://perso.univ-rennes1.fr/pierre.berthelot/

\bibitem{biswas}
  I. ~Biswas, 
  {\it Parabolic bundles as orbifold bundles},
  Duke Math. J. {\bf 88}(1997), 305--325. 

\bibitem{borne1}
  N. ~Borne, 
{\it Fibr\'es paraboliques et champ des racines}, 
Int. Math. Res. Not. IMRN  2007, no. 16, Art. ID rnm049, 38 pp. 

\bibitem{borne2}
 N. ~Borne, 
{\it Sur les repr\'esentations du groupe fondamental d'une vari\'et\'e 
priv\'ee d'un diviseur \`a croisements normaux simples}, 
Indiana Univ. Math. J. {\bf 58}(2009), 137--180. 

\bibitem{cadman}
 C. ~Cadman, 
{\it Using stacks to impose tangency conditions on curves}, 
Amer. J. of Math. {\bf 129}(2007), 405--427. 

\bibitem{caro}
  D. ~Caro, {\it Pleine fid\'elit\'e sans structure de Frobenius et 
isocristaux partiellement surconvergents}, to appear in Math. Ann. 

\bibitem{corlette}
 K. Corlette, {\it Flat G-bundles with canonical metrics}, 
J. Differential Geom. {\bf 28}(1988), 361--382.

\bibitem{crewsp}
 R. ~Crew, {\it Specialization of crystalline cohomology}, 
 Duke Math. J. {\bf 53}(1986), 749--757. 

\bibitem{crew}
  R. ~Crew, {\it $F$-isocrystals and $p$-adic representations}, in 
 Algebraic geometry, Bowdoin, 1985 (Brunswick, Maine, 1985), 111--138,
Proc. Sympos. Pure Math., 46, Part 2, Amer. Math. Soc., Providence, RI, 1987. 

\bibitem{deligne}
   P. ~Deligne, 
  {\it Equations diff\'erentielles \`a points singuliers r\'eguliers}, 
  Lecture Note in Math. {\bf 163}, Springer-Verlag, 1970. 

\bibitem{donaldson}
 S. ~K. ~Donaldson, {\it 
Infinite determinants, stable bundles and curvature}, Duke Math.
J. {\bf 54}(1987), 231--247.

\bibitem{glsq}
M. ~Gros, B. ~Le Stum and A. ~Quir\'os, 
{\it A Simpson correspondance in positive characteristic}, 
Publ. Res. Inst. Math. Sci. {\bf 46}(2010), 1--35.

\bibitem{illusie}
 L. ~Illusie, 
{\it An overview of the work of K. Fujiwara, K. Kato, and C. Nakayama on 
logarithmic \'etale cohomology}, in 
Cohomologies $p$-adiques et applications arithm\'etiques II, 
Ast\'erisque  {\bf 279}(2002), 271--322. 

\bibitem{is}
 J. ~N. ~N. ~Iyer and C. ~T. ~Simpson, 
{\it A relation between the parabolic Chern characters of 
the de Rham bundles}, Math. Ann. {\bf 338}(2007), 347--383.

\bibitem{jz}
J. ~Jost and K. ~Zuo, 
{\it Harmonic maps of infinite energy and rigidity results for
representations of fundamental groups of quasiprojective varieties}, 
J. Differential Geom. {\bf 47}(1997), 469--503.

\bibitem{kato}
 K. ~Kato, 
 {\it Logarithmic structures of Fontaine-Illusie}, in 
Algebraic analysis, geometry, and number theory (Baltimore, MD, 1988), 
191--224, Johns Hopkins Univ. Press, Baltimore, MD, 1989.

\bibitem{katzcrew}
N. ~M. ~Katz, 
{\it 
$p$-adic properties of modular schemes and modular forms}, in 
Modular functions of one variable, III, 
Lecture Notes in Mathematics {\bf 350}(1973), 69--190. 

\bibitem{katzslope}
N. ~M. ~Katz, 
{\it Slope filtration of $F$-crystals}, Ast\'erisque {\bf 63}(1979), 
113--164. 

\bibitem{katzpi1}
N. ~M. ~Katz, 
{\it 
Local-to-global extensions of representations of fundamental groups}, 
Ann. Inst. Fourier {\bf 36}(1986), 69--106. 

\bibitem{kedlayaff}
   K. ~S. ~Kedlaya, {\it Full faithfulness for overconvergent F-isocrystals}, 
in Geometric Aspects of Dwork Theory (Volume II) 819-835, 
de Gruyter, Berlin, 2004. 

\bibitem{kedlayaI}
   K. ~S. ~Kedlaya, 
   {\it Semistable reduction for overconvergent $F$-isocrystals, I$:$ 
   Unipotence and logarithmic extensions}, 
   Compositio Math.. {\bf 143}(2007), 1164--1212. 

\bibitem{kedlayaII}
   K. ~S. ~Kedlaya, 
   {\it Semistable reduction for overconvergent $F$-isocrystals, II$:$ 
   A valuation-theoretic approach}, Compositio Math. 
  {\bf 144}(2008), 657-672. 


\bibitem{kedlayaswanII}
  K. ~S. ~Kedlaya, {\it Swan conductors for p-adic differential modules, II: Global variation}, to appear in Journal de l'Institut de Mathematiques 
de Jussieu. 

\bibitem{lestum}
 B. ~Le Stum, 
 {\it Rigid cohomology}, Cambridge Tracts in Mathematics 172, 
Cambridge University Press, Cambridge, 2007.  

\bibitem{my}
  M. ~Maruyama and K. ~Yokogawa, 
{\it Moduli of parabolic stable sheaves}, 
Math. Ann. {\bf 293}(1992), 77--99. 

\bibitem{mr}
  V. ~B. ~Mehta and A. ~Ramanathan, 
{\it Restriction of stable sheaves and representations
of the fundamental group}, Invent. Math. {\bf 77}(1984), 163--172.

\bibitem{ms}
  V. ~B. ~Mehta and C. ~S. ~Seshadri, 
  {\it Moduli of vector bundles on curves with parabolic structures}, 
  Math. Ann. {\bf 248}(1980), 205--239. 

\bibitem{mochizuki1}
  T. ~Mochizuki, 
 {\it Kobayashi-Hitchin correspondence for tame harmonic bundles 
and an application}, Ast\'erisque {\bf 309}(2006). 

\bibitem{mochizuki1.5}
 T. ~Mochizuki, 
 {\it Asymptotic behaviour of tame harmonic bundles and an application to pure twistor $D$-modules. I, II}, 
 Mem. Amer. Math. Soc, {\bf 185}(2007), no. 869, 870.  

\bibitem{mochizuki2}
  T. ~Mochizuki, 
  {\it Kobayashi-Hitchin correspondence for tame harmonic bundles II},   
Geom. Topol. {\bf 13}(2009), 359--455. 

\bibitem{nakayama}
  C. ~Nakayama, {\it Logarithmic \'etale cohomology}, 
Math. Ann. {\bf 308}(1997), 365--404. 

\bibitem{ns}
  M. ~S. ~Narasimhan and C. ~S. ~Seshadri, 
  {\it Stable and unitary vector bundles on a compact Riemann surface}, 
Ann. of Math. {\bf 82}(1965), 540--567. 

\bibitem{niziol}
  W. ~Niziol, 
 {\it Toric singularities: log-blow-ups and global resolutions}, 
 J. Alg. Geom. {\bf 15}(2006), 1--29. 

\bibitem{olsson}
  M. ~C. ~Olsson, 
 {\it Logarithmic geometry and algebraic stacks}, 
Ann. Sci. \'Ecole Norm. Sup. {\bf 36}(2003), 747--791. 

\bibitem{ogus}
  A. ~Ogus, {\it $F$-isocrystals and de Rham cohomology II. 
Convergent isocrystals}, Duke Math. J. {\bf 51}(1984), 765--850. 

\bibitem{ov}
 A. ~Ogus and V. ~Vologodsky, 
 {\it Nonabelian Hodge theory in characteristic $p$}, 
Publ. Math. IHES {\bf 106}(2007), 1--138. 



\bibitem{seshadri}
  C. ~S. ~Seshadri, 
{\it Moduli of $\pi$-vector bundles over an algebraic curve}, in 
Questions on Algebraic Varieties (C.I.M.E., III Ciclo, Varenna, 1969)  pp. 139--260 Edizioni Cremonese, Rome, 1970. 

\bibitem{schepler}
 D. ~Schepler, 
{\it Logarithmic nonabelian Hodge theory in characteristic $p$}, 
thesis, 2005. 

\bibitem{simpson}
  C. ~T. ~Simpson, 
 {\it Higgs bundles and local systmes}, Publ. IHES {\bf 75}(1992), 5--95.

\bibitem{simpsonopen}
 C. ~T. ~Simpson, 
 {\it Harmonic bundles on noncompact curves}, 
J. Amer. Math. Soc. 3 (1990), 713--770. 

\bibitem{crysI}
A. ~Shiho, 
{\it Crystalline fundamental groups I --- 
Isocrystals on log crystalline site and log convergent site}, 
J. Math. Sci. Univ. Tokyo  {\bf 7}(2000), 509--656. 

\bibitem{crysII}
A. ~Shiho, 
{\it Crystalline fundamental groups II --- 
Log convergent cohomology and rigid cohomology}, 
J. Math. Sci. Univ. Tokyo  {\bf 9}(2002),  no. 1, 1--163. 

\bibitem{cech}
A. ~Shiho, {\it A note on log etale \v{C}ech cohomology}, 
Manuscripta Math. {\bf 103}(2000), 363--391. 

\bibitem{sigma}
A. ~Shiho, {\it 
On logarithmic extension of overconvergent isocrystals}, 
Math Ann. {\bf 348}(2010), 467-512

\bibitem{curvecut}
  A. ~Shiho, 
  {\it Cut-by-curves criterion for the log extendability of 
   overconvegent isocrystals}, to appear in Math. Z. 

\bibitem{relativeI}
A. ~Shiho, {\it Relative log convergent cohomology and relative 
rigid cohomology I}, arXiv:0707.1742v2. 

\bibitem{purity}
A. ~Shiho, {\it Purity for overconvergence}, 
arXiv:1007.3345v1. 

\bibitem{tsuzukicurve}
  N. ~Tsuzuki, {\it Finite local monodromy of overconvergent 
unit-root $F$-isocrystals on a curve}, Amer. J. Math. 
{\bf 120}(1998), 1165--1190. 

\bibitem{tsuzuki}
  N. ~Tsuzuki, 
  {\it Morphisms of $F$-isocrystals and the finite 
  monodromy theoerem for unit-root $F$-isocrystals}, Duke Math. J. 
  {\bf 111}(2002), 385--419. 

\bibitem{uy}
  K. ~Uhlenbeck and S. ~T. ~Yau, 
 {\it On the existence of Hermitian Yang-Mills connections
in stable bundles}, Comm. Pure Appl. Math., {\bf 39-S}(1986), 257--293.

\bibitem{weng}
 L. ~Weng, {\it Stability and arithmetic}, Adv. Study in Pure Math. 
 {\bf 58}(2010), 225--359. 

\end{thebibliography}
\end{document}